\documentclass{article}

\usepackage{microtype}
\usepackage{graphicx}
\usepackage{subcaption}
\usepackage{booktabs} %
\usepackage{authblk}
\usepackage{float}
\usepackage{fullpage}

\usepackage{mathrsfs} 
\usepackage{eufrak,euscript}
\usepackage{xcolor}
\usepackage[shortlabels]{enumitem}
\usepackage{bm}
\usepackage{charter}
\definecolor{labelkey}{rgb}{0,0.08,0.45}
\definecolor{refkey}{rgb}{0,0.6,0.0}
\definecolor{Brown}{rgb}{0.45,0.0,0.05}
\definecolor{dgreen}{rgb}{0.00,0.49,0.00}
\definecolor{dblue}{rgb}{0,0.08,0.75}
\RequirePackage[colorlinks,hyperindex]{hyperref} %
\hypersetup{linktocpage=true,citecolor=dblue,linkcolor=dgreen}
\usepackage{xurl}
\usepackage{hyperref}

\usepackage{amsmath}
\usepackage{amssymb}
\usepackage{mathtools}
\usepackage{amsthm}

\usepackage{algorithm,algorithmic}

\def\argmin{\operatornamewithlimits{arg\,min}}
\providecommand{\scalT}[2]{\left\langle{#1},{#2}\right\rangle}

\usepackage[capitalize,noabbrev]{cleveref}

\theoremstyle{plain}
\newtheorem{theorem}{Theorem}[section]
\newtheorem{proposition}[theorem]{Proposition}
\newtheorem{lemma}[theorem]{Lemma}
\newtheorem{corollary}[theorem]{Corollary}
\theoremstyle{definition}

\newtheorem{assumption}[theorem]{Assumption}
\theoremstyle{remark}
\newtheorem{remark}[theorem]{Remark}

\usepackage[textsize=tiny]{todonotes}

\author[1]{Marco Rando}
\author[2]{Samuel Vaiter}

\affil[1]{Universit\'e C\^ote d'Azur, Inria, CNRS, LJAD, Nice, France}
\affil[2]{CNRS \& Universit\'e C\^ote d'Azur, LJAD, Nice, France}

\title{ZOBA: An Efficient Single-loop Zeroth-order Bilevel Optimization Algorithm}

\date{}

\begin{document}

\maketitle

\begin{abstract}
\noindent Bilevel optimization problems consist of minimizing a value function whose evaluation depends on the solution of an inner optimization problem. These problems are typically tackled using first-order methods that require computing the gradient of the value function ({\it the hypergradient}). In several practical settings, however, first-order information is unavailable ({\it zeroth-order setting}), rendering these methods inapplicable. Finite-difference methods provide an alternative by approximating hypergradients using function evaluations along a set of directions. Nevertheless, such surrogates are notoriously expensive, and existing finite-difference bilevel methods rely on two-loop algorithms that are poorly parallelizable. In this work, we propose ZOBA, the first finite-difference single-loop algorithm for bilevel optimization. Our method leverages finite-difference hypergradient approximations based on delayed information to eliminate the need for nested loops. We analyze the proposed algorithm and establish convergence rates in the non-convex setting, achieving a complexity of $\mathcal{O}(p(d + p)^2\varepsilon^{-2})$, where $p$ and $d$ denote the dimension of inner and outer spaces respectively, which is better than prior approaches based on Hessian approximation. We further introduce and analyze HF-ZOBA, a Hessian-free variant that yields additional complexity improvements. Finally, we corroborate our findings with numerical experiments on synthetic functions and a real-world black-box task in adversarial machine learning. Our results show that our methods achieve accuracy comparable to state-of-the-art techniques while requiring less computation time.

\end{abstract}

\noindent {{\bf Keywords:} Zeroth-order Optimization, Black-box Optimization, Bilevel Optimization, Stochastic Optimization.}\\

\noindent {{\bf AMS Mathematics Subject Classification:} 90C56, 90C46, 90C15, 90C26, 90C30.}

\section{Introduction}

Bilevel optimization problems involve minimizing an outer objective ({\it the value function}) whose evaluation depends on the solution of an inner optimization problem. This class of problems arises naturally in a wide range of machine learning applications, including hyperparameter optimization \cite{bengio_gb_hyp,pedregosa16_hyperopt}, meta-learning \cite{franceschi18a_bilevel_hypopt}, representation learning \cite{ijcai2024p230}, and many others.
These problems are typically addressed using first-order algorithms that aim to compute the gradient of the value function, {\it the hypergradient} \cite{ghadimi2018approximation,ji_bilevel,chen_bilevel,soba_saba,arbel_mairal}.  In practice, computing the hypergradient exactly is prohibitively expensive or infeasible, since at each iteration it requires the action of the Hessian (and its inverse) of the inner objective, as well as a minimizer of the inner problem. For this reason, first-order bilevel methods approximate the hypergradient by relying on an approximate inner solution and inverse-Hessian actions obtained through iterative procedures. This naturally leads to a two-loop algorithmic structure, where inner loops construct these approximations by exploiting gradients and, in some cases, Hessians of the inner and outer objectives. 
However, in many practical scenarios, such information is either unavailable or prohibitively expensive to compute, and only function evaluations of the inner and outer objectives are accessible. Such problems are referred to as black-box bilevel optimization problems and are particularly relevant, as they arise in many real-world applications such as meta-training, complex control problem, fine-tuning large language models (LLMs) and many others  - see, e.g., \cite{Sapre2024,review_evo_bilevel,pujara2025reviewbileveloptimizationmethods,shirkavandbilevel,evo_for_control,Gomez_Ortega2025,conn2012bilevel}. To tackle black-box bilevel problems, several classes of algorithms have been proposed \cite{Talbi2013,metaheuristic_for_bilevel,conn2012bilevel,Aghasi2025}, and finite-difference is one of these. 
Finite-difference methods are iterative procedures that mimic first-order strategies by replacing gradients with surrogates constructed from function evaluations along a set of directions \cite{nesterov_random_2017,ghadimi_zeroth,ozd,sto_md,zoro}. Designing efficient finite-difference methods for bilevel optimization is particularly challenging.
Approximating all first- and second-order quantities involved in the hypergradient computations with finite-difference is extremely expensive, especially in high-dimensional settings, where a large number of function evaluations is required to obtain reliable gradient and Hessian surrogates. Existing finite-difference methods for black-box bilevel optimization \cite{Aghasi2025,aghasi2025optimalzerothorderbileveloptimization} typically mimic the first-order two-loop structure and attempt to mitigate the resulting computational cost by using computationally cheap but inaccurate gradient and Hessian approximations in the inner loops. To control the resulting approximation errors, these methods have to employ small step sizes and perform many inner-loop iterations, leading to significant computational overhead and poor parallelization in practice. More recent approaches \cite{aghasi2025optimalzerothorderbileveloptimization} aim to reduce this burden by regularization-based reformulations that avoid explicit Hessian approximation. 
While reducing per-iteration cost, these methods introduce bias into the bilevel objective and require careful tuning of regularization parameters, which may negatively affect solution accuracy. Moreover, the underlying two-loop structure is preserved, thereby retaining the same limitations in terms of parallelization. Recent advances in first-order bilevel optimization have shown that the two-loop structure can be avoided by using single-loop algorithms that rely on approximate hypergradients computed using delayed information \cite{soba_saba}. These methods offer substantial practical advantages, as they allow all quantities involved in hypergradient computation to be evaluated in parallel and significantly reduce per-iteration computational cost while preserving convergence guarantees. However, extending these ideas efficiently to black-box setting is far from straightforward. 
In particular, it is not evident that mimicking a single-loop scheme can achieve meaningful computational savings while sufficiently controlling finite-difference errors to preserve convergence and get reasonable complexity results. As a result, to the best of our knowledge, no single-loop finite-difference method for bilevel optimization has been previously proposed. 

\noindent In this work, we propose ZOBA, the first single-loop finite-difference algorithm for black-box bilevel optimization. Our method relies on delayed information and enables full parallelization in the approximation of the hypergradient. Moreover, ZOBA tackles the approximation cost per iteration by reusing the same function evaluations to compute multiple components of the hypergradient, thereby significantly reducing the number of required function evaluations. We analyze the proposed algorithm and derive convergence guarantees for non-convex bilevel problems. 
In particular, we show that ZOBA achieves a complexity of $\mathcal{O}(p(d+p)^2 \varepsilon^{-2})$, for an accuracy $\varepsilon \in (0,1)$ and where $p$ and $d$ denote the dimensions of the inner and outer variables, respectively, improving upon existing zeroth-order bilevel methods that rely on explicit Hessian approximations. We further introduce and analyze HF-ZOBA, a Hessian-free variant that entirely avoids explicit Hessian approximations and does not rely on regularization-based strategies, yielding additional improvements in the complexity results. Finally, we corroborate our theoretical findings through experiments on synthetic objectives and on the real-world task of computing minimal-distortion universal adversarial perturbations in black-box settings. The numerical results show that our methods, despite the additional errors introduced by delayed information and finite-difference surrogates, achieve solution quality comparable to state-of-the-art finite-difference algorithms for bilevel optimization while providing significant improvements in runtime.

\noindent The paper is organized as follows. In Section~\ref{sec:problem_setting}, we formally define the problem, introduce our algorithms, and discuss their relationship with related work. In Section~\ref{sec:main_results}, we state and discuss the main theoretical results. Section~\ref{sec:experiments} presents numerical experiments, and Section~\ref{sec:conclusion} concludes the paper with final remarks.

\section{Problem Setting \& Algorithm}\label{sec:problem_setting}

Let $(\Omega_1,\mathcal{F}_1,\mathbb{P}_1)$ and $(\Omega_2,\mathcal{F}_2,\mathbb{P}_2)$ be probability spaces, and let $\mathcal{Z}_1,\mathcal{Z}_2$ be measurable spaces. Let $\zeta:\Omega_1\to\mathcal{Z}_1$, $\xi:\Omega_2\to\mathcal{Z}_2$ be two random variables and let $f : \mathbb{R}^p \times\mathbb{R}^d \times \mathcal{Z}_1 \to \mathbb{R}$ and $g : \mathbb{R}^p \times\mathbb{R}^d \times \mathcal{Z}_2 \to \mathbb{R}$ be two functions. Consider the minimization of a function $\Psi : \mathbb{R}^d \to \mathbb{R}$ defined as 
\begin{equation}\label{eqn:problem}
\begin{aligned}
    \min\limits_{x\in\mathbb{R}^d} \Psi(x) := F(z^*(x), x) &= \mathbb{E}_\zeta[f(z^*(x), x, \zeta)]\\
\text{s.t.}\\
z^*(x) \in \argmin\limits_{z \in \mathbb{R}^p} G(z, x) &= \mathbb{E}_{\xi}[g(z,x,\xi)].
\end{aligned}
\end{equation}
Under standard regularity assumptions \cite{ghadimi2018approximation}, the function $\Psi$ is differentiable, and its gradient (the \emph{hypergradient}) admits the following representation
\begin{equation}\label{eqn:hypergrad}
\nabla \Psi(x) = \nabla_x F(z^*(x),x) + \nabla_{xz}^2 G(z^*(x),x), v^*(x),
\end{equation}
where
\begin{equation}\label{eqn:v_star}
v^*(x)= -\left[\nabla_{zz}^2 G(z^*(x),x)\right]^{-1}\nabla_z F(z^*(x),x).
\end{equation}
Here, $\nabla_x$ and $\nabla_z$ denote gradients with respect to $x$ and $z$, respectively, while $\nabla_{xz}^2$ and $\nabla_{zz}^2$ denote the corresponding Hessian blocks. We consider the \emph{stochastic zeroth-order} setting, in which no first- or higher-order information of $f$ and $g$ is available. Moreover, exact evaluations of $F$ and $G$ are inaccessible, and only noisy evaluations of $f$ and $g$ can be obtained. %
To address Problem~\eqref{eqn:problem}, we propose a stochastic zeroth-order algorithm, namely, an iterative procedure that relies solely on evaluations of $f$ and $g$. At each iteration $k$, our method constructs a hypergradient surrogate at the current iterate $x_k$ by approximating gradient and Hessian terms with mini-batch finite differences and using auxiliary sequences $z_k$ and $v_k$ to track $z^*(x_k)$ and $v^*(x_k)$. We next describe how these sequences are built ({\bf auxiliary sequence construction}) and how the hypergradient is approximated ({\bf hypergradient approximation}).

\paragraph*{Auxiliary sequence construction.} %
At every iteration $k\in \mathbb{N}$, %
let $(\xi_{i,k})_{i=1}^{b_1}$ be realizations of independent copies of the random variable $\xi$. For each $i \in [b_1] := \{1,\ldots,b_1\}$, let $\big(w_k^{(i,j)}\big)_{j=1}^{\ell_1} \subset \mathbb{R}^p$ be vectors sampled i.i.d. from $\mathcal{N}(0,I_p)$. For $h_k > 0$, let $g^{+,k}_{i,j} := g(z_k + h_k w_k^{(i,j)}, x_k, \xi_{i,k})$ and $g^{-,k}_{i,j} := g(z_k - h_k w_k^{(i,j)}, x_k, \xi_{i,k})$. Our algorithm approximates $\nabla_z G(z_k,x_k)$ with the following %
surrogate %
\begin{equation}\label{eqn:dz}
D_{z}^k := \frac{1}{b_1 \ell_1} \sum_{i=1}^{b_1} \sum_{j=1}^{\ell_1}
\frac{g_{i,j}^{+,k} - g_{i,j}^{-,k}}{2h_k} w_k^{(i,j)}.
\end{equation}
Let $\rho_k > 0$ be a stepsize, we define the sequence $(z_k)_{k\geq0}$ with the following iteration
\begin{equation}\label{eqn:z_update}
    z_{k +1} = z_k - \rho_k D_{z}^k.
\end{equation}
Then, denoting with
$g_{i}^k := g(z_k, x_k, \xi_{i,k})$ and  $c_{i,j}^k := (g^{+,k}_{i,j} + g^{-,k}_{i,j} - 2 g_i^k)/(2h_k^2)$, %
our method builds the surrogate of $\nabla_{zz}^{2}G(z_k,x_k)$ as follows
\begin{equation*}
    \begin{aligned}
        \hat{\nabla}_{zz}^2g_k := \frac{1}{b_1 \ell_1} \sum\limits_{i=1}^{b_1} \sum\limits_{j=1}^{\ell_1} c_{i,j}^k\left( w_k^{(i,j) }w_k^{(i,j)\top} - I_p \right).
    \end{aligned}
\end{equation*}
Notice that to compute this quantity we need to perform just $b_1$ evaluations since $g^{+,k}_{i,j}$ and $g^{-,k}_{i,j}$ are already computed for eq. \eqref{eqn:z_update}. 
Let $b_2,\ell_2 > 0$ and let $(\zeta_{i,k})_{i=1}^{b_2}$ be realizations of independent copies of $\zeta$. Let $f_{i,j}^{z,k,+} = f(z_k +h_k w_k^{(i,j)}, x_k, \zeta_{i,k})$ and $f_{i}^{k} := f(z_k, x_k, \zeta_{i,k})$. Our algorithm approximates $\nabla_z F(z_k,x_k)$ as
\begin{equation*}
    \begin{aligned}
        \hat{\nabla}_z f_k := \frac{1}{b_2 \ell_2}\sum\limits_{i=1}^{b_2} \sum\limits_{j=1}^{\ell_2} \frac{f_{i,j}^{z,k, +} - f_{i}^{k}}{h_k} w_k^{(i,j)}.
    \end{aligned}
\end{equation*}
Then, our procedure build the following search direction
\begin{equation}\label{eqn:dv}
    \begin{aligned}
        D_v^k := \hat{\nabla}_{zz}^2 g_kv_k + \hat{\nabla}_z f_k.
    \end{aligned}
\end{equation}
Thus, we construct the next auxiliary iterate $v_{k+1}$ as%
\begin{equation}\label{eqn:v_update}
    v_{k + 1} = v_k - \rho_k D_v^k.
\end{equation}
\paragraph*{Hypergradient approximation.} 
For $i =1,\cdots,b_1$ and $j=1,\cdots,\ell_1$, let $\hat{g}_{i,j}^{+,k} := g(z_k + h_k w_k^{(i,j)},x_k + h_k u_k^{(i,j)}, \xi_{i,k})$, $\hat{g}_{i,j}^{-,k} = g(z_k - h_k w_k^{(i,j)}, x_k - h_k u_k^{(i,j)}, \xi_{i,k})$ and $s_{i,j}^k = (\hat{g}_{i,j}^{+,k} + \hat{g}_{i,j}^{-,k} -2 g_i^k)/(2h_k^2)$. Our method approximates $\nabla_{xz}^2 G(z_k,x_k)$ as 
\begin{equation*}
    \begin{aligned}
        \hat{\nabla}_{xz}^2 g_k :=  \frac{1}{b_1 \ell_1} \sum\limits_{i=1}^{b_1} \sum\limits_{j=1}^{\ell_1} s_{i,j}^k u_k^{(i,j)}w_k^{(i,j)\top}.
    \end{aligned}
\end{equation*}
Notice that such surrogate requires only $2b_1\ell_1$ function evaluations since $(g_i^k)_{i=1}^{b_1}$ are already computed for previous quantities. Let $f_{i,j}^{x,k,+} = f(z_k, x_k + h_k u_k^{(i,j)}, \zeta_{i,k})$, consider the following approximation of $\nabla_x F$
\begin{equation*}
    \hat{\nabla}_x f_k := \frac{1}{\ell_2 b_2} \sum\limits_{i=1}^{b_2} \sum\limits_{j=1}^{\ell_2} \frac{f_{i,j}^{x,k,+} - f_i^k}{h_k}u_k^{(i,j)}.
\end{equation*}
Notice that since $(f_i^k)_{i=1}^{b_2}$ are already computed for eq. \eqref{eqn:dv}, such approximation requires $b_2\ell_2$ function evaluations. The hypergradient surrogate is thus built as follows
\begin{equation}\label{eqn:dx}
    \begin{aligned}
        D_x^k := \hat{\nabla}_{xz}^2 g_kv_k + \hat{\nabla}_x f_k.
    \end{aligned}
\end{equation}
Finally, the iterate is updated with the following iteration
\begin{equation}\label{eqn:x_update}
    x_{k + 1} = x_k - \gamma_k D_x^k.
\end{equation}
We summarize the entire procedure in Algorithm \ref{alg:zoba}.

\begin{algorithm}[H]
\caption{ZOBA: Zeroth-Order Bilevel Algorithm}
\label{alg:zoba}
\begin{algorithmic}[1]

\STATE \textbf{Input:} $x_0\in\mathbb{R}^d$, $z_0,v_0\in\mathbb{R}^p$, stepsizes $(\gamma_k)$, $(\rho_k)$, discretization parameter $h_k$, minibatch sizes $b_1,b_2$, number of directions $\ell_1,\ell_2$.
\STATE Let $b = \max(b_1,b_2)$ and $\ell = \max(\ell_1,\ell_2)$.

\FOR{$k=0,1,2,\ldots$}
\STATE Sample i.i.d.\ $\xi_{1,k},\ldots,\xi_{b,k}$ and $\zeta_{1,k},\ldots,\zeta_{b,k}$.

\STATE For each $i\in[b]$, sample $(w_{k}^{(i,j)})_{j=1}^\ell \subset\mathbb{R}^{p}$ and $(u_{k}^{(i,j)})_{j=1}^\ell \subset \mathbb{R}^d$ from $\mathcal{N}(0,I_p)$ and $\mathcal{N}(0,I_d)$.

\STATE Compute 
 \begin{equation*}
    \begin{aligned}
     z_{k + 1} &= z_k - \rho_k D_{z}^k(z_k,x_k)\\
    v_{k + 1} &= v_k - \rho_kD_{v}^k(z_k,x_k)\\
    x_{k + 1} &= x_k - \gamma_k D_{x}^k(z_k,x_k)
    \end{aligned}
 \end{equation*}

\ENDFOR

\end{algorithmic}
\end{algorithm}
\noindent 
Given initial guesses $z_0, v_0 \in \mathbb{R}^p$ and $x_0 \in \mathbb{R}^d$, at each iteration $k \in \mathbb{N}$, the algorithm gets two mini-batches of $b=\max(b_1,b_2)$ realizations $(\xi_{i,k})_{i=1}^b$ and $(\zeta_{i,k})_{i=1}^b$ of independent copies of $\xi$ and $\zeta$. For each $i \in [b]$, it then samples $\ell=\max(\ell_1,\ell_2)$ i.i.d. Gaussian directions $(w_k^{(i,j)})_{j=1}^{\ell}$ and $(u_k^{(i,j)})_{j=1}^{\ell}$ from $\mathcal{N}(0,I_p)$ and $\mathcal{N}(0,I_d)$. Using these samples, the algorithm constructs the search directions $D_z^k$, $D_v^k$, and $D_x^k$ defined in eqs.~\eqref{eqn:dz}, \eqref{eqn:dv}, and \eqref{eqn:dx}, and updates the sequences $z_k$, $v_k$, and $x_k$ accordingly. The update rules are guided by the following intuitions.
i) Eq. \eqref{eqn:z_update} corresponds to a single step of a mini-batch finite-difference method applied to the inner objective $z \mapsto G(z,x_k)$, yielding a one-step approximation of $z^*(x_k)$ (under assumptions, such as convexity in $z$). ii) Eq. \eqref{eqn:v_update} is a single mini-batch finite-difference step applied to the quadratic objective associated to the linear system in eq.~\eqref{eqn:v_star}, using $z_k$ in place of $z^*(x_k)$. iii) Finally, the iterate $x_k$ is updated using the resulting hypergradient surrogate.
Notice that the updates of the sequences are performed in parallel. This is due the use of delayed information. In particular, $v_{k+1}$ is computed using $z_k$ instead of $z_{k+1}$; since $z_k$ depends on $x_{k-1}$, this can be interpreted as a one-step delay as $z_k$ would be an approximation of $z^*(x_{k-1})$. While this delay induces additional error, it enables reuse the function evaluations employed in $D_z^k$ to estimate the Hessian block $\nabla_{zz}^2 G(z_k,x_k)$, reducing its cost from $b_1(2\ell_1+1)$ to $b_1$ evaluations. A similar delay appears in the outer update, where $x_{k+1}$ is computed using $z_k$ and $v_k$ instead of their updated counterparts, with $v_k$ itself being affected by a delayed update. While this further propagates approximation error, it enables reuse of function evaluations in the computation of $D_x^k$. As a result, the total number of function evaluations per iteration is $b_1(4\ell_1+1)+b_2(2\ell_2+1)$. 

\paragraph*{Hessian-free variant.} 
Algorithm~\ref{alg:zoba} constructs hypergradient surrogates by explicitly approximating the Hessian blocks in eq.~\eqref{eqn:hypergrad} via finite differences. Since accurate finite difference Hessian approximations requires a large number of function evaluations~\cite{aghasi2025optimalzerothorderbileveloptimization}, we introduce a Hessian-free variant that avoids explicit Hessian approximation. Instead, we approximate Hessian–vector products using first-order finite differences and estimate the resulting gradients with finite-differences. Let $\hat{g}^{+,z,k}_{i,j}(z,x) = g(z +h_kw_k^{(i,j)},x,\xi_{i,k})$, $\hat{g}^{+,x,k}_{i,j}(z,x) = g(z ,x+h_ku_k^{(i,j)},\xi_{i,k})$ and $\hat{g}^{k}_{i}(z,x) = g(z ,x,\xi_{i,k})$. We define the following minibatch gradient approximation in variable $z$ and $x$ for arbirtrary $z,x$ as
\begin{equation*}
    \hat{\nabla}_zg_k(z,x) := \frac{1}{b_1 \ell_1} \sum\limits_{i=1}^{b_1}\sum\limits_{j=1}^{\ell_1} \frac{\hat{g}^{+,z,k}_{i,j}(z,x) - \hat{g}^{k}_{i}(z,x)}{h_k}w_k^{(i,j)}.
\end{equation*}
\begin{equation*}
    \hat{\nabla}_xg_k(z,x) := \frac{1}{b_1 \ell_1} \sum\limits_{i=1}^{b_1}\sum\limits_{j=1}^{\ell_1} \frac{\hat{g}^{+,x,k}_{i,j}(z,x) - \hat{g}^{k}_{i}(z,x)}{h_k}u_k^{(i,j)}.
\end{equation*}
Let $\bar{h} > 0$, we approximate $\nabla_{zz}^2G(z_k,x_k)v_k$ and $\nabla_{xz}^2G(z_k,x_k)v_k$ with the following estimators 
\begin{equation*}
    \begin{aligned}
        H_{zz}^k &:= \frac{1}{\bar{h}_k}(\hat{\nabla}_z g_k(z_k + \bar{h}_k v_k,x_k) - \hat{\nabla}_z g_k(z_k,x_k))\\
        H_{xz}^k &:= \frac{1}{\bar{h}_k} (\hat{\nabla}_x g_k(z_k + \bar{h}_k v_k,x_k) - \hat{\nabla}_x g_k(z_k,x_k)).        
    \end{aligned}
\end{equation*}
Using these approximations, we define the search directions
\begin{equation}\label{eqn:hfzoba_search_directions}
\begin{aligned}
\hat{D}_z^k(z_k,x_k) &:= \hat{\nabla}_z g(z_k,x_k),\\
\hat{D}_v^k(z_k,x_k) &:= H_{zz}^k + \hat{\nabla}_z f(z_k,x_k),\\
\hat{D}_x^k(z_k,x_k) &:= H_{xz}^k + \hat{\nabla}_x f(z_k,x_k).
\end{aligned}
\end{equation}
Thus, we define the following algorithm

\begin{algorithm}[H]
\caption{HF-ZOBA: Hessian-free Zeroth-Order Bilevel Algorithm}
\label{alg:hfzoba}
\begin{algorithmic}[1]

\STATE \textbf{Input:} $x_0\in\mathbb{R}^d$, $z_0,v_0\in\mathbb{R}^p$, stepsizes $(\gamma_k)$, $(\rho_k)$, smoothing parameter $h_k$, minibatch sizes $b_1,b_2$, directions $\ell_1,\ell_2$.
\STATE Let $b = \max(b_1,b_2)$ and $\ell = \max(\ell_1,\ell_2)$.

\FOR{$k=0,1,2,\ldots$}
\STATE Sample i.i.d.\ $\xi_{1,k},\ldots,\xi_{b,k}$ and $\zeta_{1,k},\ldots,\zeta_{b,k}$.

\STATE For each $i\in[b]$, sample $(w_{k}^{(i,j)})_{j=1}^\ell \subset\mathbb{R}^{p}$ and $(u_{k}^{(i,j)})_{j=1}^\ell \subset \mathbb{R}^d$ from $\mathcal{N}(0,I_p)$ and $\mathcal{N}(0,I_d)$.

\STATE Compute 
 \begin{equation*}
    \begin{aligned}
     z_{k + 1} &= z_k - \rho_k \hat{D}_{z}^k(z_k,x_k)\\
    v_{k + 1} &= v_k - \rho_k \hat{D}_{v}^k(z_k,x_k)\\
    x_{k + 1} &= x_k - \gamma_k \hat{D}_{x}^k(z_k,x_k)
    \end{aligned}
 \end{equation*}
\ENDFOR
\end{algorithmic}
\end{algorithm}
The Hessian-free variant (HF-ZOBA) preserves the structure of Algorithm~\ref{alg:zoba}, differing only by replacing the search directions with the ones in eq. \eqref{eqn:hfzoba_search_directions}. The sequences $z_k$, $v_k$, and $x_k$ are updated in parallel, and delayed information is exploited to reuse function evaluations across all updates. 
In particular, function values used to compute $\hat{\nabla}_z g_k$ and $\hat{\nabla}_z f_k$ are reused in $H_{zz}^k, H_{xz}^k$ and $\hat{\nabla}_x f_k$. As a result, HF-ZOBA requires $2b_1(2\ell_1+1) + b_2(2\ell_2+1)$ function evaluations per iteration, which is higher than ZOBA, but without explicit Hessian estimation. In the next section, we review the related works and compare our algorithm with the state-of-the-arts methods.

\subsection{Related Works}
Several algorithms have been proposed for black-box bilevel optimization. We review the most related ones, highlighting differences with our methods, and briefly summarize other approaches.\\
\textbf{Finite-difference methods.} Several finite-difference approaches have been proposed for bilevel optimization. In \cite{Gu2021Optimizing} a finite-difference based algorithm for bilevel optimization is introduced. However, exact function values and access to an approximation of $z^*$ via first-order methods are assumed to be available and no convergence guarantees are provided. In \cite{shirkavandbilevel}, a finite-difference bilevel method for LLMs fine-tuning is proposed and analyzed in non-convex setting. However, their approach assumes access to gradients of the objectives, which are unavailable in our setting. In \cite{Aghasi2025}, the authors propose ZDSBA, a method for black-box bilevel optimization that approximates gradients and Hessians using single-sample, single-direction finite difference surrogates. The authors analyze their approach and obtain a complexity of $\mathcal{O}\left(\frac{(d+p)^4}{\varepsilon^{3}}\log\frac{d+p}{\varepsilon}\right)$ in non-convex setting. This is improved in ZMDSBA \cite{aghasi2025optimalzerothorderbileveloptimization} which extends ZDSBA by using mini-batch finite-difference to approximate $\nabla_x F$ and $\nabla_{xz}^2 G$, yielding a complexity of $\mathcal{O}\left(\frac{p(d+p)^2}{\varepsilon^{2}}\log\frac{1}{\varepsilon}\right)$ in the same setting. However, both our algorithms achieve better complexity. Moreover, ZDSBA and ZMDSBA rely on a two-loop structure, where the inner loops approximate $z^*$ and $v^*$. In these loops, single-sample, single-direction finite-difference surrogates are used to estimate gradients and Hessians. While cheap, these approximations are highly inaccurate, and in high dimensions many inner iterations are required for reliable estimates. In contrast, our methods adopt a single-loop design and approximate $z^*$ and $v^*$ in a single step using multi-direction, mini-batch finite differences. As observed in prior work on finite differences \cite{simple_rs_mania,salimans2017evolutionstrategiesscalablealternative,ozd,Rando2024,rando2025structured}, a single step of finite-difference methods with surrogates constructed with multiple directions provide better performance than many steps each with surrogate constructed with single-direction. Moreover, due to the single loop design, at every iteration all function evaluations can be parallelized while inner-loop iterations cannot. %
In \cite{aghasi2025optimalzerothorderbileveloptimization}, the authors also propose a Hessian-free variant (Opt-ZMDSBA) achieving a complexity of $\mathcal{O}\left((d+p)\varepsilon^{-2}\right)$. While HF-ZOBA matches this complexity, our approach offers several key advantages. First, Opt-ZMDSBA, although Hessian-free, is still a two-loop algorithm. Here the inner loop is used to approximate $z^*$, and therefore it inherits the same limitations of ZDSBA and ZMDSBA in high-dimensional settings, where many non-parallelizable inner iterations are required. Second, Opt-ZMDSBA relies on a regularization-based reformulation that transforms the bilevel problem into a single-level problem, introducing a regularization parameter $\lambda$ that must be carefully tuned to achieve accurate solutions. In contrast, HF-ZOBA operates directly on the original bilevel problem without regularization, remaining Hessian-free, single-loop, and fully parallelizable, allowing better practical performance in runtime. Finally, our gradient and Hessian surrogates are more general than those in \cite{aghasi2025optimalzerothorderbileveloptimization,Aghasi2025}, recovering the estimators of \cite{Aghasi2025} when $b_1=b_2=\ell_1=\ell_2=1$ and those of \cite{aghasi2025optimalzerothorderbileveloptimization} when $b_1=b_2>0$ and $\ell_1=\ell_2=1$.\\
\textbf{Other methods.} Different other strategies were proposed in the literature. %
In \cite{conn2012bilevel}, a zeroth-order trust-region method for bilevel optimization is proposed. However, only the deterministic setting with exact function evaluations is considered, and no convergence proof is provided. In \cite{Ehrhardt2021}, another trust-region algorithm is introduced and analyzed. However, the authors assume exact function values and first-order information of the inner objective to be available. Other strategies proposed are meta-heuristics. These are based on genetic algorithms \cite{Wang2007,ga_for_nonlinear_bilevel}, simulated annealing \cite{SAHIN199811,YU2010288}, particle swarm optimization \cite{ALVES2014547,Gao2011}, and other techniques - see \cite{metaheuristic_for_bilevel} for a review. %
While these approaches perform well in practice, they typically lack convergence guarantees.

\section{Main Results}\label{sec:main_results}

In this section, we analyze Algorithms \ref{alg:zoba} and \ref{alg:hfzoba}, providing convergence rates for different choices of parameters. More precisely, we provide bounds on the average squared norm of the gradient of the value function $\Psi$. In particular, we use the notation 
\begin{equation*}
\alpha_K := \left(\sum_{k=0}^K \gamma_k \, \mathbb{E}\!\left[ \left\| \nabla \Psi(x_k) \right\|^2 \right] \right)\bigg/ A_K,   
\end{equation*}
where $A_K = \sum_{k=0}^K \gamma_k$. In the following, we call {\it complexity} the number of function evaluations required to get $\alpha_K \leq \varepsilon$ with $\varepsilon\in (0,1)$. To improve readability, we define $R_\text{init} := \Psi(x_0) - \min \Psi + \phi_z^0 \| z_0 - z^*(x_0) \|^2 + \phi_v^0 \| v_0 - v^*(x_0) \|^2$,
where $\phi_z^0, \phi_v^0 > 0$ are constants. To perform the analysis, we consider targets satisfying the following hypothesis.
\begin{assumption}[Unbiasedness and Bounded Variance]\label{asm:bc_condition}
    Stochastic gradients and Hessians of inner and outer targets are unbiased estimators, i.e $\nabla F(z,x) = \mathbb{E}_\zeta\left[\nabla f(z,x,\zeta)\right]$, $\nabla G(z,x) = \mathbb{E}_\xi\left[\nabla g(z,x,\xi)\right]$ and $\nabla^2G(z,x) =\mathbb{E}_\xi\left[\nabla^2 g(z,x,\xi)\right]$. Moreover, there exist constants $\sigma_{1,G}, \sigma_{1,F} \geq 0$ s.t. for every $z \in \mathbb{R}^p$ and $x\in \mathbb{R}^d$, 
    \begin{equation*}
        \begin{aligned}
            \mathbb{E}_\zeta[\| \nabla f(z, x,\zeta) - \nabla F(z,x) \|^2] &\leq  \sigma_{1,F}^2\\
            \mathbb{E}_\xi[\| \nabla g(z, x,\xi) - \nabla G(z,x) \|^2] &\leq \sigma_{1,G}^2.
        \end{aligned}
    \end{equation*}
\end{assumption}

\begin{assumption}[Regularity of outer objective]\label{asm:F_smooth}
    There exist $L_{0,F}, L_{1,F} > 0$ such that for every $\zeta \in \mathcal{Z}_1$ the function $(z,x) \mapsto f(z,x,\zeta)$ is $L_{0,F}$-Lipschitz continuous, differentiable and its gradient is $L_{1,F}$-Lipschitz continuous.
\end{assumption}

\begin{assumption}[Regularity of inner objective]\label{asm:G_asm}
    There exist $L_{1,G}, L_{2,G}, \mu_G > 0$ such that, for every $\xi \in \mathcal{Z}_2$ the function $(z,x) \mapsto g(z,x,\xi)$ is twice differentiable, its gradient is $L_{1,G}$-Lipschitz and its Hessian is $L_{2,G}$-Lipschitz. Moreover, for every $x \in \mathbb{R}^d$, the function $z \mapsto g(z,x,\xi)$ is $\mu_G$-strongly convex.
\end{assumption}
\noindent This setting is standard in stochastic zeroth-order bilevel optimization and has been adopted in prior work \cite{Aghasi2025,aghasi2025optimalzerothorderbileveloptimization}. Unlike these approaches, which additionally assume boundedness of the optimal solution of the inner problem, our analysis does not rely on such an assumption. We next present convergence rates under this regime for both algorithms. %
For readability, explicit constants of theorems and corollaries are reported in the proofs - 
see Appendices \ref{app:aux_results} and \ref{app:proof_main_results}. We begin with the main result for Algorithm~\ref{alg:zoba}.
\begin{theorem}[Convergence Rate of ZOBA]\label{thm:conv_rate}
    Let Assumptions \ref{asm:bc_condition},\ref{asm:F_smooth},\ref{asm:G_asm} hold. For every $k \in \mathbb{N}$, let $z_k,v_k,x_k$ be the sequences generated by Algorithm \ref{alg:zoba}. Let $\rho_k < \min\left(1, \frac{\mu_G}{2 \omega_3 (\mu_G + 4)},  \frac{\mu_G}{2 \bar{C}_v} \right)$, $\gamma_k <c_\gamma \rho_k$ %
    and $h_k  = \mathcal{O}\left(\min\left(\sqrt{\frac{b_1\ell_1+p^3}{p^4}}, \frac{1}{\sqrt{p^4+pd^3}}, d^{-5/2}\right)\right)$, where $c_\gamma, \omega_{3} >0$ and $\bar{C}_v = \mathcal{O}(\frac{p(p + 6)^2}{b_1 \ell_1})$ are defined in the proofs. 
    Then, for $K > 0$,
    \begin{equation*}
        \begin{aligned}
        \alpha_K &\leq \frac{2}{A_K} \left( R_\text{init} + \sum\limits_{k=0}^K\bar{C}_k \right) ,
        \end{aligned}
    \end{equation*}
    where $\bar{C}_k \geq 0$ is a sequence of errors depending on $\gamma_k,\rho_k,h_k$ defined in the proof.
\end{theorem}
\noindent In the following corollary, we derive rates for specific choices of the parameters.
\begin{corollary}\label{cor:param_choices}
    Under same assumptions of Theorem \ref{thm:conv_rate}.\\
    (I) Let $\gamma_k = \gamma$, $\rho_k = \rho$ and $h_k = h$, we have
    \begin{equation*}
        \begin{aligned}
            \alpha_K &\leq \frac{R_\text{init}}{\gamma (K + 1)} +\mathcal{O}\left(\frac{\rho^3+\rho^2}{\gamma}+\gamma \right)+ \mathcal{O}\left( \frac{\rho+\rho^2+\rho^3}{\gamma}+\gamma \right) h^2 .
        \end{aligned}
    \end{equation*}
    (II) Let $\gamma_k = \gamma (k + 1)^{-\alpha_1}$, $\rho_k = \rho (k + 1)^{-\alpha_2}$ and $h_k = h (k + 1)^{-\alpha_3}$ with $\gamma, \rho, h >0$ and $\alpha_1, \alpha_2 \in (1/2, 1)$ and $\alpha_3 >1$,
    \begin{equation*}
        \begin{aligned}
        \alpha_K &\leq    \frac{R_\text{init}}{\gamma (K + 1)^{1 - \alpha_1}} + \mathcal{O} \left(\frac{1}{(K + 1)^{1 - \alpha_1}} \right).
        \end{aligned}
    \end{equation*}
    (III) Fix an accuracy $\varepsilon \in (0,1)$. Let %
    $\gamma = \sqrt{\frac{1}{C_{5,1}}\left(\frac{\Delta_1}{ K + 1} +C_3 \mu_G \bar{\phi}_z\rho^3 + C_{6,1} \rho^2 \right)}$, $\rho = \frac{\varepsilon}{\sqrt{16 C_{5,1} C_{6,1}}}$ and $h = \sqrt{\frac{\varepsilon}{2 C(\gamma, \rho)}}$ where the constants $C_3, \bar{\phi}_z, C_{6,1},\Delta_1, C_{5,1}, C(\gamma,\rho) >0$ are defined in the proof. Let $b_1\ell_1 = \mathcal{O}(\lfloor p(d + p)^2/c_1 \rfloor)$ and $b_2 \ell_2 = \mathcal{O}(\lfloor p(d + p)^2/c_2\rfloor)$ for $c_1,c_2 \in \mathbb{R}_+$. Let $K + 1\geq \mathcal{O}(\frac{p(d + p)^2}{b_1\ell_1} \varepsilon^{-2})$. Then, $\alpha_K \leq \varepsilon$ 
    and the complexity is $\mathcal{O}\left(p (d + p)^2 \varepsilon^{-2}\right)$.

\end{corollary}
\noindent In the following theorem and corollary, we show the rate for Algorithm \ref{alg:hfzoba} and we specialize the result with different choices of parameters. 
\begin{theorem}[Convergence Rate of HF-ZOBA]\label{thm:hfzoba_conv_rate}
    Under Assumptions \ref{asm:bc_condition},\ref{asm:F_smooth},\ref{asm:G_asm}. Let $\rho_k < \min\left(1, \frac1{16 C_{B,1}L_{1,G}^2}, \frac{1}{4 C_{A,1}L_{1,G}^2}\right)$, $\gamma_k < \bar{c}_\gamma \rho_k$, $\bar{h}_k = \begin{cases} \hat{h}_k / \|v_k\|, & \|v_k\| \neq 0 \\ \hat{h}_k, & \text{otherwise} \end{cases}$, and $
h_k \leq \frac{\mu_G}{4 \sqrt{(\mu_G + 2) \bar{C}_L}} \sqrt{\bar{\phi}_v}\hat{h}_k$ %
    where $\bar{c}_\gamma > 0, C_{A,1} =\mathcal{O}\left( p/ b_1\ell_1 \right)$, $C_{B,1} = \mathcal{O}\left( p/b_1 \ell_1 \right)$, $\hat{h}_k > 0$,  $\bar{C}_L = \mathcal{O}(\frac{(p + 6)^3 +(d + 6)^3}{b_1 \ell_1})$ and $\bar{\phi}_v > 0$ 
    are defined in the proof. 
Let $z_k,v_k,x_k$ be the sequences generated by eq. \eqref{eqn:hfzoba_search_directions}. Then, for $K>0$,
    \begin{equation*}
        \begin{aligned}
        \alpha_K &\leq \frac{2}{A_K} \left( R_\text{init} + \sum\limits_{k=0}^K\hat{C}_k \right).
        \end{aligned}
    \end{equation*}
    Where $\hat{C}_k \geq 0$ is a sequence of errors depending on $\gamma_k,\rho_k,h_k$ defined in the proof.
\end{theorem}

\begin{corollary}\label{cor:hfzoba_param_choices}
    Under same assumptions of Theorem \ref{thm:hfzoba_conv_rate}.\\
    (I) Let $\gamma_k = \gamma$, $\rho_k = \rho$, $h_k = h$ and $\hat{h}_k = \hat{h}$. Then,\\
    \begin{equation*}
        \begin{aligned}
            \alpha_K &\leq \frac{R_\text{init}}{\gamma (K + 1)} +\mathcal{O}\left(\gamma + \frac{\rho^3}{\gamma}\right) + \mathcal{O}\left(1 + \frac{\rho^2}{\gamma} + \gamma\right)\, \hat{h}^2\\
            &+\mathcal{O}\left(1 + \frac{\rho^2}{\gamma}  +  \frac{\rho^3}{\gamma} \right) h^2 +\mathcal{O}\left(\frac{\rho^2}{\gamma} + \gamma \right)\frac{h^2}{\hat{h}^2}.
        \end{aligned}
    \end{equation*}
    (II) Let $\gamma_k = \gamma (k + 1)^{-\alpha_1}$, $\rho_k = \rho (k + 1)^{-\alpha_2}$, $h_k = h (k + 1)^{-\alpha_3}$ and $\hat{h}_k = \hat{h}(k +1)^{-\alpha_4}$ with $\gamma, \rho, h, \hat{h} >0$, $\alpha_1, \alpha_2 \in (1/2, 1)$, $\alpha_3, \alpha_4 >1/2$ and $\alpha_3 - \alpha_4 \geq 0$,
    \begin{equation*}
        \begin{aligned}
        \alpha_K &\leq  \frac{R_\text{init}}{\gamma (K + 1)^{1 - \alpha_1}} + \mathcal{O} \left(\frac{1}{(K + 1)^{1 - \alpha_1}} \right).
        \end{aligned}
    \end{equation*}
    (III) Fix an accuracy $\varepsilon \in (0,1)$. Let $\bar{\Delta}_4 = \mathcal{O}(\frac{d + p}{b_1 \ell_1} + \frac{d + p}{b_2 \ell_2})$ and fix $K + 1\geq \frac{32 R_\text{init} \bar{\Delta}_4}{\hat{c}_\gamma} \varepsilon^{-2}$. Let $\rho =  \sqrt{\frac{2R_\text{init}}{\bar{c}_\gamma \bar{\Delta}_4 (K + 1)}}$, $\gamma = \hat{c}_\gamma \rho$ where $\hat{c}_\gamma <\bar{c}_\gamma = \mathcal{O}((b_1\ell_1)/(d + p))$. Let $h = \hat{h}^2$ and $\hat{h} \leq \min \left(1,  \sqrt{\frac{ \hat{c}_\gamma\varepsilon}{2\left( \bar{\Delta}_1 + \bar{\Delta}_2 + \bar{\Delta}_3 \right)}}\right)$ where $\bar{\Delta}_1, \bar{\Delta}_2, \bar{\Delta}_3 >0$ are defined in the proof.  %
    Let $b_1\ell_1 = \mathcal{O}(\lfloor (d + p)/c_1 \rfloor)$ and $b_2 \ell_2 = \mathcal{O}(\lfloor (d + p)/c_1\rfloor)$ for $c_1,c_2 \in \mathbb{R}_+$. Then, $ \alpha_K \leq \varepsilon$ 
    and the complexity is $\mathcal{O}((d + p) \varepsilon^{-2})$.%
\end{corollary}

\paragraph*{Discussion.} 

The bounds in Theorems \ref{thm:conv_rate} and \ref{thm:hfzoba_conv_rate} depend on two terms: the initialization and an error component. Such errors arise from the use of stochastic information, delayed information, zeroth-order approximations of stochastic gradients and Hessians, and (in case of HF-ZOBA) the first-order finite-difference approximation of Hessian–vector products.
In Corollaries \ref{cor:param_choices} and \ref{cor:hfzoba_param_choices} we provide two parameter regimes. The first (I) uses constant stepsizes, which is the most common choice in practice, while the second (II) ensures asymptotic convergence. Under choice (I), both algorithms converge to a neighborhood of a stationary point where the size is determined by stochastic variance (the term depending only on $\gamma$ and $\rho$) and finite-difference bias (i.e. the term depending on $h$), consistent with classical results for stochastic zeroth-order methods - see e.g. \cite{nesterov_random_2017,ghadimi_zeroth}. For HF-ZOBA, this neighborhood size further depends on the accuracy of the Hessian–vector product approximation, as reflected by the $\hat{h}$ and $h/\hat{h}$ terms, which respectively represent the finite-difference bias of the first-order Hessian surrogate and the zeroth-order approximation error of the gradients used within it. For both algorithms, expressing the outer stepsize $\gamma$ as a fraction of $\rho$ makes explicit that taking $\rho$ (and $h$) smaller shrinks the size of such neighborhoods but slows the decay of the initialization term through its $1/\gamma$ dependence, mirroring the behavior of first-order stochastic methods with constant stepsize.
Under choice (II), both algorithms recover the convergence rates of first-order stochastic bilevel methods \cite{soba_saba,ghadimi2018approximation,chen_bilevel} up to dependence on dimension. In particular, choosing $\alpha_1$ arbitrarily close to $1/2$ yields rates approaching $\mathcal{O}((K+1)^{-1/2})$, which is order-optimal \cite{aghasi2025optimalzerothorderbileveloptimization}. Moreover, for HF-ZOBA, guaranteeing convergence requires the parameter $h_k$ to decay faster than $\hat{h}_k$, to compensate for the zeroth-order bias in the Hessian–vector approximation.
In (III) we derive the complexities of the proposed algorithms. As shown in the corollaries, both methods improve upon ZDSBA \cite{Aghasi2025}, which attains a complexity of $\mathcal{O}\left(\frac{(d+p)^4}{\varepsilon^3}\log\left(\frac{d+p}{\varepsilon}\right)\right)$. Moreover, they also improve over ZMDSBA \cite{aghasi2025optimalzerothorderbileveloptimization}, which achieves a complexity of $\mathcal{O}\left(\frac{p(d+p)^2}{\varepsilon^2}\log\left(\frac{1}{\varepsilon}\right)\right)$. As in ZMDSBA, ZOBA retains the $\mathcal{O}(p(d+p)^2)$ term, which is consistent with the observation in \cite{aghasi2025optimalzerothorderbileveloptimization} that this dimensional dependence arises from explicit Hessian approximations. %
HF-ZOBA %
matches the complexity of Opt-ZMDSBA \cite{aghasi2025optimalzerothorderbileveloptimization}, which is indicated to be optimal in this regime. However, in contrast to Opt-ZMDSBA and ZMDSBA, whose such dimension dependence in the complexity partially comes from inner-loop iterations, our methods’ dependence arises from directions and batch sizes, which can be fully parallelized, yielding faster practical runtime.

\section{Experiments}\label{sec:experiments}
In this section, we present numerical experiments to evaluate the empirical performance of the proposed methods. We compare ZOBA and HF-ZOBA with ZDSBA \cite{Aghasi2025}, ZMDSBA \cite{aghasi2025optimalzerothorderbileveloptimization}, and Opt-ZMDSBA \cite{aghasi2025optimalzerothorderbileveloptimization}. We consider two problems: a synthetic quadratic bilevel problem and the minimal universal black-box adversarial attack. For each algorithm and each task, hyperparameters are tuned via grid search. In both experiments, we run the algorithms with a fixed budget of $10^5$ function evaluations. All experiments are repeated $10$ times, and the mean and standard deviation of the results are reported in Figures~\ref{fig:synthetic_exp} and~\ref{fig:adv_exp}. Additional details are provided in Appendix~\ref{app:exp_details}.

\paragraph*{Synthetic target.} Here, we evaluate the proposed methods on a synthetic quadratic bilevel problem. %
We consider settings with equal inner and outer dimensions $p=d$, and perform experiments considering $d \in\{25,50,100\}$. %
\begin{figure}[H]
    \centering
    \includegraphics[width=0.32\linewidth]{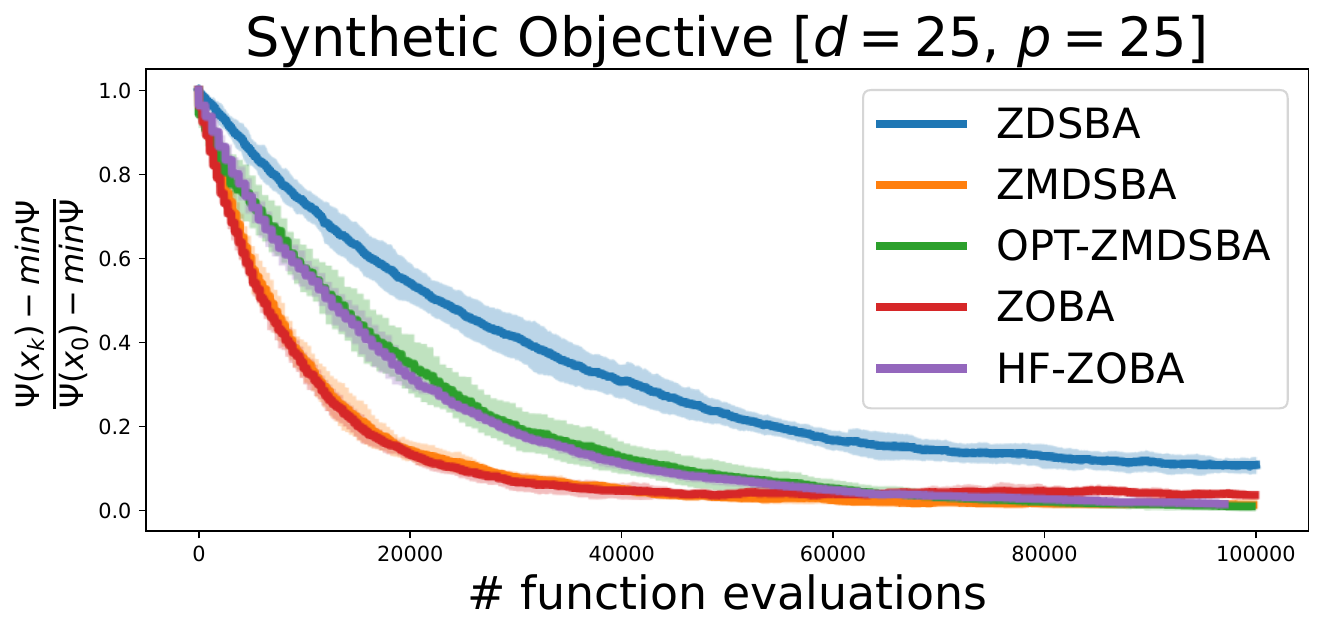}
    \includegraphics[width=0.32\linewidth]{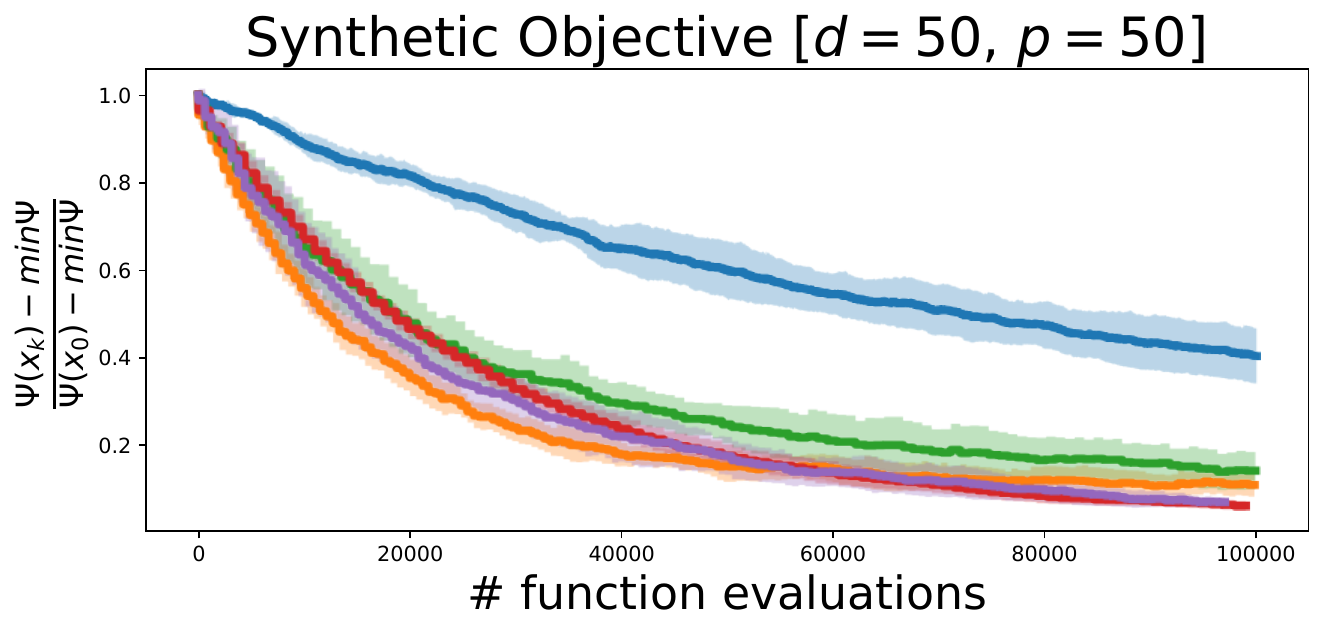}
    \includegraphics[width=0.32\linewidth]{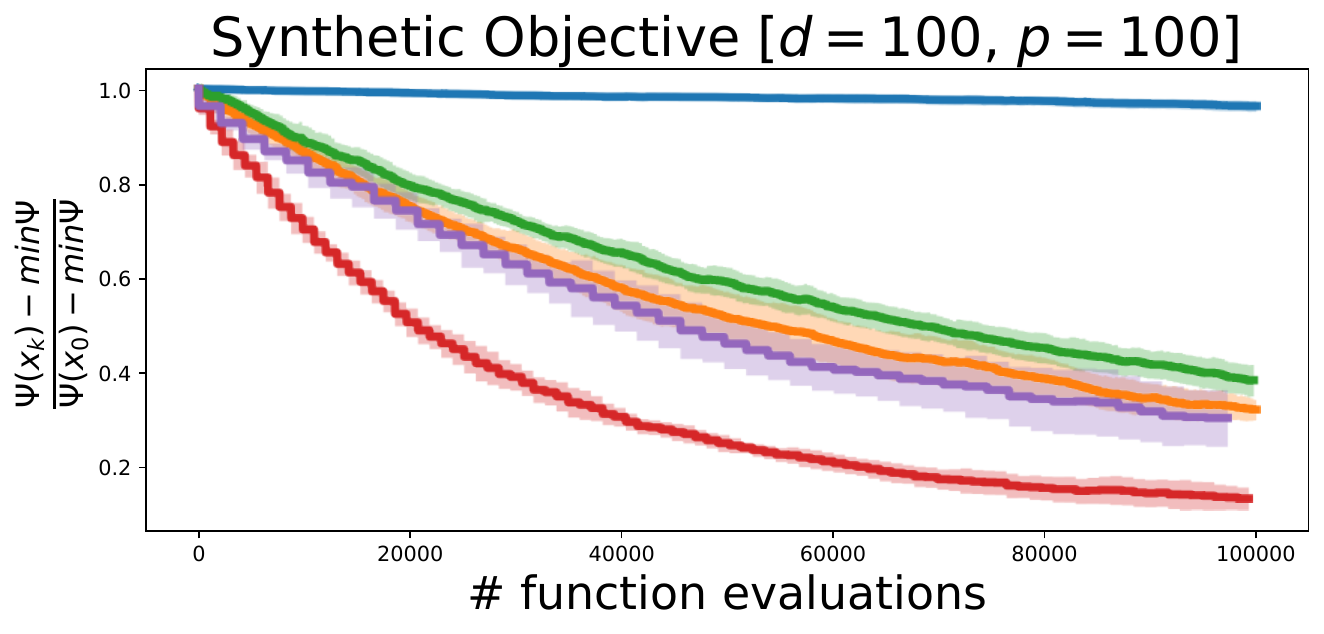}\\
    \includegraphics[width=0.32\linewidth]{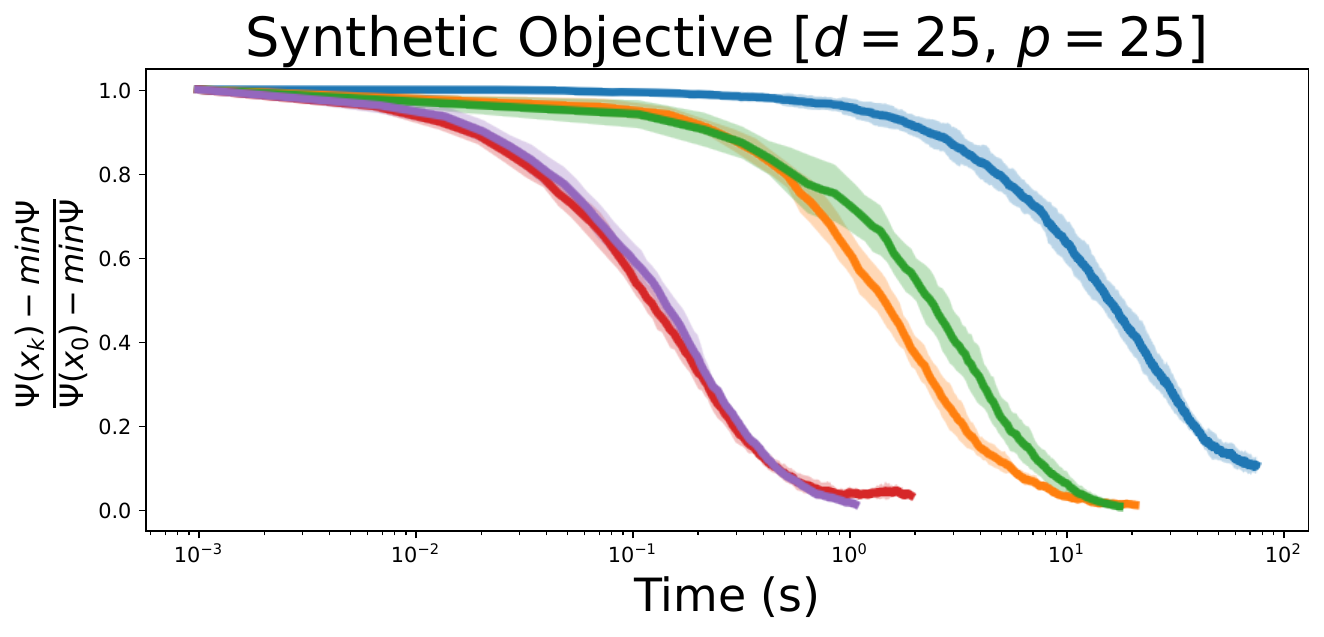}
    \includegraphics[width=0.32\linewidth]{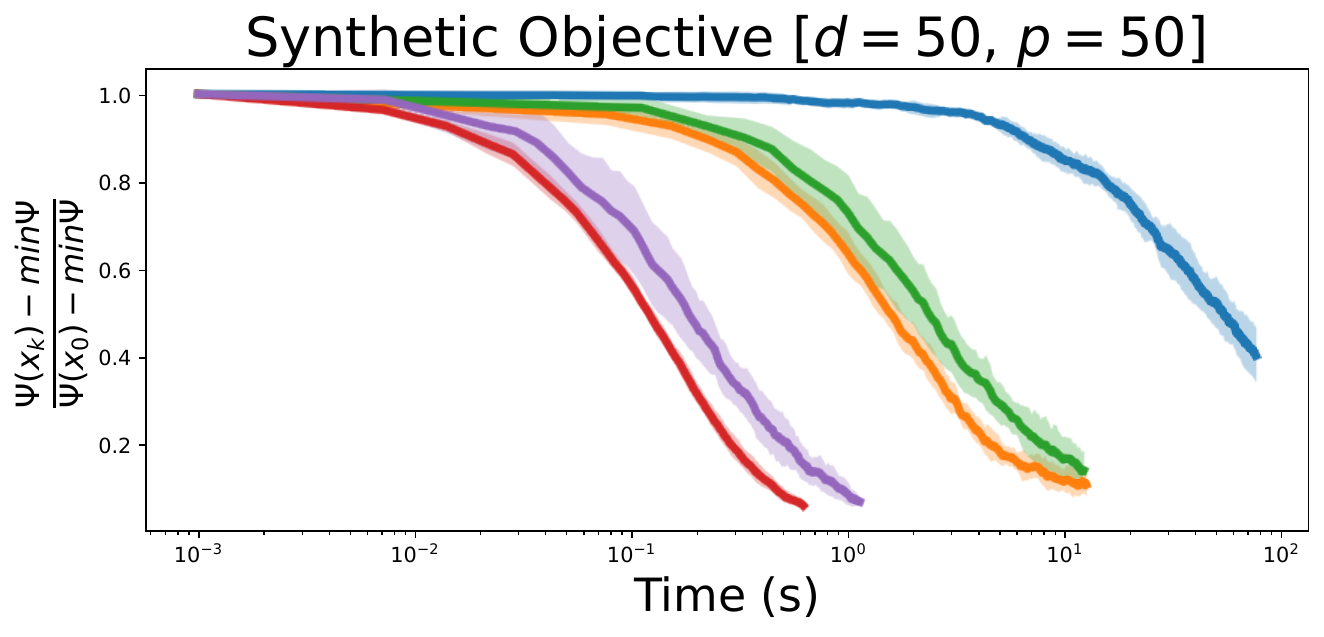}
    \includegraphics[width=0.32\linewidth]{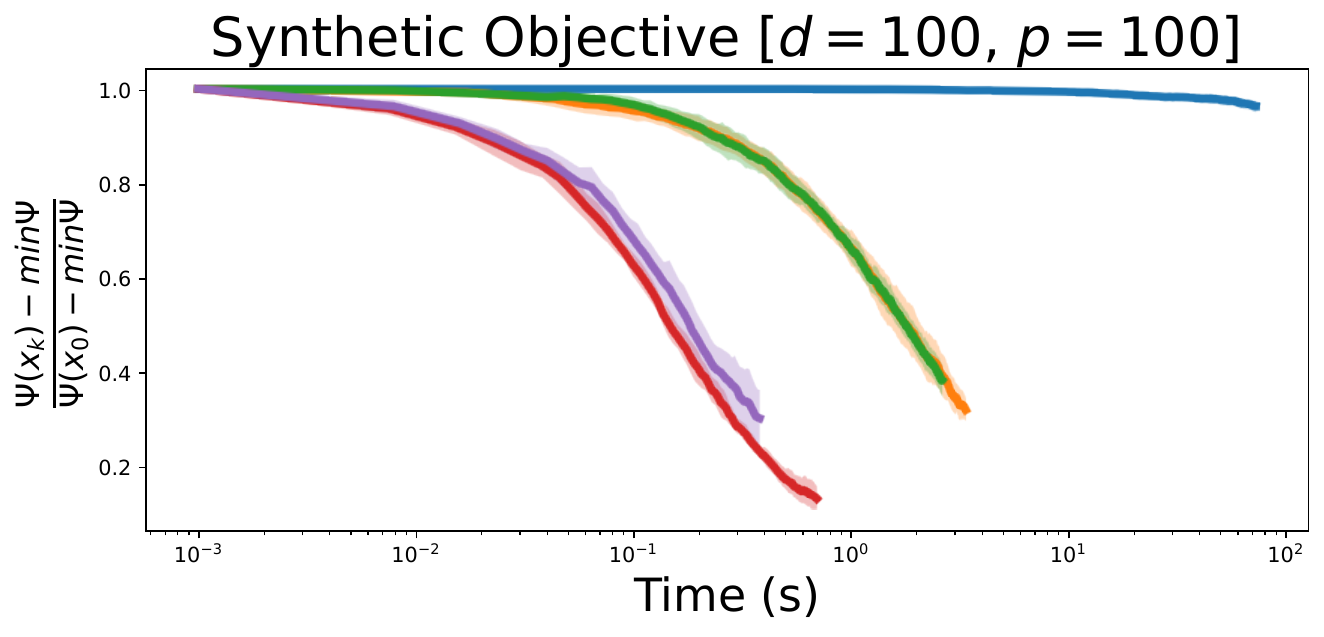}\\
    \caption{Comparison of algorithms on synthetic bilevel problems with varying dimensions: objective value versus function evaluations (first row) and versus wall-clock time in seconds (second row).}
    \label{fig:synthetic_exp}
\end{figure}
\noindent In Figure~\ref{fig:synthetic_exp} (first row) we report the normalized function value gap i.e. $(\Psi(x_k)-\min\Psi)/(\Psi(x_0)-\min\Psi)$ as a function of the number of stochastic function evaluations. We observe that ZDSBA performs worst across all problem sizes, which is consistent with its theoretical complexity. For $d=25$ and $d=50$, our methods attain performance comparable to ZMDSBA and Opt-ZMDSBA, while in the case $d=100$ they yield higher performance. This improvement can be attributed to two main factors. First, our gradient and Hessian surrogates allow flexible choices of the number of samples $b_1,b_2$ and directions $\ell_1,\ell_2$, enabling a more favorable trade-off between stochastic noise and approximation accuracy, whereas ZMDSBA and Opt-ZMDSBA are restricted to $\ell_1=\ell_2=1$ by design. Second, ZMDSBA and Opt-ZMDSBA rely on a two-loop structure in which the inner loop is used to approximate $v^*$ (or its surrogate in Opt-ZMDSBA) and the inner solution $z^*$ via multiple iterations based on gradient or Hessian approximations with a single Gaussian direction. As the dimension increases, more such inner iterations are required; %
this behavior is also reflected in their theoretical complexity, where dimension-dependent terms appear in the required number of inner-loop iterations. In contrast, our single-loop methods control the accuracy of such approximations %
through the number of directions used in the finite-difference gradient surrogates, confirming (and extending to this bilevel setting) the well-known observation in the finite-difference literature that using multiple directions leads to more accurate gradient estimates and higher practical performance than single-direction schemes - see e.g. \cite{simple_rs_mania,sto_md,salimans2017evolutionstrategiesscalablealternative,ZO-BCD,zoro,ozd,Rando2024,rando2025structuredtouroptimizationfinite,rando2025structured}. Moreover, despite the theoretical guarantees of Hessian-free approaches (Opt-ZMDSBA and HF-ZOBA), the methods based on explicit Hessian approximations (ZMDSBA and ZOBA) achieve comparable or better performance. A similar phenomenon was reported in \cite{aghasi2025optimalzerothorderbileveloptimization}, suggesting that for problems in which coarse Hessian surrogates already provide sufficiently accurate approximations, faster convergence may be observed in practice. In Figure~\ref{fig:synthetic_exp} (second row), we report the normalized function value gap as a function of wall-clock time. We observe that our methods complete the experiments in significant less time than competing methods. 
This gain stems from the parallelization and information reuse enabled by delayed updates, unlike ZDSBA, ZMDSBA, and Opt-ZMDSBA, which rely on non-parallelizable inner loops.

\paragraph*{Minimal distortion universal perturbation.}  We evaluate our methods on generating minimal-distortion universal adversarial perturbations (MD-UAP) for a multiclass MNIST classifier \cite{mnist_dataset}; see Appendix \ref{app:universal_pert_details} for details. We tackle this problem using subspace-based black-box attacks \cite{Wang2020}, which parameterize the perturbation as, e.g., $\delta = X z$, where $X \in \mathbb{R}^{d_{\text{img}} \times p}$ defines a low-dimensional subspace where $d_\text{img} = 784$ is the dimension of images and $z \in \mathbb{R}^p$ denotes the coefficients. %
While most existing approaches fix or compute $X$ heuristically, we instead optimize it by formulating MD-UAP as a black-box bilevel optimization problem, where the outer objective measures the average distortion and the inner objective measures the average confidence of correct classification. Through this reformulation, the bilevel algorithm seeks a subspace $X$ that minimizes the average distortion while finding coefficients $z^\star(X) \in \mathbb{R}^p$ that minimize the average confidence within that subspace, yielding the parameterization $\delta = X z^\star(X)$. To the best of our knowledge, this is the first work to formulate subspace-based MD-UAP by explicitly optimizing the subspace within a bilevel framework. We set $p=100$ and construct perturbations for images labeled $4$. 
\begin{figure}[h]
    \centering
    \includegraphics[width=0.8\linewidth]{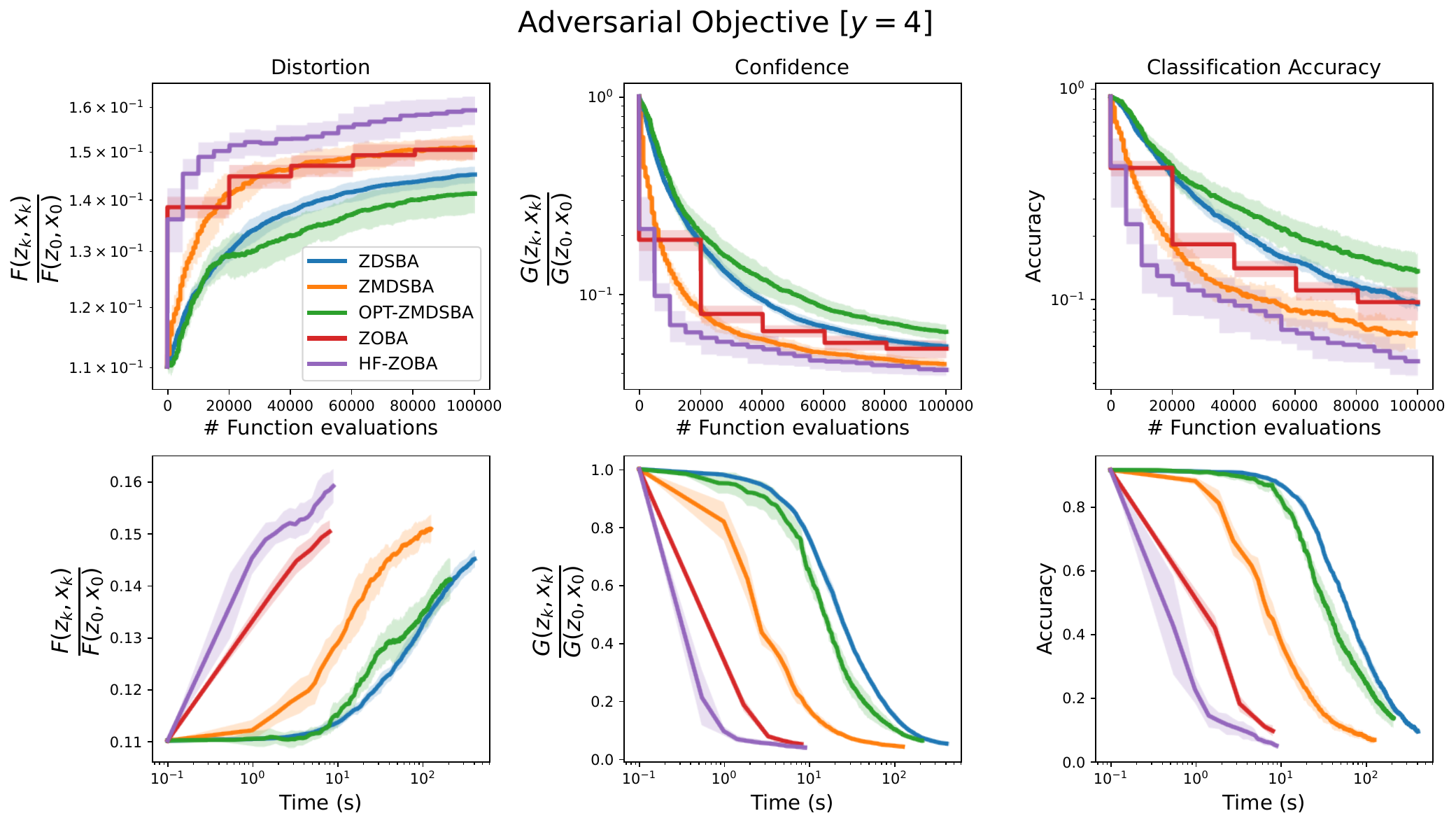}
    \caption{Comparison of algorithms on minimal distortion universal perturbation attack. Normalized outer function values and inner function values on iterations $z_k,x_k$ and classification accuracy are reported as function of number of function evaluations (first row) and versus wall-clock time in seconds (second row).}
    \label{fig:adv_exp}
\end{figure}
In Figure~\ref{fig:adv_exp} (first row), we report the outer (distortion) and inner (confidence) objective values along the sequences ${X_k}$ and ${z_k}$, together with the classification accuracy, as functions of the number of stochastic function evaluations. As in the previous experiment, values are repeated according to the number of stochastic evaluations performed per iteration. We observe that our methods achieve performance comparable to ZMDSBA and Opt-ZMDSBA, and better performance than ZDSBA. Notice that for all methods, the average distortion initially increases before reaching a plateau; this behavior is expected since the perturbation is initialized with very small magnitude and must grow in norm to effectively reduce the classification confidence. Although the final performance is similar across the considered methods, our approaches reach comparable accuracy and confidence levels in fewer iterations by leveraging multiple function evaluations per iteration to construct high-quality gradient and Hessian surrogates. This yields a practical advantage due to full parallelization of the evaluations. Such an effect is illustrated in Figure~\ref{fig:adv_exp} (second row), where the same quantities are plotted against wall-clock time. Here, we observe that our methods achieve substantially better runtime performance than competing approaches. Notice that ZOBA requires many function evaluations per iteration due to costly Hessian-block estimation. In contrast, HF-ZOBA requires fewer evaluations for estimating gradients and, by preserving a single-loop structure, achieves higher performance within the same wall-clock time.

\section{Conclusion}\label{sec:conclusion}
In this work, we introduced ZOBA, the first single-loop algorithm for fully zeroth-order bilevel optimization, and HF-ZOBA, its Hessian-free variant. We analyzed both algorithms and derived convergence rates. %
Our experiments show that single-loop approaches provide clear practical advantages in wall-clock time over current two-loop state-of-the-art methods, making them particularly suitable for computationally expensive problems. Future directions include extending the analysis to non-smooth settings, relaxing the strong-convexity assumptions and developing adaptive parameter selection strategies.%

\paragraph*{Acknowledgments.} This work has been supported by the French government, through the 3IA Cote d’Azur Investments in the project managed by the National Research Agency (ANR) with the reference number ANR-23-IACL-0001, the ANR project PRC MAD ANR-24-CE23-1529 and the support of the ``France 2030'' funding ANR-23-PEIA-0004 (PDE-AI).

\bibliography{bibliography}
\bibliographystyle{plain}

\newpage
\appendix

\section{Cost Comparison with Finite-Difference Bilevel Optimization Methods}

In this appendix, we compare our methods with state-of-the-art finite-difference algorithms in terms of the function evaluations performed at each iteration. More precisely, we analyze how many function evaluations are required per iteration by state-of-the-art methods and how these evaluations are used.

\paragraph*{ZDSBA \cite{Aghasi2025}.}
At each iteration $k \in \mathbb{N}$, ZDSBA first approximates the inner solution $z^*(x_k)$ at the current iterate $x_k$ by constructing the sequence $(z_k){i=1}^{N\text{inner}}$ through $N_\text{inner} > 0$ iterations of a stochastic zeroth-order method. This procedure involves the construction of a stochastic gradient surrogate of the inner objective at every step. Such a surrogate is built using finite differences with a single sample and a single Gaussian direction. Each inner iteration therefore requires two function evaluations, resulting in a total cost of $2N_\text{inner}$ function evaluations. Next, the method approximates $v^*(x_k)$ by constructing a sequence $(v_k){k=1}^{N\text{inverse}}$ via $N_\text{inverse}$ iterations of a stochastic zeroth-order method applied to the quadratic form in eq. \eqref{eqn:v_star}, where $z^*(x_k)$ is replaced by $z_{N_\text{inner}}$. Each iteration again relies on single-sample, single-direction finite-difference estimators. Since this quadratic form requires approximations of both the gradient of the outer objective with respect to the first variable and the corresponding Hessian block, each iteration incurs a cost of $5$ function evaluations. Finally, the Hessian cross-block and the gradient of the outer objective with respect to the second variable are approximated. Again, such approximations are built using single-sample, single Gaussian-direction estimators, incurring an additional cost of $5$ function evaluations. Therefore, the total number of function evaluations per iteration $k$ is
\begin{equation*}
2N_\text{inner} + 5N_\text{inverse} + 5.
\end{equation*}
\paragraph*{ZMDSBA \cite{aghasi2025optimalzerothorderbileveloptimization}.} ZMDSBA generalizes ZDSBA by allowing multiple samples in the estimation of the Hessian cross-block and the gradient of the outer objective with respect to the second variable. Specifically, these quantities are approximated using finite differences with a single Gaussian direction and $N_\text{sample}$ samples (i.e. our gradient/hessian estimators with $\ell_1 = \ell_2 = 1$ and $b_1 = b_2 = N_\text{sample}$). The other steps are identical to ZDSBA. As a result, the total number of function evaluations per iteration becomes
\begin{equation*}
2N_\text{inner} + 5N_\text{inverse} + 5N_\text{sample}.
\end{equation*}

\paragraph*{OPT-ZMDSBA \cite{aghasi2025optimalzerothorderbileveloptimization}.} OPT-ZMDSBA is based on a regularization-based reformulation of the bilevel problem, where the inner optimization is incorporated via a penalty function. This approach allows the algorithm to avoid explicit Hessian approximations. At each iteration $k$, the algorithm constructs approximations of $z^*(x_k)$ with a sequence $(z_k)_{k=1}^{N_\text{inner}}$ and an auxiliary sequence $(y_k)_{k=1}^{N_\text{inner}}$ by performing $N_\text{inner}$ iterations of a stochastic zeroth-order method using single-sample, single Gaussian direction gradient surrogates. More precisely, these sequences require computing three gradient surrogates of the inner and outer objectives per iteration, resulting in a total cost of $6N_\text{inner}$ function evaluations. Then, the algorithm constructs a search direction to update the iterate. Such a directions requires computing three surrogate gradients of the inner and outer objectives using single-direction estimators with $N_\text{sample}$ samples, incurring a cost of $6N_\text{sample}$ function evaluations. Therefore, the total number of function evaluations per iteration is
\begin{equation*}
6N_\text{inner} + 6N_\text{sample}.
\end{equation*}
In Table \ref{tab:alg_comparison}, we summarize these methods by reporting the number of function evaluations per iteration (\# FE) and their theoretical complexity, i.e. the number of function evaluations required to achieve $\varepsilon \in (0,1)$ accuracy.

\begin{table}[H]
    \centering
    \caption{Comparison of bilevel zeroth-order algorithms in terms of complexity and number of function evaluations performed at every iteration (\# FE).} \label{tab:alg_comparison}
    \begin{tabular}{llll}
    \toprule
       Method & Loops & \# FE & Complexity   \\
     \midrule
       ZDSBA & two-loops & $2N_\text{inner} + 5N_\text{inverse} + 5$ & $\mathcal{O}((d + p)^4 \varepsilon^{-3} \log((d + p)/\varepsilon))$\\
       ZMDSBA & two-loops & $2N_\text{inner} + 5N_\text{inverse} + 5N_\text{sample}$ & $\mathcal{O}(p(d + p)^2 \varepsilon^{-2} \log(1/\varepsilon))$\\
       OPT-ZMDSBA & two-loops & $6N_\text{inner} + 6N_\text{sample}$ & $\mathcal{O}((d + p) \varepsilon^{-2})$\\
       ZOBA & single-loop & $b_1(4\ell_1+1)+b_2(2\ell_2+1)$& $\mathcal{O}(p(d + p)^2 \varepsilon^{-2})$\\
       HF-ZOBA & single-loop & $2b_1(2\ell_1+1) + b_2(2\ell_2+1)$ & $\mathcal{O}((d + p) \varepsilon^{-2})$\\
    \bottomrule
    \end{tabular}
\end{table}
\noindent In Table \ref{tab:alg_comparison}, we observe that in the limit case $b_1 = b_2 = \ell_1 = \ell_2 = N_\text{inner} = N_\text{sample} = N_\text{inverse} = 1$, our methods require fewer function evaluations per iteration than other finite-difference bilevel algorithms. This reduction stems from the reuse of function evaluations enabled by the use of delayed information in ZOBA and HF-ZOBA.  In addition, the function evaluations required by ZOBA and HF-ZOBA can be performed in parallel across directions and batch samples, while the evaluations in ZDSBA, ZMDSBA and OPT-ZMDSBA are inherently sequential due to their inner-loop structure. As a result, the effective wall-clock cost of ZOBA and HF-ZOBA can be significantly lower in practice, even when the number of function evaluations is comparable or even higher.

\section{Experimental Details}\label{app:exp_details}
In this appendix we provide all details on the experimental setup and parameter tuning used to perform the experiments reported in Section \ref{sec:experiments}. We implemented all scripts in Python~3 (version 3.11) and used the NumPy (version 1.25.0) \cite{numpy}, PyTorch (version 2.0.1) \cite{pytorch}, and Matplotlib (version 3.10.8) \cite{matplotlib} libraries. This appendix is organized as follows. In Appendix \ref{app:quad_problem_details}, we provide details of the experiments on the synthetic quadratic problem while in Appendix \ref{app:universal_pert_details}, we provide details of the experiments on the minimal-distortion universal perturbation task. %
Details on the machine used to perform the experiments are reported in Table \ref{tab:machine_details}. %

\begin{table}[H]
    \centering
    \caption{Machine used to perform the experiments} \label{tab:machine_details}
    \begin{tabular}{ll}
    \toprule
       Feature  &   \\
     \midrule
       CPU  & 64 x AMD EPYC 7301 16-Core\\
       GPU & 1 x NVIDIA Quadro RTX 6000\\
       RAM & 256 GB\\
    \bottomrule
    \end{tabular}
\end{table}

\subsection{Synthetic Problem}\label{app:quad_problem_details}
Here we describe the synthetic problem considered in Section \ref{sec:experiments}. We consider the minimization of the following function
\begin{equation}\label{eqn:app_synthetic_problem}
    \begin{aligned}
 \min_{x \in \mathbb{R}^d} \Psi(x) &:= F(z^*(x), x) :=\frac{1}{2n}\sum_{i=1}^{n}\left(C_{i} z^*(x) - D_{i} x - b_i\right)^2 + \frac{1}{2} \|x - \bar{x}\|_2^2 \\
\text{s.t.}\quad z^*(x) &\in \arg\min_{z \in \mathbb{R}^p}G(z,x) :=\frac{1}{2m}\sum_{j=1}^{m} \left( A_{j} z -B_{j} x - a_j\right)^2,
    \end{aligned}
\end{equation}
where $A \in \mathbb{R}^{m \times p},B \in \mathbb{R}^{m \times d},C \in \mathbb{R}^{n \times p},D \in \mathbb{R}^{n \times d}$ are Gaussian matrices where every entry is sampled from $\mathcal{N}(0,1)$. For fixed $\bar{z} \in \mathbb{R}^p$ and $\bar{x} \in \mathbb{R}^d$, we define for every $j=1,\cdots,m$ and $i=1,\cdots, n$
\begin{equation*}
    \begin{aligned}
        a := A \bar{z} - B\bar{x} \quad \text{and} \quad b := C\bar{z} - D\bar{x}.
    \end{aligned}
\end{equation*}
In the experiments in Section \ref{sec:experiments}, we fixed $n = m = 1000$, $\bar{z} = [2,\cdots,2] \in \mathbb{R}^p$ and $\bar{x} = [1,\dots, 1] \in \mathbb{R}^d$. We considered the setting $d = p =25, 50, 100$.

\paragraph{Parameter Tuning.}%

All parameters of every algorithm were tuned via grid search. For every algorithm and each parameter configuration in its parameter grid, the algorithm was run $6$ times with a fixed budget of $10^5$ function evaluations, using $6$ different random initializations. The same set of initializations was used across all parameter configurations to ensure a fair comparison. The optimal parameter configuration was selected as the one minimizing, on average over the runs, the normalized function value gap at the final iterate,
\begin{equation*}
    \begin{aligned}
        \frac{\Psi(x_K) - \min \Psi}{\Psi(x_0) - \min \Psi},
    \end{aligned}
\end{equation*}
where $x_K$ denotes the iterate obtained at the last iteration performed $K$. Initializations of $x_0$ and $z_0$ were sampled independently from the uniform distribution over the hypercubes $[-5.0, 10.0]^d$ and $[-5.0, 10.0]^p$, respectively. The initialization of $v_0$ (i.e. the approximation of the inverse Hessian-vector product) was set to the zero vector. Once the optimal parameter configuration was identified, the algorithm was run $11$ additional times using the same evaluation budget. To mitigate the effect of initialization overhead (e.g., library loading) on runtime measurements, the first run was discarded, and all reported results were computed by averaging over the remaining $10$ runs.

\paragraph*{Search Spaces.} 
For all algorithms, every discretization parameter used are fixed to $h = 10^{-3}$. Every stepsize parameter (used by all algorithms) was tuned over the set $\{10^{-5}, 10^{-4}, 10^{-3}, 10^{-2}\}$. For ZOBA and HF-ZOBA, the number of directions $\ell_1$ and $\ell_2$ used to approximate stochastic gradients and Hessians was selected from $\{1, 10,  50, 100\}$. Every batch sizes ($b_1, b_2$ in ZOBA, HF-ZOBA and $s_k$ in ZMDSBA and OPT-ZMDSBA of \cite{aghasi2025optimalzerothorderbileveloptimization}) were tuned over $\{1, 10, 50, 100\}$. Notice that the number of directions and batch sizes are not required to be equal, i.e., $\ell_1, \ell_2, b_1, b_2$ can all take different values. The number of inner-loop iterations used to approximate $z^*$ and $v^*$ for ZDSBA, ZMDSBA, and OPT-ZMDSBA was tuned over $\{1, 10, 50\}$. Finally, the regularization parameter $\lambda$ for OPT-ZMDSBA was selected from $\{1.0, 10.0, 100.0\}$. All parameters tuned, their ranges and the algorithms that used them are summarized in Table \ref{tab:parameter_grids}.  For each method, the full search grid was defined as the Cartesian product of the candidate sets corresponding to its parameters.

\begin{table}[h]
\centering
\caption{Parameter grids used for each algorithm.}
\label{tab:parameter_grids}
\begin{tabular}{l l l}
\toprule
\textbf{Parameter} & \textbf{Algorithm(s)} & \textbf{Candidate Values} \\
\midrule
Stepsizes & All algorithms & $\{10^{-5}, 10^{-4}, 10^{-3}, 10^{-2}\}$ \\
Number of directions & ZOBA, HF-ZOBA & $\{1, 10, 25, 50, 100\}$ \\
Batch sizes & ZOBA, HF-ZOBA, ZMDSBA, OPT-ZMDSBA & $\{1, 10, 50, 100\}$ \\
Inner-loop iterations & ZDSBA, ZMDSBA, OPT-ZMDSBA & $\{1, 10, 50\}$ \\
Regularization  & OPT-ZMDSBA & $\{1.0, 10.0, 100.0\}$ \\
\bottomrule
\end{tabular}
\end{table}
\paragraph*{Impact of the number of directions and batchsize.} We study how the performance of our algorithms is affected by the choice of the number of directions and the batch-size parameters. Specifically, we consider the optimization of the synthetic objective defined in eq. \eqref{eqn:app_synthetic_problem} with $p = d = 50$ and $n = m = 1000$. We fix a budget of $10^4$ function evaluations (i.e., fewer than those used in Section \ref{sec:experiments}) and run ZOBA and HF-ZOBA with a constant stepsize $\rho$, chosen from a grid of 20 logarithmically spaced values between $10^{-5}$ and $1.0$. The outer stepsize is set as a fraction of $\rho$, namely $\gamma = \rho/c$ and, in particular, we fix $c = 5.0$. The discretization parameters are fixed to $h = \hat{h} = 10^{-3}$. To reduce the number of hyperparameters, we impose $b := b_1 = b_2$ and $\ell := \ell_1 = \ell_2$. Each experiment is repeated five times, and we report the mean and standard deviation of the results. In Figure \ref{fig:app_same_bl_zoba}, we plot the normalized objective value progress $(\Psi(x_K) - \min \Psi)/(\Psi(x_0) - \min \Psi)$ at the final iterate $x_K$ obtained by ZOBA as a function of the stepsize $\rho$. If the algorithm diverges, this quantity is clipped to one.
\begin{figure}[h]
    \centering
    \includegraphics[width=0.8\linewidth]{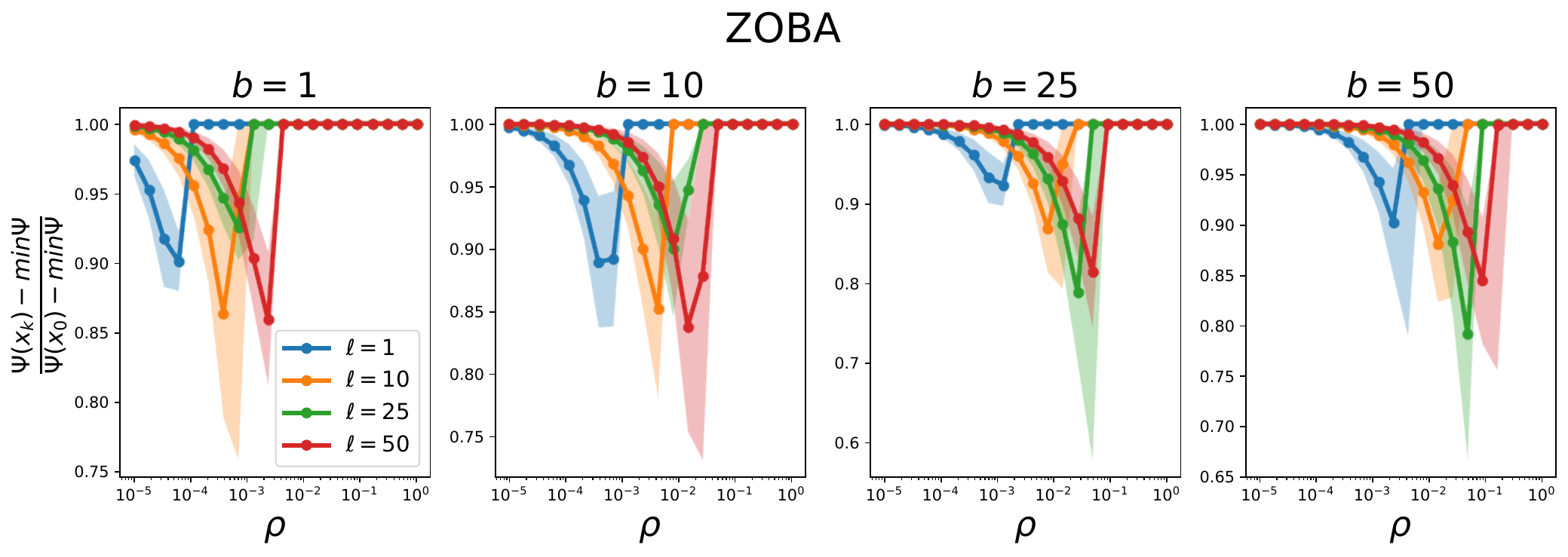}
    \includegraphics[width=0.8\linewidth]{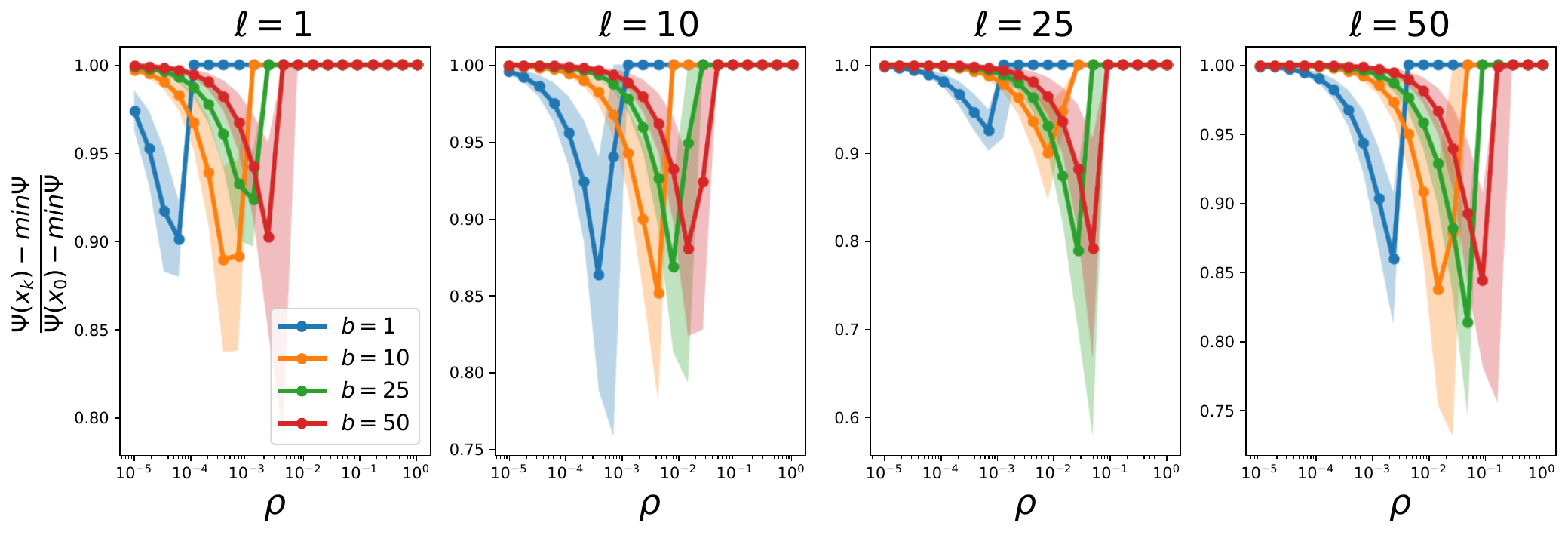}
    \caption{Normalized objective value progress at the final iterate obtained by running Algorithm \ref{alg:zoba} with different stepsizes $\rho$, numbers of directions $\ell$, and batch sizes $b$. If the algorithm diverges, the quantity $(\Psi(x_K) - \min \Psi)/(\Psi(x_0) - \min \Psi)$ is clipped to one. In the first row, every plot show how this quantity change for different values of $\ell$ by fixing the batch size $b$; in the second row, instead, every plot show how such value change for different values of $b$ by fixing $\ell$.}
    \label{fig:app_same_bl_zoba}
\end{figure}
From Figure \ref{fig:app_same_bl_zoba}, we observe that increasing the number of directions $\ell$, i.e., improving the quality of the stochastic gradient approximations for both the inner and outer objectives, makes ZOBA more stable and allows for the use of larger stepsize $\rho$. Moreover, increasing $\ell$ generally leads to better or comparable final performance in terms of solution quality, while also being potentially more time-efficient in practice since function evaluations can be computed in parallel. When $\ell$ (and eventually $b$) becomes too large, performance deteriorates. This behavior is due to the limited budget of function evaluations: large values of $\ell$ and $b$ consume many evaluations per iteration, resulting in too few optimization steps to make meaningful progress. Comparing the first and second rows of Figure \ref{fig:app_same_bl_zoba}, we further observe that although $\ell$ and $b$ play a similar conceptual role (both improving the quality of gradient and Hessian surrogates), fixing $b$ and increasing $\ell$ may yield better results in practice. For example, fixing $b = 1$ and increasing $\ell$ achieves a lower normalized objective gap $(\Psi(x_K) - \min \Psi)/(\Psi(x_0) - \min \Psi)$ than fixing $\ell = 1$ and increasing $b$. Moreover, this advantage diminishes or vanishes when the fixed parameter takes a large value. This behavior can again be attributed to the limited evaluation budget, or to the fact that when either $\ell$ or $b$ is sufficiently large, the variance becomes small and the two parameters play increasingly similar roles. One possible explanation of that phenomenon is that, in some regimes, stochastic (or mini-batch) gradients are already sufficiently accurate and it is also easier to obtain good zeroth-order approximations via multiple directions than improving surrogate quality through batching with poor approximations i.e. good approximations of high-variance stochastic gradients (large $\ell$, small $b$) may thus be preferable to poor approximations of low-variance minibatch gradients (small $\ell$, large $b$). %
A similar phenomenon has been observed in \cite{rando2025structured}, suggesting that this behavior is not specific in bilevel optimization but it can also arise in other stochastic settings. These findings provide empirical evidence suggesting that achieving a good balance between $\ell$ and $b$ is important for maximizing the performance of the method, and they support the use of flexible surrogates such as ours, rather than special cases in which $\ell$ is fixed to $1$ and only $b$ varies, as in \cite{Aghasi2025,aghasi2025optimalzerothorderbileveloptimization}.
\begin{figure}[H]
    \centering
    \includegraphics[width=0.8\linewidth]{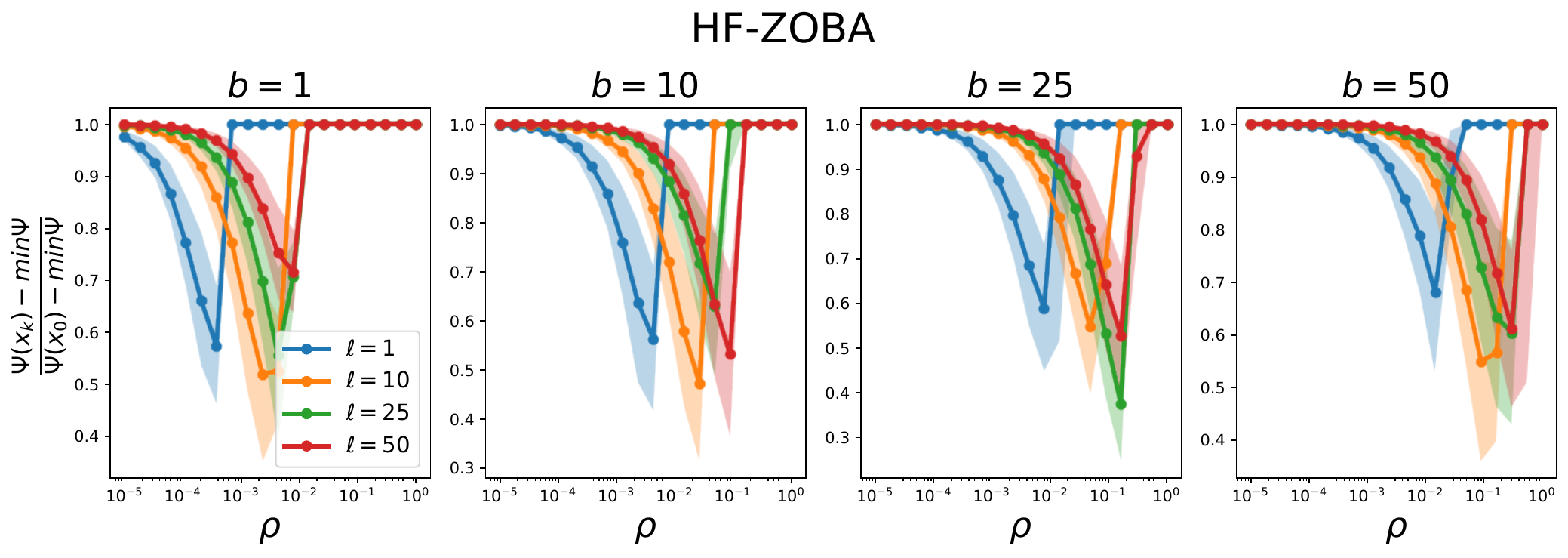}    \includegraphics[width=0.8\linewidth]{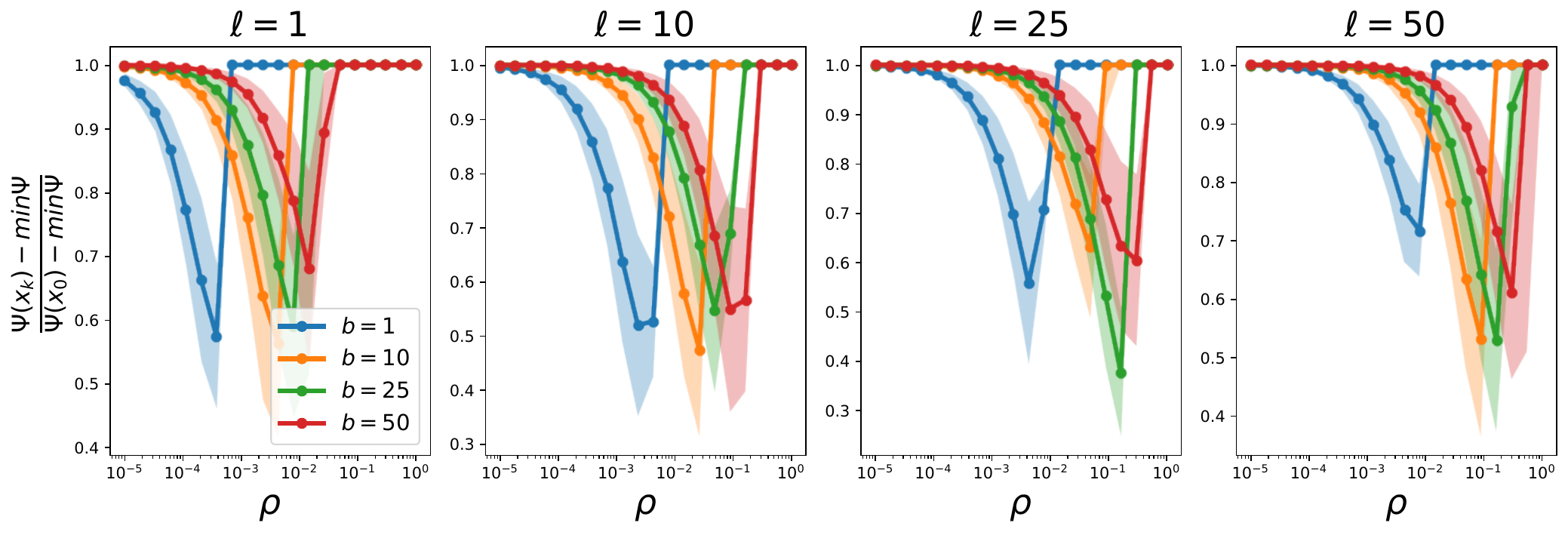}
    \caption{Normalized objective value progress at the final iterate obtained by running Algorithm \ref{alg:hfzoba} with $\gamma = \rho/5$ for different stepsizes $\rho$, numbers of directions $\ell$, and batch sizes $b$. If the algorithm diverges, the quantity $(\Psi(x_K) - \min \Psi)/(\Psi(x_0) - \min \Psi)$ is clipped to one. In the first row, every plot show how this quantity change for different values of $\ell$ by fixing the batch size $b$; in the second row, instead, every plot show how such value change for different values of $b$ by fixing $\ell$.}%
    \label{fig:app_same_bl_hfzoba}
\end{figure}
\noindent In Figure \ref{fig:app_same_bl_hfzoba}, HF-ZOBA exhibits a similar pattern: increasing $\ell$ or $b$ improves stability and allows for larger stepsizes $\rho$, and fixing a small $b$ while increasing $\ell$ achieves lower normalized objective gaps than the opposite. Comparing Figures \ref{fig:app_same_bl_hfzoba} and \ref{fig:app_same_bl_zoba}, HF-ZOBA achieves significantly better results than ZOBA under the limited evaluation budget, even though both methods perform similarly in Section \ref{sec:experiments}. This can be attributed to the smaller budget: while higher budgets allow using more directions and larger batch sizes, under limited evaluations fewer directions or smaller batch sizes must be used. ZOBA relies on explicit Hessian approximations, which can be noisy with few directions, leading to higher variance in the updates. As a result, HF-ZOBA, which relies solely on gradient surrogates, does not suffer from this variance, allowing for more precise steps and thus better performance under limited evaluation budget settings.

\subsection{Minimal-distortion Universal Perturbation}\label{app:universal_pert_details}

In this appendix, we report the experimental details of the minimal-distortion universal adversarial attack experiment presented in Section~\ref{sec:experiments}. More precisely, we describe the problem and we propose our bilevel formulation. We describe how we implemented and trained the classifier and how we choose the parameters of the optimizer.

\paragraph*{Problem Description \& Bilevel Formulation.} 
The goal of the minimal-distortion universal perturbation problem is the following: let $\hat{f} : \mathbb{R}^{d_{\text{img}}} \to \mathbb{R}^C$ be a trained multiclass classifier and let $S_{\text{test}}=\{(\hat{x}_i,y_i)\}_{i=1}^m$ with $\hat{x}_i \in \mathbb{R}^{d_{\mathrm{img}}}$ and $y_i \in \{1, \dots, C\}$ be a test dataset. Let $\psi : \mathbb{R}^{d_{\text{img}}} \times \mathbb{R}^{d_{\text{img}}} \to \mathbb{R}^{d_{\text{img}}}$ be a manipulation function that, given a data point $\hat{x}$ and a perturbation $\delta$, returns the perturbed data point with perturbation $\delta$. Given a function $L(y,\hat{f}(x))$ that measures the confidence of classifying $x$ with label $y$ and a distortion measure $D(\hat{x}, \delta)$ that quantifies the amount of distortion obtained by perturbing $\hat{x}$ with $\delta$, the minimal-distortion universal perturbation problem can be formulated as the following constrained optimization problem:
\begin{equation*}
    \begin{aligned}
        \min\limits_{\delta \in \mathbb{R}^{d_{\text{img}}}}& \frac{1}{m} \sum\limits_{i=1}^m D(\hat{x}_i, \delta) \quad \text{s.t.} \quad  \delta \in \argmin\limits_{\delta^\prime \in \mathbb{R}^{d_{\text{img}}}} \frac{1}{m} \sum\limits_{i=1}^m L(y_i,\hat{f}(\psi(\hat{x}_i, \delta))).
    \end{aligned}
\end{equation*}
\noindent Based on prior observations that adversarial perturbations tend to concentrate in low-dimensional subspaces of the input space (see, e.g., \cite{zoo,Bayer2025}) and the fact that zeroth-order methods can perform poorly in high-dimensional settings due to their dependence on dimension in the convergence rates (see, e.g., \cite{nesterov_random_2017}), we restrict the attack to a low-dimensional subspace. More precisely, we parametrize the perturbation as $\delta = Az$, with $A \in \mathbb{R}^{d_{\text{img}} \times p}$ and $z \in \mathbb{R}^p$ with $p < d_{\text{img}}$, and we rewrite the minimal-distortion universal perturbation problem as a zeroth-order bilevel optimization problem, where the outer problem chooses the transformation $A$ that minimizes the distortion measure and the inner problem selects the coefficients $z$ that minimize the classification confidence. The idea of restricting the attack to a lower-dimensional subspace has been proposed in previous works and is common, for instance, in spanning attacks~\cite{Wang2020}. In such works, the subspace is generally either heuristically fixed or selected at random; in contrast, we propose to optimize it. This approach relies on the intuition that minimizing the average distortion is an "easier" problem than minimizing classification confidence, which is instead recast as the inner optimization problem in $\mathbb{R}^p$ with a smaller $p$. In our experiments, we trained a multiclass classifier $\hat{f}$ on the MNIST dataset~\cite{mnist_dataset}. We fixed the manipulation function to be the Carlini \& Wagner transformation~\cite{carlini_wagner}, i.e., for all $\hat{x}, \delta \in \mathbb{R}^{d_{\text{img}}}$,
\begin{equation*}
    \psi(\hat{x}, \delta) = \frac{1}{2} \tanh{\left(\tanh^{-1}{\left(2\hat{x} \right)} +\delta \right)}.
\end{equation*} 
We considered $D$ to be the $\ell_1$ norm as the distortion measure, i.e., for all $\hat{x}, \delta \in \mathbb{R}^{d_{\text{img}}}$,
\begin{equation*}
    D(\hat{x}, \delta) = \| \hat{x} - \psi(\hat{x}, \delta)\|_1.
\end{equation*}
The accuracy criterion $L$ is fixed to be the black-box untargeted loss proposed in~\cite{zoo}, i.e.,
\begin{equation*}
    L\left(y,\hat{f} \left(\hat{x} \right) \right) = \max\left(\max\limits_{t \neq y} \log \left[\hat{f}(\hat{x}) \right]_t  - \log \left[ \hat{f}(\hat{x}) \right]_y, -\kappa \right),
\end{equation*}
with $\kappa = -1.0$. Informally, this loss measures how much the confidence of predicting the true class $y$ is smaller than the largest confidence of classfing the input in one of the other classes. If the true class already has the lowest confidence among all classes, the loss is clipped at $-\kappa$ to prevent excessively negative values. By minimizing this loss, the optimizer increases the relative confidence of an incorrect class compared to the true class, encouraging missclassification. We therefore consider the following bilevel optimization problem.
\begin{equation*}
    \begin{aligned}
        \min\limits_{A \in \mathbb{R}^{d_{\text{img}} \times p }}& \frac{1}{m} \sum\limits_{i=1}^m \| \hat{x}_i - \psi(\hat{x}_i, Az^*(A))\|_1 \quad \text{s.t.} \quad  z^*(A) \in \arg\min\limits_{z \in \mathbb{R}^{p}} \frac{1}{m} \sum\limits_{i=1}^m L(y_i,\hat{f}(\psi(\hat{x}_i, Az))).
    \end{aligned}
\end{equation*}
In our experiments, since we considered the MNIST dataset, the number of pixels in the images is $d_{\text{img}} = 784$, and we fixed $p = 100$.

\paragraph*{Data Preprocessing.} We downloaded the training and test sets of MNIST dataset \cite{mnist_dataset} and we normalized the images to be in $[-0.5, 0.5]^{d_\text{img}}$ as suggested in previous works - see e.g. \cite{liu_svr,rando2025structuredtouroptimizationfinite}.

\paragraph*{Network Architecture.} The multiclass classifier has been implemented as a convolutional neural network (CNN), following the architecture described in~\cite{rando2025structuredtouroptimizationfinite}. The network consists of five convolutional layers with ReLU activations, interleaved with max-pooling layers, and followed by fully connected layers.  Specifically, the first two convolutional layers use $3 \times 3$ kernels. The first layer has $1$ input channel and $32$ output channels, while the second layer has $32$ input channels and $64$ output channels. These two layers are followed by a $2 \times 2$ max-pooling layer with stride $2$. The third and fourth convolutional layers have $64$ input and output and use $3 \times 3$ kernels. A second $2 \times 2$ max-pooling layer with stride $2$ is applied thereafter. The resulting feature maps are flattened and passed through two fully connected layers, each with $200$ hidden units and ReLU activation. The final output layer is a fully connected layer with $10$ units and a softmax activation, corresponding to the $10$ classes of the MNIST dataset. The network architecture is summarized in Table~\ref{tab:network_architecture}.

\begin{table}[H]
\centering
\caption{CNN architecture}\label{tab:network_architecture}
\begin{tabular}{llll}
\toprule
Layer & Type & Parameters & Activation \\
\midrule
1 & Convolutional & in=$1$, out=$32$, kernel=$3\times3$ & ReLU \\
2 & Convolutional & in=$32$, out=$64$, kernel=$3\times 3$ & ReLU  \\
3 & Max Pooling & kernel=$2 \times 2$, stride=$2$ & --  \\
4 & Convolutional & in=$64$, out=$64$, kernel=$3 \times 3$ & ReLU  \\
5 & Convolutional & in=$64$, out=$64$, kernel=$3 \times 3$ & ReLU  \\
6 & Max Pooling & kernel=$2 \times 2$, stride=$2$ & --  \\
7 & Dense & in=$1024$, out=$200$ & ReLU  \\
8 & Dense & in=$200$, out=$200$ & ReLU  \\
9 & Dense & in=$200$, out=$10$ & Softmax  \\
\bottomrule
\end{tabular}
\end{table}

\paragraph*{Training Details.} We describe here the training procedure and hyperparameters used to train the classifier. In particular, we adopt the defensive distillation technique introduced in~\cite{defensive_distillation} to improve robustness against adversarial perturbations. Defensive distillation is a training procedure based on a teacher-student framework. Let $\mathcal{D} = \{(\hat{x}_i, y_i)\}_{i=1}^n$ denote the labeled training dataset, where $\hat{x}_i \in \mathbb{R}^{d_{\text{img}}}$ and $y_i \in \{1, \dots, C\}$ are the data points and labels respectively, with $C > 1$ denoting the number of classes. In the first stage, a teacher network $\hat{f}_{\text{teacher}}: \mathbb{R}^{d_{\text{img}}} \rightarrow \mathbb{R}^C$ is trained using a softened softmax output with temperature $T > 1$. Specifically, given the logits $x^{\text{log}} \in \mathbb{R}^C$ produced by the network for an input $\hat{x}$, the temperature-scaled softmax is defined as
\begin{equation*}
[\text{softmax}_T(x^{\text{log}})]_i = \frac{\exp(x^{\text{log}}_i / T)}{\sum_{j=1}^C \exp(x^{\text{log}}_j / T)}, \quad i = 1, \dots, C.    
\end{equation*}
The teacher model is trained by minimizing the standard cross-entropy loss between the softened predictions and the ground-truth labels. In the second stage, the trained teacher network is used to generate soft labels $\bar{y}_i = \text{softmax}_T(\hat{f}_{\text{teacher}}(\hat{x}_i))$ for $i = 1, \dots, n$ and a student network $\hat{f}_{\text{student}}$, with the same architecture as the teacher, is then trained to mimic such soft predictions by minimizing the cross-entropy loss with respect to these instead of the hard labels $y_i$. After training, the temperature is reset to $T = 1$ and the student model is used for inference. In our experiments, the temperature was set to $T = 100$ during both the teacher and student training stages, as suggested in~\cite{defensive_distillation}. Both the student and teacher networks are initialized from the weights obtained after one SGD step with temperature $T=1$, starting from random initialization, with all weights using Kaiming uniform initialization \cite{He_2015_ICCV} and biases set to zero. The training dataset is split into two disjoint subsets, with $80\%$ of the data used for training and $20\%$ for validation. All models were trained on training part using the SGD optimizer with learning rate $0.01$, momentum $0.9$, and batch size $32$. Early stopping based on the validation loss was employed, with a patience of $5$ epochs, as suggested in \cite{rando2025structuredtouroptimizationfinite}. Training was performed for up to $100$ epochs. Additionally, dropout with rate $0.8$ was applied to the last hidden layer before the output layer. A summary of the hyperparameters is provided in Table~\ref{tab:training_params}. After training, we assessed the performance of the student model by evaluating its training and test accuracy, achieving a training accuracy of $99.16\%$ and a test accuracy of $98.87\%$.%
\begin{table}[h]
    \centering
    \caption{Training hyperparameters}\label{tab:training_params}
    \begin{tabular}{ll}
    \toprule
    Parameter &   \\
    \midrule
    Optimizer & SGD with momentum \\
    Learning rate & $0.01$ \\
    Momentum & $0.9$ \\
    Batch size & $32$ \\
    Max epochs & $100$ \\
    Temperature & $100$ \\
    Early stopping & Patience of $5$ epochs (on validation loss) \\
    Dropout & $0.8$  \\
    \bottomrule
    \end{tabular}

\end{table}

\paragraph*{Parameter Tuning.} The procedure to tune parameters of the optimization algorithms is the same of the one adopted in the synthetic problem. Parameters are tuned with grid search. For every algorithm and every parameter configuration of the respective grid, we run such an algorithm $6$ times with a fixed budget of $10^5$ function evaluations. Since we cannot access to the function values of $\Psi$, we selected the best configuration of parameters for each algorithm as the one which minimize the average classification accuracy all over the runs computed at the last iteration performed. Initializations of $x_0$ and $z_0$ are sampled from Gaussian distributions with $0$ mean and $0.5$ of standard deviation while the initialization of $v_0$ is set to the zero vector. As in the synthetic problem, algorithms with optimal configuration are re-executed $5$ times to produce the plots in Section \ref{sec:experiments}. Parameter grids used are the same for the synthetic problems (see Table \ref{tab:parameter_grids}) with the only difference that the stepsize parameters are tuned over the set $\{10^{-7}, 10^{-6}, 10^{-5}, 10^{-4}, 10^{-3}, 10^{-2} \}$.

\paragraph*{Additional Experiments.} We repeated the adversarial experiment of Section \ref{sec:experiments} by constructing universal perturbations for images labeled with $y=1$, $y=6$, and $y=9$. In Figure \ref{fig:app_other_adv_exps}, we report, for each group of images with labels $1$, $6$, $4$ (used in Section \ref{sec:experiments}), and $9$, the normalized outer (distortion) and inner (confidence) objective values along the sequences ${x_k}$ and ${z_k}$, together with classification accuracy, as functions of the number of stochastic function evaluations (top row of each plot) and wall-clock time (bottom row of each plot). As in Section \ref{sec:experiments}, values plotted against the number of stochastic evaluations are repeated according to the number of evaluations performed per iteration. From Figure \ref{fig:app_other_adv_exps}, we observe again that our methods achieve comparable performance to state-of-the-art algorithms in significantly less time. In particular, ZOBA requires a large number of function evaluations per iteration, since obtaining reliable estimators of Hessian blocks is costly. This limits performance in budget-constrained settings, where only a few updates can be performed. In contrast, HF-ZOBA avoids this limitation by requiring fewer function evaluations per iteration, as it approximates Hessian–vector products using finite differences rather than explicitly estimating Hessian blocks. Moreover, by leveraging delayed information and a single-loop structure, HF-ZOBA reuses function evaluations to compute multiple terms of the search direction while retaining the benefits of parallelization. As a result, HF-ZOBA runs faster than other approaches while maintaining comparable performance, and without relying on a regularized version of the problem (as in Opt-ZMDSBA \cite{aghasi2025optimalzerothorderbileveloptimization}), which can introduce additional bias and reduce performance.

\begin{figure}[H]
    \centering
    \includegraphics[width=0.49\linewidth]{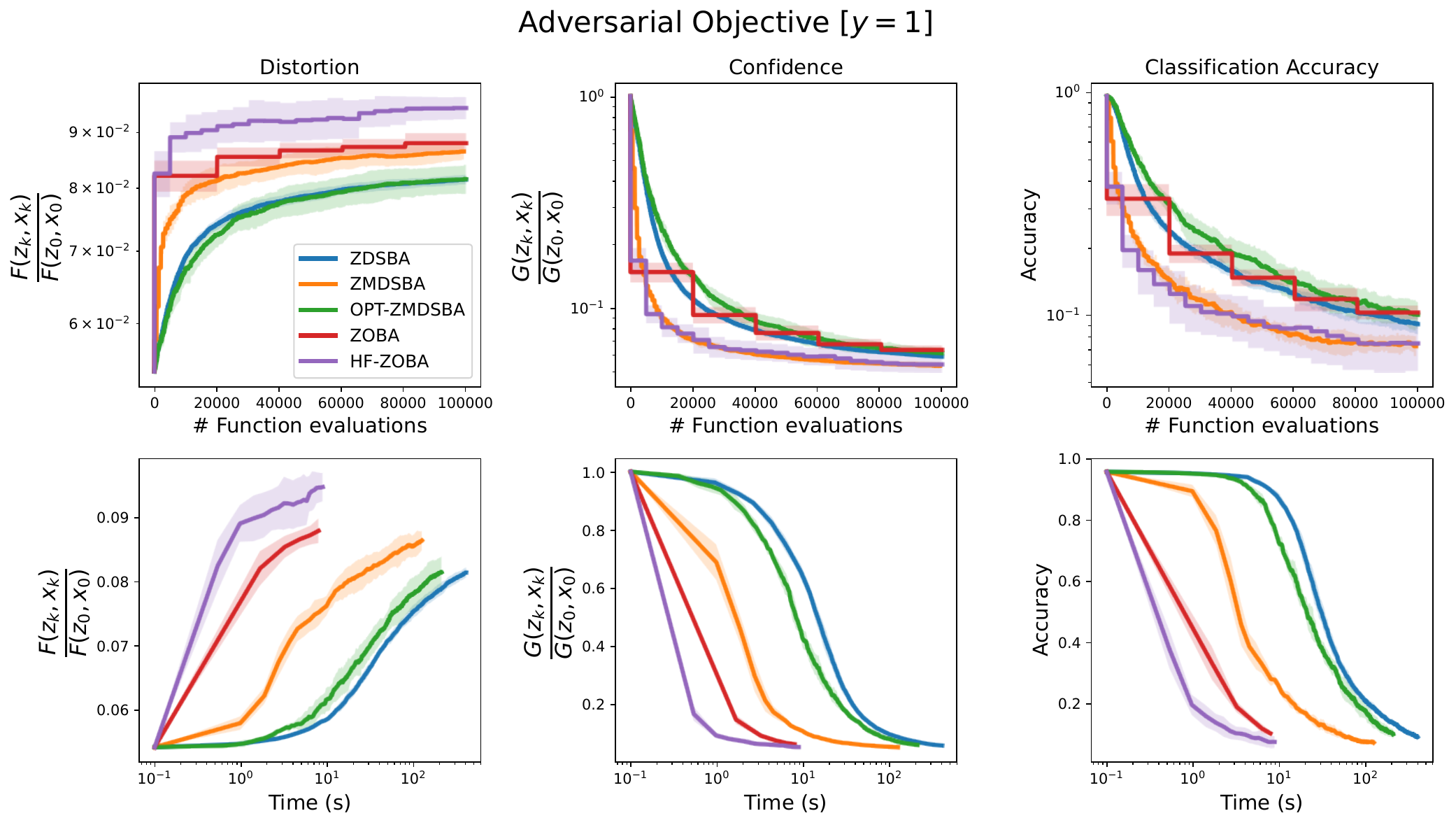}
    \includegraphics[width=0.49\linewidth]{img/adversarial/appendix/Adversarial_4_100000_5_100_comparison.pdf}
    \includegraphics[width=0.49\linewidth]{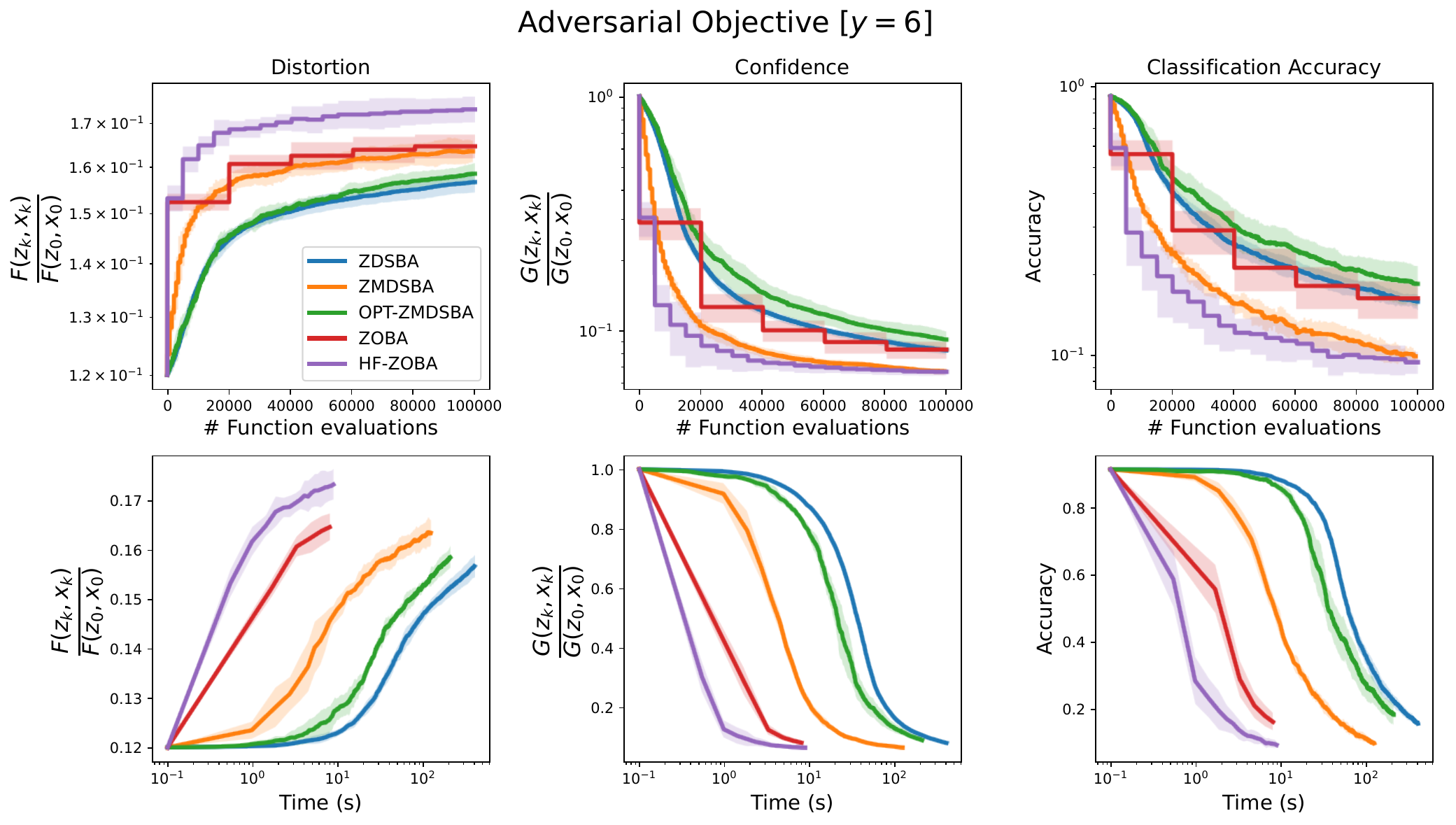}
    \includegraphics[width=0.49\linewidth]{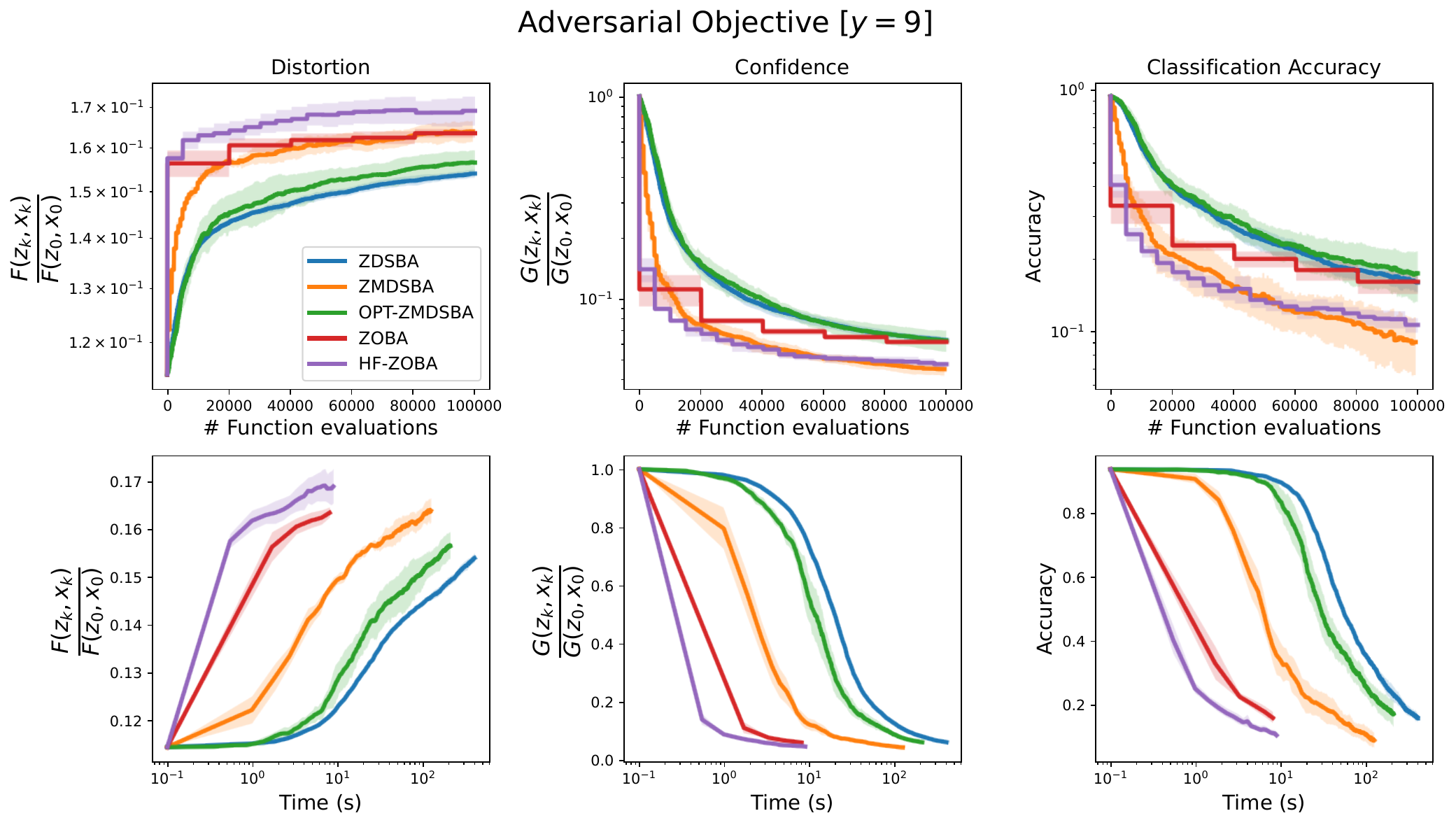}
    \caption{Comparison of algorithms on minimal-distortion universal perturbation attacks across different test sets with different labels $y$. For each test set, we report the normalized outer function values, inner function values at iterations $z_k, x_k$, and classification accuracy. The top row of each block shows these metrics as a function of the number of function evaluations, while the bottom row shows them versus wall-clock time (in seconds).}
    \label{fig:app_other_adv_exps}
\end{figure}

\section{Auxiliary Results}\label{app:aux_results}
In this appendix, we introduce the main concepts and report and prove the lemmas and propositions required to prove the main theorems. 
\paragraph*{Notation.} In the following, we denote by $I_p$ the identity matrix $p \times p$. For a function $f : \mathbb{R}^p \times \mathbb{R}^d \to \mathbb{R}$, we indicate with $\nabla_z f(z,x)$ and $\nabla_x f(z,x)$, the gradient with respect to the first and second variable respectively. Moreover, we denote with $\mathbb{E}_k[\cdot]$ the conditional expectation conditioning on the entire history of every random variable $(w_{t}^{(i,j)})_{t < k}, (u_{t}^{(i,j)})_{t < k}, (\xi_{i,t})_{t < k}$ and $(\zeta_{i,t})_{t < k}$ for every $i=1,\cdots, b_1$ and $j=1,\cdots,\ell_1$ i.e. the expectation is taken only on random variables at iteration $k$
\begin{equation*}
    \mathbb{E}_k[\cdot] := \mathbb{E}_{\xi_{1,k},\zeta_{1,k}}\left[\cdots \left[\mathbb{E}_{\xi_{b_1,k},\zeta_{b_1,k}} \left[ \mathbb{E}_{w_k^{(1,1)},u_k^{(i,j)}} \left[\cdots\mathbb{E}_{w_k^{(b_1,\ell_1)},u_k^{(b_1,\ell_1)}} \left[\cdot \right] \right] \right] \right] \right].
\end{equation*}
Moreover, we denote with $\mathbb{E}_{W_k^{(i)}}[\cdot]$ the conditional expectation on $w_k^{(i,j)}$ for $j = 1,\cdots,\ell_1$ i.e. the expectation is taken only on $w_k^{(i,j)}$ for every $j = 1,\cdots,\ell_1$ (i.e. $i$ is fixed)
\begin{equation*}
    \mathbb{E}_{W_k^{(i)}}[\cdot] :=  \mathbb{E}_{w_k^{(i,1)}} \left[ \mathbb{E}_{w_k^{(i,2)}}\left[\cdots \mathbb{E}_{w_k^{(i,\ell_1)}}\left[ \cdot \right] \right]  \right].
\end{equation*}
Similarly, we denote with $\mathbb{E}_{W_k}[\cdot]$ the conditional expectation on $w_k^{(i,j)}$ for every $i=1,\cdots,b_1$ and $j = 1,\cdots,\ell_1$ i.e. the expectation is taken only on $w_k^{(i,j)}$ for every $i=1,\cdots,b_1$ and $j = 1,\cdots,\ell_1$
\begin{equation*}
    \begin{aligned}
        \mathbb{E}_{W_k}[\cdot] &:= \mathbb{E}_{W_k^{(1)}} \left[ \cdots \mathbb{E}_{W_k^{(b_1)}} \left[ \cdot \right] \right] =\mathbb{E}_{w_k^{(1,1)}} \left[ \mathbb{E}_{w_k^{(1,2)}}\left[\cdots \mathbb{E}_{w_k^{(1,\ell_1)}}\left[ \mathbb{E}_{w_k^{(2,1)}} \left[\cdots\mathbb{E}_{w_k^{(b_1,\ell_1)}} \left[\cdot \right] \right] \right] \right]  \right].
    \end{aligned}
\end{equation*}
For $i = 1,\cdots, b_2$, we denote with 
\begin{equation*}
    \begin{aligned}
        \mathbb{E}_{U_k^{(i)}}[\cdot] = \mathbb{E}_{u_k^{(i,1)}} \left[ \mathbb{E}_{u_k^{(i,2)}}\left[\cdots \mathbb{E}_{u_k^{(i,\ell_2)}}\left[ \cdot \right] \right]  \right] \quad \text{and} \quad \mathbb{E}_{U_k}[\cdot] = \mathbb{E}_{U_k^{(1)}} \left[ \cdots \mathbb{E}_{U_k^{(b_2)}} \left[ \cdot \right] \right].
    \end{aligned}
\end{equation*}
Similarly, we denote with 
\begin{equation*}
    \mathbb{E}_{\xi_k} [\cdot] = \mathbb{E}_{\xi_{1,k}}[\cdots \mathbb{E}_{\xi_{b_1,k}}[\cdot]] \quad \text{and} \quad \mathbb{E}_{\zeta_k} [\cdot] = \mathbb{E}_{\zeta_{1,k}}[\cdots \mathbb{E}_{\zeta_{b_2,k}}[\cdot]].
\end{equation*}

\begin{lemma}\label{lem:norm_gaus_vector}
    Let $u \sim \mathcal{N}(0,I_d)$. Then, for $p \in [0,2]$,
    \begin{equation*}
        \mathbb{E}_u[\|u\|^p] \leq d^{p/2}.
    \end{equation*}
    For $p \geq 2$,
    \begin{equation*}
        \mathbb{E}_u[\|u\|^p] \leq (d + p)^{p/2}.
    \end{equation*}
    Moreover, let $A,B \in \mathbb{R}^{d \times d}$ be symmetric matrices. Then,
    \begin{equation*}
        \begin{aligned}
            \mathbb{E}_u[u^\top A u] &= tr(A), \quad \mathbb{E}_u[u^\top A u u^\top B u] = 2tr(AB) + tr(A)tr(B),\\
            \mathbb{E}_u[uu^\top u u^\top] &= (d + 2)I_d.
        \end{aligned}
    \end{equation*}
    where $tr(A)$ denotes the trace of matrix $A$.
\end{lemma}
\begin{proof}
    These are standard results for Gaussian variables and can be found in several textbooks - see e.g. \cite{nesterov_random_2017,petersen2008matrix,seber2003linear}.%
\end{proof}

\noindent The following appendices collect auxiliary results used in the proofs of the main theorems and corollaries. In Appendix \ref{app:aux_results_bilevel}, we present lemmas establishing bounds and regularity properties of the quantities arising in bilevel optimization under the considered assumptions. In Appendix \ref{app:gaussian_smoothing}, we review the Gaussian smoothing framework for the bilevel problem and states the associated lemmas and properties used to analyze our algorithms. In Appendix \ref{app:aux_res_zoba}, we provide bounds on the sequences generated by Algorithm \ref{alg:zoba} together with the corresponding descent lemma and in Appendix \ref{app:aux_res_hfzoba} we report analogous results for Algorithm \ref{alg:hfzoba}.

\subsection{Auxiliary results for Bilevel Optimization}\label{app:aux_results_bilevel}
\begin{lemma}[Regularity of $\Psi$, $z^*$ and $v^*$]\label{lem:smt_setting_smooth_value}
    Under Assumptions \ref{asm:F_smooth},\ref{asm:G_asm}, the following holds:\\
    (I) the function $\Psi$ is $L_{\Psi}$-smooth for some $L_{\Psi} > 0$.\\
    (II) there exists $L_* > 0$ such that for every $x_1,x_2 \in \mathbb{R}^d$,
    \begin{equation*}
        \| z^*(x_1) - z^*(x_2)\| \leq L_* \|x_1 - x_2 \| \quad \text{and} \quad \|v^*(x_1) - v^*(x_2)\| \leq L_* \|x_1 - x_2 \|,
    \end{equation*}
    where $z^*(\cdot) \in \argmin\limits_{z \in \mathbb{R}^p} G(z,\cdot)$ and $v^*(\cdot)$ is defined in eq. \eqref{eqn:v_star}.
\end{lemma}
\begin{proof}
    The point (I)  has been proved in \cite{ghadimi2018approximation} and (II) has been proved in \cite{soba_saba}.%
\end{proof}

\begin{lemma}[Bound on norm of $v^*$]\label{lem:bound_norm_v}
    Under Assumptions \ref{asm:F_smooth}, \ref{asm:G_asm}. Let $v^*(\cdot)$ be defined as in eq. \eqref{eqn:v_star}. Then, for every $x \in \mathbb{R}^d$,
    \begin{equation*}
        \| v^*(x) \| \leq \frac{L_{0,F}}{\mu_G}.
    \end{equation*}
\end{lemma}
\begin{proof}
    Since $v^*(x) = - [\nabla_{zz}^2 G(z^*(x),x)]^{-1} \nabla_z F(z^*(x),x)$, $G$ is $\mu_G$-strongly convex and $F$ is $L_{0,F}$-Lipschitz continuous, we have
    \begin{equation*}
        \left\| v^*(x) \right\| = \left\| [\nabla_{zz}^2 G(z^*(x),x)]^{-1} \nabla_z F(z^*(x),x) \right\| \leq \frac{L_{0,F}}{\mu_G}.
    \end{equation*}
\end{proof}

\subsection{Gaussian Smoothing for Bilevel Optimization}\label{app:gaussian_smoothing}
The analysis of Algorithm \ref{alg:zoba} and \ref{alg:hfzoba} rely on the fact that the surrogates used to construct the search directions (in eq. \eqref{eqn:dz}, \eqref{eqn:dv} and \eqref{eqn:dx} for Algorithm \ref{alg:zoba} and eq. \eqref{eqn:hfzoba_search_directions} for Algorithm \ref{alg:hfzoba}) are unbiased estimators of gradients and Hessians of a smooth approximation of the inner and outer objective functions $f$ and $g$. More precisely, let $(\Omega, \mathcal{F}, \mathbb{P})$ be a probability space and $\mathcal{Z}$ be a measurable space. Let $f : \mathbb{R}^p \times \mathbb{R}^d \times \mathcal{Z} \to \mathbb{R}$ and let $\zeta : \Omega \to \mathcal{Z}$ be a random varianble. Then for any  $h, \eta \geq 0$, we define the following smoothed version of $f$ for every $z \in \mathbb{R}^p$, $x \in \mathbb{R}^d$ and every realization of $\zeta$.
\begin{equation}\label{eqn:gaussian_smoothing}
f_{h,\eta}(z, x, \zeta)= \mathbb{E}_{w, u}\left[f(z + h w, x + \eta u, \zeta)\right],
\end{equation}
where $w \sim \mathcal{N}(0, I_p)$ and $u \sim \mathcal{N}(0, I_d)$. For $h, \eta > 0$, the function $f_{h,\eta}$ is differentiable even when $f$ is not, and its properties depends on those of $f$ \cite{Aghasi2025} - see Proposition \ref{prop:smoothing_properties}. Note that this construction differs from the standard Gaussian smoothing considered in~\cite{nesterov_random_2017}, as it applies distinct smoothing parameters to different variable blocks. This flexibility is advantageous in bilevel optimization, where we want to compute gradients and hessian only a of subset of variables. 
In the following proposition we show some properties of the smoothing.
\begin{proposition}[Properties of the smoothing]\label{prop:smoothing_properties}
Let $(\Omega, \mathcal{F}, \mathbb{P})$ be a probability space and let $\mathcal{Z}$ be a measurable space. Let $f : \mathbb{R}^p \times \mathbb{R}^d \times \mathcal{Z} \to \mathbb{R}$ and let $f_{h,\eta}$ be the smooth surrogate of $f$ defined as in eq. \eqref{eqn:gaussian_smoothing}. Then, the following hold:\\
\noindent (I) If $f$ is convex (or $\mu_f$ strongly convex) then $f_{h,\eta}$ is convex (or $\mu_f$ strongly convex).\\
\noindent (II) If $f$ is $L_0$-Lipschitz then we have that $f_{h,\eta}$ is $L_0$-Lipschitz.\\
\noindent (III) If $f$ is $L_1$-smooth then we have that $f_{h,\eta}$ is $L_1$-smooth.

\end{proposition}
\begin{proof}
The results on convexity, Lipschitz continuity, and smoothness were proved in \cite{Aghasi2025}. 
\end{proof}
\noindent In the following lemma we show that gradients and hessians of the smoothed approximation of a target can be expressed as expectation of finite difference.
\begin{lemma}[Gradient and Hessian of Smoothing]\label{lem:grad_hess_smoothing}
   Let $(\Omega, \mathcal{F}, \mathbb{P})$ be a probability space and let $\mathcal{Z}$ be a measurable space. Let $\zeta : \Omega \to \mathcal{Z}$, and let $w \sim \mathcal{N}(0, I_p)$ and $u \sim \mathcal{N}(0, I_d)$. Let $f : \mathbb{R}^p \times \mathbb{R}^d \times \mathcal{Z} \to \mathbb{R}$, and let $f_{h,\eta}$ denote the smooth surrogates of $f$ defined as in eq.~\eqref{eqn:gaussian_smoothing}. Then, for every $z \in \mathbb{R}^p$, and $x \in \mathbb{R}^d$,
    \begin{equation}\label{eqn:smt_lem_grad}
        \begin{aligned}
            \nabla_z f_{h,\eta}(z,x,\zeta) &=\mathbb{E}_{w,u} \left[ \frac{f(z + hw, x + \eta u,\zeta) }{h}w \right] \quad \text{and} \quad\nabla_x f_{h,\eta}(z,x) = \mathbb{E}_{w,u} \left[ \frac{f(z + hw, x + \eta u,\zeta) }{\eta} u \right].
        \end{aligned}
    \end{equation}    
    Moreover, 
    \begin{equation}\label{eqn:smt_lam_hess}
        \begin{aligned}
            \nabla_{zz}^2 f_{h,\eta}(z,x,\zeta) &=\mathbb{E}_{w,u} \left[ \frac{f(z + hw, x + \eta u, \zeta)}{h^2}(ww^\top - I_p) \right],\\
            \nabla_{xz}^2 f_{h,\eta}(z,x,\zeta) &=\mathbb{E}_{w,u} \left[ \frac{f(z + hw, x + \eta u, \zeta)}{\eta h}uw^\top \right].
        \end{aligned}
    \end{equation}    
\end{lemma}
\begin{proof}
    Although this result was stated in \cite{Aghasi2025}, no explicit proof was given. We therefore provide a proof here for completeness. By the definition of the smooth surrogate (eq. \eqref{eqn:gaussian_smoothing}), we have
    \begin{equation*}
        f_{h,\eta}(z,x,\zeta) = \mathbb{E}_{w,u}[f(z + hw, x + \eta u, \zeta)].
    \end{equation*}
    Since $w,u$ are independent and $f$ is absolutely measurable, by Fubini's Theorem, we have
    \begin{equation}\label{eqn:smt_lem_eq1}
        f_{h,\eta}(z,x,\zeta) = \mathbb{E}_{w,u}[f(z + hw, x + \eta u, \zeta)] = \mathbb{E}_{w}[ \mathbb{E}_u[f(z + hw, x + \eta u, \zeta)]].
    \end{equation}
    Now, we prove eq. \eqref{eqn:smt_lem_grad}. Let $y = z + hw$. Thus, by substitution, we have 
    \begin{equation*}
        \begin{aligned}
        f_{h,\eta}(z,x,\zeta) &= \frac{1}{(\sqrt{2\pi})^{p/2}}\int  \mathbb{E}_u[f(z + hw, x + \eta u, \zeta)] e^{-\frac{\|w\|^2}{2}} dw\\
        &= \frac{1}{(\sqrt{2\pi})^{p/2}}\int  \mathbb{E}_u[f(y, x + \eta u, \zeta)] e^{-\frac{\|y - x\|^2}{2h^2}} \frac{dy}{h^p}.    
        \end{aligned}
    \end{equation*}
    Therefore, computing the gradient in $z$, we have
    \begin{equation*}
        \begin{aligned}
        \nabla_z f_{h,\eta}(z,x,\zeta) &= \frac{1}{(\sqrt{2\pi})^{p/2}}\int  \mathbb{E}_u[f(y, x + \eta u, \zeta)] \nabla_z e^{-\frac{\|y - z\|^2}{2h^2}} \frac{dy}{h^p}\\
        &=\frac{1}{(\sqrt{2\pi})^{p/2}}\int  \mathbb{E}_u[f(y, x + \eta u, \zeta)] \left(\frac{y - z}{h^2} \right)  e^{-\frac{\|y - z\|^2}{2h^2}} \frac{dy}{h^p}\\
        &=\frac{1}{(\sqrt{2\pi})^{p/2}}\int  \mathbb{E}_u \left[ \frac{f(z + hw, x + \eta u, \zeta)}{h} w \right]  e^{-\frac{\|w\|^2}{2}} dw\\
        &=\mathbb{E}_{w,u} \left[\frac{f(z + hw, x + \eta u, \zeta)}{h}w \right]. 
        \end{aligned}
    \end{equation*}
    Similarly we can compute the gradient in $x$ i.e. $\nabla_x f_{h,\eta}(z,x,\zeta)$. Now, we prove eq. \eqref{eqn:smt_lam_hess}. Setting $y = z + hw$, by substitution,
    \begin{equation}\label{eqn:gaus_smt_sub1}
        \begin{aligned}
        f_{h,\eta}(z,x,\zeta) &=\frac{1}{(\sqrt{2\pi})^{p/2}}\int  \mathbb{E}_u[f(y, x + \eta u, \zeta)] e^{-\frac{\|y - z\|^2}{2h^2}} \frac{dy}{h^p}.    
        \end{aligned}
    \end{equation}
    Computing the Hessian block in $zz$, we have
    \begin{equation*}
        \begin{aligned}
        \nabla_{zz}^2f_{h,\eta}(z,x,\zeta) &=\frac{1}{(\sqrt{2\pi})^{p/2}}\int  \mathbb{E}_u[f(y, x + \eta u, \zeta)] \nabla_{zz}^2 \left(e^{-\frac{\|y - z\|^2}{2h^2}} \right) \frac{dy}{h^p}\\
        &=\frac{1}{(\sqrt{2\pi})^{p/2}}\int  \mathbb{E}_u[f(y, x + \eta u, \zeta)] \left( \frac{(y - z)(y - z)^\top}{h^4} - \frac{1}{h^2}I_p \right)e^{-\frac{\|y - z\|^2}{2h^2}}  \frac{dy}{h^p}\\
        &=\frac{1}{(\sqrt{2\pi})^{p/2}}\int  \mathbb{E}_u \left[\frac{f(y, x + \eta u, \zeta)}{h^2}  \left( ww^\top-I_p \right) \right]e^{-\frac{\|w\|^2}{2}}  dw\\
        &=\mathbb{E}_{w,u} \left[\frac{f(y, x + \eta u, \zeta)}{h^2}  \left( ww^\top-I_p \right) \right].
        \end{aligned}
    \end{equation*}
    Now, to get the expression of the cross-block, we restart from eq. \eqref{eqn:gaus_smt_sub1}. Let $\bar{y}=x + \eta u$, by substitution we have
    \begin{equation*}
        \begin{aligned}
        f_{h,\eta}(z,x,\zeta) &=\frac{1}{(\sqrt{2\pi})^{p/2}} \frac{1}{(\sqrt{2\pi})^{d/2}}\int \left(\int f(y, x + \eta u, \zeta)e^{-\frac{\|u\|^2}{2}}du \right)e^{-\frac{\|y - z\|^2}{2h^2}} \frac{dy}{h^p}\\
        &=\frac{1}{(\sqrt{2\pi})^{p/2}} \frac{1}{(\sqrt{2\pi})^{d/2}}\int \left(\int f(y, \bar{y}, \zeta) e^{-\frac{\|\bar{y}-x\|^2}{2\eta^2}} \frac{d\bar{y}}{\eta^d} \right)e^{-\frac{\|y - z\|^2}{2h^2}} \frac{dy}{h^p}\\
        &=\frac{1}{(\sqrt{2\pi})^{p/2}} \frac{1}{(\sqrt{2\pi})^{d/2}}\int \int f(y, \bar{y}, \zeta) e^{-\frac{\|\bar{y}-x\|^2}{2\eta^2}}  e^{-\frac{\|y - z\|^2}{2h^2}} \frac{d\bar{y}}{\eta^d} \frac{dy}{h^p}.
        \end{aligned}
    \end{equation*}
    Computing the mixed block $xz$ we have,
    \begin{equation*}
        \begin{aligned}
        \nabla_{xz}^2 f_{h,\eta}(z,x,\zeta) &= \frac{1}{(\sqrt{2\pi})^{p/2}} \frac{1}{(\sqrt{2\pi})^{d/2}}\int \int f(y, \bar{y}, \zeta) \nabla_{xz}^2\left(e^{-\frac{\|\bar{y}-x\|^2}{2\eta^2}}  e^{-\frac{\|y - z\|^2}{2h^2}} \right) \frac{d\bar{y}}{\eta^d} \frac{dy}{h^p}\\
        &= \frac{1}{(\sqrt{2\pi})^{p/2}} \frac{1}{(\sqrt{2\pi})^{d/2}}\int \int f(y, \bar{y}, \zeta) \left( \frac{(y - z) (\bar{y} - x)^\top}{h^2 \eta^2} \right) e^{-\frac{\|\bar{y}-x\|^2}{2\eta^2}}  e^{-\frac{\|y - z\|^2}{2h^2}}  \frac{d\bar{y}}{\eta^d} \frac{dy}{h^p}\\
        &= \frac{1}{(\sqrt{2\pi})^{p/2}} \frac{1}{(\sqrt{2\pi})^{d/2}}\int \int \frac{f(z + hw, x + \eta u, \zeta)}{h \eta} uw^\top  e^{-\frac{\|u\|^2}{2}}  e^{-\frac{\|w\|^2}{2}}  du dw\\
        &=\mathbb{E}_{w,u} \left[ \frac{f(z + hw, x + \eta u, \zeta)}{h \eta} uw^\top \right].
        \end{aligned}
    \end{equation*}
\end{proof}
\begin{remark}
    Notice that by symmetry, the gradient of smoothing can be expressed as expectation of forward or central finite difference
    \begin{equation*}
        \begin{aligned}
            \nabla_z f_{h,\eta}(z,x,\zeta) &=\mathbb{E}_{w,u} \left[ \frac{f(z + hw, x + \eta u,\zeta) }{h}w \right]\\
            &= \mathbb{E}_{w,u} \left[ \frac{f(z + hw, x + \eta u,\zeta) - f(z,x,\zeta) }{h}w \right]\\
            &= \mathbb{E}_{w,u} \left[ \frac{f(z + hw, x + \eta u,\zeta) - f(z - hw, x - \eta u,\zeta) }{2h}w \right].
        \end{aligned}
    \end{equation*}
    Analogously, we can express the gradient with respect to the second variable.   Moreover, also the hessian blocks can be expressed as finite difference
    \begin{equation*}
        \begin{aligned}
            \nabla_{zz}^2 f_{h,\eta}(z,x,\zeta) &=\mathbb{E}_{w,u} \left[ \frac{f(z + hw, x + \eta u, \zeta)}{h^2}(ww^\top - I_p) \right]\\
            &=\mathbb{E}_{w,u} \left[ \frac{f(z + hw, x + \eta u, \zeta) + f(z - hw, x - \eta u, \zeta) - 2 f(z,x,\zeta)}{2h^2}(ww^\top - I_p) \right],\\
            \nabla_{xz}^2 f_{h,\eta}(z,x,\zeta) &=\mathbb{E}_{w,u} \left[ \frac{f(z + hw, x + \eta u, \zeta)}{\eta h}uw^\top \right]\\
            &=\mathbb{E}_{w,u} \left[ \frac{f(z + hw, x + \eta u, \zeta) + f(z - hw, x - \eta u, \zeta) - 2 f(z,x,\zeta)}{2\eta h}uw^\top \right].
        \end{aligned}
    \end{equation*}
\end{remark}
\noindent Notice that if we define $F(z,x) = \mathbb{E}_\zeta[f(z,x,\zeta)]$, we can define the smoothed surrogate $F_{h,\eta}(z,x)$ as in eq. \eqref{eqn:gaussian_smoothing} i.e. let $w \sim \mathcal{N}(0,I_p)$ and $u \sim \mathcal{N}(0, I_d)$,
\begin{equation*}
    F_{h,\eta}(z,x) = \mathbb{E}_{w,u}[F(z + hw,x + \eta u)].
\end{equation*}
Moreover, we have
\begin{equation*}
    \begin{aligned}
    F_{h,\eta}(z,x) &= \mathbb{E}_{w,u}[F(z + hw, x + \eta u)]\\
    &= \mathbb{E}_{w,u}[\mathbb{E}_\zeta[f(z + hw, x + \eta u,\zeta)]]\\
    &= \mathbb{E}_\zeta[\mathbb{E}_{w,u}[f(z + hw, x + \eta u,\zeta)]] = \mathbb{E}_\zeta[f_{h,\eta}(z,x,\zeta)].      
    \end{aligned}
\end{equation*}
\noindent In the following lemma, we report the error of using the gradient of the smoothing instead of the exact gradient of the target.
\begin{lemma}[Smoothing Error]\label{lem:smoothing_error}
    Let $F : \mathbb{R}^{p} \times\mathbb{R}^d \to \mathbb{R}$ be a $L_{1,F}$-smooth function and let $F_{h,\eta}$ denotes its gaussian smoothed surrogate defined as in eq. \eqref{eqn:gaussian_smoothing}. Then, for every $z \in \mathbb{R}^p$ and $x \in \mathbb{R}^d$,
    \begin{equation}\label{eqn:smt_error_c1}
        \begin{aligned}
            (\forall h>0,\eta\geq0) \quad \| \nabla_zF_{h,\eta}(z, x) - \nabla_z F(z,x) \| &\leq \frac{L_{1,F}}{2}(p + 3)^{3/2}h + \frac{L_{1,F} \eta^2}{2h} d \sqrt{p}.\\
            (\forall h\geq0,\eta>0) \quad\| \nabla_xF_{h,\eta}(z, x) - \nabla_x F(z,x) \| &\leq \frac{L_{1,F} p\sqrt{d}}{2} \frac{h^2}{\eta} + \frac{L_{1,F} }{2} (d + 3)^{3/2} \eta.
        \end{aligned}
    \end{equation}
    Let $F$ be a twice differentiable functions with $L_{2,G}$-Lipschitz continuous hessian. Then for every $h > 0$ and $\eta \geq0$ we have
    \begin{equation*}
        \| \nabla_z F_{h,\eta}(z,x) - \nabla_z F(z,x) \| \leq \frac{2 L_{G,2}}{3} \left(h^2(p + 4)^2 + \frac{\eta^3}{h}\sqrt{p} (d + 3)^{2/3}  \right),
    \end{equation*}
    and,
    \begin{equation}\label{eqn:hessian_block1_error}
        \|\nabla^2_{zz} F_{h, \eta}(z,x) - \nabla_{zz}^2 F(z,x) \| \leq \frac{2L_{2,G}}{3}\left(h (d + 5)^{5/2} + h (d +3)^{3/2} + \frac{\eta^3}{h^2}(d + 1) (p +3)^{3/2} \right).
    \end{equation}

\end{lemma}
\begin{proof}
    These results have been proved in \cite{Aghasi2025}.
\end{proof}

\noindent Now, building on the previous lemmas, we introduce a result that bounds the error of a smoothed approximate hypergradient computed for arbitrary $z \in \mathbb{R}^p$ and $v \in \mathbb{R}^p$, rather than for the exact inner solution and the true inverse Hessian action.
\begin{lemma}[Bilevel Smoothing Error]\label{lem:bilevel_smth_error}
    Let Assumptions \ref{asm:F_smooth} and \ref{asm:G_asm} holds. For every $h_1, \eta_1,h_2,\eta_2 >0$ and every $x \in \mathbb{R}^d,z\in\mathbb{R}^p$, let $G_{h_1,\eta_1}(z, x)$ and $F_{h_2,\eta_2}(z,x)$ be the smooth surrogates of the inner and outer targets in eq. \eqref{eqn:problem} respectively. Then, for every $z,v\in \mathbb{R}^p$ and $x \in \mathbb{R}^d$,
    \begin{equation*}
        \begin{aligned}
            \| \nabla^2_{xz}G_{h_1,\eta_1}(z, x)v + \nabla_x F_{h_2,\eta_2}(z,x) - \nabla \Psi(x) \| &\leq  \left(L_{1,F} + \frac{L_{0,F} L_{1,G}}{\mu_G} \right)\|z -z^*(x)\|+  L_{1,G} \|v - v^*(x)\|\\
           &+ \frac{L_{0,F}}{\mu_G}\frac{2L_{2,G}}{3}\bigg(h_1 (d + 5)^{5/2}\\
           &+ h_1 (d +3)^{3/2} + \frac{\eta_1^3}{h_1^2}(d + 1) (p +3)^{3/2} \bigg)\\
           & +\frac{L_{1,F} p\sqrt{d}}{2} \frac{h_2^2}{\eta_2} + \frac{L_{1,F} }{2} (d + 3)^{3/2} \eta_2.\\
        \end{aligned}
    \end{equation*}
\end{lemma}
\begin{proof}
    Since $\nabla \Psi(x) := \nabla_x F(z^*(x),x) + \nabla^2_{2,1} G(z^*(x),x)v^*(x)$, we have
    \begin{equation*}
        \begin{aligned}
        \| \nabla^2_{xz}G_{h_1,\eta_1}(z, x)v + \nabla_x F_{h_2,\eta_2}(z,x) - \nabla \Psi(x) \| &\leq \underbrace{\| \nabla_x F_{h_2,\eta_2}(z,x) - \nabla_x F(z^*(x),x) \|}_{A}\\
        &+ \underbrace{\| \nabla^2_{xz}G_{h_1,\eta_1}(z, x)v - \nabla^2_{xz}G(z^*(x), x)v^*(x)\|}_{B}.
        \end{aligned}
    \end{equation*}
    To bound $A$, we add and subtract $\nabla_x F(z,x)$,
    \begin{equation*}
        \begin{aligned}
            \| \nabla_x F_{h_2,\eta_2}(z,x) - \nabla_x F(z^*(x),x) \| &\leq \| \nabla_x F_{h_2,\eta_2}(z,x) - \nabla_x F(z,x)\| + \|\nabla_x F(z,x) - \nabla_x F(z^*(x),x) \|.
        \end{aligned}
    \end{equation*}
    By eq. \eqref{eqn:smt_error_c1} and $L_{1,F}$-smoothness we have
    \begin{equation*}
        \begin{aligned}
            \| \nabla_x F_{h_2,\eta_2}(z,x) - \nabla_x F(z^*(x),x) \| &\leq \frac{L_{1,F} p\sqrt{d}}{2} \frac{h_2^2}{\eta_2} + \frac{L_{1,F} }{2} (d + 3)^{3/2} \eta_2+L_{1,F}\|z -z^*(x)\|.
        \end{aligned}
    \end{equation*}
    To bound $B$, we add and subtract $\nabla^2_{xz}G_{h_{1},\eta_1}(z,x) v^*(x)$,
    \begin{equation*}
        \begin{aligned}
            \| \nabla^2_{xz}G_{h_1,\eta_1}(z, x)v - \nabla^2_{xz}G(z^*(x), x)v^*(x)\| &\leq \underbrace{\| \nabla^2_{xz}G_{h_1,\eta_1}(z, x)v - \nabla^2_{xz}G_{h_1,\eta_1}(z, x)v^*(x) \|}_{C}\\
            &+\underbrace{\|\nabla^2_{xz}G_{h_1,\eta_1}(z, x)v^*(x) - \nabla^2_{xz}G(z^*(x), x)v^*(x)\|}_{D}.
        \end{aligned}
    \end{equation*}
    We bound $C$, by $L_{1,G}$-smoothness of $G$, 
    \begin{equation*}
        \begin{aligned}
           \| \nabla^2_{xz}G_{h_1,\eta_1}(z, x)v - \nabla^2_{xz}G_{h_1,\eta_1}(z, x)v^*(x) \| &= \|  \nabla^2_{xz}G_{h_1,\eta_1}(z, x)(v - v^*(x)) \|\\
           &\leq \|\nabla^2_{xz}G_{h_1,\eta_1}(z, x) \| \|v - v^*(x)\|\\
           &\leq L_{1,G} \|v - v^*(x)\|.         
        \end{aligned}
    \end{equation*}
    We bound $D$, by adding and subtracting $\nabla^2_{xz}G(z, x)v^*(x)$
    \begin{equation*}
        \begin{aligned}
           \|\nabla^2_{xz}G_{h_1,\eta_1}(z, x)v^*(x) - \nabla^2_{xz}G(z^*(x), x)v^*(x)\| &\leq  
           \|(\nabla^2_{xz}G(z, x) - \nabla^2_{xz}G(z^*(x), x))v^*(x)\|\\
           &+ \|\nabla^2_{xz}G_{h_1,\eta_1}(z, x)v^*(x) - \nabla^2_{xz}G(z, x)v^*(x) \|
        \end{aligned}
    \end{equation*}
    Since hessian of $G$ is $L_{2,G}$-Lipschitz, adding and subtracting $\nabla_{xz}^2 G(z,x)v^*(x)$,
    \begin{equation*}
        \begin{aligned}
           \|\nabla^2_{xz}G_{h_1,\eta_1}(z, x)v^*(x) - \nabla^2_{xz}G(z^*(x), x)v^*(x)\| &\leq  \|(\nabla^2_{xz}G(z, x) - \nabla^2_{xz}G(z^*(x), x))v^*(x)\|\\
           &+ \|\nabla^2_{xz}G_{h_1,\eta_1}(z, x)v^*(x) - \nabla^2_{xz}G(z, x)v^*(x) \|\\
           &\leq L_{1,G}\| z - z^*(x)\| \|v^*(x)\|\\
           &+ \|\nabla^2_{xz}G_{h_1,\eta_1}(z, x) - \nabla^2_{xz}G(z, x)\| \|v^*(x) \|
        \end{aligned}
    \end{equation*}
    By Lemma \ref{lem:smoothing_error} (eq. \eqref{eqn:hessian_block1_error}) and Lemma \ref{lem:bound_norm_v}, we have
    \begin{equation*}
        \begin{aligned}
           \|\nabla^2_{xz}G_{h_1,\eta_1}(z, x)v^*(x) - \nabla^2_{xz}G(z^*(x), x)v^*(x)\| &\leq \frac{L_{0,F} L_{1,G}}{\mu_G}\| z - z^*(x)\|\\
           &+ \frac{L_{0,F}}{\mu_G}\frac{2L_{2,G}}{3}\bigg(h_1 (d + 5)^{5/2} + h_1 (d +3)^{3/2}\\
           &+ \frac{\eta_1^3}{h_1^2}(d + 1) (p +3)^{3/2} \bigg).
        \end{aligned}
    \end{equation*}
    Putting together all these bounds, we get the claim.
\end{proof}

\noindent Notice that in Algorithms \ref{alg:zoba} and \ref{alg:hfzoba}, the search directions are constructed using unbiased estimators of the gradient and Hessian of the smoothed surrogates of the outer and inner objectives  $f$ and $g$. In the following appendices, we study such estimators and bound their approximation errors. In particular, in Appendix \ref{app:aux_res_zoba} we study the approximation of the gradients and Hessians used in Algorithm \ref{alg:zoba}, while in Appendix \ref{app:aux_res_hfzoba} we study the surrogate constructed in Algorithm \ref{alg:hfzoba}.%

\subsection{Auxiliary results for Algorithm \ref{alg:zoba}}\label{app:aux_res_zoba}
In this appendix, we analyze the approximation error of the finite-difference estimators and the search directions constructed in Algorithm \ref{alg:zoba}. To simplify reading, for this section, we introduce some notation that we will use in the proofs. 
\paragraph*{Notation.} Let $b = \max(b_1, b_2)$ and $\ell = \max(\ell_1, \ell_2)$. For every iteration $k \in \mathbb{N}$, let $(\xi_{i,k})_{i=1}^b$ and $(\zeta{i,k})_{i=1}^b$ be independent realizations of the random variables $\xi$ and $\zeta$ sampled at iteration $k$. For every $i = 1, \ldots, b$, let $(w_k^{(i,j)}){j=1}^{\ell}$ and $(u_k^{(i,j)}){j=1}^{\ell}$ be random vectors sampled i.i.d from $\mathcal{N}(0, I_p)$ and $\mathcal{N}(0, I_d)$, respectively, at iteration $k \in \mathbb{N}$. We denote by $\hat{\nabla}_z g_{h_k}(z_k, x_k, \xi{i,k})$, $\hat{\nabla}_z f_{h_k}(z_k, x_k, \zeta_{i,k})$, and $\hat{\nabla}_x f_{h_k}(z_k, x_k, \zeta_{i,k})$ the finite-difference surrogates of the stochastic gradients of the inner and outer objectives $g$ and $f$ with respect to the variables $z$ and $x$, computed at iteration $k \in \mathbb{N}$ using $z_k$, $x_k$, and $\xi_{i,k}$ (or $\zeta_{i,k}$), as follows. %
\begin{equation}\label{eqn:zoba_grad_approximations}
    \begin{aligned}
        \hat{\nabla}_z g_{h_{k}}(z_k,x_k,\xi_{i,k}) &= \frac{1}{\ell_1} \sum\limits_{j=1}^{\ell_1} \frac{g(z_k + h_kw_k^{(i,j)},x_k,\xi_{i,k}) -g(z_k - h_kw_k^{(i,j)},x_k,\xi_{i,k})}{2 h_k} w_k^{(i,j)}\\
        \hat{\nabla}_z f_{h_{k}}(z_k,x_k,\zeta_{i,k}) &= \frac{1}{\ell_2} \sum\limits_{j=1}^{\ell_2} \frac{f(z_k + h_kw_k^{(i,j)},x_k,\zeta_{i,k}) - f(z_k,x_k,\zeta_{i,k})}{ h_k} w_k^{(i,j)}\\    
        \hat{\nabla}_x f_{h_{k}}(z_k,x_k,\zeta_{i,k}) &= \frac{1}{\ell_2} \sum\limits_{j=1}^{\ell_2} \frac{f(z_k ,x_k + h_ku_k^{(i,j)},\zeta_{i,k}) - f(z_k,x_k,\zeta_{i,k})}{ h_k} u_k^{(i,j)}.\\    
    \end{aligned}
\end{equation}
Let
\begin{equation*}
    \begin{aligned}
        c_{k}^{(i,j)}(z_k,x_k,\xi_{i,k}) &=\frac{g(z_k + h_kw_k^{(i,j)},x_k,\xi_{i,k}) +g(z_k - h_kw_k^{(i,j)},x_k,\xi_{i,k}) - 2g(z_k ,x_k,\xi_{i,k})}{2 h_k^2},\\
        s_k^{(i,j)}(z_k,x_k,\xi_{i,k}) &= \frac{g(z_k + h_kw_k^{(i,j)},x_k +h_ku_k^{(i,j)},\xi_{i,k}) +g(z_k - h_kw_k^{(i,j)},x_k -h_ku_k^{(i,j)},\xi_{i,k}) - 2g(z_k ,x_k,\xi_{i,k})}{2 h_k^2}
    \end{aligned}
\end{equation*}
We denote the finite-difference surrogates of Hessian block $zz$ and $xz$ of the inner target as follows.
\begin{equation}\label{eqn:zoba_hessian_approximations}
    \begin{aligned}
        \hat{\nabla}_{zz}^2 g_{h_{k}}(z_k,x_k,\xi_{i,k}) &= \frac{1}{\ell_1} \sum\limits_{j=1}^{\ell_1} c_{k}^{(i,j)}(z_k,x_k,\xi_{i,k}) (w_k^{(i,j)}w_k^{(i,j)\top} - I_p).\\
        \hat{\nabla}_{xz}^2 g_{h_{k}}(z_k,x_k,\xi_{i,k}) &= \frac{1}{\ell_1} \sum\limits_{j=1}^{\ell_1} s_k^{(i,j)}(z_k,x_k,\xi_{i,k})  u_k^{(i,j)}w_k^{(i,j)\top}.
    \end{aligned}
\end{equation}
Moreover, we denote by $\nabla_z g_{h_k}$, $\nabla_z f_{h_k}$, and $\nabla_x f_{h_k}$ the gradients with respect to the variables $z$ and $x$ of the smooth surrogates (eq. \eqref{eqn:gaussian_smoothing}) of $g$ and $f$ at iteration $k \in \mathbb{N}$. These are defined as indicated in Lemma \ref{lem:grad_hess_smoothing} e.g.,
\begin{equation*}
\begin{aligned}
\nabla_z g_{h_k}(z_k, x_k, \xi_{i,k})
&= \frac{1}{\ell_1} \sum\limits_{j=1}^{\ell_1}
\mathbb{E}_{w_k^{(i,j)}} \left[
\frac{g(z_k + h_k w_k^{(i,j)}, x_k, \xi{i,k}) - g(z_k - h_k w_k^{(i,j)}, x_k, \xi_{i,k})}{2 h_k}  w_k^{(i,j)}
\right] \\
&= \mathbb{E}_{w} \left[
\frac{g(z_k + h_k w, x_k, \xi{i,k}) - g(z_k - h_k w, x_k, \xi_{i,k})}{2 h_k}  w
\right],
\end{aligned}
\end{equation*}
where the second equality follows from the fact that all $w_k^{(i,j)}$ are i.i.d. Analogously, $\nabla_z f_{h_k}$ and $\nabla_x f_{h_k}$ are defined. We begin by providing a lemma that bounds the expected norm of the surrogate gradient approximations.

\begin{lemma}[Approximation Error Bounds]\label{lem:approx_error}
    Let Assumptions \ref{asm:bc_condition},\ref{asm:F_smooth},\ref{asm:G_asm} hold. Let $\hat{\nabla}_z g_{h_{k}}$,$\hat{\nabla}_z f_{h_{k}}$ and $\hat{\nabla}_x f_{h_{k}}$ be the stochastic gradient surrogates defined in eq. \eqref{eqn:zoba_grad_approximations}. Let $(z_k)_{k \in \mathbb{N}}$ and $(x_k)_{k \in \mathbb{N}}$ be the sequences generated by Algorithm \ref{alg:zoba}. Then, for every $k \in \mathbb{N}$ 
    \begin{equation}\label{eqn:app_err_approx_g}
        \begin{aligned}
        \mathbb{E}_k \left[ \left\| \frac{1}{b_1} \sum\limits_{i=1}^{b_1}  \hat{\nabla}_z g_{h_{k}}(z_k,x_k,\xi_{i,k})  \right\|^2\right] &\leq 2 \left(2 + \frac{p + 2}{b_1 \ell_1} \right) \left\|\nabla_z G(z_k, x_k) \right\|^2\\
        &+ \left( \frac{(p + 6)^3}{2 b_1 \ell_1} + \left( \frac{3}{b_1} + 1 \right) (p + 3)^3 \right)L_{1,G}^2 h_{k}^2\\
        &+ 2 \left(3 +\frac{p + 2}{\ell_1} \right)\frac{\sigma_{1,G}^2}{b_1},
        \end{aligned}
    \end{equation}
    \begin{equation*}
        \begin{aligned}
            \mathbb{E}_k \left[ \left\| \frac{1}{b_2} \sum\limits_{i=1}^{b_2} \hat{\nabla}_z f_{h_{k}}(z_k,x_k,\zeta_{i,k})  \right\|^2\right] &\leq 2 \left(1 + \frac{p + 2}{b_2 \ell_2} \right) \left\|\nabla_z F(z_k, x_k) \right\|^2\\
        & +\left( \frac{(p + 6)^3}{2 b_2 \ell_2} + \left(\frac{3}{b_2} + 1 \right)(p + 3)^{3}  \right) L_{1,F}^2 h_{k}^2\\
        &+ 2\left(\frac{p + 2}{\ell_2} + 3 \right ) \frac{\sigma_{1,F}^2}{b_2}, \\
        \end{aligned}
    \end{equation*}
    and,
    \begin{equation*}
        \begin{aligned}
            \mathbb{E}_k \left[  \left\| \frac{1}{b_2} \sum\limits_{i=1}^{b_2}\hat{\nabla}_x f_{h_{k}}(z_k,x_k,\zeta_{i,k})  \right\|^2\right] &\leq 2 \left(1 + \frac{d + 2}{b_2 \ell_2} \right) \left\|\nabla_x F(z_k, x_k) \right\|^2\\
        & +\left( \frac{(d + 6)^3}{2 b_2 \ell_2} + \left(\frac{3}{b_2} + 1 \right)(d + 3)^{3}  \right) L_{1,F}^2 h_{k}^2\\
        &+ 2\left(\frac{d + 2}{b_2 \ell_2} +\frac{3}{b_2} \right )\sigma_{1,F}^2.
        \end{aligned}
    \end{equation*}
\end{lemma}

\begin{proof}
    We start proving eq. \eqref{eqn:app_err_approx_g}. %
    Define
    \begin{equation*}
        \begin{aligned}
        g^{(i,j)}_k&=\frac{g(z_k + h_{k}w_k^{(i,j)}, x_k, \xi_{i,k}) - g(z_k - h_{k}w_k^{(i,j)}, x, \xi_{i,k})}{2h_{k}}w_{k}^{(i,j)},\\
        \delta^{(i,j)}_k &= g^{(i,j)}_k - \nabla_zg_{h_{k}}(z_k,x_k,\xi_{i,k}), \quad \text{and} \quad \delta^{(i)}_k = \sum\limits_{j=1}^{\ell_1} \delta^{(i,j)}_k. 
        \end{aligned}
    \end{equation*}
    Adding and subtracting $\frac{1}{b_1\ell_1} \sum\limits_{i = 1}^{b_1} \sum\limits_{j = 1}^{\ell_1} \nabla_z g_{h_{k}}(z_k, x_k, \xi_{i,k})$, we have
    \begin{equation*}
        \begin{aligned}
        \left\| \frac{1}{b_1} \sum\limits_{i=1}^{b_1}  \hat{\nabla}_z g_{h_{k}}(z_k,x_k,\xi_{i,k})  \right\|^2 &= \frac{1}{b_1^2 \ell_1^2} \left\| \sum\limits_{i =1}^{b_1}\sum\limits_{j =1}^{\ell_1} \delta^{(i,j)}_k +\nabla_z g_{h_k}(z_k,x_k,\xi_{i,k}) \right\|^2\\
        &= \frac{1}{b_1^2 \ell_1^2} \left\| \sum\limits_{i =1}^{b_1}\sum\limits_{j =1}^{\ell_1} \delta^{(i,j)}_k \right\|^2 +\frac{1}{b^2_1} \left\| \sum\limits_{i = 1}^{b_1}\nabla_z g_{h_{k}}(z_k,x_k,\xi_{i,k}) \right\|^2\\
        &+ \frac{2}{b^2_1 \ell^2_1}\scalT{\sum\limits_{i =1}^{b_1}\sum\limits_{j =1}^{\ell_1} \delta^{(i,j)}_k}{\sum\limits_{i =1}^{b_1}\sum\limits_{j =1}^{\ell_1} \nabla_z g_{h_{k}}(z_k,x_k,\xi_{i,k})}.
        \end{aligned}
    \end{equation*}
    Taking the expectation in $w_{k}^{(i,j)}$ for every $i = 1,\cdots,b_1$ and $j = 1,\cdots,\ell_1$, we get
    \begin{equation*}
        \begin{aligned}
        \mathbb{E}_{W_k}\left[\left\| \frac{1}{b_1} \sum\limits_{i=1}^{b_1}  \hat{\nabla}_z g_{h_{k}}(z_k,x_k,\xi_{i,k})  \right\|^2\right] &= \frac{1}{b_1^2 \ell_1^2} \mathbb{E}_{W_k}\left[ \left\| \sum\limits_{i =1}^{b_1}\sum\limits_{j =1}^{\ell_1} \delta^{(i,j)}_k \right\|^2\right] + \frac{1}{b^2_1} \left\| \sum\limits_{i = 1}^{b_1}\nabla_z g_{h_{k}}(z_k,x_k,\xi_{i,k}) \right\|^2\\
        &+ \frac{2}{b_1^2 \ell_1^2}\scalT{\sum\limits_{i =1}^{b_1}\sum\limits_{j =1}^{\ell_1} \mathbb{E}_{w^{(i,j)}_k}[\delta^{(i,j)}_k]}{\sum\limits_{i =1}^{b_1}\sum\limits_{j =1}^{\ell_1} \nabla_z g_{h_{k}}(z_k,x_k,\xi_{i,k})}.
        \end{aligned}
    \end{equation*}
    By Lemma \ref{lem:grad_hess_smoothing}, for every $i$ and $j$, we have $\mathbb{E}_{w^{(i,j)}_k}\left[g^{(i,j)}_k \right] = \nabla_z g_{h_{k}}(z_k, x_k, \xi_{i,k})$. Thus, observing that
    \begin{equation*}
        \begin{aligned}
            \mathbb{E}_{w^{(i,j)}_k}\left[\delta^{(i,j)}_k\right] = \mathbb{E}_{w^{(i,j)}_k}\left[g^{(i,j)}_k - \nabla_zg_{h_{k}}(z_k,x_k,\xi_{i,k})\right] = \mathbb{E}_{w^{(i,j)}_k}\left[g^{(i,j)}_k \right] - \nabla_zg_{h_{k}}(z_k,x_k,\xi_{i,k}) =0,
        \end{aligned}
    \end{equation*}
    we have,
    \begin{equation}\label{eqn:approx_err_grad_cfd}
        \begin{aligned}
        \mathbb{E}_{W_k}\left[\left\| \frac{1}{b_1} \sum\limits_{i=1}^{b_1}  \hat{\nabla}_z g_{h_{k}}(z_k,x_k,\xi_{i,k})  \right\|^2\right] &= \underbrace{\frac{1}{b_1^2 \ell_1^2} \mathbb{E}_{W_k}\left[ \left\| \sum\limits_{i =1}^{b_1}\sum\limits_{j =1}^{\ell_1} \delta^{(i,j)}_k \right\|^2\right]}_{A} + \underbrace{\frac{1}{b^2_1} \left\| \sum\limits_{i = 1}^{b_1}\nabla_z g_{h_{k}}(z,x,\xi_i) \right\|^2}_{B}.
        \end{aligned}
    \end{equation}
    The term $B$ can be bounded by adding and subtracting the full-gradient of the smoothing $\nabla_z G_{h_{k}}(z_k,x_k)$, we get%
    \begin{equation*}
        \begin{aligned}
            \frac{1}{b^2_1} \left\| \sum\limits_{i = 1}^b\nabla_z g_{h_{k}}(z_k,x_k,\xi_{i,k}) \right\|^2 %
            &\leq \frac{2}{b^2_1} \left\| \sum\limits_{i = 1}^{b_1}\nabla_z g_{h_{k}}(z_k,x_k,\xi_{i,k}) - \nabla_z G_{h_{k}}(z_k, x_k) \right\|^2 + 2 \left \|\nabla_z G_{h_{k}}(z_k, x_k) \right\|^2.
        \end{aligned}
    \end{equation*}
    Let $\bar{\delta}_{i,k} := g_{h_{k}}(z_k,x_k,\xi_{i,k}) - \nabla_z G_{h_{k}}(z_k, x_k)$. Then, taking the conditional expectation on $\xi_{i,k}$ for every $i =1,\cdots,b_1$, we have
    \begin{equation*}
        \begin{aligned}
            \frac{1}{b^2_1} \mathbb{E}_{\xi_k}\left[\left\| \sum\limits_{i = 1}^{b_1}\nabla_z g_{h_{k}}(z_k,x_k,\xi_{i,k}) \right\|^2\right] &\leq \frac{2}{b_1^2} \mathbb{E}_{\xi_k} \left[\left\| \sum\limits_{i = 1}^{b_1}\bar{\delta}_{i,k} \right\|^2 \right] + 2 \left \|\nabla_z G_{h_{k}}(z_k, x_k) \right\|^2\\
            &= \frac{2}{b_1^2} \bigg( \sum\limits_{i=1}^{b_1}\mathbb{E}_{\xi_{i,k}} \left[\left\| \bar{\delta}_{i,k} \right\|^2 \right] + \sum\limits_{i=1}^{b_1}\sum\limits_{t \neq i}\mathbb{E}_{\xi_{i,k}} \left[\scalT{\bar{\delta}_{i,k}}{\bar{\delta}_{t,k}} \right]\bigg)\\
            &+ 2 \left \|\nabla_z G_{h_{k}}(z_k, x_k) \right\|^2.
        \end{aligned}
    \end{equation*}
    By Assumption \ref{asm:bc_condition} and, since for every $i\neq t$, $\xi_{i,k}$ and $\xi_{t,k}$ are independent, we have
    \begin{equation*}
        \begin{aligned}
            \frac{1}{b^2_1} \mathbb{E}_{\xi_k}\left[\left\| \sum\limits_{i = 1}^{b_1}\nabla_z g_{h_{k}}(z_k,x_k,\xi_{i,k}) \right\|^2\right] &\leq \frac{2}{b_1^2} \bigg( \sum\limits_{i=1}^{b_1}\mathbb{E}_{\xi_{i,k}} \left[\left\| \bar{\delta}_{i,k} \right\|^2 \right] + \sum\limits_{i=1}^{b_1}\sum\limits_{t \neq i}\underbrace{\scalT{\mathbb{E}_{\xi_{i,k}} \left[\bar{\delta}_{i,k}\right] }{\mathbb{E}_{\xi_{t,k}} \left[\bar{\delta}_{t,k}\right]}}_{=0}\bigg)\\
            &+ 2 \left \|\nabla_z G_{h_{k}}(z_k, x_k) \right\|^2.
        \end{aligned}
    \end{equation*}
    Adding and subtracting $\nabla_z g(z_k, x_k,\xi_{i,k})$ and $\nabla_z G(z_k,x_k)$, we get
    \begin{equation*}
        \begin{aligned}
            \frac{1}{b^2_1} \mathbb{E}_{\xi_k}\left[\left\| \sum\limits_{i = 1}^{b_1}\nabla_z g_{h_{k}}(z_k,x_k,\xi_{i,k}) \right\|^2\right] &\leq \frac{2}{b_1^2}  \sum\limits_{i=1}^{b_1} \bigg(3\mathbb{E}_{\xi_{i,k}} \left[\left\| \nabla_z g_{h_{k}}(z_k,x_k,\xi_{i,k}) - \nabla_z g(z_k, x_k, \xi_{i,k}) \right\|^2 \right]\\
            &+3 \| \nabla_z G(z_k, x_k) - \nabla_z G_{h_{k}}(z_k, x_k) \|^2 \\
            &+3 \mathbb{E}_{\xi_{i,k}} \left[ \| \nabla_z g(z_k, x_k, \xi_{i,k}) - \nabla_z G(z_k, x_k) \|^2 \right]\bigg)\\
            &+ 2 \left \|\nabla_z G_{h_{k}}(z_k, x_k) \right\|^2.
        \end{aligned}
    \end{equation*}    
    By Lemma \ref{lem:smoothing_error} and Assumption \ref{asm:bc_condition},
    \begin{equation*}
        \begin{aligned}
            \frac{1}{b^2_1} \mathbb{E}_{\xi_k}\left[\left\| \sum\limits_{i = 1}^{b_1}\nabla_z g_{h_{k}}(z_k,x_k,\xi_{i,k}) \right\|^2\right] &\leq 
            \frac{6}{b_1} \bigg(\frac{L_{1,G}^2}{2}(p + 3)^{3}h_{k}^2 + \sigma_{1,G}^2\bigg) + 2 \left \|\nabla_z G_{h_{k}}(z_k, x_k) \right\|^2.
        \end{aligned}
    \end{equation*}    
    Adding and subtracting $\nabla_z G(z_k,x_k)$,
    \begin{equation*}
        \begin{aligned}
            \frac{1}{b^2_1} \mathbb{E}_{\xi_k}\left[\left\| \sum\limits_{i = 1}^{b_1}\nabla_z g_{h_{k}}(z_k,x_k,\xi_{i,k}) \right\|^2\right] &\leq \frac{6}{b_1} \bigg(\frac{L_{1,G}^2}{2}(p + 3)^{3}h_{k}^2 + \sigma_{1,G}^2\bigg)\\
            &+ 2 \left \|\nabla_z G_{h_{k}}(z_k, x_k) - \nabla_z G(z_k, x_k) + \nabla_z G(z_k, x_k) \right\|^2\\
            &\leq \frac{6}{b_1} \bigg(\frac{L_{1,G}^2}{2}(p + 3)^{3}h_{k}^2 + \sigma_{1,G}^2\bigg) + 4 \left \|\nabla_z G_{h_{k}}(z_k, x_k) - \nabla_z G(z_k, x_k) \right\|^2\\
            &+ 4 \left\|\nabla_z G(z_k, x_k) \right\|^2.
        \end{aligned}
    \end{equation*}    
    By Lemma \ref{lem:smoothing_error}, we get
    \begin{equation}\label{eqn:approx_err_grad_cfd_B}
        \begin{aligned}
            \frac{1}{b^2_1} \mathbb{E}_{\xi_k}\left[\left\| \sum\limits_{i = 1}^{b_1}\nabla_z g_{h_{k}}(z_k,x_k,\xi_{i,k}) \right\|^2\right] &\leq  4 \left\|\nabla_z G(z_k, x_k) \right\|^2 + \left(\frac{3}{b_1} + 1 \right)L_{1,G}^2(p + 3)^{3}h_{k}^2 + \frac{6\sigma_{1,G}^2}{b_1}.
        \end{aligned}
    \end{equation}    
    Now, we focus on bounding the term $A$. %
    We have that 
    \begin{equation*}
        \begin{aligned}
            \frac{1}{b_1^2 \ell_1^2} \mathbb{E}_{W_k}\left[ \left\| \sum\limits_{i =1}^{b_1}\sum\limits_{j =1}^{\ell_1} \delta^{(i,j)}_k \right\|^2\right] &= \frac{1}{b_1^2 \ell_1^2} \left( \sum\limits_{i=1}^{b_1} \mathbb{E}_{W_k^{(i)}}\left[ \| \delta^{(i)}_k \|^2 \right] + \sum\limits_{i =1}^{b_1} \sum\limits_{t \neq i} \mathbb{E}_{W_k^{(i)}}\left[\scalT{\delta^{(i)}_k}{\delta^{(t)}_k} \right] \right).
        \end{aligned}
    \end{equation*}
    Notice that for $i \neq t$, we have that for every $j= 1,\cdots, \ell_1$, $w^{(i,j)}, w^{(t,j)}$ are all independent. Therefore, we have
    \begin{equation*}
        \begin{aligned}
            \frac{1}{b_1^2 \ell_1^2} \mathbb{E}_{W_k}\left[ \left\| \sum\limits_{i =1}^{b_1}\sum\limits_{j =1}^{\ell_1} \delta^{(i,j)}_k \right\|^2\right] %
            &=\frac{1}{b_1^2 \ell_1^2} \left( \sum\limits_{i=1}^{b_1} \mathbb{E}_{W^{(i)}_k}\left[ \| \delta^{(i)}_k \|^2 \right] + \sum\limits_{i =1}^{b_1} \sum\limits_{t \neq i} \underbrace{\scalT{\mathbb{E}_{W_k^{(i)}}[\delta^{(i)}_k] }{\mathbb{E}_{W_k^{(t)}}[\delta^{(t)}_k] }}_{=0} \right).
        \end{aligned}
    \end{equation*}
    Similarly, we get
    \begin{equation*}
        \begin{aligned}
            \frac{1}{b_1^2 \ell_1^2} \mathbb{E}_{W_k}\left[ \left\| \sum\limits_{i =1}^{b_1}\sum\limits_{j =1}^{\ell_1} \delta^{(i,j)}_k \right\|^2\right] %
            &= \frac{1}{b_1^2 \ell_1^2}\sum\limits_{i=1}^{b_1} \left( \sum\limits_{j=1}^{\ell_1} \mathbb{E}_{w_k^{(i,j)}}\left[ \| \delta^{(i,j)}_k \|^2 \right] + \sum\limits_{j=1}^{\ell_1} \sum\limits_{t\neq j} \mathbb{E}_{w_k^{(i,j)}}\left[ \scalT{\delta^{(i,j)}_k}{\delta^{(i,t)}_k}  \right] \right).
        \end{aligned}
    \end{equation*}
    Again, observing that for $j\neq t$, we have that $w^{(i,j)},w^{(i,t)}$ are independent. Thus, we have
    \begin{equation*}
        \begin{aligned}
            \frac{1}{b_1^2 \ell_1^2} \mathbb{E}_{W_k}\left[ \left\| \sum\limits_{i =1}^{b_1}\sum\limits_{j =1}^{\ell_1} \delta^{(i,j)}_k \right\|^2\right] &= \frac{1}{b_1^2 \ell_1^2}\sum\limits_{i=1}^{b_1} \left( \sum\limits_{j=1}^{\ell_1} \mathbb{E}_{w_k^{(i,j)}}\left[ \| \delta^{(i,j)}_k \|^2 \right] + \sum\limits_{j=1}^{\ell_1} \sum\limits_{t\neq j}  \underbrace{\scalT{\mathbb{E}_{w_k^{(i,j)}}\left[\delta^{(i,j)}_k\right]}{\mathbb{E}_{w_k^{(i,t)}}\left[\delta^{(i,t)}_k\right]}}_{=0}\right)\\
            &=\frac{1}{b_1^2 \ell_1^2}\sum\limits_{i=1}^{b_1} \sum\limits_{j=1}^{\ell_1} \mathbb{E}_{w_k^{(i,j)}}\left[ \| \delta^{(i,j)}_k \|^2 \right].
        \end{aligned}
    \end{equation*}
    Developing the square, we get
    \begin{equation}\label{eqn:approx_err_grad_cfd_A}
        \begin{aligned}
            \frac{1}{b_1^2 \ell_1^2} \mathbb{E}_{W_k}\left[ \left\| \sum\limits_{i =1}^{b_1}\sum\limits_{j =1}^{\ell_1} \delta^{(i,j)}_k \right\|^2\right] 
            &=\frac{1}{b_1^2 \ell_1^2}\sum\limits_{i=1}^{b_1} \sum\limits_{j=1}^{\ell_1} \mathbb{E}_{w_k^{(i,j)}}\left[ \| \delta^{(i,j)}_k \|^2 \right]\\
            &=\frac{1}{b_1^2 \ell_1^2}\sum\limits_{i=1}^{b_1} \sum\limits_{j=1}^{\ell_1} \bigg(  \mathbb{E}_{w_k^{(i,j)}}\left[ \left\| g^{(i,j)}_k \right\|^2 \right]+ \| \nabla_z g_{h_{k}}(z_k, x_k, \xi_{i,k}) \|^2\\
            &- 2\mathbb{E}_{w_k^{(i,j)}}\left[ \scalT{g^{(i,j)}_k}{\nabla_zg_{h_{k}}(z_k, x_k, \xi_{i,k})}  \right]\bigg)\\
            &=\frac{1}{b_1^2 \ell_1^2}\sum\limits_{i=1}^{b_1} \sum\limits_{j=1}^{\ell_1} \left(  \mathbb{E}_{w_k^{(i,j)}}\left[ \left\| g^{(i,j)}_k \right\|^2 \right]-  \| \nabla_zg_{h_{k}}(z_k, x_k, \xi_{i,k}) \|^2 \right)\\
            &\leq \frac{1}{b_1^2 \ell_1^2}\sum\limits_{i=1}^{b_1} \sum\limits_{j=1}^{\ell_1} \mathbb{E}_{w_k^{(i,j)}}\left[ \left\| g^{(i,j)}_k \right\|^2 \right].
        \end{aligned}
    \end{equation}
    Therefore, taking the conditional expectation on $\xi_{i,k}$ for every $i$ in eq. \eqref{eqn:approx_err_grad_cfd} and using inequalities \eqref{eqn:approx_err_grad_cfd_A} and \eqref{eqn:approx_err_grad_cfd_B}, we have
    \begin{equation}\label{eqn:approx_err_grad_cfd_AB}
        \begin{aligned}
        \mathbb{E}_{\xi_k} \left[\mathbb{E}_{W_k}\left[\left\| \frac{1}{b_1} \sum\limits_{i=1}^{b_1}  \hat{\nabla}_z g_{h_{k}}(z_k,x_k,\xi_{i,k})  \right\|^2\right] \right] &\leq \frac{1}{b_1^2 \ell_1^2}\sum\limits_{i=1}^{b_1} \sum\limits_{j=1}^{\ell_1} \underbrace{\mathbb{E}_{\xi_{i,k}} \left[\mathbb{E}_{w_k^{(i,j)}}\left[ \left\| g^{(i,j)}_k \right\|^2 \right] \right]}_{C} + 4 \left\|\nabla_z G(z_k, x_k) \right\|^2\\
        &+ \left(\frac{3}{b_1} + 1 \right)L_{1,G}^2(p + 3)^{3}h_{k}^2 + \frac{6\sigma_{1,G}^2}{b_1}.
        \end{aligned}
    \end{equation}
    We focus on bounding $C$. Adding and subtracting $g(z_k,x_k,\xi_{i,k})$, we get
    \begin{equation*}
        \begin{aligned}
            \mathbb{E}_{\xi_{i,k}} \left[\mathbb{E}_{w_k^{(i,j)}}\left[ \left\| g^{(i,j)}_k \right\|^2 \right] \right] &\leq \frac{1}{2h_{k}^2} \mathbb{E}_{\xi_{i,k}} \left[\mathbb{E}_{w_k^{(i,j)}} \left[ \left(g(z_k + h_{k} w_k^{(i,j)}, x_k, \xi_{i,k}) - g(z_k, x_k,\xi_{i,k})\right)^2\|w_k^{(i,j)}\|^2 \right]\right]\\
            &+ \frac{1}{2h_{k}^2} \mathbb{E}_{\xi_{i,k}} \left[\mathbb{E}_{w_k^{(i,j)}} \left[ \left(g(z_k, x_k,\xi_{i,k}) - g(z_k - h_{k} w_k^{(i,j)}, x_k, \xi_{i,k}) \right)^2 \|w_k^{(i,j)}\|^2 \right] \right].
        \end{aligned}
    \end{equation*}
    By symmetry, we have
    \begin{equation*}
        \begin{aligned}
            \mathbb{E}_{\xi_{i,k}} \left[\mathbb{E}_{w_k^{(i,j)}}\left[ \left\| g^{(i,j)}_k \right\|^2 \right] \right] &\leq \frac{1}{h_{k}^2} \mathbb{E}_{\xi_{i,k}} \left[\mathbb{E}_{w_k^{(i,j)}} \left[ \left(g(z_k + h_{k} w_k^{(i,j)}, x_k, \xi_{i,k}) - g(z_k, x_k,\xi_{i,k})\right)^2\|w_k^{(i,j)}\|^2 \right] \right].
        \end{aligned}
    \end{equation*}
    Adding and subtracting $\scalT{\nabla g(z_k,x_k,\xi_{i,k})}{[h_{k} w_k^{(i,j)}, 0]}$ where $[h_{k} w_k^{(i,j)}, 0] \in \mathbb{R}^{p + d}$ denotes the concatenation between $h_{k} w_k^{(i,j)}$ and the vector of zeros of $d$ entries.
    \begin{equation*}
        \begin{aligned}
            \mathbb{E}_{\xi_{i,k}} \left[\mathbb{E}_{w_k^{(i,j)}}\left[ \left\| g^{(i,j)}_k \right\|^2 \right] \right] &\leq \frac{2}{h_{k}^2} \mathbb{E}_{\xi_{i,k}} \bigg[\mathbb{E}_{w_k^{(i,j)}} \bigg[ \bigg(g(z_k + h_{k} w_k^{(i,j)}, x_k, \xi_{i,k}) - g(z_k, x_k,\xi_{i,k})\\
            &-\scalT{\nabla g(z_k,x_k,\xi_{i,k})}{[h_{k} w_k^{(i,j)}, 0]} \bigg)^2\|w_k^{(i,j)}\|^2 \bigg] \bigg]\\
            &+ \frac{2}{h_{k}^2} \mathbb{E}_{\xi_{i,k}} \left[\mathbb{E}_{w_k^{(i,j)}} \left[ \left(\scalT{\nabla g(z_k,x_k,\xi_{i,k})}{[h_{k} w_k^{(i,j)}, 0]} \right)^2\|w_k^{(i,j)}\|^2 \right] \right].
        \end{aligned}
    \end{equation*}
    Notice that $\scalT{\nabla g(z_k,x_k,\xi_{i,k})}{[h_{k} w_k^{(i,j)}, 0]}  = \scalT{\nabla_z g(z_k,x_k,\xi_{i,k})}{h_{k} w_k^{(i,j)}}$. By the Descent Lemma \cite{polyak1987introduction},
    \begin{equation*}
        \begin{aligned}
            \mathbb{E}_{\xi_{i,k}} \left[\mathbb{E}_{w_k^{(i,j)}}\left[ \left\| g^{(i,j)}_k \right\|^2 \right] \right] &\leq \frac{L_{1,G}^2}{2} h_{k}^2\mathbb{E}_{w_k^{(i,j)}} \bigg[\|w_k^{(i,j)}\|^6 \bigg] \bigg]\\
            &+ 2 \mathbb{E}_{\xi_{i,k}} \left[ \nabla_z g(z_k,x_k,\xi_{i,k})^\top \mathbb{E}_{w_k^{(i,j)}} \left[ w_k^{(i,j)}w_k^{(i,j) \top} w_k^{(i,j)}w_k^{(i,j) \top} \right] \nabla_z g(z_k,x_k,\xi_{i,k}) \right].
        \end{aligned}
    \end{equation*}
    By Lemma \ref{lem:norm_gaus_vector}, we have
    \begin{equation*}
        \begin{aligned}
            \mathbb{E}_{\xi_{i,k}} \left[\mathbb{E}_{w_k^{(i,j)}}\left[ \left\| g^{(i,j)}_k \right\|^2 \right] \right] &\leq \frac{L_{1,G}^2}{2} (p + 6)^3 h_{k}^2 + 2 (p + 2)\mathbb{E}_{\xi_{i,k}} \left[ \| \nabla_z g(z_k,x_k,\xi_{i,k}) \|^2 \right].%
        \end{aligned}
    \end{equation*}
    Adding and subtracting $\nabla_z G(z_k, x_k)$ and by Assumption \ref{asm:bc_condition}, we get
    \begin{equation*}
        \begin{aligned}
            \mathbb{E}_{\xi_{i,k}} \left[\mathbb{E}_{w_k^{(i,j)}}\left[ \left\| g^{(i,j)}_k \right\|^2 \right] \right] &\leq \frac{L_{1,G}^2}{2} (p + 6)^3 h_{k}^2\\
            &+ 2 (p + 2)\mathbb{E}_{\xi_{i,k}} \left[ \| \nabla_z g(z_k,x_k,\xi_{i,k}) - \nabla_z G(z_k, x_k) + \nabla_z G(z_k, x_k) \|^2 \right]\\
            &= \frac{L_{1,G}^2}{2} (p + 6)^3 h_{k}^2+ 2 (p + 2)\mathbb{E}_{\xi_{i,k}} \left[ \| \nabla_z g(z_k,x_k,\xi_{i,k}) - \nabla_z G(z_k, x_k) \|^2 \right]\\
            &+ 2 (p + 2)\|\nabla_z G(z_k, x_k) \|^2\\
            &+ 4(p + 2) \underbrace{\scalT{\mathbb{E}_{\xi_{i,k}} \left[ \nabla_z g(z_k,x_k,\xi_{i,k}) - \nabla_z G(z_k, x_k) \right]}{\nabla_z G(z_k,x_k)}}_{=0}\\
            &\leq \frac{L_{1,G}^2}{2} (p + 6)^3 h_{k}^2+ 2 (p + 2) \sigma_{1,G}^2 + 2 (p + 2)\|\nabla_z G(z_k, x_k) \|^2.
        \end{aligned}
    \end{equation*}
    Using this inequality in eq. \eqref{eqn:approx_err_grad_cfd_AB}, we get the first claim
    \begin{equation*}
        \begin{aligned}
        \mathbb{E}_{\xi_k} \left[\mathbb{E}_{W_k}\left[\left\|\frac{1}{b_1} \sum\limits_{i=1}^{b_1}  \hat{\nabla}_z g_{h_{k}}(z_k,x_k,\xi_{i,k})   \right\|^2\right] \right] 
        &\leq 2 \left(2 + \frac{p + 2}{b_1 \ell_1} \right) \left\|\nabla_z G(z_k, x_k) \right\|^2\\
        &+ \left( \frac{(p + 6)^3}{2 b_1 \ell_1} + \left( \frac{3}{b_1} + 1 \right) (p + 3)^3 \right)L_{1,G}^2 h_{k}^2\\
        &+ 2 \left(3 +\frac{p + 2}{\ell_1} \right)\frac{\sigma_{1,G}^2}{b_1}.
        \end{aligned}
    \end{equation*}
    The proof of the second claim follows the same line. We report below the proof since the finite-difference surrogate used is different. %
    Define
    \begin{equation*}
        \begin{aligned}
        f^{(i,j)}_k&=\frac{f(z_k + h_{k}w_k^{(i,j)}, x_k, \zeta_{i,k}) - f(z_k, x_k, \zeta_{i,k})}{h_{k}}w_{k}^{(i,j)},\\
        \delta^{(i,j)}_{f,k} &= f^{(i,j)}_k - \nabla_z f_{h_{k}}(z_k,x_k,\zeta_{i,k}), \quad \text{and} \quad \delta^{(i)}_{f,k} = \sum\limits_{j=1}^{\ell_2} \delta^{(i,j)}_{f,k}. 
        \end{aligned}
    \end{equation*}
    Following the same steps of the previous claim, we get
    \begin{equation}\label{eqn:grad_approx_f_z_AB}
        \begin{aligned}
        \mathbb{E}_{\zeta_k} \left[\mathbb{E}_{W_k}\left[\left\|  \frac{1}{b_2} \sum\limits_{i=1}^{b_2} \hat{\nabla}_z f_{h_{k}}(z_k,x_k,\zeta_{i,k})  \right\|^2\right] \right] &\leq \frac{1}{b_2^2 \ell_2^2}\sum\limits_{i=1}^{b_2} \sum\limits_{j=1}^{\ell_2} \underbrace{\mathbb{E}_{\zeta_{i,k}} \left[\mathbb{E}_{w_k^{(i,j)}}\left[ \left\| f^{(i,j)}_k \right\|^2 \right] \right]}_{D} + 4 \left\|\nabla_z F(z_k, x_k) \right\|^2\\
        &+ \left(\frac{3}{b_2} + 1 \right)L_{1,F}^2(p + 3)^{3}h_{k}^2 + \frac{6\sigma_{1,F}^2}{b_2}.
        \end{aligned}
    \end{equation}
    Similarly to the previous $C$ term, we bound the term $D$. Adding and subtracting $\scalT{\nabla f(z_k,x_k,\zeta_{i,k})}{[h_{k} w_k^{(i,j)}, 0]}$, we have    
    \begin{equation*}
        \begin{aligned}
            \mathbb{E}_{\zeta_{i,k}} \left[\mathbb{E}_{w_k^{(i,j)}}\left[ \left\| f^{(i,j)}_k \right\|^2 \right] \right] &= \frac{1}{h_{k}^2} \mathbb{E}_{\zeta_{i,k}} \left[\mathbb{E}_{w_k^{(i,j)}} \left[ \left(f(z_k + h_{k} w_k^{(i,j)}, x_k, \zeta_{i,k}) - f(z_k, x_k)\right)^2\|w_k^{(i,j)}\|^2 \right] \right]\\
            &\leq \frac{2}{h_{k}^2} \mathbb{E}_{\zeta_{i,k}} \bigg[\mathbb{E}_{w_k^{(i,j)}} \bigg[ \bigg(f(z_k + h_{k} w_k^{(i,j)}, x_k, \zeta_{i,k}) - f(z_k, x_k)\\
            &-\scalT{\nabla f(z_k,x_k,\zeta_{i,k})}{(h_{k}w_{k}^{(i,j)},0)}\bigg)^2\|w_k^{(i,j)}\|^2 \bigg] \bigg]\\
            &+ \frac{2}{h_{k}^2} \mathbb{E}_{\zeta_{i,k}} \left[\mathbb{E}_{w_k^{(i,j)}} \left[ \left(\scalT{\nabla f(z_k,x_k,\zeta_{i,k})}{(h_{k}w_{k}^{(i,j)},0)}\right)^2\|w_k^{(i,j)}\|^2 \right] \right].
        \end{aligned}
    \end{equation*}
    Notice that $\scalT{\nabla f(z_k,x_k,\zeta_{i,k})}{[h_{k} w_k^{(i,j)}, 0]}  = \scalT{\nabla_z f(z_k,x_k,\xi_{i,k})}{h_{k} w_k^{(i,j)}}$. By the Descent Lemma \cite{polyak1987introduction},
    \begin{equation*}
        \begin{aligned}
            \mathbb{E}_{\zeta_{i,k}} \left[\mathbb{E}_{w_k^{(i,j)}}\left[ \left\| f^{(i,j)}_k \right\|^2 \right] \right] &\leq \frac{L_{1,F}^2}{2} h_{k}^2\mathbb{E}_{w_k^{(i,j)}} \bigg[\|w_k^{(i,j)}\|^6 \bigg] \bigg]\\
            &+ 2 \mathbb{E}_{\zeta_{i,k}} \left[ \nabla_z f(z_k,x_k,\zeta_{i,k})^\top \mathbb{E}_{w_k^{(i,j)}} \left[ w_k^{(i,j)}w_k^{(i,j) \top} w_k^{(i,j)}w_k^{(i,j) \top} \right] \nabla_z f(z_k,x_k,\zeta_{i,k}) \right].
        \end{aligned}
    \end{equation*}
    By Lemma \ref{lem:norm_gaus_vector}, we have
    \begin{equation*}
        \begin{aligned}
            \mathbb{E}_{\zeta_{i,k}} \left[\mathbb{E}_{w_k^{(i,j)}}\left[ \left\| f^{(i,j)}_k \right\|^2 \right] \right] &\leq \frac{L_{1,F}^2}{2} (p + 6)^3 h_{k}^2 + 2 (p + 2)\mathbb{E}_{\zeta_{i,k}} \left[ \| \nabla_z f(z_k,x_k,\zeta_{i,k}) \|^2 \right].
        \end{aligned}
    \end{equation*}
    Adding and subtracting $\nabla_z F(z_k, x_k)$ and by Assumption \ref{asm:bc_condition}, we get
    \begin{equation*}
        \begin{aligned}
            \mathbb{E}_{\zeta_{i,k}} \left[\mathbb{E}_{w_k^{(i,j)}}\left[ \left\| f^{(i,j)}_k \right\|^2 \right] \right] &\leq \frac{L_{1,F}^2}{2} (p + 6)^3 h_{k}^2\\
            &+ 2 (p + 2)\mathbb{E}_{\zeta_{i,k}} \left[ \| \nabla_z f(z_k,x_k,\xi_{i,k}) - \nabla_z F(z_k, x_k) + \nabla_z F(z_k, x_k) \|^2 \right]\\
            &= \frac{L_{1,F}^2}{2} (p + 6)^3 h_{k}^2+ 2 (p + 2)\mathbb{E}_{\zeta_{i,k}} \left[ \| \nabla_z f(z_k,x_k,\zeta_{i,k}) - \nabla_z F(z_k, x_k) \|^2 \right]\\
            &+ 2 (p + 2)\|\nabla_z F(z_k, x_k) \|^2\\
            &+ 4(p + 2) \underbrace{\scalT{\mathbb{E}_{\xi_{i,k}} \left[ \nabla_z f(z_k,x_k,\zeta_{i,k}) - \nabla_z F(z_k, x_k) \right]}{\nabla_z F(z_k,x_k)}}_{=0}\\
            &\leq \frac{L_{1,F}^2}{2} (p + 6)^3 h_{k}^2+ 2 (p + 2) \sigma_{1,F}^2 + 2 (p + 2)\|\nabla_z F(z_k, x_k) \|^2.
        \end{aligned}
    \end{equation*}
    Using this inequality in eq. \eqref{eqn:grad_approx_f_z_AB}, we get the second claim
    \begin{equation*}
        \begin{aligned}
        \mathbb{E}_{\zeta_k} \left[\mathbb{E}_{W_k}\left[\left\|  \frac{1}{b_2} \sum\limits_{i=1}^{b_2} \hat{\nabla}_z f_{h_{k}}(z_k,x_k,\zeta_{i,k})  \right\|^2\right] \right] &\leq 2 \left(1 + \frac{p + 2}{b_2 \ell_2} \right) \left\|\nabla_z F(z_k, x_k) \right\|^2\\
        & +\left( \frac{(p + 6)^3}{2 b_2 \ell_2} + \left(\frac{3}{b_2} + 1 \right)(p + 3)^{3}  \right) L_{1,F}^2 h_{k}^2\\
        &+ 2\left(\frac{p + 2}{b_2 \ell_2} +\frac{3}{b_2} \right )\sigma_{1,F}^2\\  %
        \end{aligned}
    \end{equation*}
    The proof of the last claim follows the same line.
\end{proof}
\noindent We now focus on deriving bounds on the expected norm of the product between a surrogate Hessian approximation and a vector. We begin by stating an auxiliary proposition that bounds the expected norm of the product of a surrogate Hessian approximation (constructed using a single direction and a single sample) with a vector. This result is a special case of Proposition 2.5 in \cite{Aghasi2025}; we report it here because it will be used and subsequently extended to derive bounds for our estimators.
\begin{proposition}[Single direction Hessian estimator approximation error]\label{prop:aghasi_hess_approx}
    Let Assumption \ref{asm:G_asm} holds. Let $w,u$ be sampled i.i.d from $\mathcal{N}(0,I_p)$ and $\mathcal{N}(0, I_d)$, respectively. For $h > 0$, define 
    \begin{equation*}
        \begin{aligned}
            c_h(z,x,w,\xi) &= \frac{g(z +hw,x,\xi) +g(z - hw,x,\xi) - 2g(z,x,\xi)}{2h^2}\\
            s_h(z,x,w,\xi) &= \frac{g(z + hw,x + hu,\xi) +g(z - hw,x- hu,\xi) - 2g(z,x,\xi)}{2h^2}.
        \end{aligned}
    \end{equation*}
    Then, for every $z,v \in \mathbb{R}^p$, $x \in \mathbb{R}^d$, we have
   \begin{equation*}
       \begin{aligned}
            \mathbb{E}_{\xi}\left[ \mathbb{E}_{w} \left[ \| c_h(z,x,w,\xi)(ww^{\top} - I_p)v) \|^2 \right] \right] &\leq \bigg(  4 L_{2,G}^2 (p +16)^4 h^2  +  \frac{15}{2} (p + 6)^2 \mathbb{E}_{\xi} \left[\|\nabla_{zz}^2 g(z,x,\xi)\|_F^2 \right]  \bigg)\|v\|^2,\\
        \end{aligned}
   \end{equation*} 
   and,
   \begin{equation*}
       \begin{aligned}
            \mathbb{E}_{\xi}\left[ \mathbb{E}_{w} \left[ \| s_h(z,x,w,\xi)uw^{\top} v) \|^2 \right] \right] &\leq \bigg(8L_{2,G}^2 h^2 \left( (p + 8)^4 + 2p(d + 12)^3 \right)\\
            &+ \bigg( 6 (p + 4)(p + 2) \mathbb{E}_{\xi}[\| \nabla_{zz}^2 g_{h}(z,x,\xi) \|_F^2]\\
            &+ 36 (p + 2) \mathbb{E}_{\xi}[\| \nabla_{xz}^2 g_{h}(z,x,\xi) \|_F^2]\\
            &+ 30 p(d + 2) \mathbb{E}_{\xi}[\|\nabla^2_{xx} g_{h}(z, x, \xi) \|^2_F]\bigg)\bigg) \| v\|^2.
        \end{aligned}
   \end{equation*} 
\end{proposition}
\begin{proof}
The first inequality is a special case of Proposition 2.5(a) in \cite{Aghasi2025} with $\eta = h$ and $\mu = 0$, while the second inequality is a special case of Proposition 2.5(b) in \cite{Aghasi2025} with $\eta = \mu = h$.
\end{proof}
\noindent Now, with the following lemma, we extend this result for hessian block estimators used in Algorithm \ref{alg:zoba}.
\begin{lemma}[Hessian Blocks approximation error]\label{lem:hess_approx_error}
    Let Assumptions \ref{asm:G_asm} holds. For every $k \in \mathbb{N}$, let $\hat{\nabla}_{zz}^2 g_{h_{k}}$ and $\hat{\nabla}_{xz}^2 g_{h_{k}}$ be the Hessian block surrogates defined in eq. \eqref{eqn:zoba_hessian_approximations}. Let $(x_k)_{k\in\mathbb{N}}$, $(v_k)_{k \in \mathbb{N}}$ and $(z_k)_{k \in \mathbb{N}}$ be the sequences generated by Algorithm \ref{alg:zoba}. Let $(\xi_{i,k})_{i=1}^{b_1}$ be realizations of independent copies of random variable $\xi$ obtained at iteration $k$. Then,
    \begin{equation}\label{eqn:approx_hess_b11}
        \begin{aligned}
            \mathbb{E}_{\xi_k}\left[ \mathbb{E}_{W_k} \left[ \left\| \frac{1}{b_1} \sum\limits_{i=1}^{b_1} \hat{\nabla}_{zz}^2 g_{h_{k}}(z_k,x_k,\xi_{i,k})v_k  \right\|^2 \right] \right] &\leq \bigg( \frac{4 L_{2,G}^2 (p +16)^4}{{b_1} \ell_1} h_{k}^2 + \bigg(1 +\frac{15 (p + 6)^2p}{2 b_1 \ell_1} \bigg)   L_{1,G}^2 \bigg)\|v_k\|^2,%
        \end{aligned}
    \end{equation}
    and, 
    \begin{equation}\label{eqn:approx_hess_cross}
    \begin{aligned}
        \mathbb{E}_{\xi_k}\left[ \mathbb{E}_{W_k,U_k} \left[ \left\| \frac{1}{b_1} \sum\limits_{i=1}^{b_1} \hat{\nabla}_{xz}^2 g_{h_{k}}(z_k,x_k,\xi_{i,k})v_k  \right\|^2 \right] \right] &\leq \bigg(\frac{1}{b_1 \ell_1} \bigg(8L_{2,G}^2 h_k^2 \left( (p + 8)^4 + 2p(d + 12)^3 \right)\\
            &+ 6(p + 4)(p + 2)p + 36 (p + 2) \min(p, d)\\
            &+ 30 p(d + 2)d L_{1,G}^2\bigg) + L_{1,G}^2\bigg)\| v_k\|^2.
    \end{aligned}        
    \end{equation}
\end{lemma}
\begin{proof}
    We start by proving eq. \eqref{eqn:approx_hess_b11}. For $i = 1,\cdots,b_1$ and $j=1,\cdots,\ell_1$, we define%
    \begin{equation*}
        \begin{aligned}
            c^{(i,j)}_k &= c_k^{(i,j)}(z_k,x_k,\xi_{i,k}) =\frac{g(z_k + h_{k}w^{(i,j)}_k, x_k, \xi_{i,k}) + g(z_k - h_{k}w^{(i,j)}_k, x_k, \xi_{i,k}) - 2g(z_k,x_k, \xi_{i,k})}{2h_{k}^2},\\
            \delta^{(i,j)}_k &= c^{(i,j)}_k \left(w^{(i,j)}_k w^{(i,j) \top}_k - I_p \right)v_k - \nabla_{zz}^2 g_{h_{k}}(z_k,x_k,\xi_{i,k})v_k,\\
            \Delta^{(i)}_k &= \sum\limits_{j=1}^{\ell_1} \delta^{(i,j)}_k.%
        \end{aligned}
    \end{equation*}
    Adding and subtracting $\nabla_{zz}^2 g_{h_{k}}(z_k,x_k,\xi_{i,k})v_k$ and developing the square, we get
    \begin{equation*}
        \begin{aligned}
            \left\| \frac{1}{b_1} \sum\limits_{i=1}^{b_1} \hat{\nabla}_{zz}^2 g_{h_{k}}(z_k,x_k,\xi_{i,k})v_k  \right\|^2 %
            &= \frac{1}{b_1^2 \ell_1^2} \left\|  \sum\limits_{i=1}^{b_1}  \Delta_k^{(i)} \right\|^2 + \frac{1}{b_1^2} \left\| \sum\limits_{i=1}^{b_1} \nabla_{zz}^2 g_{h_{k}}(z_k,x_k,\xi_{i,k})v_k \right\|^2\\
            &+ \frac{2}{b_1^2 \ell_1^2}\scalT{\sum\limits_{i=1}^{b_1}  \Delta_k^{(i)}}{\sum\limits_{i=1}^{b_1} \sum\limits_{j=1}^{\ell_1} \nabla_{zz}^2 g_{h_{k}}(z_k,x_k,\xi_{i,k})v_k}.
        \end{aligned}
    \end{equation*}
    Taking the expectation on $w^{(i,j)}_k$ for every $i=1,\cdots,b_1$ and $j = 1,\cdots,\ell_1$, and observing that by Lemma \ref{lem:grad_hess_smoothing}, we have, 
    \begin{equation*}
        \mathbb{E}_{W_k}\left[ \sum\limits_{i=1}^{b_1}\Delta_k^{(i)} \right] = \sum\limits_{i=1}^{b_1} \sum\limits_{j=1}^{\ell_1} \mathbb{E}_{w^{(i,j)}_k}\left[c^{(i,j)}_k \left(w^{(i,j)}_k w^{(i,j) \top}_k - I_p \right) \right]v_k - \nabla_{zz}^2 g_{h_{k}}(z_k,x_k,\xi_{i,k})v_k =0,
    \end{equation*}
    we get
    \begin{equation}\label{eqn:approx_hblock_b11_1}
        \begin{aligned}
            \mathbb{E}_{W_k}\left[ \left\| \frac{1}{b_1} \sum\limits_{i=1}^{b_1} \hat{\nabla}_{zz}^2 g_{h_{k}}(z_k,x_k,\xi_{i,k})v_k \right\|^2\right] &= \frac{1}{b_1^2 \ell_1^2} \mathbb{E}_{W_k} \left[\left\|  \sum\limits_{i=1}^{b_1} \Delta_k^{(i)} \right\|^2\right]+ \frac{1}{b^2_1} \left\| \sum\limits_{i=1}^{b_1} \nabla_{zz}^2 g_{h_{k}}(z_k,x_k,\xi_{i,k})v_k \right\|^2\\
            &+ \frac{2}{b_1^2 \ell_1^2}\scalT{\underbrace{ \mathbb{E}_{W_k} \left[\sum\limits_{i=1}^{b_1} \Delta_k^{(i)} \right]}_{=0}}{\sum\limits_{i=1}^{b_1} \nabla_{zz}^2 g_{h_{k}}(z_k,x_k,\xi_{i,k})v_k}.
        \end{aligned}
    \end{equation}
    By independence, we have
    \begin{equation*}
        \begin{aligned}
            \mathbb{E}_{W_k} \left[\left\|  \sum\limits_{i=1}^b  \Delta_k^{(i)} \right\|^2\right] &=  \sum\limits_{i=1}^{b_1} \mathbb{E}_{W_k^{(i)}}\left[\| \Delta_k^{(i)} \|^2\right] + \underbrace{\sum\limits_{i=1}^{b_1} \sum\limits_{t \neq i} \scalT{\mathbb{E}_{W_k^{(i)}}\left[\Delta_k^{(i)} \right]}{\mathbb{E}_{W_k^{(t)}}\left[\Delta_k^{(t)}\right]}}_{=0}\\
            &=\sum\limits_{i=1}^{b_1} \left(\sum\limits_{j=1}^{\ell_1} \mathbb{E}_{w_k^{(i,j)}}\left[ \left\|\delta_k^{(i,j)} \right\|^2\right] + \underbrace{\sum\limits_{j=1}^{\ell_1} \sum\limits_{t \neq j} \scalT{\mathbb{E}_{w_k^{(i,j)}}\left[\delta_k^{(i,j)} \right]}{\mathbb{E}_{w_k^{(i,t)}}\left[\delta_k^{(i,t)} \right]}}_{=0}  \right)\\
            &=\sum\limits_{i=1}^{b_1} \sum\limits_{j=1}^{\ell_1} \mathbb{E}_{w_k^{(i,j)}} \left[ \left\| \left(c^{(i,j)}_k \left(w^{(i,j)}_k w^{(i,j) \top}_k - I_p \right) - \nabla_{zz}^2 g_{h_{k}}(z_k,x_k,\xi_{i,k}) \right)v_k \right\|^2 \right]\\
            &\leq\sum\limits_{i=1}^{b_1} \sum\limits_{j=1}^{\ell_1} \mathbb{E}_{w_k^{(i,j)}} \left[ \left\|c^{(i,j)}_k \left(w^{(i,j)}_k w^{(i,j) \top}_k - I_p \right)v_k \right\|^2 \right] .
        \end{aligned}
    \end{equation*}
    Here, the last inequality holds by the fact that $\mathbb{E}_{w_k^{(i,j)}} \left[ c^{(i,j)}_k \left(w^{(i,j)}_k w^{(i,j) \top}_k - I_p \right)v_k\right] =\nabla_{zz}^2 g_{h_{k}}(z_k,x_k,\xi_{i,k}) v_k$. Taking the expectation on $\xi_{i,k}$ for every $i=1,\cdots,b_1$, by Proposition \ref{prop:aghasi_hess_approx},%
    \begin{equation*}
        \begin{aligned}
            \mathbb{E}_{\xi_k} \left[\mathbb{E}_{W_k} \left[\left\|  \sum\limits_{i=1}^{b_1}  \Delta_k^{(i)} \right\|^2\right] \right] &\leq \sum\limits_{i=1}^{b_1} \sum\limits_{j=1}^{\ell_1} \mathbb{E}_{\xi_{i,k}} \left[\mathbb{E}_{w_k^{(i,j)}} \left[ \left\|c^{(i,j)}_k \left(w^{(i,j)}_k w^{(i,j) \top}_k - I_p \right)v_k \right\|^2 \right] \right] \\
            &\leq \sum\limits_{i=1}^{b_1} \sum\limits_{j=1}^{\ell_1}  \bigg(  4 L_{2,G}^2 (p +16)^4 h_{k}^2  +  \frac{15}{2} (p + 6)^2 \mathbb{E}_{\xi_{i,k}} \left[\|\nabla_{zz}^2 g(z_k,x_k,\xi_{i,k})\|_F^2 \right]  \bigg)\|v_k\|^2\\
            &\leq \sum\limits_{i=1}^{b_1} \sum\limits_{j=1}^{\ell_1} \bigg( 4 L_{2,G}^2 (p +16)^4 h_{k}^2%
            + %
            \frac{15}{2} (p + 6)^2 p \mathbb{E}_{\xi_{i,k}} \left[\|\nabla_{zz}^2 g(z_k,x_k,\xi_{i,k})\|^2 \right] \bigg) \|v_k\|^2.%
        \end{aligned}
    \end{equation*}
    Since for every $\xi$, we have that the function $(z, x) \mapsto g(z, x, \xi)$ is $L_{1,G}$-smooth, we have that $\| \nabla^2_{zz} g(z,x, \xi) \| \leq \| \nabla^2 g(z,x, \xi) \|  \leq L_{1,G}$. %
    Therefore, we get
    \begin{equation*}
        \begin{aligned}
            \mathbb{E}_{\xi_k} \left[\mathbb{E}_{W_k} \left[\left\|  \sum\limits_{i=1}^{b_1}  \Delta_k^{(i)} \right\|^2\right] \right] &\leq \sum\limits_{i=1}^{b_1} \sum\limits_{j=1}^{\ell_1} \bigg( 4 L_{2,G}^2 (p +16)^4 h_{k}^2 + \frac{15}{2} (p + 6)^2 p L_{1,G}^2 \bigg)\|v_k\|^2.
        \end{aligned}
    \end{equation*}
    Using this inequality in eq. \eqref{eqn:approx_hblock_b11_1}, we get%
    \begin{equation*}
        \begin{aligned}
            \mathbb{E}_{\xi_k} \left[\mathbb{E}_{W_k}\left[ \left\| \frac{1}{b_1} \sum\limits_{i=1}^{b_1} \hat{\nabla}_{zz}^2 g_{h_{k}}(z_k,x_k,\xi_{i,k})v_k \right\|^2\right] \right] &\leq 
            \frac{1}{{b_1} \ell_1}  \bigg( 4 L_{2,G}^2 (p +16)^4 h_{k}^2 + \frac{15}{2} (p + 6)^2 p L_{1,G}^2 \bigg)\|v_k\|^2\\
            &+ \frac{1}{b^2_1} \mathbb{E}_{\xi_k} \left[\left\| \sum\limits_{i=1}^{b_1} \nabla_{zz}^2 g_{h_{k}}(z_k,x_k,\xi_{i,k})v_k \right\|^2 \right].
        \end{aligned}
    \end{equation*}
    By Proposition \ref{prop:smoothing_properties}, we have that $g_{h_{k}}$ is $L_{1,G}$-smooth. Thus $\| \nabla_{zz}^2 g_{h_{k}}(z_k, x_k, \xi_{i,k}) \| \leq L_{1,G}$ and, therefore, we have
    \begin{equation*}
        \begin{aligned}
            \mathbb{E}_{\xi_k} \left[\mathbb{E}_{W_k}\left[ \left\|\frac{1}{b_1} \sum\limits_{i=1}^{b_1} \hat{\nabla}_{zz}^2 g_{h_{k}}(z_k,x_k,\xi_{i,k})v_k  \right\|^2\right] \right] &\leq \bigg( \frac{4 L_{2,G}^2 (p +16)^4}{{b_1} \ell_1} h_{k}^2 + \left(1 +\frac{15 (p + 6)^2p}{2 b_1 \ell_1} \right)   L_{1,G}^2 \bigg)\|v_k\|^2.%
        \end{aligned}
    \end{equation*}
    Now, we prove eq. \eqref{eqn:approx_hess_cross}. For $i = 1,\cdots,b$ and $j=1,\cdots,\ell$ we define
    \begin{equation*}
        \begin{aligned}
        s^{(i,j)}_k &= s_k^{(i,j)}(z_k,x_k, \xi_{i,k}), \qquad %
        \bar{\delta}^{(i,j)}_k =  s^{(i,j)}_k  u^{(i,j)}_k w^{(i,j)\top}_k v_k - \nabla_{xz}^2 g_{h_k}(z_k,x_k,\xi_{i,k})v_k,\\
        \bar{\Delta}^{(i)}_k &= \sum\limits_{j=1}^{\ell_1} \bar{\delta}^{(i,j)}_k.            
        \end{aligned}
    \end{equation*}
    Adding and subtracting $\nabla_{xz}^2 g_{h_k}(z_k,x_k,\xi_{i,k})v_k$ and developing the square, we get
    \begin{equation*}
        \begin{aligned}
            \left\| \frac{1}{b_1} \sum\limits_{i=1}^{b_1} \hat{\nabla}_{xz}^2 g(z_k,x_k,\xi_{i,k})v_k  \right\|^2 %
            &= \frac{1}{b_1^2 \ell_1^2} \left\|  \sum\limits_{i=1}^{b_1} \bar{\Delta}^{(i)}_k\right\|^2+ \frac{1}{b^2_1} \left\| \sum\limits_{i=1}^{b_1} \nabla_{xz}^2 g_{h_k}(z_k,x_k,\xi_{i,k})v_k \right\|^2\\
            &+ \frac{2}{b_1^2 \ell_1^2}\scalT{\sum\limits_{i=1}^{b_1} \bar{\Delta}^{(i)}_k}{\sum\limits_{i=1}^{b_1} \nabla_{xz}^2 g_{h_k}(z_k,x_k,\xi_{i,k})v_k}.
        \end{aligned}
    \end{equation*}
    Taking the conditional expectation on $w^{(i,j)}_k,u^{(i,j)}_k$ for every $i=1,\cdots,b_1$ and $j = 1,\cdots, \ell_1$, we have that since $w^{(i,j)}_k, u^{(i,j)}_k, \xi_{i,k}$ are independent,
    \begin{equation}\label{eqn:approx_hblock_cross}
        \begin{aligned}
            \mathbb{E}_{W_k,U_k}\left[\left\| \frac{1}{b_1} \sum\limits_{i=1}^{b_1} \hat{\nabla}_{xz}^2 g(z_k,x_k,\xi_{i,k})v_k  \right\|^2\right] &= \frac{1}{b_1^2 \ell_1^2} \underbrace{\mathbb{E}_{W_k,U_k} \left[\left\|  \sum\limits_{i=1}^{b_1} \bar{\Delta}^{(i)}_k  \right\|^2\right]}_{A} + \frac{1}{b^2_1} \underbrace{\left\| \sum\limits_{i=1}^{b_1} \nabla_{xz}^2 g_{h_k}(z_k,x_k,\xi_{i,k})v_k \right\|^2}_{B}\\
            &+ \frac{2}{b_1^2 \ell_1^2}\scalT{\underbrace{\sum\limits_{i=1}^{b_1} \mathbb{E}_{W_k,U_k}[\bar{\Delta}_k^{(i)}]}_{=0}}{\sum\limits_{i=1}^{b_1} \nabla_{xz}^2 g_{h_k}(z_k,x_k,\xi_{i,k})v_k}.
        \end{aligned}
    \end{equation}
    We bound the terms $A$ and $B$ separately. We start from $A$,
    \begin{equation*}
        \begin{aligned}
            \mathbb{E}_{W_k,U_k} \left[\left\|  \sum\limits_{i=1}^{b_1} \bar{\Delta}^{(i)}_k  \right\|^2\right] = \sum\limits_{i=1}^{b_1} \mathbb{E}_{W_k^{(i)},U_k^{(i)}} \left[\left\|  \bar{\Delta}^{(i)}_k  \right\|^2\right] + \sum\limits_{i=1}^{b_1} \sum\limits_{t \neq i} \mathbb{E}_{W_k,U_k} \left[ \scalT{\bar{\Delta}^{(i)}_k }{\bar{\Delta}^{(t)}_k}\right].
        \end{aligned}
    \end{equation*}
    Observing that, for every $i \neq t$, $w_k^{(i,j)}, w_k^{(t,j)}, u_k^{{(i,j)}}, u_{k}^{(t,j)}$ are independent, and that by Lemma \ref{lem:grad_hess_smoothing}, we have
    \begin{equation*}
        \begin{aligned}
        \mathbb{E}_{W_k^{(i)},U_k^{(i)}} \left[\bar{\Delta}^{(i)}_k \right] %
        &= \sum\limits_{j=1}^{\ell_1} \mathbb{E}_{w_k^{(i,j)},u_k^{(i,j)}}\left[s^{(i,j)}_k  u^{(i,j)}_k w^{(i,j)\top}_k \right] v_k - \nabla_{xz}^2 g_{h_k}(z_k,x_k,\xi_{i,k})v_k = 0,\\
        \end{aligned}
    \end{equation*}
    we get
    \begin{equation*}
        \begin{aligned}
            \mathbb{E}_{W_k,U_k} \left[\left\|  \sum\limits_{i=1}^{b_1} \bar{\Delta}^{(i)}_k  \right\|^2\right] = \sum\limits_{i=1}^{b_1} \mathbb{E}_{W_k^{(i)},U_k^{(i)}} \left[\left\|  \bar{\Delta}^{(i)}_k  \right\|^2\right] + \sum\limits_{i=1}^{b_1} \sum\limits_{t \neq i}  \underbrace{\scalT{\mathbb{E}_{W_k^{(i)}, U_k^{(i)}}\left[\bar{\Delta}^{(i)}_k \right] }{\mathbb{E}_{W_k^{(t)}, U_k^{(t)}} \left[\bar{\Delta}^{(t)}_k \right]}}_{=0}.
        \end{aligned}
    \end{equation*}
    Similarly, we have
    \begin{equation*}
        \begin{aligned}
            \mathbb{E}_{W_k,U_k} \left[\left\|  \sum\limits_{i=1}^{b_1} \bar{\Delta}^{(i)}_k  \right\|^2\right] &= \sum\limits_{i=1}^{b_1} \mathbb{E}_{W_k^{(i)},U_k^{(i)}} \left[\left\|  \bar{\Delta}^{(i)}_k  \right\|^2\right]\\
            &= \sum\limits_{i=1}^{b_1} \bigg(\sum\limits_{j=1}^{\ell_1} \mathbb{E}_{w_k^{(i,j)},u_k^{(i,j)}} \left[\left\|  \bar{\delta}^{(i,j)}_k  \right\|^2\right] + \sum\limits_{j=1}^{\ell_1} \sum\limits_{t \neq j} \mathbb{E}_{W_k^{(i)},U_k^{(i)}} \left[ \scalT{\bar{\delta}^{(i,j)}_k}{\bar{\delta}_k^{(i,t)}} \right]\bigg).
        \end{aligned}
    \end{equation*}
    Again, since for every $j \neq t$, $w_{k}^{(i,j)},u_{k}^{(i,j)},w_{k}^{(i,t)},u_{k}^{(i,t)}$ are independent and, by Lemma \ref{lem:grad_hess_smoothing}, we have
    \begin{equation*}
        \begin{aligned}
            \mathbb{E}_{w_k^{(i,j)}, u_k^{(i,j)}}[ \bar{\delta}_k^{(i,j)}] %
            &=\mathbb{E}_{w_k^{(i,j)},u_k^{(i,j)}}\left[s^{(i,j)}_k  u^{(i,j)}_k w^{(i,j)\top}_k \right] v_k - \nabla_{xz}^2 g_k(z_k,x_k,\xi_{i,k})v_k = 0\\
        \end{aligned}
    \end{equation*}
    we get
    \begin{equation*}
        \begin{aligned}
            \mathbb{E}_{W_k,U_k} \left[\left\|  \sum\limits_{i=1}^{b_1} \bar{\Delta}^{(i)}_k  \right\|^2\right] &= \sum\limits_{i=1}^{b_1} \bigg(\sum\limits_{j=1}^{\ell_1} \mathbb{E}_{w_k^{(i,j)},u_k^{(i,j)}} \left[\left\|  \bar{\delta}^{(i,j)}_k  \right\|^2\right] + \sum\limits_{j=1}^{\ell_1} \sum\limits_{t \neq j} \underbrace {\scalT{ \mathbb{E}_{w_k^{(i,j)},u_k^{(i,j)}} \left[\bar{\delta}^{(i,j)}_k \right]}{\mathbb{E}_{w_k^{(i,t)},u_k^{(i,t)}}\left[\bar{\delta}_k^{(i,t)}\right]}}_{=0}\bigg)\\
            &=\sum\limits_{i=1}^{b_1} \sum\limits_{j=1}^{\ell_1} \mathbb{E}_{w_k^{(i,j)},u_k^{(i,j)}} \left[\left\|s_k^{(i,j)} u_k^{(i,j)} w_k^{(i,j)\top} v_k - \nabla_{xz}^2 g_{h_k}(z_k, x_k, \xi_{i,k})v_k \right\|^2\right]\\
            &\leq \sum\limits_{i=1}^{b_1} \sum\limits_{j=1}^{\ell_1}\mathbb{E}_{w_k^{(i,j)},u_k^{(i,j)}} \left[\left\|s_k^{(i,j)} u_k^{(i,j)} w_k^{(i,j)\top} v_k \right\|^2 \right].
        \end{aligned}
    \end{equation*}
    Taking the expectation on $\xi_k$, by Proposition \ref{prop:aghasi_hess_approx}, %
    and since $\| A\|^2_F \leq rank(A) \|A\|^2$ where $\|A\|_F^2$ is the squared Frobenius norm of a matrix $A$, we have
    \begin{equation*}
        \begin{aligned}
            \mathbb{E}_{\xi_k} \left[\mathbb{E}_{W_k,U_k} \left[\left\|  \sum\limits_{i=1}^{b_1} \bar{\Delta}^{(i)}_k  \right\|^2\right] \right] %
            &\leq \sum\limits_{i=1}^{b_1} \sum\limits_{j=1}^{\ell_1} \bigg(8L_{2,G}^2 h_k^2 \left( (p + 8)^4 + 2p(d + 12)^3 \right)\\
            &+ \bigg(6(p + 4)(p + 2)p \mathbb{E}_{\xi_{i,k}}[\| \nabla_{zz}^2 g_{h_k}(z_k,x_k,\xi_{i,k}) \|^2]\\
            &+ 36 (p + 2) \min(p, d) \mathbb{E}_{\xi_{i,k}}[\| \nabla_{xz}^2 g_{h_k}(z_k,x_k,\xi_{i,k}) \|^2]\\
            &+ 30 p(d + 2)d \mathbb{E}_{\xi_{i,k}}[\|\nabla^2_{xx} g_{h_k}(z_k, x_k, \xi_{i,k}) \|^2]\bigg)\bigg) \| v_k\|^2.
        \end{aligned}
    \end{equation*}
    Since, by Proposition \ref{prop:smoothing_properties}, the function $g_{h_k}$ is $L_{1,G}$-smooth, we have
    \begin{equation}\label{eqn:bound_cross_hess_A}
        \begin{aligned}
            \mathbb{E}_{\xi_k} \left[\mathbb{E}_{W_k,U_k} \left[\left\|  \sum\limits_{i=1}^{b_1} \bar{\Delta}^{(i)}_k  \right\|^2\right] \right] &\leq \sum\limits_{i=1}^{b_1} \sum\limits_{j=1}^{\ell_1} \bigg(8L_{2,G}^2 h_k^2 \left( (p + 8)^4 + 2p(d + 12)^3 \right)\\
            &+ \bigg(6(p + 4)(p + 2)p \\
            &+ 36 (p + 2) \min(p, d) + 30 p(d + 2)d \bigg) L_{1,G}^2\bigg) \| v_k\|^2.
        \end{aligned}
    \end{equation}
    Now, we bound the term $B$ of eq. \eqref{eqn:approx_hblock_cross}. By Cauchy-Schwartz, and since $g_{h_k}$ is $L_{1,G}$-smooth, we have
    \begin{equation*}
        \begin{aligned}
        \left\| \sum\limits_{i=1}^{b_1} \nabla_{xz}^2 g_{h_k}(z_k,x_k,\xi_{i,k})v_k \right\|^2 &\leq b_1 \sum\limits_{i=1}^{b_1} \| \nabla_{xz}^2 g_{h_k}(z_k, x_k, \xi_{i,k})v_k\|^2\\
        &\leq b_1 \sum\limits_{i=1}^{b_1} \| \nabla_{xz}^2 g_{h_k}(z_k, x_k, \xi_{i,k}) \|^2 \|v_k\|^2 \\  
        &\leq b_1^2 L_{1,G}^2 \|v_k\|^2.
        \end{aligned}
    \end{equation*}
    Using this inequality and eq. \eqref{eqn:bound_cross_hess_A} in eq. \eqref{eqn:approx_hblock_cross}, we get the claim.
\end{proof}
\noindent In the following lemma, we provide bound on the norm of the search directions used in Algorithm \ref{alg:zoba}.
\begin{lemma}[Bound on direction norm]\label{lem:bound_search_directions}
    Let Assumptions \ref{asm:bc_condition},\ref{asm:F_smooth},\ref{asm:G_asm} holds. Let $D_z^k$, $D_v^k$ and $D_x^k$ be the search directions computed at iteration $k \in \mathbb{N}$ with eq. \eqref{eqn:dz}, \eqref{eqn:dv} and \eqref{eqn:dx}. Let $(z_k)_{k \in \mathbb{N}},(v_k)_{k \in \mathbb{N}}$ and $(x_k)_{k \in \mathbb{N}}$ be the sequences generated by Algorithm \ref{alg:zoba}. Then, for every $k \in \mathbb{N}$, we have
    \begin{equation}\label{eqn:bound_norm_dz}
        \begin{aligned}
        \mathbb{E}_k \left[\| D_z^k \|^2 \right]  &\leq 2 \left(2 + \frac{p + 2}{b_1 \ell_1} \right) L_{1,G}^2 \| z_k - z^*(x_k) \|^2\\
            &+ \left( \frac{(p + 6)^3}{2 b_1 \ell_1} + \left( \frac{3}{b_1} + 1 \right) (p + 3)^3 \right)L_{1,G}^2 h_{k}^2 + 2 \left(3 +\frac{p + 2}{\ell_1} \right)\frac{\sigma_{1,G}^2}{b_1}.
        \end{aligned}
    \end{equation}
    Moreover, let $C_v^k = 4 \bigg( \frac{4 L_{2,G}^2 (p +16)^4}{{b_1} \ell_1} h_{k}^2 + \left(1 +\frac{15 (p + 6)^2p}{2 b_1 \ell_1} \right)   L_{1,G}^2 \bigg)$. Then, for every $k \in \mathbb{N}$,
    \begin{equation}\label{eqn:bound_norm_dv}
        \begin{aligned}
            \mathbb{E}_k \left[ \| D_v^k \|^2 \right] &\leq C_v^k\|v_k - v^*(x_k)\|^2+ C_v^k \frac{L_{0,F}^2}{\mu_G^2} \\
            &+ 4 \left(1 + \frac{p + 2}{b_2 \ell_2} \right)L_{0,F}^2 + 4\left(\frac{p + 2}{\ell_2} + 3 \right ) \frac{\sigma_{1,F}^2}{b_2}\\
            & +2\left( \frac{(p + 6)^3}{2 b_2 \ell_2} + \left(\frac{3}{b_2} + 1 \right)(p + 3)^{3}  \right) L_{1,F}^2 h_{k}^2,
        \end{aligned}        
    \end{equation}
    and,
    \begin{equation}\label{eqn:bound_norm_dx}
        \begin{aligned}
            \mathbb{E}_k\left[ \left\| D_x^k \right\|^2 \right] &\leq 2C_x^k\| v_k - v^*(x_k)\|^2+2C_x^k \frac{L_{0,F}^2}{\mu_G^2}+2\left( \frac{(d + 6)^3}{2 b_2 \ell_2} + \left(\frac{3}{b_2} + 1 \right)(d + 3)^{3}  \right) L_{1,F}^2 h_{k}^2\\
            &+4 \left(1 + \frac{d + 2}{b_2 \ell_2} \right) L_{0,F}^2 + 4\left(\frac{d + 2}{b_2 \ell_2} +\frac{3}{b_2} \right )\sigma_{1,F}^2,
        \end{aligned}
    \end{equation}
    with $C_x^k=2 \bigg(\frac{1}{b_1 \ell_1} \bigg(  C_d^{(1)}h_k^2+ C_d^{(2)}\bigg) + L_{1,G}^2\bigg)$ where 
    \begin{equation*}
        \begin{aligned}
            C_d^{(1)} &= 8L_{2,G}^2\left( (p + 8)^4 + 2p(d + 12)^3 \right) \\
            C_d^{(2)} &= 6(p + 4)(p + 2)p + 36 (p + 2) \min(p, d)+ 30 p(d + 2)d L_{1,G}^2 \\
        \end{aligned}
    \end{equation*}
\end{lemma}
\begin{proof}
    We start by proving eq. \eqref{eqn:bound_norm_dz}. By eq. \eqref{eqn:dz} and Lemma \ref{lem:approx_error},
    \begin{equation*}
        \begin{aligned}
            \mathbb{E}_k \left[\| D_z^k \|^2 \right] &\leq 2 \left(2 + \frac{p + 2}{b_1 \ell_1} \right) \left\|\nabla_z G(z_k, x_k) \right\|^2\\
        &+ \left( \frac{(p + 6)^3}{2 b_1 \ell_1} + \left( \frac{3}{b_1} + 1 \right) (p + 3)^3 \right)L_{1,G}^2 h_{k}^2\\
        &+ 2 \left(3 +\frac{p + 2}{\ell_1} \right)\frac{\sigma_{1,G}^2}{b_1}.
        \end{aligned}
    \end{equation*}
    Since $\nabla_z G(z^*(x_k), x_k) = 0$, by $L_{1,G}$-smoothness, we have
    \begin{equation*}
        \begin{aligned}
            \mathbb{E}_k \left[\| D_z^k \|^2 \right] &\leq 2 \left(2 + \frac{p + 2}{b_1 \ell_1} \right) L_{1,G}^2 \| z_k - z^*(x_k) \|^2\\
            &+ \left( \frac{(p + 6)^3}{2 b_1 \ell_1} + \left( \frac{3}{b_1} + 1 \right) (p + 3)^3 \right)L_{1,G}^2 h_{k}^2\\
            &+ 2 \left(3 +\frac{p + 2}{\ell_1} \right)\frac{\sigma_{1,G}^2}{b_1}.
        \end{aligned}
    \end{equation*}
    Now we prove eq. \eqref{eqn:bound_norm_dv}. By eq. \eqref{eqn:dv}, we have
    \begin{equation*}
        \begin{aligned}
            \mathbb{E}_k \left[ \| D_v^k \|^2 \right] &= \mathbb{E}_k \left[ \left\| \left(\frac{1}{b_1} \sum\limits_{i=1}^{b_1} \hat{\nabla}^2_{zz}g_{h_{k}}(z_k,x_k,\xi_{i,k}, W_{i,k}) \right)v_k + \frac{1}{b_2}\sum\limits_{i=1}^{b_2} \hat{\nabla}_z f_{h_{k}}(z_k,x_k,\zeta_{i,k}, \bar{W}_{i,k}) \right\|^2 \right]\\
            &\leq 2 \mathbb{E}_k \left[ \left\| \left(\frac{1}{b_1} \sum\limits_{i=1}^{b_1} \hat{\nabla}^2_{zz}g_{h_{k}}(z_k,x_k,\xi_{i,k}, W_{i,k}) \right)v_k \right\|^2 \right]\\
            &+ 2 \mathbb{E}_k \left[ \left\| \frac{1}{b_2}\sum\limits_{i=1}^{b_2} \hat{\nabla}_z f_{h_{k}}(z_k,x_k,\zeta_{i,k}, \bar{W}_{i,k}) \right\|^2 \right].
        \end{aligned}
    \end{equation*}
    By Lemma \ref{lem:approx_error} and Lemma \ref{lem:hess_approx_error}, we get
    \begin{equation*}
        \begin{aligned}
            \mathbb{E}_k \left[ \| D_v^k \|^2 \right] &\leq 2 \bigg( \frac{4 L_{2,G}^2 (p +16)^4}{{b_1} \ell_1} h_{k}^2 + \left(1 +\frac{15 (p + 6)^2p}{2 b_1 \ell_1} \right)   L_{1,G}^2 \bigg)\|v_k\|^2\\
            &+ 4 \left(1 + \frac{p + 2}{b_2 \ell_2} \right) \left\|\nabla_z F(z_k, x_k) \right\|^2\\
        & +2\left( \frac{(p + 6)^3}{2 b_2 \ell_2} + \left(\frac{3}{b_2} + 1 \right)(p + 3)^{3}  \right) L_{1,F}^2 h_{k}^2 + 4\left(\frac{p + 2}{\ell_2} + 3 \right ) \frac{\sigma_{1,F}^2}{b_2}.
        \end{aligned}
    \end{equation*}
    Since $F$ is $L_{0,F}$-Lipschitz continuous, we have 
    \begin{equation*}
        \begin{aligned}
            \mathbb{E}_k \left[ \| D_v^k \|^2 \right] &\leq 2 \bigg( \frac{4 L_{2,G}^2 (p +16)^4}{{b_1} \ell_1} h_{k}^2 + \left(1 +\frac{15 (p + 6)^2p}{2 b_1 \ell_1} \right)   L_{1,G}^2 \bigg)\|v_k\|^2\\
            &+ 4 \left(1 + \frac{p + 2}{b_2 \ell_2} \right)L_{0,F}^2\\
        & +2\left( \frac{(p + 6)^3}{2 b_2 \ell_2} + \left(\frac{3}{b_2} + 1 \right)(p + 3)^{3}  \right) L_{1,F}^2 h_{k}^2 + 4\left(\frac{p + 2}{\ell_2} + 3 \right ) \frac{\sigma_{1,F}^2}{b_2}.
        \end{aligned}
    \end{equation*}
    Adding and subtracting $v^*(x_k)$, we get
    \begin{equation*}
        \begin{aligned}
            \mathbb{E}_k \left[ \| D_v^k \|^2 \right] &\leq 4 \bigg( \frac{4 L_{2,G}^2 (p +16)^4}{{b_1} \ell_1} h_{k}^2 + \left(1 +\frac{15 (p + 6)^2p}{2 b_1 \ell_1} \right)   L_{1,G}^2 \bigg)\|v_k - v^*(x_k)\|^2\\
            &+ 4 \bigg( \frac{4 L_{2,G}^2 (p +16)^4}{{b_1} \ell_1} h_{k}^2 + \left(1 +\frac{15 (p + 6)^2p}{2 b_1 \ell_1} \right)   L_{1,G}^2 \bigg) \| v^*(x_k)\|^2\\
            &+ 4 \left(1 + \frac{p + 2}{b_2 \ell_2} \right)L_{0,F}^2\\
        & +2\left( \frac{(p + 6)^3}{2 b_2 \ell_2} + \left(\frac{3}{b_2} + 1 \right)(p + 3)^{3}  \right) L_{1,F}^2 h_{k}^2 + 4\left(\frac{p + 2}{\ell_2} + 3 \right ) \frac{\sigma_{1,F}^2}{b_2}.
        \end{aligned}
    \end{equation*}
    Let $C_v^k = 4 \bigg( \frac{4 L_{2,G}^2 (p +16)^4}{{b_1} \ell_1} h_{k}^2 + \left(1 +\frac{15 (p + 6)^2p}{2 b_1 \ell_1} \right)   L_{1,G}^2 \bigg)$. By Lemma \ref{lem:bound_norm_v}, we have
    \begin{equation*}
        \begin{aligned}
            \mathbb{E}_k \left[ \| D_v^k \|^2 \right] &\leq C_v^k\|v_k - v^*(x_k)\|^2+ C_v^k \frac{L_{0,F}^2}{\mu_G^2} \\
            &+ 4 \left(1 + \frac{p + 2}{b_2 \ell_2} \right)L_{0,F}^2 + 4\left(\frac{p + 2}{\ell_2} + 3 \right ) \frac{\sigma_{1,F}^2}{b_2}\\
            & +2\left( \frac{(p + 6)^3}{2 b_2 \ell_2} + \left(\frac{3}{b_2} + 1 \right)(p + 3)^{3}  \right) L_{1,F}^2 h_{k}^2.
        \end{aligned}
    \end{equation*}
    Finally, we prove eq. \eqref{eqn:bound_norm_dx}. By eq. \eqref{eqn:dx},
    \begin{equation*}
        \begin{aligned}
            \mathbb{E}_k\left[ \left\| D_x^k \right\|^2 \right] &= \mathbb{E}_k \left[ \left\| \left(\frac{1}{b_1} \sum\limits_{i=1}^{b_1} \hat{\nabla}^2_{xz}g_k(z_k,x_k,\xi_{i,k}) \right)v_k + \frac{1}{b_2}\sum\limits_{i=1}^{b_2} \hat{\nabla}_x f_k(z_k,x_k,\zeta_{i,k})\right\|^2 \right]\\
            &\leq 2 \mathbb{E}_k\left[\left\| \left(\frac{1}{b_1} \sum\limits_{i=1}^{b_1} \hat{\nabla}^2_{xz}g_k(z_k,x_k,\xi_{i,k}) \right)v_k  \right\|^2 \right]\\
            &+2\mathbb{E}_k \left[ \left\| \frac{1}{b_2}\sum\limits_{i=1}^{b_2} \hat{\nabla}_x f_k(z_k,x_k,\zeta_{i,k}) \right\|^2 \right].
        \end{aligned}
    \end{equation*}
    By Lemma \ref{lem:approx_error} and Lemma \ref{lem:hess_approx_error}, we have
    \begin{equation*}
        \begin{aligned}
            \mathbb{E}_k\left[ \left\| D_x^k \right\|^2 \right] &\leq 2 \bigg(\frac{1}{b_1 \ell_1} \bigg(8L_{2,G}^2 h_k^2 \left( (p + 8)^4 + p(d + 12)^3 \right)\\
            &+ 6(p + 4)(p + 2)p + 36 (p + 2) \min(p, d)\\
            &+ 30 p(d + 2)d L_{1,G}^2\bigg) + L_{1,G}^2\bigg)\| v_k\|^2\\
            &+4 \left(1 + \frac{d + 2}{b_2 \ell_2} \right) \left\|\nabla_x F(z_k, x_k) \right\|^2\\
        & +2\left( \frac{(d + 6)^3}{2 b_2 \ell_2} + \left(\frac{3}{b_2} + 1 \right)(d + 3)^{3}  \right) L_{1,F}^2 h_{k}^2\\
        &+ 4\left(\frac{d + 2}{b_2 \ell_2} +\frac{3}{b_2} \right )\sigma_{1,F}^2.
        \end{aligned}
    \end{equation*}
    By $L_{0,F}$-Lipschitz continuity of $F$,
    \begin{equation*}
        \begin{aligned}
            \mathbb{E}_k\left[ \left\| D_x^k \right\|^2 \right] &\leq 2 \bigg(\frac{1}{b_1 \ell_1} \bigg(8L_{2,G}^2 h_k^2 \left( (p + 8)^4 + 2p(d + 12)^3 \right)\\
            &+ 6(p + 4)(p + 2)p + 36 (p + 2) \min(p, d)\\
            &+ 30 p(d + 2)d L_{1,G}^2\bigg) + L_{1,G}^2\bigg)\| v_k\|^2\\        
            & +2\left( \frac{(d + 6)^3}{2 b_2 \ell_2} + \left(\frac{3}{b_2} + 1 \right)(d + 3)^{3}  \right) L_{1,F}^2 h_{k}^2\\
            &+4 \left(1 + \frac{d + 2}{b_2 \ell_2} \right) L_{0,F}^2 + 4\left(\frac{d + 2}{b_2 \ell_2} +\frac{3}{b_2} \right )\sigma_{1,F}^2.
        \end{aligned}
    \end{equation*}
    Let $C_x^k=2 \bigg(\frac{1}{b_1 \ell_1} \bigg(8L_{2,G}^2 h_k^2 \left( (p + 8)^4 + 2p(d + 12)^3 \right) + 6(p + 4)(p + 2)p + 36 (p + 2) \min(p, d)+ 30 p(d + 2)d L_{1,G}^2\bigg) + L_{1,G}^2\bigg)$. Adding and subtracting $v^*(x_k)$, by Lemma \ref{lem:bound_norm_v}, we get
    \begin{equation*}
        \begin{aligned}
            \mathbb{E}_k\left[ \left\| D_x^k \right\|^2 \right] &\leq 2C_x^k\| v_k - v^*(x_k)\|^2\\
            &+2C_x^k \frac{L_{0,F}^2}{\mu_G^2} +2\left( \frac{(d + 6)^3}{2 b_2 \ell_2} + \left(\frac{3}{b_2} + 1 \right)(d + 3)^{3}  \right) L_{1,F}^2 h_{k}^2\\
            &+4 \left(1 + \frac{d + 2}{b_2 \ell_2} \right) L_{0,F}^2 + 4\left(\frac{d + 2}{b_2 \ell_2} +\frac{3}{b_2} \right )\sigma_{1,F}^2.
        \end{aligned}
    \end{equation*}

\end{proof}
\noindent Now, in the next lemma, we provides bounds on the error of the sequences $z_k, v_k$ generated by Algorithm~\ref{alg:zoba} relative to the corresponding optimal solutions $z^*(x_k), v^*(x_k)$ at each iterate $x_k$.
\begin{lemma}[Bound of sequences]\label{lem:smooth_bound_zv}
    Let Assumptions \ref{asm:F_smooth}, \ref{asm:G_asm} holds. Let $(z_k)_{k\in \mathbb{N}},(v_k)_{k\in \mathbb{N}}$ be the sequences generated by Algorithm \ref{alg:zoba}. Let $\omega_k^z = \left( 1 + \frac{4}{\mu_G \rho_k} \right)$, the for every $k \in \mathbb{N}$
    \begin{equation}\label{eqn:bound_z_smooth}
        \begin{aligned}
        \mathbb{E}_k\left[\| z_{k + 1} - z^*(x_{k + 1})\|^2\right] &\leq  \left(1 - \frac{\mu_G \rho_k}{2}\right)\|z_{k} -z^*(x_k) \|^2 + 2\rho_k^2 \mathbb{E}_k\left[\left\| D_z^k \right\|^2\right]\\
        &+2 \omega_k^zL_*^2 \gamma_k^2 \mathbb{E}_k \left[\left\| D_x^k \right\|^2 \right] + \frac{L_{1,G}^2}{4 \mu_G}(d + 3)^{3} \rho_k h_{k}^2.   
        \end{aligned}
    \end{equation}
    Moreover, let $\omega_{v,1} = \left( L_{1,F} + \frac{L_{0,F} L_{2,G}}{\mu_G} \right)$, $\omega_{v,2} = \frac{16 L_{2,G}^2}{9} \left( (d + 5)^{5/2} + (d + 3)^{3/2} \right)^2 $, $\omega_k^v = \left( 1 + \frac{2}{\mu_G \rho_k} \right)$ and $h_{k} \leq \frac{\mu_G}{4 \sqrt{\omega_{v,2}}}$. Then, we have
    \begin{equation}\label{eqn:bound_v_smooth}
        \begin{aligned}
        \mathbb{E}_k\left[ \| v_{k + 1} - v^*(x_{k + 1}) \|^2 \right] &\leq \left(1 - \mu_G\rho_k \right)\|v_k - v^*(x_k) \|^2 +4\frac{\omega_{v,1}^2}{\mu_G}\rho_k \|z_k - z^*(x_k)\|^2\\
        &+ 2\rho_k^2 \mathbb{E}_k\left[\|D_v^k\|^2 \right] + 2\omega_k^vL_*^2 \gamma_k^2  \mathbb{E}_k \left[\| D_x^k\|^2\right]\\
        &+ \left(2  \frac{L_{1,G}^2(d + 3)^{3}}{\mu_G}  + 4 \frac{\omega_{v,2}L_{0,F}^2}{\mu_G^3} \right) \rho_k h_{k}^2.        
        \end{aligned}
    \end{equation}
\end{lemma}
\begin{proof}
We start by proving eq. \eqref{eqn:bound_z_smooth}. Adding and subtracting $z^*(x_k)$,
\begin{equation}\label{eqn:seq_z_1}
    \begin{aligned}
        \| z_{k + 1} - z^*(x_{k + 1})\|^2 &= \| z_{k + 1} - z^*(x_k) + z^*(x_k) - z^*(x_{k + 1})\|^2\\
        &=\|z_{k+1} -z^*(x_k) \|^2 + \|z^*(x_{k + 1}) - z^*(x_k) \|^2 + 2 \scalT{z_{k + 1} - z^*(x_k)}{z^*(x_k) - z^*(x_{k + 1})}\\
        &=\underbrace{\|z_{k+1} -z^*(x_k) \|^2}_{A} + \underbrace{\|z^*(x_{k + 1}) - z^*(x_k) \|^2}_{B} \underbrace{- 2 \scalT{z_{k + 1} - z^*(x_k)}{z^*(x_{k + 1}) - z^*(x_k)}}_{C}.
    \end{aligned}
\end{equation}
We start by bounding $A$. By eq. \eqref{eqn:z_update}, we have
\begin{equation*}
    \|z_{k+1} -z^*(x_k) \|^2 = \|z_{k} -z^*(x_k) \|^2 + \rho_k^2 \left\| D_z^k \right\|^2 - 2 \rho_k\scalT{D_z^k}{z_k - z^*(x_k)}.
\end{equation*}
Taking the conditional expectation, by eq. \eqref{eqn:dz},%
\begin{equation*}
    \mathbb{E}_k\left[\|z_{k+1} -z^*(x_k) \|^2\right] = \|z_{k} -z^*(x_k) \|^2 + \rho_k^2 \mathbb{E}_k\left[\left\| D_z^k \right\|^2\right] - 2 \rho_k\scalT{\nabla_z G_{h_{k}}(z_k,x_k) }{z_k - z^*(x_k)}.
\end{equation*}
Adding and subtracting $\nabla_z G(z_k, x_k)$, we get
\begin{equation*}
    \begin{aligned}
        \mathbb{E}_k\left[\|z_{k+1} -z^*(x_k) \|^2\right] &= \|z_{k} -z^*(x_k) \|^2 + \rho_k^2 \mathbb{E}_k\left[\left\| D_z^k \right\|^2\right]\\
        &- 2 \rho_k\scalT{\nabla_z G(z_k,x_k)}{z_k - z^*(x_k)}\\
        &- 2 \rho_k\scalT{\nabla_z G_{h_{k}}(z_k,x_k) - \nabla_z G(z_k, x_k) }{z_k - z^*(x_k)}.
    \end{aligned}
\end{equation*}
By $\mu_G$-strong convexity of $G$, we have
\begin{equation*}
    \begin{aligned}
        \mathbb{E}_k\left[\|z_{k+1} -z^*(x_k) \|^2\right] &= (1 - 2\mu_G \rho_k)\|z_{k} -z^*(x_k) \|^2 + \rho_k^2 \mathbb{E}_k\left[\left\| D_z^k\right\|^2\right]\\
        &- 2 \rho_k\scalT{\nabla_z G_{h_{k}}(z_k,x_k) - \nabla_z G(z_k, x_k) }{z_k - z^*(x_k)}.
    \end{aligned}
\end{equation*}
By Cauchy-Schwartz and Lemma \ref{lem:smoothing_error},
\begin{equation*}
    \begin{aligned}
        \mathbb{E}_k\left[\|z_{k+1} -z^*(x_k) \|^2\right] &= (1 - 2\mu_G \rho_k)\|z_{k} -z^*(x_k) \|^2 + \rho_k^2 \mathbb{E}_k\left[\left\| D_z^k \right\|^2\right]\\
        &+ 2 \rho_k \left( \frac{L_{1,G}}{2}(d + 3)^{3/2}h_{k} \right) \|z_k - z^*(x_k)\|.
    \end{aligned}
\end{equation*}
By Young's Lemma with parameter $\nu_1 > 0$,
\begin{equation*}
    \begin{aligned}
        \mathbb{E}_k\left[\|z_{k+1} -z^*(x_k) \|^2\right] &= (1 - 2\mu_G \rho_k)\|z_{k} -z^*(x_k) \|^2 + \rho_k^2 \mathbb{E}_k\left[\left\| D_z^k \right\|^2\right]\\
        &+ \frac{\rho_k}{\nu_1} \left( \frac{L_{1,G}}{2}(d + 3)^{3/2}h_{k} \right)^2 + \nu_1 \rho_k  \|z_k - z^*(x_k)\|^2.
    \end{aligned}
\end{equation*}
Let $\nu_1 = \mu_G$,
\begin{equation}\label{eqn:seq_z_bound_a}
    \begin{aligned}
        \mathbb{E}_k\left[\|z_{k+1} -z^*(x_k) \|^2\right] &= (1 - \mu_G \rho_k)\|z_{k} -z^*(x_k) \|^2 + \rho_k^2 \mathbb{E}_k\left[\left\| D_z^k \right\|^2\right]\\
        &+ \frac{L_{1,G}^2}{4 \mu_G}(d + 3)^{3} \rho_k h_{k}^2.
    \end{aligned}
\end{equation}
Now we bound the term $B$. By Lemma \ref{lem:smt_setting_smooth_value} and eq. \eqref{eqn:x_update}, we have
\begin{equation}\label{eqn:seq_z_bound_b}
    \|z^*(x_{k + 1}) - z^*(x_k) \|^2 \leq L_*^2 \|x_{k + 1} - x_k \|^2 = L_*^2 \gamma_k^2 \left\| D_x^k \right\|^2.
\end{equation}
Finally, we bound the term $C$. By eq. \eqref{eqn:x_update} and Cauchy-Schwartz,
\begin{equation*}
    \begin{aligned}
        - 2 \scalT{z_{k + 1} - z^*(x_k)}{z^*(x_{k + 1}) - z^*(x_k)} &= - 2 \scalT{z_{k} - z^*(x_k)}{z^*(x_{k + 1}) - z^*(x_k)}\\
        &+2\rho_k \scalT{D_z^k}{z^*(x_{k + 1}) - z^*(x_k)}\\
        &\leq - 2 \scalT{z_{k} - z^*(x_k)}{z^*(x_{k + 1}) - z^*(x_k)}\\
        &+2\rho_k \| D_z^k\| \| z^*(x_{k + 1}) - z^*(x_k) \|.
    \end{aligned}
\end{equation*}
By Lemma \ref{lem:smt_setting_smooth_value}, $z^*$ is $L_*$-Lipschitz. Thus,
\begin{equation*}
    \begin{aligned}
        - 2 \scalT{z_{k + 1} - z^*(x_k)}{z^*(x_{k + 1}) - z^*(x_k)} &\leq - 2 \scalT{z_{k} - z^*(x_k)}{z^*(x_{k + 1}) - z^*(x_k)}\\
        &+2 L_* \rho_k \| D_z^k\| \| x_{k + 1} - x_k \|.
    \end{aligned}
\end{equation*}
By Young's inequality and eq. \eqref{eqn:x_update},
\begin{equation*}
    \begin{aligned}
        - 2 \scalT{z_{k + 1} - z^*(x_k)}{z^*(x_{k + 1}) - z^*(x_k)} &\leq 
        - 2 \scalT{z_{k} - z^*(x_k)}{z^*(x_{k + 1}) - z^*(x_k)}\\
        &+\rho_k^2 \left\| D_z^k \right\|^2 + L_*^2 \gamma_k^2 \left\| D_x^k \right\|^2.
    \end{aligned}
\end{equation*}
By Cauchy-Schwartz and Lemma \ref{lem:smt_setting_smooth_value},%
\begin{equation*}
    \begin{aligned}
        - 2 \scalT{z_{k + 1} - z^*(x_k)}{z^*(x_{k + 1}) - z^*(x_k)} &\leq 2 L_* \| z_{k} - z^*(x_k) \| \|x_{k + 1} - x_k\|\\
        &+\rho_k^2 \left\| D_z^k \right\|^2 + L_*^2 \gamma_k^2 \left\| D_x^k \right\|^2.
    \end{aligned}
\end{equation*}
Let $\nu_2 > 0$, by $ab \leq \nu_2a^2 +\frac{b^2}{\nu_2}$
\begin{equation*}
    \begin{aligned}
        - 2 \scalT{z_{k + 1} - z^*(x_k)}{z^*(x_{k + 1}) - z^*(x_k)} &\leq 2 \nu_2 \| z_{k} - z^*(x_k) \|^2 +2 \frac{L_*^2}{\nu_2} \|x_{k + 1} - x_k\|^2\\
        &+\rho_k^2 \left\| D_z^k \right\|^2 + L_*^2 \gamma_k^2 \left\| D_x^k \right\|^2.
    \end{aligned}
\end{equation*}
Let $\nu_2 = \frac{\mu_G \rho_k}{4}$,
\begin{equation}\label{eqn:seq_z_bound_c}
    \begin{aligned}
        - 2 \scalT{z_{k + 1} - z^*(x_k)}{z^*(x_{k + 1}) - z^*(x_k)} &\leq \frac{\mu_G}{2} \rho_k \| z_{k} - z^*(x_k) \|^2 + \frac{8 L_*^2 \gamma_k^2}{\mu_G \rho_k} \| D_x^k\|^2\\
        &+\rho_k^2 \left\| D_z^k \right\|^2 + L_*^2 \gamma_k^2 \left\| D_x^k \right\|^2.
    \end{aligned}
\end{equation}
Restarting from eq. \eqref{eqn:seq_z_1}, taking the conditional expectation and using equations \eqref{eqn:seq_z_bound_a}, \eqref{eqn:seq_z_bound_b} and \eqref{eqn:seq_z_bound_c}, we get
\begin{equation*}
    \begin{aligned}
        \mathbb{E}_k\left[\| z_{k + 1} - z^*(x_{k + 1})\|^2\right] &\leq  \left(1 - \frac{\mu_G \rho_k}{2}\right)\|z_{k} -z^*(x_k) \|^2 + 2\rho_k^2 \mathbb{E}_k\left[\left\| D_z^k \right\|^2\right]\\
        &+2 \left( 1 + \frac{4}{\mu_G \rho_k} \right)L_*^2 \gamma_k^2 \mathbb{E}_k \left[\left\| D_x^k \right\|^2 \right]\\
        &+ \frac{L_{1,G}^2}{4 \mu_G}(d + 3)^{3} \rho_k h_{k}^2\\
    \end{aligned}
\end{equation*}
Now we prove eq. \eqref{eqn:bound_v_smooth}. Adding and subtracting $v^*(x_k)$, we have
\begin{equation}\label{eqn:bound_v_main}
    \begin{aligned}
    \|v_{k + 1} - v^*(x_{k + 1})\|^2 &= \| v_{k + 1} - v^*(x_k) + v^*(x_k) - v^*(x_{k+1}) \|^2\\
    &= \| v_{k + 1} - v^*(x_k) \|^2 + \|v^*(x_{k + 1}) - v^*(x_k)\|^2 + 2 \scalT{v_{k + 1} - v^*(x_k)}{v^*(x_k) - v^*(x_{k + 1})}\\
    &= \underbrace{\| v_{k + 1} - v^*(x_k) \|^2}_{A} + \underbrace{\|v^*(x_{k + 1}) - v^*(x_k)\|^2}_{B} \underbrace{- 2 \scalT{v_{k + 1} - v^*(x_k)}{v^*(x_{k + 1}) - v^*(x_k)}}_{C}.
    \end{aligned}
\end{equation}
We start bounding the term $A$. By eq. \eqref{eqn:v_update} we have
\begin{equation*}
    \begin{aligned}
        \| v_{k + 1} - v^*(x_k) \|^2 &=\| v_{k} - \rho_k D_v^k - v^*(x_k) \|^2\\
        &=\|v_k - v^*(x_k) \|^2 + \rho_k^2 \|D_v^k\|^2 - 2 \rho_k \scalT{D_v^k}{v_k - v^*(x_k)}.
    \end{aligned}
\end{equation*}
Let $\hat{D}_v^k = \nabla_{zz}G_{h_{k}}(z_k,x_k)v_k + \nabla_z F_{h_{k}}(z_k,x_k)$. Observing that $\mathbb{E}_k[D_v^k] = \hat{D}_v^k$, taking the conditional expectation, we have
\begin{equation*}
    \begin{aligned}
        \mathbb{E}_k \left[ \| v_{k + 1} - v^*(x_k) \|^2 \right] &=\|v_k - v^*(x_k) \|^2 + \rho_k^2 \mathbb{E}_k\left[\|D_v^k\|^2 \right] - 2 \rho_k \scalT{\hat{D}_v^k}{v_k - v^*(x_k)}.
    \end{aligned}
\end{equation*}
Adding and subtracting $\bar{D}_v^k = \nabla_{zz}^2G(z_k,x_k)v_k + \nabla_z F(z_k,x_k)$,
\begin{equation}\label{eqn:bound_v_A_1}
    \begin{aligned}
        \mathbb{E}_k \left[ \| v_{k + 1} - v^*(x_k) \|^2 \right] &=\|v_k - v^*(x_k) \|^2 + \rho_k^2 \mathbb{E}_k\left[\|D_v^k\|^2 \right]\\
        &- 2 \rho_k \scalT{\hat{D}_v^k -\bar{D}_v^k}{v_k - v^*(x_k)} \underbrace{- 2\rho_k\scalT{\bar{D}_v^k}{v_k - v^*(x_k)}}_{A_1}.
    \end{aligned}
\end{equation}
We focus on $A_1$. Let $\bar{D}_{v,*}^k = \nabla_{zz}^2G(z^*(x_k),x_k)v^*(x_k) + \nabla_z F(z^*(x_k), x_k) = 0$. Thus, we have
\begin{equation*}
    \begin{aligned}
        - 2\rho_k\scalT{\bar{D}_v^{k}}{v_k - v^*(x_k)} &= - 2\rho_k\scalT{\bar{D}_v^{k} - \bar{D}_{v,*}^k}{v_k - v^*(x_k)}\\
        &= - 2\rho_k\scalT{\nabla_{zz}^2G(z_k,x_k)v_k - \nabla_{zz}^2G(z^*(x_k),x_k)v^*(x_k)}{v_k - v^*(x_k)}\\
        &- 2\rho_k\scalT{\nabla_z F(z_k,x_k)-  \nabla_z F(z^*(x_k), x_k)}{v_k - v^*(x_k)}.
    \end{aligned}
\end{equation*}
Adding and subtracting $\nabla_{zz}^2 G(z_k,x_k)v^*(x_k)$,
\begin{equation*}
    \begin{aligned}
        - 2\rho_k\scalT{\bar{D}_v^{k}}{v_k - v^*(x_k)} &=- 2\rho_k\scalT{\nabla_{zz}^2G(z_k,x_k)v_k - \nabla_{zz}^2 G(z_k,x_k)v^*(x_k) }{v_k - v^*(x_k)}\\
        &- 2\rho_k\scalT{\nabla_{zz}^2 G(z_k,x_k)v^*(x_k) - \nabla_{zz}^2G(z^*(x_k),x_k)v^*(x_k)}{v_k - v^*(x_k)}\\
        &- 2\rho_k\scalT{\nabla_z F(z_k,x_k)-  \nabla_z F(z^*(x_k), x_k)}{v_k - v^*(x_k)}.
    \end{aligned}
\end{equation*}
By strong-convexity of $G$, Cauchy-Schwartz and Lemma \ref{lem:bound_norm_v}, we get %
\begin{equation*}
    \begin{aligned}
        - 2\rho_k\scalT{\bar{D}_v^{k}}{v_k - v^*(x_k)} &\leq - 2\rho_k\mu_G \| v_k - v^*(x_k) \|^2\\
        &+ \frac{2L_{0,F}}{\mu_G}\rho_k \|\nabla_{zz}^2 G(z_k,x_k) - \nabla_{zz}^2G(z^*(x_k),x_k)\| \|v_k - v^*(x_k)\|\\
        &+ 2\rho_k \| \nabla_z F(z_k,x_k)-  \nabla_z F(z^*(x_k), x_k) \| \|v_k - v^*(x_k)\|.
    \end{aligned}
\end{equation*}
Since hessian of $G$ are $L_{2,G}$-Lipschitz and $F$ is $L_{1,F}$-smooth, we have
\begin{equation*}
    \begin{aligned}
        - 2\rho_k\scalT{\bar{D}_v^{k}}{v_k - v^*(x_k)} &\leq - 2\rho_k\mu_G \| v_k - v^*(x_k) \|^2\\
        &+ \frac{2L_{0,F}}{\mu_G} L_{2,G}\rho_k \|z_k - z^*(x_k)\| \|v_k - v^*(x_k)\|\\
        &+ 2L_{1,F}\rho_k \| z_k - z^*(x_k) \| \|v_k - v^*(x_k)\|.
    \end{aligned}
\end{equation*}
Let $\omega_{v,1} = \left( L_{1,F} + \frac{L_{0,F} L_{2,G}}{\mu_G} \right)$, we get
\begin{equation*}
    \begin{aligned}
        - 2\rho_k\scalT{\bar{D}_v^{k}}{v_k - v^*(x_k)} &\leq - 2\rho_k\mu_G \| v_k - v^*(x_k) \|^2 + 2\rho_k \omega_{v,1} \| z_k - z^*(x_k) \| \|v_k - v^*(x_k)\|.
    \end{aligned}
\end{equation*}
By Young's inequality with parameter $\nu_1 > 0$, we have
\begin{equation*}
    \begin{aligned}
        - 2\rho_k\scalT{\bar{D}_v^{k}}{v_k - v^*(x_k)} &\leq - 2\rho_k\mu_G \| v_k - v^*(x_k) \|^2+ 2\rho_k \left(\frac{\omega_{v,1}^2}{2\nu_1} \| z_k - z^*(x_k) \|^2 + \frac{\nu_1}{2} \|v_k - v^*(x_k)\|^2 \right)\\
        &=-(2 \mu_G - \nu_1) \rho_k \|v_k - v^*(x_k)\|^2 + \frac{\omega_{v,1}^2}{\nu_1}\rho_k \|z_k - z^*(x_k)\|^2.
    \end{aligned}
\end{equation*}
Using this inequality in eq. \eqref{eqn:bound_v_A_1}, we get
\begin{equation*}
    \begin{aligned}
        \mathbb{E}_k \left[ \| v_{k + 1} - v^*(x_k) \|^2 \right] &=\left(1 - 2\mu_G\rho_k + \nu_1 \rho_k \right)\|v_k - v^*(x_k) \|^2 + \rho_k^2 \mathbb{E}_k\left[\|D_v^k\|^2 \right]\\
        &- 2 \rho_k \scalT{\hat{D}_v^k -\bar{D}_v^k}{v_k - v^*(x_k)}+ \frac{\omega_{v,1}^2}{\nu_1}\rho_k \|z_k - z^*(x_k)\|^2.
    \end{aligned}
\end{equation*}
By Cauchy-Schwartz and Young's inequality with parameter $\nu_2 > 0$, we have 
\begin{equation}\label{eqn:bound_v_A_2}
    \begin{aligned}
        \mathbb{E}_k \left[ \| v_{k + 1} - v^*(x_k) \|^2 \right] &=\left(1 - 2\mu_G\rho_k + \nu_1 \rho_k + \nu_2 \rho_k \right)\|v_k - v^*(x_k) \|^2 + \rho_k^2 \mathbb{E}_k\left[\|D_v^k\|^2 \right]\\
        &+\frac{\rho_k}{\nu_2} \underbrace{\| \hat{D}_v^k -\bar{D}_v^k \|^2}_{A_2} +\frac{\omega_{v,1}^2}{\nu_1}\rho_k \|z_k - z^*(x_k)\|^2.
    \end{aligned}
\end{equation}
Now we bound the term $A_2$.
\begin{equation*}
    \begin{aligned}
        \| \hat{D}_v^k -\bar{D}_v^k \|^2 &= \| \nabla_{zz}^2G_{h_{k}}(z_k, x_k)v_k + \nabla_z F_{h_{k}}(z_k, x_k) - \nabla_{zz}^2G(z_k, x_k)v_k - \nabla_z F(z_k, x_k) \|^2\\
        &\leq 2\| \left(\nabla_{zz}^2G_{h_{k}}(z_k, x_k) - \nabla_{zz}^2G(z_k, x_k)\right)v_k \|^2 + 2 \| \nabla_z F_{h_{k}}(z_k, x_k) - \nabla_z F(z_k, x_k) \|^2\\
        &\leq 2 \| \nabla_{zz}^2G_{h_{k}}(z_k, x_k) - \nabla_{zz}^2G(z_k, x_k) \|^2 \|v_k\|^2 + 2 \| \nabla_z F_{h_{k}}(z_k, x_k) - \nabla_z F(z_k, x_k) \|^2. 
    \end{aligned}
\end{equation*}
By Lemma \ref{lem:smoothing_error}, we have
\begin{equation*}
    \begin{aligned}
        \| \hat{D}_v^k -\bar{D}_v^k \|^2 &\leq  \frac{8L_{2,G}^2}{9}\left( (d + 5)^{5/2} + (d +3)^{3/2} \right)^2h_{k}^2 \|v_k\|^2\\
        &+ \frac{L_{1,G}^2}{2}(d + 3)^{3}h_{k}^2. 
    \end{aligned}
\end{equation*}
Adding and subtracting $v^*(x_k)$, we get
\begin{equation*}
    \begin{aligned}
        \| \hat{D}_v^k -\bar{D}_v^k \|^2 &\leq  \frac{16 L_{2,G}^2}{9}\left( (d + 5)^{5/2} + (d +3)^{3/2} \right)^2h_{k}^2 \|v_k - v^*(x_k)\|^2\\
        &+ \frac{16 L_{2,G}^2}{9}\left( (d + 5)^{5/2} + (d +3)^{3/2} \right)^2h_{k}^2 \|v^*(x_k)\|^2\\
        &+ \frac{L_{1,G}^2}{2}(d + 3)^{3}h_{k}^2. 
    \end{aligned}
\end{equation*}
Again by Lemma \ref{lem:bound_norm_v}, we have %
\begin{equation*}
    \begin{aligned}
        \| \hat{D}_v^k -\bar{D}_v^k \|^2 &\leq  \frac{16 L_{2,G}^2}{9}\left( (d + 5)^{5/2} + (d +3)^{3/2} \right)^2h_{k}^2 \|v_k - v^*(x_k)\|^2\\
        &+ \frac{16 L_{2,G}^2}{9}\left( (d + 5)^{5/2} + (d +3)^{3/2} \right)^2 \frac{L_{0,F}^2}{\mu_G^2}h_{k}^2\\
        &+ \frac{L_{1,G}^2}{2}(d + 3)^{3}h_{k}^2. 
    \end{aligned}
\end{equation*}
Let $\omega_{v,2} = \frac{16 L_{2,G}^2}{9}\left( (d + 5)^{5/2} + (d +3)^{3/2} \right)^2$.
Using this inequality in eq. \eqref{eqn:bound_v_A_2}, we get
\begin{equation*}
    \begin{aligned}
        \mathbb{E}_k \left[ \| v_{k + 1} - v^*(x_k) \|^2 \right] &\leq\left(1 - 2\mu_G\rho_k + \nu_1 \rho_k + \nu_2 \rho_k +\omega_{v,2} h_{k}^2 \frac{\rho_k}{\nu_2} \right)\|v_k - v^*(x_k) \|^2\\
        &+ \rho_k^2 \mathbb{E}_k\left[\|D_v^k\|^2 \right] +\frac{\omega_{v,1}^2}{\nu_1}\rho_k \|z_k - z^*(x_k)\|^2\\
        &+ \frac{L_{1,G}^2}{2}(d + 3)^{3} \frac{\rho_k}{\nu_2}h_{k}^2 + \omega_{v,2} \frac{L_{0,F}^2}{\mu_G^2} \frac{\rho_k}{\nu_2} h_{k}^2.
    \end{aligned}
\end{equation*}
Setting $\nu_1 =\nu_2 = \frac{\mu_G}{4}$ and since $h_{k}^2 \leq \frac{\mu_G^2}{16 \omega_{v,2}}$, we have
\begin{equation}\label{eqn:bound_v_A_3}
    \begin{aligned}
        \mathbb{E}_k \left[ \| v_{k + 1} - v^*(x_k) \|^2 \right] &\leq\left(1 - \frac{5 \mu_G}{4}\rho_k \right)\|v_k - v^*(x_k) \|^2\\
        &+ \rho_k^2 \mathbb{E}_k\left[\|D_v^k\|^2 \right] +4\frac{\omega_{v,1}^2}{\mu_G}\rho_k \|z_k - z^*(x_k)\|^2\\
        &+ 2  \frac{L_{1,G}^2(d + 3)^{3}}{\mu_G} \rho_kh_{k}^2 + 4 \frac{\omega_{v,2}L_{0,F}^2}{\mu_G^3} \rho_k h_{k}^2.
    \end{aligned}
\end{equation}
Now, we bound the term $B$ of eq. \eqref{eqn:bound_v_main}. By Lemma \ref{lem:smt_setting_smooth_value}, $v^*$ is $L_*$-Lipschitz continuous. Thus, by eq. \eqref{eqn:x_update}, we have
\begin{equation}\label{eqn:bound_v_B}
    \| v^*(x_{k + 1}) - v^*(x_k) \|^2 \leq L_*^2 \|x_{k + 1} - x_k \|^2 = L_*^2 \gamma_k^2 \| D_x^k\|^2.
\end{equation}
Finally, we bound the term $C$ of eq. \eqref{eqn:bound_v_main}. By eq. \eqref{eqn:v_update}, we have
\begin{equation*}
    \begin{aligned}
        -2\scalT{v_{k + 1} - v^*(x_k)}{v^*(x_{k + 1}) - v^*(x_k)} &= -2\scalT{v_{k} - v^*(x_k)}{v^*(x_{k + 1}) - v^*(x_k)}\\
        &+ 2 \rho_k \scalT{D_v^k}{v^*(x_{k + 1}) - v^*(x_k)}.
    \end{aligned}
\end{equation*}
By Cauchy-Schwartz and Lemma \ref{lem:smt_setting_smooth_value}, we have
\begin{equation*}
    \begin{aligned}
        -2\scalT{v_{k + 1} - v^*(x_k)}{v^*(x_{k + 1}) - v^*(x_k)} &\leq -2\scalT{v_{k} - v^*(x_k)}{v^*(x_{k + 1}) - v^*(x_k)}\\
        &+ 2 L_* \rho_k \| D_v^k \| \| x_{k + 1} - x_k \|.
    \end{aligned}
\end{equation*}
By Young's inequality with $\nu_3 > 0$ and eq. \eqref{eqn:x_update}, we get
\begin{equation*}
    \begin{aligned}
        -2\scalT{v_{k + 1} - v^*(x_k)}{v^*(x_{k + 1}) - v^*(x_k)} &\leq -2\scalT{v_{k} - v^*(x_k)}{v^*(x_{k + 1}) - v^*(x_k)}\\
        &+ \rho_k^2 \nu_3 \| D_v^k \|^2 + \frac{L_*^2}{\nu_3} \gamma_k^2 \| D_x^k \|^2.
    \end{aligned}
\end{equation*}
Similarly, by Cauchy-Schwartz, Lemma \ref{lem:smt_setting_smooth_value} (i.e. $L_*$-Lipschitz of $v^*$) and by Young's inequality with parameter $\nu_4 > 0$,
\begin{equation*}
    \begin{aligned}
        -2\scalT{v_{k + 1} - v^*(x_k)}{v^*(x_{k + 1}) - v^*(x_k)} &\leq 2 L_* \| v_{k} - v^*(x_k) \| \| x_{k + 1} - x_k \|\\
        &+ \rho_k^2 \nu_3 \| D_v^k \|^2 + \frac{L_*^2}{\nu_3} \gamma_k^2 \| D_x^k \|^2\\
        &\leq  \nu_4 \| v_{k} - v^*(x_k) \|^2 +  \frac{1}{\nu_4} L^2_*\| x_{k + 1} - x_k \|^2\\
        &+ \rho_k^2 \nu_3 \| D_v^k \|^2 + \frac{L_*^2}{\nu_3} \gamma_k^2 \| D_x^k \|^2.
    \end{aligned}
\end{equation*}
Let $\nu_3 = 1$ and $\nu_4 = \frac{\mu_G \rho_k}{4}$, we have
\begin{equation}\label{eqn:bound_v_C}
    \begin{aligned}
        -2\scalT{v_{k + 1} - v^*(x_k)}{v^*(x_{k + 1}) - v^*(x_k)} &\leq  \frac{\mu_G}{4} \rho_k \| v_{k} - v^*(x_k) \|^2 +  \frac{4}{\mu_G \rho_k} L^2_* \gamma_k^2 \| D_x^k \|^2\\
        &+ \rho_k^2  \| D_v^k \|^2 + L_*^2 \gamma_k^2 \| D_x^k \|^2.
    \end{aligned}
\end{equation}
Let $\omega_k^v = \left(1 + \frac{2}{\mu_G \rho_k} \right)$. Restarting from eq. \eqref{eqn:bound_v_main}, taking the conditional expectation and by equations \eqref{eqn:bound_v_A_3}, \eqref{eqn:bound_v_B},\eqref{eqn:bound_v_C}, we get the claim
\begin{equation*}
    \begin{aligned}
        \mathbb{E}_k\left[ \| v_{k + 1} - v^*(x_{k + 1}) \|^2 \right] &\leq \left(1 - \mu_G\rho_k \right)\|v_k - v^*(x_k) \|^2 +4\frac{\omega_{v,1}^2}{\mu_G}\rho_k \|z_k - z^*(x_k)\|^2\\
        &+ 2\rho_k^2 \mathbb{E}_k\left[\|D_v^k\|^2 \right] + 2\omega_k^vL_*^2 \gamma_k^2  \mathbb{E}_k \left[\| D_x^k\|^2\right]\\
        &+ 2  \frac{L_{1,G}^2(d + 3)^{3}}{\mu_G} \rho_kh_{k}^2 + 4 \frac{\omega_{v,2}L_{0,F}^2}{\mu_G^3} \rho_k h_{k}^2.
    \end{aligned}
\end{equation*}

\end{proof}

\noindent The following lemma characterizes the descent of the value function $\Psi$ along the sequence of the iterates $x_k$ generated by Algorithm~\ref{alg:zoba}.
\begin{lemma}[Value Function sequence bound]\label{lem:smooth_val_fun_descent}
    Under Assumption \ref{asm:F_smooth}, \ref{asm:G_asm}. Let $(x_k)_{k \in \mathbb{N}}$ be the sequence generated by Algorithm \ref{alg:zoba}. Let $\omega_{v,1} =\left( L_{1,F} + \frac{L_{0,F} L_{2,G}}{\mu_G} \right)$. Then,
    \begin{equation*}
        \begin{aligned}
            \mathbb{E}_k\left[\Psi(x_{k + 1})\right] &\leq \Psi(x_k) - \frac{\gamma_k}{2}\| \nabla \Psi(x_k) \|^2 -\frac{\gamma_k}{2} \| \hat{D}_x^{k} \|^2 + \frac{L_\Psi}{2} \gamma_k^2 \mathbb{E}_k\left[\|D_x^k\|^2\right]\\
            &+ 2\gamma_k \omega_{v,1}^2 \| z_k - z^*(x_k) \|^2 + 2 \gamma_k L_{1,G}^2 \|v_k - v^*(x_k) \|^2\\ %
            &+ C_1^2 \gamma_k h_{k}^2,
        \end{aligned}
    \end{equation*}
    where $C_1= \left(2\frac{L_{0,F} L_{2,G}}{3\mu_G} \left( (d + 5)^{5/2} +  (d + 3)^{3/2} + (d + 1) (p + 3)^{3/2} \right) + \frac{L_{1,F}}{2}(d + 3)^{3/2} \right)$.
\end{lemma}
\begin{proof}
    By Lemma \ref{lem:smt_setting_smooth_value}, $\Psi$ is $L_\Psi$-smooth. Thus, by the Descent Lemma \cite{polyak1987introduction} we have
    \begin{equation*}
        \Psi(x_{k + 1}) \leq \Psi(x_k) + \scalT{\nabla \Psi(x_k)}{x_{k + 1} - x_k} + \frac{L_\Psi}{2} \| x_{k + 1} - x_k \|^2.
    \end{equation*}
    By eq. \eqref{eqn:x_update}, we have
    \begin{equation*}
        \begin{aligned}
            \Psi(x_{k + 1}) \leq \Psi(x_k) - \gamma_k \scalT{\nabla \Psi(x_k)}{D_x^k} + \frac{L_\Psi}{2} \gamma_k^2\|D_x^k\|^2.
        \end{aligned}
    \end{equation*}
    Let $\hat{D}_x^{k} = \nabla_{xz}^2 G_{h_{k}}(z_k,x_k)v_k +\nabla_x F_{h_{k}}(z_k, x_k)$. Taking the conditional expectation, we get
    \begin{equation*}
        \mathbb{E}_k[\Psi(x_{k + 1})] \leq \Psi(x_k) - \gamma_k \scalT{\nabla \Psi(x_k)}{\hat{D}_x^{k}} + \frac{L_\Psi}{2} \gamma_k^2 \mathbb{E}_k\left[\|D_x^k\|^2\right].
    \end{equation*}
    By $\scalT{a}{b} = \frac{1}{2}\left( \|a\|^2 + \|b\|^2 - \|a - b\|^2  \right)$, we have
    \begin{equation*}
        \begin{aligned}
            \mathbb{E}_k\left[\Psi(x_{k + 1})\right] &\leq \Psi(x_k) - \frac{\gamma_k}{2} \left( \| \nabla \Psi(x_k) \|^2 + \| \hat{D}_x^{k} \|^2 - \| \hat{D}_x^{k} - \nabla \Psi(x_k) \|^2 \right)\\
            &+ \frac{L_\Psi}{2} \gamma_k^2 \mathbb{E}_k\left[\|D_x^k\|^2\right]\\
            &= \Psi(x_k) - \frac{\gamma_k}{2}\| \nabla \Psi(x_k) \|^2 -\frac{\gamma_k}{2} \| \hat{D}_x^{k} \|^2\\
            &+ \frac{\gamma_k}{2}\| \hat{D}_x^{k} - \nabla \Psi(x_k) \|^2+ \frac{L_\Psi}{2} \gamma_k^2 \mathbb{E}_k\left[\|D_x^k\|^2\right].
        \end{aligned}
    \end{equation*}
    Let $\omega_{v,1} = \left( L_{1,F} + \frac{L_{0,F} L_{2,G}}{\mu_G} \right)$. By Lemma \ref{lem:bilevel_smth_error} with $h_1 = h_{k}$, $\eta_1 = h_{k}$, $h_2 = 0$ and $\eta_2 = h_{k}$, we get
    \begin{equation*}
        \begin{aligned}
            \mathbb{E}_k\left[\Psi(x_{k + 1})\right] &\leq  \Psi(x_k) - \frac{\gamma_k}{2}\| \nabla \Psi(x_k) \|^2 -\frac{\gamma_k}{2} \| \hat{D}_x^{k} \|^2\\
            &+ \frac{L_\Psi}{2} \gamma_k^2 \mathbb{E}_k\left[\|D_x^k\|^2\right]+ \frac{\gamma_k}{2} \bigg(\omega_{v,1} \| z_k - z^*(x_k) \| + L_{1,G} \|v_k - v^*(x_k) \|\\
            &+ 2\frac{L_{0,F} L_{2,G}}{3\mu_G} \bigg( h_{k} (d + 5)^{5/2} + h_{k} (d + 3)^{3/2} + h_k(d + 1) (p + 3)^{3/2} \bigg)\\
            &+ \frac{L_{1,F} }{2} (d + 3)^{3/2} h_k^2 \bigg)^2\\
        \end{aligned}
    \end{equation*}
    Let $C_1:=\left(2\frac{L_{0,F} L_{2,G}}{3\mu_G} \bigg(  (d + 5)^{5/2} + (d + 3)^{3/2} + (d + 1) (p + 3)^{3/2} \bigg)+ \frac{L_{1,F}}{2}(d + 3)^{3/2}\right) h_{k}$.
    \begin{equation*}
        \begin{aligned}
            \mathbb{E}_k\left[\Psi(x_{k + 1})\right] &\leq  \Psi(x_k) - \frac{\gamma_k}{2}\| \nabla \Psi(x_k) \|^2 -\frac{\gamma_k}{2} \| \hat{D}_x^{k} \|^2\\
            &+ \frac{L_\Psi}{2} \gamma_k^2 \mathbb{E}_k\left[\|D_x^k\|^2\right]+ \frac{\gamma_k}{2} \bigg(\omega_{v,1} \| z_k - z^*(x_k) \| + L_{1,G} \|v_k - v^*(x_k) \|\\
            &+ C_1 h_{k}\bigg)^2.
        \end{aligned}
    \end{equation*}
    By $(a + b)^2 \leq 2a^2 + 2b^2$, we get
    \begin{equation*}
        \begin{aligned}
            \mathbb{E}_k\left[\Psi(x_{k + 1})\right] &\leq\Psi(x_k) - \frac{\gamma_k}{2}\| \nabla \Psi(x_k) \|^2 -\frac{\gamma_k}{2} \| \hat{D}_x^{k} \|^2\\
            &+ \frac{L_\Psi}{2} \gamma_k^2 \mathbb{E}_k\left[\|D_x^k\|^2\right]\\
            &+ \gamma_k \bigg(\omega_{v,1} \| z_k - z^*(x_k) \| + L_{1,G} \|v_k - v^*(x_k) \| \bigg)^2\\
            &+ C_1^2 \gamma_k h_{k}^2\\
            &\leq\Psi(x_k) - \frac{\gamma_k}{2}\| \nabla \Psi(x_k) \|^2 -\frac{\gamma_k}{2} \| \hat{D}_x^{k} \|^2 + \frac{L_\Psi}{2} \gamma_k^2 \mathbb{E}_k\left[\|D_x^k\|^2\right]\\
            &+ 2\gamma_k \omega_{v,1}^2 \| z_k - z^*(x_k) \|^2 + 2 \gamma_k L_{1,G}^2 \|v_k - v^*(x_k) \|^2\\ %
            &+ C_1^2 \gamma_k h_{k}^2.
        \end{aligned}
    \end{equation*}
    
\end{proof}

\subsection{Auxiliary results for Algorithm \ref{alg:hfzoba}}\label{app:aux_res_hfzoba}

In this appendix, we collect all lemmas required for the analysis of Algorithm \ref{alg:hfzoba}. Unlike Algorithm \ref{alg:zoba}, which relies on unbiased estimators of the Hessian of the smoothing, Algorithm \ref{alg:hfzoba} directly approximates Hessian–vector products using first-order finite differences, and then replaces the resulting gradients with zeroth-order surrogates. We begin by introducing a lemma that bounds the approximation error of the first-order finite-difference for a twice-differentiable function with Lipschitz-continuous Hessians.

\begin{lemma}[First-order approximation of Hessian-vector products] \label{lem:biased_hess_approx}
Let $f: \mathbb{R}^p \times\mathbb{R}^d \to \mathbb{R}$ be a twice-differentiable function with $L_{2,G}$-Lipschitz continuous hessian. Then,  for every $z,v \in \mathbb{R}^p$, $x \in \mathbb{R}^d$ and $\bar{h} > 0$, 
\begin{equation*}
    \begin{aligned}
    \left\| \frac{1}{\bar{h}} \left( \nabla_z f(z + \bar{h} v, x) -\nabla_z f(z,x) \right) -\nabla_{zz}^2 f(z,x)v \right\|&\leq \frac{L_{2,G}}{2} \bar{h} \|v\|^2,\\
    \left\| \frac{1}{\bar{h}} \left( \nabla_x f(z + \bar{h} v, x) -\nabla_x f(z,x) \right) -\nabla_{xz}^2 f(z,x)v \right\|&\leq \frac{L_{2,G}}{2} \bar{h} \|v\|^2.
    \end{aligned}
\end{equation*}
\end{lemma}
\begin{proof}
The proof of both inequalities follows the same line of Lemma 1.2.4 of \cite{nesterov2018lectures}. For completeness, we include the proof of the first inequality (the second follows analogously). Since $f$ is twice differentiable, we have
\begin{equation*}
\nabla_z f(z + \bar h v, x) = \nabla_z f(z, x) + \int_0^{1} \nabla^2_{zz} f(z + t \bar{h}v, x) \, \bar{h}v \, dt.    
\end{equation*}
Adding and subtracting $\nabla^2_{zz} f(z, x) \, \bar{h}v$
\begin{equation*}
\nabla_z f(z + \bar h v, x) = \nabla_z f(z, x) + \nabla^2_{zz} f(z, x) \, \bar{h}v + \int_0^{1} (\nabla^2_{zz} f(z + t \bar{h}v, x) - \nabla^2_{zz} f(z, x))  \, \bar{h}v \, dt.    
\end{equation*}
Therefore, we have
\begin{equation*}
    \begin{aligned}
        \left\| \frac{1}{\bar{h}} \left( \nabla_z f(z + \bar{h} v, x) -\nabla_z f(z,x) \right) -\nabla_{zz}^2 f(z,x)v \right\| %
        &= \frac{1}{\bar{h}}\left\|  \int_0^{1} (\nabla^2_{zz} f(z + t \bar{h}v, x) - \nabla^2_{zz} f(z, x))  \, \bar{h}v \, dt \right\|\\
        &\leq \frac{1}{\bar{h}}\int_0^{1}\left\|   (\nabla^2_{zz} f(z + t \bar{h}v, x) - \nabla^2_{zz} f(z, x))  \, \bar{h}v  \right\| \, dt\\
        &\leq \frac{1}{\bar{h}}\int_0^{1}\left\|\nabla^2_{zz} f(z + t \bar{h}v, x) - \nabla^2_{zz} f(z, x)\right\| \| \bar{h}v\| \, dt
    \end{aligned}
\end{equation*}
By $L_{2,G}$-Lipschitz continuity of the Hessian, we get the claim
\begin{equation*}
    \begin{aligned}
        \left\| \frac{1}{\bar{h}} \left( \nabla_z f(z + \bar{h} v, x) -\nabla_z f(z,x) \right) -\nabla_{zz}^2 f(z,x)v \right\| &\leq L_{2,G} \bar{h} \|v\|^2\int_0^{1}t \, dt= \frac{L_{2,G}}{2}\bar{h}\|v\|^2.
    \end{aligned}
\end{equation*}
\end{proof}
\noindent In the following lemma, we derive bounds on the norm of the finite-difference approximations of the Hessian-block–vector products used in Algorithm~\ref{alg:hfzoba}.
\begin{lemma}[Bound on surrogate difference]\label{lem:hfzoba_sur_diff}
    Let Assumptions \ref{asm:F_smooth} and \ref{asm:G_asm} hold. Let $z_k,v_k$ and $x_k$ be the sequences generated by Algorithm \ref{alg:hfzoba}. Let $h_k, \bar{h}_k > 0$. Then, for every $k \in \mathbb{N}$
    \begin{equation*}
        \begin{aligned}
            \mathbb{E}_k\left[\left\| \hat{\nabla}_zg_k(z_k + \bar{h}_kv_k,x_k) - \hat{\nabla}_z g_k(z_k,x_k) \right\|^2\right] &\leq C_1L_{1,G}^2 \bar{h}_k^2 \|v_k - v^*(x_k)\|^2 +C_1\frac{L_{1,G}^2 L_{0,F}^2}{\mu_G^2}\bar{h}_k^2 +\frac{3L_{1,G}^2 (p + 6)^3}{b_1 \ell_1}h_k^2,\\
            \mathbb{E}_k\left[\left\| \hat{\nabla}_xg_k(z_k + \bar{h}_kv_k,x_k) - \hat{\nabla}_x g_k(z_k,x_k) \right\|^2\right] &\leq C_2L_{1,G}^2 \bar{h}_k^2 \|v_k - v^*(x_k)\|^2 +C_2\frac{L_{1,G}^2 L_{0,F}^2}{\mu_G^2}\bar{h}_k^2 +\frac{3L_{1,G}^2 (d + 6)^3}{b_1 \ell_1}h_k^2.
        \end{aligned}
    \end{equation*}    
    where $C_1 = 4\left(1 + \frac{3(p + 2)}{b_1 \ell_1} \right)$ and $C_2 = 4\left(1 + \frac{3(d + 2)}{b_1 \ell_1} \right)$.
\end{lemma}
\begin{proof}
    We start by proving the first inequality. To simplify the reading we will use the following notation.
    \begin{equation*}
        \begin{aligned}
            g_{v,k}^{(i,j)} &= \frac{g(z_k +\bar{h}_k v_k + h_kw_k^{(i,j)},x_k,\xi_{i,k}) - g(z_k +\bar{h}_k v_k,x_k,\xi_{i,k})}{h_k}w_k^{(i,j)},\\
            g_k^{(i,j)} &= \frac{g(z_k + h_kw_k^{(i,j)},x_k,\xi_{i,k}) - g(z_k,x_k,\xi_{i,k})}{h_k}w_k^{(i,j)},\\
            \delta_k^{(i,j)} &= g_{v,k}^{(i,j)} - g_k^{(i,j)} - (\nabla_z g_{h_k}(z_k +\bar{h}_k v_k, x_k, \xi_{i,k}) - \nabla_z g_{h_k}(z_k, x_k, \xi_{i,k})),\\
            \delta_k^{(i)} &= \sum\limits_{j=1}^{\ell_1} \delta_k^{(i,j)}.
        \end{aligned}
    \end{equation*}
    Adding and subtracting $\frac{1}{b_1\ell_1} \sum\limits_{i=1}^{b_1} \sum\limits_{j=1}^{\ell_1} \nabla_z g_{h_k}(z_k + \bar{h}_k v_k, x_k, \xi_{i,k} ) - \nabla_z g_{h_k}(z_k, x_k, \xi_{i,k} )$, we get
    \begin{equation*}
        \begin{aligned}
            \mathbb{E}_k \left[\left\| \hat{\nabla}_zg_k(z_k + \bar{h}_kv_k,x_k) - \hat{\nabla}_z g_k(z_k,x_k) \right\|^2 \right] &\leq 2 \mathbb{E}_k \left[\left\| \frac{1}{b_1 \ell_1} \sum\limits_{i=1}^{b_1} \delta_k^{(i)} \right\|^2 \right]\\
            &+ 2\mathbb{E}_k\left[ \left\| \nabla_z g_{h_k}(z_k +\bar{h}_k v_k, x_k, \xi_{i,k}) - \nabla_z g_{h_k}(z_k, x_k, \xi_{i,k}) \right\|^2 \right].
        \end{aligned}
    \end{equation*}
    By $L_{1,G}$-smoothness of $g$, we get
    \begin{equation*}
        \begin{aligned}
            \mathbb{E}_k \left[\left\| \hat{\nabla}_zg_k(z_k + \bar{h}_kv_k,x_k) - \hat{\nabla}_z g_k(z_k,x_k) \right\|^2 \right] &\leq 2 \mathbb{E}_k \left[\left\| \frac{1}{b_1 \ell_1} \sum\limits_{i=1}^{b_1} \delta_k^{(i)} \right\|^2 \right]\\
            &+ 2 L_{1,G}^2 \bar{h}_k^2 \left\|  v_k \right\|^2.
        \end{aligned}
    \end{equation*}
    Adding and subtracting $v^*(x_k)$ and by Lemma \ref{lem:bound_norm_v}, we get
    \begin{equation*}
        \begin{aligned}
            \mathbb{E}_k \left[\left\| \hat{\nabla}_zg_k(z_k + \bar{h}_kv_k,x_k) - \hat{\nabla}_z g_k(z_k,x_k) \right\|^2 \right] &\leq 2 \mathbb{E}_k \left[\left\| \frac{1}{b_1 \ell_1} \sum\limits_{i=1}^{b_1} \delta_k^{(i)} \right\|^2 \right]\\
            &+ 4 L_{1,G}^2 \bar{h}_k^2 \left\|  v_k - v^*(x_k) \right\|^2 + 4 \frac{L_{1,G}^2sL_{0,F}^2}{\mu_G^2} \bar{h}_k^2\\
        \end{aligned}
    \end{equation*}
    Since for every $i =1,\cdots,b_1$ and $t =1,\cdots,b_1$ with $i \neq t$, $w_k^{(i,j)}$ and $w_{k}^{(t,j)}$ are independent for every $j = 1,\cdots,\ell_1$, we get
    \begin{equation*}
        \begin{aligned}
            \mathbb{E}_k \left[\left\| \hat{\nabla}_zg_k(z_k + \bar{h}_kv_k,x_k) - \hat{\nabla}_z g_k(z_k,x_k) \right\|^2 \right] &\leq  \frac{2}{b_1^2 \ell_1^2} \left(\sum\limits_{i=1}^{b_1}\mathbb{E}_k \left[\left\| \delta_k^{(i)} \right\|^2 \right] + \sum\limits_{i=1}^{b_1} \sum\limits_{t \neq i} \scalT{\mathbb{E}_k[\delta_k^{(i)}]}{\mathbb{E}_k[\delta_k^{(t)}]} \right)\\
            &+ 4 L_{1,G}^2 \bar{h}_k^2 \left\|  v_k - v^*(x_k) \right\|^2 + 4 \frac{L_{1,G}^2sL_{0,F}^2}{\mu_G^2} \bar{h}_k^2.
        \end{aligned}
    \end{equation*}
    Observing that, by Lemma \ref{lem:grad_hess_smoothing}, we have
    \begin{equation*}
        \mathbb{E}_k[\delta_k^{(i)}] =  \sum\limits_{j=1}^{\ell_1} \mathbb{E}_{w_k^{(i,j)}}[g_{v,k}^{(i,j)} - g_k^{(i,j)}] - (\nabla_z g_{h_k}(z_k +\bar{h}_k v_k, x_k, \xi_{i,k}) - \nabla_z g_{h_k}(z_k, x_k, \xi_{i,k}))=0,
    \end{equation*}
    we get
    \begin{equation*}
        \begin{aligned}
            \mathbb{E}_k \left[\left\| \hat{\nabla}_zg_k(z_k + \bar{h}_kv_k,x_k) - \hat{\nabla}_z g_k(z_k,x_k) \right\|^2 \right] &\leq  \frac{2}{b_1^2 \ell_1^2} \sum\limits_{i=1}^{b_1}\mathbb{E}_k \left[\left\| \delta_k^{(i)} \right\|^2 \right] \\
            &+ 4 L_{1,G}^2 \bar{h}_k^2 \left\|  v_k - v^*(x_k) \right\|^2 + 4 \frac{L_{1,G}^2sL_{0,F}^2}{\mu_G^2} \bar{h}_k^2.
        \end{aligned}
    \end{equation*}
    Similarly, recalling that $\delta_{k}^{(i)} = \sum\limits_{j=1}^{\ell_1} \delta_{k}^{(i,j)}$, we have
    \begin{equation*}
        \begin{aligned}
            \mathbb{E}_k \left[\left\| \hat{\nabla}_zg_k(z_k + \bar{h}_kv_k,x_k) - \hat{\nabla}_z g_k(z_k,x_k) \right\|^2 \right] &\leq  \frac{2}{b_1^2 \ell_1^2} \sum\limits_{i=1}^{b_1} \left(\sum\limits_{j=1}^{\ell_1} \mathbb{E}_k \left[\left\| \delta_k^{(i,j)} \right\|^2 \right] + \sum\limits_{j=1}^{\ell_1} \sum\limits_{t \neq j} \mathbb{E}_k\left[ \scalT{\delta_k^{(i,j)}}{\delta_k^{(i,t)}} \right] \right) \\
            &+ 4 L_{1,G}^2 \bar{h}_k^2 \left\|  v_k - v^*(x_k) \right\|^2 + 4 \frac{L_{1,G}^2sL_{0,F}^2}{\mu_G^2} \bar{h}_k^2.
        \end{aligned}
    \end{equation*}
    For $j = 1,\cdots,\ell_1$ and $t = 1,\cdots,\ell_1$ with $j\neq t$, $w_k^{(i,j)}$ and $w_k^{(i,t)}$ are independent for every $i = 1,\cdots,b_1$. Thus, we get
    \begin{equation} \label{eqn:bound_grad_diff_1}
        \begin{aligned}
            \mathbb{E}_k \left[\left\| \hat{\nabla}_zg_k(z_k + \bar{h}_kv_k,x_k) - \hat{\nabla}_z g_k(z_k,x_k) \right\|^2 \right] &\leq \frac{2}{b_1^2 \ell_1^2} \sum\limits_{i=1}^{b_1} \sum\limits_{j=1}^{\ell_1} \mathbb{E}_k \left[\left\| \delta_k^{(i,j)} \right\|^2 \right] \\
            &+ 4 L_{1,G}^2 \bar{h}_k^2 \left\|  v_k - v^*(x_k) \right\|^2 + 4 \frac{L_{1,G}^2sL_{0,F}^2}{\mu_G^2} \bar{h}_k^2\\
            &\leq\frac{2}{b_1^2 \ell_1^2} \sum\limits_{i=1}^{b_1} \sum\limits_{j=1}^{\ell_1} \underbrace{\mathbb{E}_k \left[\left\| g_{v,k}^{(i,j)} - g_k^{(i,j)}\right\|^2 \right]}_{A} \\
            &+ 4 L_{1,G}^2 \bar{h}_k^2 \left\|  v_k - v^*(x_k) \right\|^2 + 4 \frac{L_{1,G}^2sL_{0,F}^2}{\mu_G^2} \bar{h}_k^2.
        \end{aligned}
    \end{equation} %
    Now, we focus on bounding $A$. Adding and subtracting $\scalT{\nabla g(z_k + \bar{h}_kv_k,x_k,\xi_{i,k})}{ [w_k^{(i,j)},0]}$ and $\scalT{\nabla g(z_k,x_k,\xi_{i,k})}{ [w_k^{(i,j)},0]}$, we have
    \begin{equation*}
        \begin{aligned}
        \mathbb{E}_k \left[\left\| g_{v,k}^{(i,j)} - g_k^{(i,j)}\right\|^2\right]&= \frac{1}{h_k^2} \mathbb{E}_k \bigg[ (g(z_k + \bar{h}_k v_k + h_kw_k^{(i,j)},x_k,\xi_{i,k}) - g(z_k + \bar{h}_kv_k,x_k,\xi_{i,k})\\
        &- (g(z_k + h_kw_k^{(i,j)},x_k,\xi_{i,k}) - g(z_k,x_k,\xi_{i,k})))^2 \bigg\|w_k^{(i,j)} \bigg\|^2\bigg]\\
        &\leq \frac{3}{h_k^2} \mathbb{E}_k \bigg[ (g(z_k + \bar{h}_k v_k + h_kw_k^{(i,j)},x_k,\xi_{i,k}) - g(z_k + \bar{h}_kv_k,x_k,\xi_{i,k})\\&- \scalT{\nabla g(z_k + \bar{h}_kv_k, x_k, \xi_{i,k})}{[h_k w_k^{(i,j)},0]})^2\|w_k^{(i,j)}\|^2\bigg]\\
        &+\frac{3}{h_k^2} \mathbb{E}_k \bigg[ (g(z_k + h_kw_k^{(i,j)},x_k,\xi_{i,k}) - g(z_k,x_k,\xi_{i,k})\\&- \scalT{\nabla g(z_k, x_k, \xi_{i,k})}{[h_k w_k^{(i,j)},0]})^2\|w_k^{(i,j)}\|^2\bigg]\\
        &+\frac{3}{h_k^2} \mathbb{E}_k \bigg[ (\scalT{\nabla g(z_k + \bar{h}_kv_k, x_k, \xi_{i,k}) - \nabla g(z_k,x_k,\xi_{i,k})}{[h_k w_k^{(i,j)},0]})^2\|w_k^{(i,j)}\|^2\bigg].
        \end{aligned}
    \end{equation*}
    Since $g$ is $L_{1,G}$-smooth, by the Descent Lemma \cite{polyak1987introduction},
    \begin{equation*}
        \begin{aligned}
        \mathbb{E}_k \left[\left\| g_{v,k}^{(i,j)} - g_k^{(i,j)}\right\|^2\right]&\leq  \frac{3L_{1,G}^2}{2}h_k^2 \mathbb{E}_k[\| w_k^{(i,j)} \|^6] \\
        &+\frac{3}{h_k^2} \mathbb{E}_k \bigg[ (\scalT{\nabla g(z_k + \bar{h}_kv_k, x_k, \xi_{i,k}) - \nabla g(z_k,x_k,\xi_{i,k})}{[h_k w_k^{(i,j)},0]})^2\|w_k^{(i,j)}\|^2\bigg]\\
        &=\frac{3L_{1,G}^2}{2}h_k^2 \mathbb{E}_k[\| w_k^{(i,j)} \|^6] \\
        &+3 \mathbb{E}_{\xi_{i,k}} \bigg[ (\nabla_z g(z_k + \bar{h}_kv_k, x_k, \xi_{i,k}) - \nabla_z g(z_k,x_k,\xi_{i,k}))^\top\\
        &\mathbb{E}_{w_k^{(i,k)}}[w_k^{(i,j)}w_k^{(i,j)\top}w_k^{(i,j)}w_k^{(i,j)\top}] (\nabla_z g(z_k + \bar{h_k}v_k , x_k, \xi_{i,k}) - \nabla_z g(z_k,x_k,\xi_{i,k})) \bigg].
        \end{aligned}
    \end{equation*}
    By Proposition \ref{lem:norm_gaus_vector}, we have
    \begin{equation*}
        \begin{aligned}
        \mathbb{E}_k \left[\left\| g_{v,k}^{(i,j)} - g_k^{(i,j)}\right\|^2\right]&\leq  \frac{3L_{1,G}^2}{2}(p + 6)^3h_k^2 +3 (p + 2) \mathbb{E}_{\xi_{i,k}} \bigg[\|\nabla_z g(z_k + \bar{h}_kv_k, x_k, \xi_{i,k}) - \nabla_z g(z_k,x_k,\xi_{i,k}) \|^2 \bigg].
        \end{aligned}
    \end{equation*}
    Again by $L_{1,G}$-smoothness of $g$, we have
    \begin{equation*}
        \begin{aligned}
        \mathbb{E}_k \left[\left\| g_{v,k}^{(i,j)} - g_k^{(i,j)}\right\|^2\right]&\leq  \frac{3L_{1,G}^2}{2}(p + 6)^3h_k^2 +3 (p + 2) L_{1,G}^2\bar{h}_k^2 \|v_k\|^2.
        \end{aligned}
    \end{equation*}
    Adding and subtracting $v^*(x_k)$, by Lemma \ref{lem:bound_norm_v}, we get
    \begin{equation*}
        \begin{aligned}
        \mathbb{E}_k \left[\left\| g_{v,k}^{(i,j)} - g_k^{(i,j)}\right\|^2\right]&\leq  \frac{3L_{1,G}^2}{2}(p + 6)^3h_k^2 +6 (p + 2) L_{1,G}^2\bar{h}_k^2 \|v_k - v^*(x_k)\|^2\\
        &+ \frac{6 (p + 2) L_{1,G}^2 L_{0,F}^2}{\mu_G^2}\bar{h}_k^2.
        \end{aligned}
    \end{equation*}
    Using this bound in eq. \eqref{eqn:bound_grad_diff_1}, we get the first claim
    \begin{equation*}
        \begin{aligned}
            \mathbb{E}_k \left[\left\| \hat{\nabla}_zg_k(z_k + \bar{h}_kv_k,x_k) - \hat{\nabla}_z g_k(z_k,x_k) \right\|^2 \right] &\leq%
            4\left(1 + \frac{3(p + 2)}{b_1 \ell_1} \right)L_{1,G}^2 \bar{h}_k^2 \|v_k - v^*(x_k)\|^2\\
            &+4\left(1 + \frac{3(p + 2)}{b_1 \ell_1} \right)\frac{L_{1,G}^2 L_{0,F}^2}{\mu_G^2}\bar{h}_k^2 +\frac{3L_{1,G}^2 (p + 6)^3}{b_1 \ell_1}h_k^2.
        \end{aligned}
    \end{equation*} %
    The proof of the second inequality follows the same line.
\end{proof}
\noindent In the next lemma, we provide bounds on the norm of the search directions computed in Algorithm \ref{alg:hfzoba}. 

\begin{lemma}[Bound on search directions]\label{lem:hfzoba_search_directions}
    Let Assumptions \ref{asm:F_smooth},\ref{asm:bc_condition} and \ref{asm:G_asm} hold. Let $\hat{D}_z^k, \hat{D}_v^k$ and $\hat{D}_x^k$ be the search directions defined in eq. \eqref{eqn:hfzoba_search_directions}. Let $z_k,v_k$ and $x_k$ be the sequences generated by Algorithm \ref{alg:hfzoba}. Then, for every $k \in \mathbb{N}$,
    \begin{equation*}
        \begin{aligned}
            \mathbb{E}_k \left[\| \hat{D}_z^k \|^2 \right]  &\leq 2 \left(2 + \frac{p + 2}{b_1 \ell_1} \right) L_{1,G}^2 \| z_k - z^*(x_k) \|^2\\
            &+ \left( \frac{(p + 6)^3}{2 b_1 \ell_1} + \left( \frac{3}{b_1} + 1 \right) (p + 3)^3 \right)L_{1,G}^2 h_{k}^2 + 2 \left(3 +\frac{p + 2}{\ell_1} \right)\frac{\sigma_{1,G}^2}{b_1}.
        \end{aligned}
    \end{equation*}
    Let $C_1 = 4\left(1 + \frac{3(p + 2)}{b_1 \ell_1} \right)$, $B_1 = 2 \left(1 + \frac{p + 2}{b_2 \ell_2} \right)$, $B_2 = \left( \frac{(p + 6)^3}{2 b_2 \ell_2} + \left(\frac{3}{b_2} + 1 \right)(p + 3)^{3}  \right) L_{1,F}^2$ and $B_3 = 2\left(\frac{p + 2}{\ell_2} + 3 \right ) \frac{\sigma_{1,F}^2}{b_2}$. Then for every $k \in \mathbb{N}$,
    \begin{equation*}
        \begin{aligned}
            \mathbb{E}_k\left[\left\| \hat{D}_v^k \right\|^2\right] &\leq 2 C_1L_{1,G}^2 \|v_k - v^*(x_k)\|^2 +2C_1\frac{L_{1,G}^2 L_{0,F}^2}{\mu_G^2}\\
            &+\frac{6L_{1,G}^2 (p + 6)^3}{b_1 \ell_1}\frac{h_k^2}{\bar{h}_k^2}+ 2 B_1 L_{0,F}^2 +2 B_2 h_{k}^2 + 2B_3.
        \end{aligned}
    \end{equation*}
    Moreover, let $C_2 = 4\left(1 + \frac{3(d + 2)}{b_1 \ell_1} \right)$, $B_4 = 2 \left(1 + \frac{d + 2}{b_2 \ell_2} \right)$, $B_5 = \left( \frac{(d + 6)^3}{2 b_2 \ell_2} + \left(\frac{3}{b_2} + 1 \right)(d + 3)^{3}  \right) L_{1,F}^2$ and $B_6 = 2\left(\frac{d + 2}{b_2 \ell_2} +\frac{3}{b_2} \right )\sigma_{1,F}^2$. Then,
    \begin{equation*}
        \begin{aligned}
            \mathbb{E}_k\left[\left\| \hat{D}_x^k \right\|^2\right] &\leq 2 C_2L_{1,G}^2 \|v_k - v^*(x_k)\|^2 +2C_2\frac{L_{1,G}^2 L_{0,F}^2}{\mu_G^2}\\
            &+\frac{6L_{1,G}^2 (d + 6)^3}{b_1 \ell_1}\frac{h_k^2}{\bar{h}_k^2}+ 2 B_4 L_{0,F}^2 +2 B_5 h_{k}^2 + 2B_6.
        \end{aligned}
    \end{equation*}
\end{lemma}
\begin{proof}
    The proof of the first inequality follows the same line of the proof of eq. \eqref{eqn:bound_norm_dz}. We focus on proving the second inequality. By $\|a + b\|^2 \leq 2\|a\|^2 + 2\|b\|^2$, 
    \begin{equation*}
        \begin{aligned}
            \mathbb{E}_k\left[\left\| \hat{D}_v^k \right\|^2\right] &= \mathbb{E}_k\left[\left\| \frac{1}{\bar{h}_k}\left( \hat{\nabla}_zg_k(z_k + \bar{h}_k v_k, x_k) - \hat{\nabla}_zg_k(z_k, x_k) \right) + \hat{\nabla}_z f_k(z_k,x_k) \right\|^2\right]\\
            &\leq 2 \underbrace{\mathbb{E}_k\left[\left\| \frac{1}{\bar{h}_k}\left( \hat{\nabla}_zg_k(z_k + \bar{h}_k v_k, x_k) - \hat{\nabla}_zg_k(z_k, x_k) \right)  \right\|^2\right]}_{A}+ 2 \underbrace{\mathbb{E}_k \left[ \left\| \hat{\nabla}_z f_k(z_k,x_k) \right\|^2 \right]}_{B}.
        \end{aligned}
    \end{equation*}
    We bound the $B$ term by Lemma \ref{lem:approx_error},
    \begin{equation*}
        \begin{aligned}
           \mathbb{E}_k \left[ \left\| \hat{\nabla}_z f_k(z_k,x_k) \right\|^2 \right] &\leq \underbrace{2 \left(1 + \frac{p + 2}{b_2 \ell_2} \right)}_{B_1} \left\|\nabla_z F(z_k, x_k) \right\|^2\\
        & +\underbrace{\left( \frac{(p + 6)^3}{2 b_2 \ell_2} + \left(\frac{3}{b_2} + 1 \right)(p + 3)^{3}  \right) L_{1,F}^2}_{B_2} h_{k}^2 + \underbrace{2\left(\frac{p + 2}{\ell_2} + 3 \right ) \frac{\sigma_{1,F}^2}{b_2}}_{B_3}.
        \end{aligned}
    \end{equation*}
    We bound the $A$ term by Lemma \ref{lem:hfzoba_sur_diff}. Let $C_1 = 4\left(1 + \frac{3(p + 2)}{b_1 \ell_1} \right)$. Then,
    \begin{equation*}
        \begin{aligned}
            \frac{2}{\bar{h}_k^2} \mathbb{E}_k [\| \hat{\nabla}_zg_k(z_k + \bar{h}_k v_k, x_k) - \hat{\nabla}_zg_k(z_k, x_k) \|^2]&\leq C_1L_{1,G}^2 \|v_k - v^*(x_k)\|^2 +C_1\frac{L_{1,G}^2 L_{0,F}^2}{\mu_G^2}\\
            &+\frac{3L_{1,G}^2 (p + 6)^3}{b_1 \ell_1}\frac{h_k^2}{\bar{h}_k^2}.
        \end{aligned}
    \end{equation*}
    Thus, by these bounds, we have
    \begin{equation*}
        \begin{aligned}
            \mathbb{E}_k\left[\left\| \hat{D}_v^k \right\|^2\right] &\leq 2 C_1L_{1,G}^2 \|v_k - v^*(x_k)\|^2 +2C_1\frac{L_{1,G}^2 L_{0,F}^2}{\mu_G^2}\\
            &+\frac{6L_{1,G}^2 (p + 6)^3}{b_1 \ell_1}\frac{h_k^2}{\bar{h}_k^2}+ 2 B_1 \left\|\nabla_z F(z_k, x_k) \right\|^2\\
        & +2 B_2 h_{k}^2 + 2B_3.
        \end{aligned}
    \end{equation*}
    Since $F$ is $L_{0,F}$-Lipschitz continuous, we get the first claim
    \begin{equation*}
        \begin{aligned}
            \mathbb{E}_k\left[\left\| \hat{D}_v^k \right\|^2\right] &\leq 2 C_1L_{1,G}^2 \|v_k - v^*(x_k)\|^2 +2C_1\frac{L_{1,G}^2 L_{0,F}^2}{\mu_G^2}\\
            &+\frac{6L_{1,G}^2 (p + 6)^3}{b_1 \ell_1}\frac{h_k^2}{\bar{h}_k^2}+ 2 B_1 L_{0,F}^2 +2 B_2 h_{k}^2 + 2B_3.
        \end{aligned}
    \end{equation*}
    The proof of the last inequality follows the same line.
        \begin{equation*}
        \begin{aligned}
            \mathbb{E}_k\left[\left\| \hat{D}_x^k \right\|^2\right] &= \mathbb{E}_k\left[\left\| \frac{1}{\bar{h}_k}\left( \hat{\nabla}_xg_k(z_k + \bar{h}_k v_k, x_k) - \hat{\nabla}_xg_k(z_k, x_k) \right) + \hat{\nabla}_x f_k(z_k,x_k) \right\|^2\right]\\
            &\leq 2 \underbrace{\mathbb{E}_k\left[\left\| \frac{1}{\bar{h}_k}\left( \hat{\nabla}_xg_k(z_k + \bar{h}_k v_k, x_k) - \hat{\nabla}_xg_k(z_k, x_k) \right)  \right\|^2\right]}_{A}+ 2 \underbrace{\mathbb{E}_k \left[ \left\| \hat{\nabla}_x f_k(z_k,x_k) \right\|^2 \right]}_{B}.
        \end{aligned}
    \end{equation*}
    We bound the $B$ term by Lemma \ref{lem:approx_error},
    \begin{equation*}
        \begin{aligned}
           \mathbb{E}_k \left[ \left\| \hat{\nabla}_x f_k(z_k,x_k) \right\|^2 \right] &\leq 
           \underbrace{2 \left(1 + \frac{d + 2}{b_2 \ell_2} \right)}_{B_4} \left\|\nabla_x F(z_k, x_k) \right\|^2\\
        & +\underbrace{\left( \frac{(d + 6)^3}{2 b_2 \ell_2} + \left(\frac{3}{b_2} + 1 \right)(d + 3)^{3}  \right) L_{1,F}^2}_{B_5} h_{k}^2\\
        &+ \underbrace{2\left(\frac{d + 2}{b_2 \ell_2} +\frac{3}{b_2} \right )\sigma_{1,F}^2}_{B_6}.
        \end{aligned}
    \end{equation*}
    We bound the $A$ term by Lemma \ref{lem:hfzoba_sur_diff}. Let $C_2 = 4\left(1 + \frac{3(d + 2)}{b_1 \ell_1} \right)$. Then,
    \begin{equation*}
        \begin{aligned}
            \frac{2}{\bar{h}_k^2} \mathbb{E}_k [\| \hat{\nabla}_xg_k(z_k + \bar{h}_k v_k, x_k) - \hat{\nabla}_xg_k(z_k, x_k) \|^2]&\leq C_2L_{1,G}^2 \|v_k - v^*(x_k)\|^2 +C_2\frac{L_{1,G}^2 L_{0,F}^2}{\mu_G^2}\\
            &+\frac{3L_{1,G}^2 (d + 6)^3}{b_1 \ell_1}\frac{h_k^2}{\bar{h}_k^2}.
        \end{aligned}
    \end{equation*}
    Thus, by these bounds, we have
    \begin{equation*}
        \begin{aligned}
            \mathbb{E}_k\left[\left\| \hat{D}_x^k \right\|^2\right] &\leq 2 C_2L_{1,G}^2 \|v_k - v^*(x_k)\|^2 +2C_2\frac{L_{1,G}^2 L_{0,F}^2}{\mu_G^2}\\
            &+\frac{6L_{1,G}^2 (d + 6)^3}{b_1 \ell_1}\frac{h_k^2}{\bar{h}_k^2}+ 2 B_4 \left\|\nabla_z F(z_k, x_k) \right\|^2\\
        & +2 B_5 h_{k}^2 + 2B_6.
        \end{aligned}
    \end{equation*}
    Since $F$ is $L_{0,F}$-Lipschitz continuous, we get the first claim
    \begin{equation*}
        \begin{aligned}
            \mathbb{E}_k\left[\left\| \hat{D}_x^k \right\|^2\right] &\leq 2 C_2L_{1,G}^2 \|v_k - v^*(x_k)\|^2 +2C_2\frac{L_{1,G}^2 L_{0,F}^2}{\mu_G^2}\\
            &+\frac{6L_{1,G}^2 (d + 6)^3}{b_1 \ell_1}\frac{h_k^2}{\bar{h}_k^2}+ 2 B_4 L_{0,F}^2 +2 B_5 h_{k}^2 + 2B_6.
        \end{aligned}
    \end{equation*}
\end{proof}
\noindent In the next lemma, we bound the error of the sequence $(v_k)_{k \in \mathbb{N}}$ generated by Algorithm~\ref{alg:hfzoba} relative to the corresponding optimal solutions $v^*(x_k)$ at each iterate $x_k$. We do not need to bound the error of the sequence $(z_k)_{k \in \mathbb{N}}$, since it differs from that of ZOBA only in the construction of the gradient estimator (as no Hessian is involved in the $z_k$ update of ZOBA). Indeed, in Algorithm~\ref{alg:hfzoba} this surrogate is constructed using forward finite differences instead of central differences. As shown in Lemma~\ref{lem:approx_error}, this change yields the same bound on the surrogate norm, analogously to the bound obtained for the gradient surrogate of the outer objective $f$ in ZOBA.
\begin{lemma}[Bound on $v_k$ iteration]\label{lem:hfzoba_vbound}
  Let Assumptions \ref{asm:F_smooth} and \ref{asm:G_asm} hold. Let $(v_k)_{k \in \mathbb{N}}$ be the sequence generated in Algorithm \ref{alg:hfzoba} and let $\bar{h}_k = \begin{cases}
    \frac{\hat{h}_k}{\|v_k\|} & \text{if} \quad \|v_k\| \neq 0\\
    \hat{h}_k
\end{cases}$ with $\hat{h}_k \leq \frac{\mu_G}{\sqrt{8} L_{2,G}}$. Let $\omega_{v, 1} = \left(L_{1,F} + \frac{L_{0,F} L_{2,G}}{\mu_G} \right)$, $\omega_{v,2} = \frac{16 L_{2,G}^2}{9} \left( (d + 5)^{5/2} + (d + 3)^{3/2} \right)^2$ and $\omega_k^v = \left(1 + \frac{2}{\mu_G \rho_k} \right)$. Let $h_k \leq \frac{\mu_G}{\sqrt{16 \omega_{v,2}}}$. Then, for every $k \in \mathbb{N}$,
  \begin{equation*}
    \begin{aligned}
      \mathbb{E}_k \left[ \| v_{k + 1} - v^*(x_{k + 1}) \|^2 \right] &\leq \left(1 - \frac{\mu_G}{2} \rho_k \right) \| v_k - v^*(x_k) \|^2 +4\frac{\omega_{v,1}^2}{\mu_G} \rho_k \|z_k - z^*(x_k)\|^2\\
        &+ 2\rho_k^2 \mathbb{E}_k[\| \hat{D}_v^k\|^2] +2 \omega_k^v L_*^2 \gamma_k^2 \mathbb{E}_k[\| \hat{D}_x^k\|^2]\\
        &+2 \frac{L_{1,G}^2(d + 3)^2}{\mu_G} \rho_k h_k^2 + 4\frac{\omega_{v,2} L_{0,F}^2}{\mu_G^3} \rho_k h_k^2\\
        &+  \frac{8}{\mu_G} \frac{L_{2,G}^2}{4} \frac{L_{0,F}^2}{\mu_G^2} \rho_k\hat{h}_k^2.
    \end{aligned}
  \end{equation*}
\end{lemma}
\begin{proof}
    The proof follows the same line of Lemma \ref{lem:smooth_bound_zv}. Adding and subtracting $v^*(x_k)$, we have
    \begin{equation}\label{eqn:hfzoba_bound_v}
        \begin{aligned}
            \| v_{k + 1} - v^*(x_{k + 1}) \|^2 &=\underbrace{\| v_{k + 1} - v^*(x_k)\|^2}_{A} + \underbrace{\| v^*(x_{k + 1}) - v^*(x_k) \|^2}_{B} \underbrace{-2\scalT{v_{k + 1} - v^*(x_k)}{v^*(x_{k + 1}) - v^*(x_k)}}_{C}.
        \end{aligned}
    \end{equation}
    We start by bounding $A$.  By Algorithm \ref{alg:hfzoba} and taking the conditional expectation
    \begin{equation*}
        \begin{aligned}
        \mathbb{E}_k[\| v_{k + 1} - v^*(x_k)\|^2] &= \| v_k - v^*(x_k) \|^2 + \rho_k^2 \mathbb{E}_k[\| \hat{D}_v^k\|^2]\\
        &- 2\rho_k \scalT{\frac{1}{\bar{h}_k} ( \nabla_z G_{h_k}(z_k + \bar{h}_k v_k, x_k) - \nabla_z G(z_k,x_k)) + \nabla_z F_{h_k}(z_k,x_k)}{v_k - v^*(x_k)}.            
        \end{aligned}
    \end{equation*}
    Let $\tilde{D}_v^k = \nabla^2_{zz}G_{h_k}(z_k,x_k)v_k + \nabla_z F_{h_k}(z_k,x_k)$. Adding and subtracting $\nabla_{zz}G_{h_k}(z_k,x_k)v_k$, we get
        \begin{equation*}
        \begin{aligned}
        \mathbb{E}_k[\| v_{k + 1} - v^*(x_k)\|^2] &= \| v_k - v^*(x_k) \|^2 + \rho_k^2 \mathbb{E}_k[\| \hat{D}_v^k\|^2]\\
        &- 2\rho_k \scalT{\tilde{D}_v^k}{v_k - v^*(x_k)}\\
        &- 2\rho_k \scalT{\frac{1}{\bar{h}_k} ( \nabla_z G_{h_k}(z_k + \bar{h}_k v_k, x_k) - \nabla_z G(z_k,x_k)) - \nabla_{zz}^2 G(z_k, x_k) v_k}{v_k - v^*(x_k)}.
        \end{aligned}
    \end{equation*}
    Let $\omega_{v, 1} = \left(L_{1,F} + \frac{L_{0,F} L_{2,G}}{\mu_G} \right)$, $\omega_{v,2} = \frac{16 L_{2,G}^2}{9} \left( (d + 5)^{5/2} + (d + 3)^{3/2} \right)^2$. For $h_k^2 \leq \frac{\mu_G^2}{16 \omega_{v,2}}$, following the same line of the proof of Lemma \ref{lem:smooth_bound_zv},  we get
    \begin{equation*}
        \begin{aligned}
        \mathbb{E}_k[\| v_{k + 1} - v^*(x_k)\|^2] &= \left(1 - \frac{5\mu_G}{4} \rho_k \right) \| v_k - v^*(x_k) \|^2\\
        &+ \rho_k^2 \mathbb{E}_k[\| \hat{D}_v^k\|^2] +4\frac{\omega_{v,1}^2}{\mu_G} \rho_k \|z_k - z^*(x_k)\|^2\\
        &+2 \frac{L_{1,G}^2(d + 3)^2}{\mu_G} \rho_k h_k^2 + 4\frac{\omega_{v,2} L_{0,F}^2}{\mu_G^3} \rho_k h_k^2\\
        &- 2\rho_k \scalT{\frac{1}{\bar{h}_k} ( \nabla_z G_{h_k}(z_k + \bar{h}_k v_k, x_k) - \nabla_z G(z_k,x_k)) - \nabla_{zz}^2 G(z_k, x_k) v_k}{v_k - v^*(x_k)}.
        \end{aligned}
    \end{equation*}
    Let 
    \begin{equation*}
        \Delta_k = \frac{1}{\bar{h}_k} ( \nabla_z G_{h_k}(z_k + \bar{h}_k v_k, x_k) - \nabla_z G(z_k,x_k)) - \nabla_{zz}^2 G(z_k, x_k) v_k.
    \end{equation*}
    By Cauchy-Schwartz and Young's inequality with parameter $\frac{\mu_G}{4}$, we have
    \begin{equation*}
        \begin{aligned}
- 2\rho_k \scalT{\Delta_k}{v_k - v^*(x_k)} &\leq \frac{4\rho_k}{\mu_G} \| \Delta_k \|^2 + \frac{\mu_G}{4} \rho_k \|v_k - v^*(x_k)\|^2.
        \end{aligned}
    \end{equation*}
    By Lemma \ref{lem:biased_hess_approx} and  by the choice, 
    \begin{equation*}
        \begin{aligned}
            \bar{h}_k = \begin{cases}
                \frac{\hat{h}_k}{\|v_k\|} & \|v_k\| \neq 0\\
                \hat{h}_k
            \end{cases},
        \end{aligned}
    \end{equation*}
    we have
    \begin{equation*}
        \begin{aligned}
- 2\rho_k \scalT{\Delta_k}{v_k - v^*(x_k)} &\leq %
\frac{4\rho_k}{\mu_G} \frac{L_{2,G}^2}{4} \bar{h}_k^2\|v_k\|^4 + \frac{\mu_G}{4} \rho_k \|v_k - v^*(x_k)\|^2\\
&= \frac{4\rho_k}{\mu_G} \frac{L_{2,G}^2}{4} \hat{h}_k^2\|v_k\|^2 + \frac{\mu_G}{4} \rho_k \|v_k - v^*(x_k)\|^2.
        \end{aligned}
    \end{equation*}
    Adding and subtracting $v^*(x_k)$ and by Lemma \ref{lem:bound_norm_v},
    \begin{equation*}
        \begin{aligned}
- 2\rho_k \scalT{\Delta_k}{v_k - v^*(x_k)} &\leq  \frac{8\rho_k}{\mu_G} \frac{L_{2,G}^2}{4} \hat{h}_k^2\|v_k -v^*(x_k)\|^2 + \frac{\mu_G}{4} \rho_k \|v_k - v^*(x_k)\|^2 + \frac{8\rho_k}{\mu_G} \frac{L_{2,G}^2}{4} \hat{h}_k^2\|v^*(x_k)\|^2\\
&\leq  \left(\frac{8}{\mu_G} \frac{L_{2,G}^2}{4} \hat{h}_k^2 + \frac{\mu_G}{4}\right) \rho_k\|v_k -v^*(x_k)\|^2 +  \frac{8\rho_k}{\mu_G} \frac{L_{2,G}^2}{4} \frac{L_{0,F}^2}{\mu_G^2} \hat{h}_k^2.%
        \end{aligned}
    \end{equation*}
    Since $\hat{h}_k \leq \frac{\mu_G}{\sqrt{8}L_{2,G}}= \sqrt{\frac{\mu_G}{4} \frac{\mu_G}{8} \frac{4}{L_{2,G}^2}}$,
    \begin{equation*}
        \begin{aligned}
- 2\rho_k \scalT{\Delta_k}{v_k - v^*(x_k)} &\leq   \frac{\mu_G}{2} \rho_k\|v_k -v^*(x_k)\|^2 +  \frac{8}{\mu_G} \frac{L_{2,G}^2}{4} \frac{L_{0,F}^2}{\mu_G^2} \rho_k\hat{h}_k^2.%
        \end{aligned}
    \end{equation*}
    Thus, we conclude the bound of term $A$ obtaining
    \begin{equation*}
        \begin{aligned}
        \mathbb{E}_k[\| v_{k + 1} - v^*(x_k)\|^2] &\leq \left(1 - \frac{5\mu_G}{4} \rho_k + \frac{\mu_G}{2} \rho_k \right) \| v_k - v^*(x_k) \|^2\\
        &+ \rho_k^2 \mathbb{E}_k[\| \hat{D}_v^k\|^2] +4\frac{\omega_{v,1}^2}{\mu_G} \rho_k \|z_k - z^*(x_k)\|^2\\
        &+2 \frac{L_{1,G}^2(d + 3)^2}{\mu_G} \rho_k h_k^2 + 4\frac{\omega_{v,2} L_{0,F}^2}{\mu_G^3} \rho_k h_k^2\\
        &+  \frac{8}{\mu_G} \frac{L_{2,G}^2}{4} \frac{L_{0,F}^2}{\mu_G^2} \rho_k\hat{h}_k^2.
        \end{aligned}
    \end{equation*}
    We bound the $B$ term by Lemma \ref{lem:smt_setting_smooth_value} and Algorithm \ref{alg:hfzoba},
    \begin{equation*}
        \begin{aligned}
            \|v^*(x_{k +1}) - v^*(x_k) \|^2 \leq L_*^2 \|x_{k + 1} - x_k \|^2 = L_*^2 \gamma_k^2 \| \hat{D}_x^k\|^2.
        \end{aligned}
    \end{equation*}
    The bound on term $C$ follows the same line of the analogous bound in the proof of Lemma \ref{lem:smooth_bound_zv},
    \begin{equation*}
        \begin{aligned}
            -2\scalT{v_{k + 1} - v^*(x_k)}{v^*(x_{k + 1}) - v^*(x_k)} &\leq \frac{\mu_G}{4} \rho_k \| v_{k} - v^*(x_k) \|^2 +  \frac{4}{\mu_G \rho_k} L^2_* \gamma_k^2 \| \hat{D}_x^k \|^2\\
        &+ \rho_k^2  \| \hat{D}_v^k \|^2 + L_*^2 \gamma_k^2 \| \hat{D}_x^k \|^2.
        \end{aligned}
    \end{equation*}
    Let $\omega_k^v = \left(1 + \frac{2}{\mu_G \rho_k} \right)$. Restarting from eq. \eqref{eqn:hfzoba_bound_v}, taking the conditional expectation and using the bounds on the terms $A,B$ and $C$, we get the claim
    \begin{equation*}
        \begin{aligned}
            \mathbb{E}_k[\|v_{k+1} - v^*(x_{k + 1}) \|^2] &\leq \left(1 - \frac{\mu_G}{2} \rho_k \right) \| v_k - v^*(x_k) \|^2 +4\frac{\omega_{v,1}^2}{\mu_G} \rho_k \|z_k - z^*(x_k)\|^2\\
        &+ 2\rho_k^2 \mathbb{E}_k[\| \hat{D}_v^k\|^2] +2 \omega_k^v L_*^2 \gamma_k^2 \mathbb{E}_k[\| \hat{D}_x^k\|^2]\\
        &+2 \frac{L_{1,G}^2(d + 3)^2}{\mu_G} \rho_k h_k^2 + 4\frac{\omega_{v,2} L_{0,F}^2}{\mu_G^3} \rho_k h_k^2\\
        &+  \frac{8}{\mu_G} \frac{L_{2,G}^2}{4} \frac{L_{0,F}^2}{\mu_G^2} \rho_k\hat{h}_k^2.
        \end{aligned}
    \end{equation*}
\end{proof}
\noindent In the following lemma, we derive the bound on the value function $\Psi$ along the sequence $(x_k)_{k \in \mathbb{N}}$ generated by Algorithm \ref{alg:hfzoba}.
\begin{lemma}[Function value decrease]\label{lem:hfzoba_fun_dec}
    Let Assumptions \ref{asm:F_smooth}, \ref{asm:bc_condition}, \ref{asm:G_asm} hold. Let $(x_k)_{k \in \mathbb{N}}$ be the sequence generated by Algorithm \ref{alg:hfzoba}. Then for every $k \in \mathbb{N}$,
    \begin{equation*}
        \begin{aligned}
            \mathbb{E}_k[\Psi(x_{k + 1})] &\leq \Psi(x_k) - \frac{\gamma_k}{2} \|\nabla \Psi(x_k)\|^2 + \frac{L_\Psi}{2} \gamma_k^2 \mathbb{E}_k[\| \hat{D}_x^k \|^2]\\
            &+4 C_{1,\Psi}^2 \gamma_k \|z_k - z^*(x_k)\|^2 + C_{4,\Psi} \gamma_k \|v_k - v^*(x_k)\|^2\\
            &+ 2 C_{2,\Psi}^2 \gamma_kh_k^2 + C_{3,\Psi} \hat{h}_k^2\gamma_k,
        \end{aligned}
    \end{equation*}
    where 
    \begin{equation*}
        \begin{aligned}
            C_{1,\Psi} &= L_{1,F} + \frac{L_{0,F} L_{1,G}}{\mu_G},\\
            C_{2,\Psi} &= \frac{2L_{2,G}L_{0,F}}{3\mu_G}\left((d + 5)^{5/2} +  (d +3)^{3/2} +  (d + 1) (p +3)^{3/2} \right) + \frac{L_{1,F} d\sqrt{p}}{2}\\
            C_{3,\Psi} &= \frac{L_{2,G}^2L_{0,F}^2}{2\mu_G^2} \quad \text{and}\quad C_{4,\Psi} =4L_{1,G}^2 + \frac{\mu_G^2}{16}.
        \end{aligned}
    \end{equation*}
\end{lemma}
\begin{proof}
    The proof follows the same line of Lemma \ref{lem:smooth_val_fun_descent}. By Lemma \ref{lem:smt_setting_smooth_value}, $\Psi$ is $L_\Psi$-smooth. Thus, by the Descent Lemma \cite{polyak1987introduction} and Algorithm \ref{alg:hfzoba}, we have
    \begin{equation*}
        \Psi(x_{k + 1}) \leq \Psi(x_k) - \gamma_k\scalT{\nabla \Psi(x_k)}{\hat{D}_x^k} + \frac{L_\Psi}{2} \gamma_k^2 \| \hat{D}_x^k \|^2. 
    \end{equation*}
    Taking the conditional expectation,
    \begin{equation*}
        \begin{aligned}
            \mathbb{E}_k[\Psi(x_{k + 1})] &\leq \Psi(x_k)\\
            &- \gamma_k\scalT{\nabla \Psi(x_k)}{\frac{1}{\bar{h}_k}(\nabla_x G_{h_k}(z_k +\bar{h}_k v_k,x_k) -\nabla_x G_{h_k}(z_k,x_k)) +\nabla_x F_{h_k}(z_k, x_k)}\\
            &+ \frac{L_\Psi}{2} \gamma_k^2 \mathbb{E}_k[\| \hat{D}_x^k \|^2]. 
        \end{aligned}
    \end{equation*}
    By $\scalT{a}{b} = \frac{1}{2} (\|a\|^2 +\|b\|^2 - \|a - b\|^2)$, we get
    \begin{equation}\label{eqn:val_dec_hfzoba_1}
        \begin{aligned}
            \mathbb{E}_k[\Psi(x_{k + 1})] &\leq \Psi(x_k)\\
            &- \frac{\gamma_k}{2} \bigg( \|\nabla \Psi(x_k)\|^2\\
            &+ \|\frac{1}{\bar{h}_k}(\nabla_x G_{h_k}(z_k +\bar{h}_k v_k,x_k) -\nabla_x G_{h_k}(z_k,x_k)) +\nabla_x F_{h_k}(z_k, x_k) \|^2\\
            &- \|\frac{1}{\bar{h}_k}(\nabla_x G_{h_k}(z_k +\bar{h}_k v_k,x_k) -\nabla_x G_{h_k}(z_k,x_k)) +\nabla_x F_{h_k}(z_k, x_k) - \nabla \Psi(x_k) \|^2\bigg)\\
            &+ \frac{L_\Psi}{2} \gamma_k^2 \mathbb{E}_k[\| \hat{D}_x^k \|^2]\\
            &\leq \Psi(x_k) - \frac{\gamma_k}{2} \|\nabla \Psi(x_k)\|^2 + \frac{L_\Psi}{2} \gamma_k^2 \mathbb{E}_k[\| \hat{D}_x^k \|^2]\\
            &+\underbrace{\frac{\gamma_k}{2}\|\frac{1}{\bar{h}_k}(\nabla_x G_{h_k}(z_k +\bar{h}_k v_k,x_k) -\nabla_x G_{h_k}(z_k,x_k)) +\nabla_x F_{h_k}(z_k, x_k) - \nabla \Psi(x_k) \|^2}_{A}. 
        \end{aligned}
    \end{equation}
    Now we bound $A$. Let $\bar{\Delta}_k = \frac{1}{\bar{h}_k}(\nabla_x G_{h_k}(z_k +\bar{h}_k v_k,x_k) -\nabla_x G_{h_k}(z_k,x_k)) +\nabla_x F_{h_k}(z_k, x_k) $, $\Delta_k = \frac{1}{\bar{h}_k}(\nabla_x G_{h_k}(z_k +\bar{h}_k v_k,x_k) -\nabla_x G_{h_k}(z_k,x_k))- \nabla_{xz}^2G_{h_k}(z_k,x_k)v_k$ and $\tilde{D}_x^k =\nabla_{xz}^2 G_{h_k}(z_k,x_k)v_k +\nabla_x F_{h_k}(z_k, x_k)$. Adding and subtracting $\nabla_{xz}^2G_{h_k}(z_k,x_k)v_k$ and by $\|a + b\|^2 \leq 2\|a\|^2 + 2\|b\|^2$,
    \begin{equation*}
        \begin{aligned}
            \frac{\gamma_k}{2}\|\bar{\Delta}_k - \nabla \Psi(x_k) \|^2 &\leq \gamma_k \| \Delta_k \|^2 + \gamma_k \| \tilde{D}_x^k - \nabla \Psi(x_k)\|^2.
        \end{aligned}
    \end{equation*}
    By Lemma \ref{lem:biased_hess_approx} and Lemma \ref{lem:bilevel_smth_error} with $h_1 = \eta_1 = \eta_2 = h_k$, we have
    \begin{equation*}
        \begin{aligned}
            \frac{\gamma_k}{2}\|\bar{\Delta}_k - \nabla \Psi(x_k) \|^2 &\leq \frac{L_{2,G}^2}{4} \bar{h}_k^2\gamma_k \|v_k \|^4\\
            &+ \gamma_k \bigg( \underbrace{\left(L_{1,F} + \frac{L_{0,F} L_{1,G}}{\mu_G} \right)}_{C_{1,\Psi}}\|z_k -z^*(x_k)\| +  L_{1,G} \|v_k - v^*(x_k)\|\\
           &+ \underbrace{\left(\frac{L_{0,F}}{\mu_G}\frac{2L_{2,G}}{3}\left((d + 5)^{5/2} +  (d +3)^{3/2} +  (d + 1) (p +3)^{3/2} \right) + \frac{L_{1,F} d\sqrt{p}}{2} \right)}_{=:C_{2,\Psi}} h_k\bigg)^2\\
           &\leq \frac{L_{2,G}^2}{4} \bar{h}_k^2\gamma_k \|v_k \|^4\\
           &+ 4 C_{1,\Psi}^2 \gamma_k \|z_k - z^*(x_k)\|^2 + 4 L_{1,G}^2 \gamma_k \|v_k - v^*(x_k)\|^2 + 2 C_{2,\Psi}^2 \gamma_k h_k^2.
        \end{aligned}
    \end{equation*}
    By the choice 
    \begin{equation*}
        \begin{aligned}
            \bar{h}_k = \begin{cases}
                \frac{\hat{h}_k}{\|v_k\|} & \|v_k\| \neq 0\\
                \hat{h}_k
            \end{cases},
        \end{aligned}
    \end{equation*}
    we have
    \begin{equation*}
        \begin{aligned}
            \frac{\gamma_k}{2}\|\bar{\Delta}_k - \nabla \Psi(x_k) \|^2 &\leq \frac{L_{2,G}^2}{4} \hat{h}_k^2\gamma_k \|v_k \|^2\\
           &+ 4 C_{1,\Psi}^2 \gamma_k \|z_k - z^*(x_k)\|^2 + 4 L_{1,G}^2 \gamma_k \|v_k - v^*(x_k)\|^2 + 2 C_{2,\Psi}^2 \gamma_k h_k^2.
        \end{aligned}
    \end{equation*}
    Adding and subtracting $v^*(x_k)$ and by Lemma \ref{lem:bound_norm_v}, we get
    \begin{equation*}
        \begin{aligned}
            \frac{\gamma_k}{2}\|\bar{\Delta}_k - \nabla \Psi(x_k) \|^2 &\leq \underbrace{\frac{L_{2,G}^2}{2} \frac{L_{0,F}^2}{\mu_G^2}}_{=:C_{3,\Psi}} \hat{h}_k^2\gamma_k \\
           &+ 4 C_{1,\Psi}^2 \gamma_k \|z_k - z^*(x_k)\|^2 + \left(4 L_{1,G}^2 + \frac{L_{2,G}^2}{2} \hat{h}_k^2 \right) \gamma_k \|v_k - v^*(x_k)\|^2 + 2 C_{2,\Psi}^2 \gamma_kh_k^2.
        \end{aligned}
    \end{equation*}
    Since $\hat{h}_k \leq \frac{\mu_G}{\sqrt{8}L_{2,G}}$, we have
    \begin{equation*}
        \begin{aligned}
            \frac{\gamma_k}{2}\|\bar{\Delta}_k - \nabla \Psi(x_k) \|^2 &\leq 4 C_{1,\Psi}^2 \gamma_k \|z_k - z^*(x_k)\|^2 + \underbrace{\left(4 L_{1,G}^2 + \frac{\mu_G^2}{16} \right)}_{=:C_{4,\Psi}} \gamma_k \|v_k - v^*(x_k)\|^2\\
            &+ 2 C_{2,\Psi}^2 \gamma_kh_k^2 + C_{3,\Psi} \hat{h}_k^2\gamma_k.
        \end{aligned}
    \end{equation*}
    Using this inequality in eq. \eqref{eqn:val_dec_hfzoba_1}, we get
    \begin{equation*}
        \begin{aligned}
            \mathbb{E}_k[\Psi(x_{k + 1})] &\leq \Psi(x_k) - \frac{\gamma_k}{2} \|\nabla \Psi(x_k)\|^2 + \frac{L_\Psi}{2} \gamma_k^2 \mathbb{E}_k[\| \hat{D}_x^k \|^2]\\
            &+4 C_{1,\Psi}^2 \gamma_k \|z_k - z^*(x_k)\|^2 + C_{4,\Psi} \gamma_k \|v_k - v^*(x_k)\|^2\\
            &+ 2 C_{2,\Psi}^2 \gamma_kh_k^2 + C_{3,\Psi} \hat{h}_k^2\gamma_k. 
        \end{aligned}
    \end{equation*}
\end{proof}

\section{Proofs of Main Results}\label{app:proof_main_results}

\subsection*{Proof of Theorem \ref{thm:conv_rate}} 
    Let $\delta_{z_k} = \| z_k - z^*(x_k) \|^2$ and $\delta_{v_k} = \| v_k - v^*(x_k) \|^2$. For $\phi_z^k, \phi_v^k > 0$ s.t. $\phi_v^{k + 1} \leq \phi_v^k$ and $\phi_z^{k+1} \leq \phi_z^k$, define
    \begin{equation*}
        L_k = \Psi(x_k) + \phi_z^k \delta_{z_k} +\phi_v^k \delta_{v_k}.
    \end{equation*}
    Consider
    \begin{equation*}
        \begin{aligned}
        L_{k + 1} - L_k &= \Psi(x_{k + 1}) - \Psi(x_k) + \phi_z^{k+ 1} \delta_{z_{k + 1}} -   \phi_z^k \delta_{z_k} + \phi_v^{k + 1} \delta_{v_{k + 1}} - \phi_v^k \delta_{v_k}  \\
        &\leq\Psi(x_{k + 1}) - \Psi(x_k) + \phi_z^k (\delta_{z_{k + 1}} -  \delta_{z_k}) + \phi_v^k( \delta_{v_{k + 1}} - \delta_{v_k}).            
        \end{aligned}
    \end{equation*}
    Taking the conditional expectation, we have
    \begin{equation*}
        \mathbb{E}_k\left[L_{k + 1}\right] - L_k \leq \mathbb{E}_k\left[\Psi(x_{k + 1})\right] - \Psi(x_k) + \phi_z^k (\mathbb{E}_k\left[\delta_{z_{k + 1}}\right] -  \delta_{z_k}) + \phi_v^k( \mathbb{E}_k\left[\delta_{v_{k + 1}}\right] - \delta_{v_k}).
    \end{equation*}    
    By Lemma \ref{lem:smooth_val_fun_descent}, we get
    \begin{equation*}
        \begin{aligned}
            \mathbb{E}_k\left[L_{k + 1}\right] - L_k &\leq - \frac{\gamma_k}{2}\| \nabla \Psi(x_k) \|^2 -\frac{\gamma_k}{2} \| \hat{D}_x^{k} \|^2 + \frac{L_\Psi}{2} \gamma_k^2 \mathbb{E}_k\left[\|D_x^k\|^2\right]\\
            &+ 2\gamma_k \omega_{v,1}^2 \delta_{z_k} + 2 \gamma_k L_{1,G}^2 \delta_{v_k}\\ %
            &+\phi_z^k (\mathbb{E}_k\left[\delta_{z_{k + 1}}\right] -  \delta_{z_k}) + \phi_v^k( \mathbb{E}_k\left[\delta_{v_{k + 1}}\right] - \delta_{v_k})\\
            &+ C_1^2 \gamma_k h_{k}^2.
        \end{aligned}
    \end{equation*}
    By Lemma \ref{lem:smooth_bound_zv}, we have
    \begin{equation*}
        \begin{aligned}
            \mathbb{E}_k\left[L_{k + 1}\right] - L_k &\leq - \frac{\gamma_k}{2}\| \nabla \Psi(x_k) \|^2 -\frac{\gamma_k}{2} \| \hat{D}_x^{k} \|^2 \\
        &+2 \phi_v^k \rho_k^2 \mathbb{E}_k\left[\|D_v^k\|^2 \right] +2 \phi_z^k \rho_k^2 \mathbb{E}_k\left[\left\| D_z^k \right\|^2\right]\\
        &- \bigg( \mu_G \phi_v^k \rho_k - 2 L_{1,G}^2 \gamma_k \bigg) \delta_{v_k}\\
        &- \bigg( \frac{\mu_G}{2}\phi_z^k \rho_k - 4\frac{\omega_{v,1}^2}{\mu_G} \phi_v^k \rho_k  - 2 \omega_{v,1}^2 \gamma_k \bigg) \delta_{z_k}\\
        &+ \bigg( \frac{L_\Psi}{2} + 2 w_k^z L_*^2 \phi_z^k + 2\omega_k^vL_*^2 \phi_v^k   \bigg)\gamma_k^2 \mathbb{E}_k \left[ \| D_x^k \|^2 \right] \\
            &+ C_1^2 \gamma_k h_{k}^2\\ %
        &+ \left( \frac{L_{1,G}^2}{4 \mu_G}(d + 3)^{3} \phi_z^k + 2  \frac{L_{1,G}^2(d + 3)^{3}}{\mu_G} \phi_v^k  + 4 \frac{\omega_{v,2}L_{0,F}^2}{\mu_G^3} \phi_v^k \right) \rho_k h_{k}^2.
        \end{aligned}
    \end{equation*}
    Let $\omega_3 = 2 \left(2 + \frac{p + 2}{b_1 \ell_1} \right) L_{1,G}^2$ and $\omega_{dx}^k = \bigg( \frac{L_\Psi}{2} + 2 w_k^z L_*^2 \phi_z^k + 2\omega_k^vL_*^2 \phi_v^k   \bigg)$.
    By Lemma \ref{lem:bound_search_directions}, we get
    \begin{equation*}
        \begin{aligned}
            \mathbb{E}_k\left[L_{k + 1}\right] - L_k &\leq - \frac{\gamma_k}{2}\| \nabla \Psi(x_k) \|^2 -\frac{\gamma_k}{2} \| \hat{D}_x^{k} \|^2 \\
            &- \bigg( \mu_G \phi_v^k \rho_k - 2  C_v^k \phi_v^k \rho_k^2 - 2 L_{1,G}^2 \gamma_k - \omega_{dx}^k  2C_x^k \gamma_k^2  \bigg) \delta_{v_k}\\
            &- \bigg( \frac{\mu_G}{2}\phi_z^k \rho_k - 4\frac{\omega_{v,1}^2}{\mu_G} \phi_v^k \rho_k   - 2\omega_3 \phi_z^k \rho_k^2 - 2 \omega_{v,1}^2 \gamma_k \bigg) \delta_{z_k}\\
                &+ C_1^2 \gamma_k h_{k}^2\\ %
             &+ \left( \frac{L_{1,G}^2}{4 \mu_G}(d + 3)^{3} \phi_z^k + 2  \frac{L_{1,G}^2(d + 3)^{3}}{\mu_G} \phi_v^k  + 4 \frac{\omega_{v,2}L_{0,F}^2}{\mu_G^3} \phi_v^k \right) \rho_k h_{k}^2\\
            &+ 2 \left( \frac{(p + 6)^3}{2 b_1 \ell_1} + \left( \frac{3}{b_1} + 1 \right) (p + 3)^3 \right)L_{1,G}^2 \phi_z^k \rho_k^2 h_{k}^2\\
            &+4  \left(3 +\frac{p + 2}{\ell_1} \right)\frac{\sigma_{1,G}^2}{b_1} \phi_z^k \rho_k^2\\
            &+2 \phi_v^k \rho_k^2 \bigg( C_v^k \frac{L_{0,F}^2}{\mu_G^2} + 4 \left(1 + \frac{p + 2}{b_2 \ell_2} \right)L_{0,F}^2 + 4\left(\frac{p + 2}{\ell_2} + 3 \right ) \frac{\sigma_{1,F}^2}{b_2}\\
            & +2\left( \frac{(p + 6)^3}{2 b_2 \ell_2} + \left(\frac{3}{b_2} + 1 \right)(p + 3)^{3}  \right) L_{1,F}^2 h_{k}^2 \bigg)\\
            &+\omega_{dx}^k \gamma_k^2 \bigg(2C_x^k \frac{L_{0,F}^2}{\mu_G^2}+2\left( \frac{(d + 6)^3}{2 b_2 \ell_2} + \left(\frac{3}{b_2} + 1 \right)(d + 3)^{3}  \right) L_{1,F}^2 h_{k}^2\\
            &+4 \left(1 + \frac{d + 2}{b_2 \ell_2} \right) L_{0,F}^2 + 4\left(\frac{d + 2}{b_2 \ell_2} +\frac{3}{b_2} \right )\sigma_{1,F}^2\bigg).
        \end{aligned}
    \end{equation*}    
    Let 
    \begin{equation*}
        \begin{aligned}
            C_2^k &:= \left( \frac{L_{1,G}^2}{4 \mu_G}(d + 3)^{3} \phi_z^k + 2  \frac{L_{1,G}^2(d + 3)^{3}}{\mu_G} \phi_v^k  + 4 \frac{\omega_{v,2}L_{0,F}^2}{\mu_G^3} \phi_v^k \right)\\
            C_3 &:= 4  \left(3 +\frac{p + 2}{\ell_1} \right)\frac{\sigma_{1,G}^2}{b_1}\\
            C_4&:= 2 \left( \frac{(p + 6)^3}{2 b_1 \ell_1} + \left( \frac{3}{b_1} + 1 \right) (p + 3)^3 \right)L_{1,G}^2\\
            C_5^k(h_k)&:=\bigg(2C_x^k \frac{L_{0,F}^2}{\mu_G^2}+2\left( \frac{(d + 6)^3}{2 b_2 \ell_2} + \left(\frac{3}{b_2} + 1 \right)(d + 3)^{3}  \right) L_{1,F}^2 h_{k}^2\\
            &+4 \left(1 + \frac{d + 2}{b_2 \ell_2} \right) L_{0,F}^2 + 4\left(\frac{d + 2}{b_2 \ell_2} +\frac{3}{b_2} \right )\sigma_{1,F}^2\bigg)\\
            C_6^k(h_k)&:=\bigg( C_v^k \frac{L_{0,F}^2}{\mu_G^2} + 4 \left(1 + \frac{p + 2}{b_2 \ell_2} \right)L_{0,F}^2 + 4\left(\frac{p + 2}{\ell_2} + 3 \right ) \frac{\sigma_{1,F}^2}{b_2}\\
            & +2\left( \frac{(p + 6)^3}{2 b_2 \ell_2} + \left(\frac{3}{b_2} + 1 \right)(p + 3)^{3}  \right) L_{1,F}^2 h_{k}^2 \bigg).
        \end{aligned}
    \end{equation*}
    Thus, we have
    \begin{equation*}
        \begin{aligned}
            \mathbb{E}_k\left[L_{k + 1}\right] - L_k &\leq - \frac{\gamma_k}{2}\| \nabla \Psi(x_k) \|^2 -\frac{\gamma_k}{2} \| \hat{D}_x^{k} \|^2 \\
            &- \bigg( \mu_G \phi_v^k \rho_k - 2  C_v^k \phi_v^k \rho_k^2 - 2 L_{1,G}^2 \gamma_k - \omega_{dx}^k  2C_x^k \gamma_k^2  \bigg) \delta_{v_k}\\
            &- \bigg( \frac{\mu_G}{2}\phi_z^k \rho_k - 4\frac{\omega_{v,1}^2}{\mu_G} \phi_v^k \rho_k   - 2\omega_3 \phi_z^k \rho_k^2 - 2 \omega_{v,1}^2 \gamma_k \bigg) \delta_{z_k}\\
                &+ C_1^2 \gamma_k h_{k}^2 + C_2^k \rho_k h_{k}^2 + C_4 \phi_z^k \rho_k^2 h_{k}^2\\
            &+C_3 \phi_z^k \rho_k^2 +2 \phi_v^k C_6^k(h_k) \rho_k^2  +\omega_{dx}^k C_5^k(h_k) \gamma_k^2 .
        \end{aligned}
    \end{equation*}        
    Let $\bar{C}_k = C_1^2 \gamma_k h_{k}^2 + C_2^k \rho_k h_{k}^2 + C_4 \phi_z^k \rho_k^2 h_{k}^2 +C_3 \phi_z^k \rho_k^2 +2 \phi_v^k C_6^k(h_k) \rho_k^2  +\omega_{dx}^k C_5^k(h_k) \gamma_k^2 $. Observing that $-\frac{\gamma_k}{2} \| \hat{D}_x^{k} \|^2 \leq 0$, we have
        \begin{equation*}
        \begin{aligned}
            \mathbb{E}_k\left[L_{k + 1}\right] - L_k &\leq - \frac{\gamma_k}{2}\| \nabla \Psi(x_k) \|^2\\
            &- \bigg( \mu_G \phi_v^k \rho_k - 2  C_v^k \phi_v^k \rho_k^2 - 2 L_{1,G}^2 \gamma_k - \omega_{dx}^k  2C_x^k \gamma_k^2  \bigg) \delta_{v_k}\\
            &- \bigg( \frac{\mu_G}{2}\phi_z^k \rho_k - 4\frac{\omega_{v,1}^2}{\mu_G} \phi_v^k \rho_k   - 2\omega_3 \phi_z^k \rho_k^2 - 2 \omega_{v,1}^2 \gamma_k \bigg) \delta_{z_k}\\
            &+ \bar{C}_k.
        \end{aligned}
    \end{equation*}    
    Choosing $h_k$ such that $h_{k}^2 \leq  \min \left(\frac{{b_1} \ell_1}{4 L_{2,G}^2 (p +16)^4} \left(1 +\frac{15 (p + 6)^2p}{2 b_1 \ell_1} \right)L_{1,G}^2, \frac{\mu_G^2}{16 \omega_{v,2}}, \frac{1}{16L_{2,G}^2\left((p + 8)^4 + p(d + 12)^3 \right)}\right)$, we have,
    \begin{equation*}
        C_v^k = 8 \left(1 +\frac{15 (p + 6)^2p}{2 b_1 \ell_1} \right) L_{1,G}^2 =:  \bar{C}_v .
    \end{equation*}
    Therefore, we get
    \begin{equation*}
        \begin{aligned}
            \mathbb{E}_k\left[L_{k + 1}\right] - L_k &\leq - \frac{\gamma_k}{2}\| \nabla \Psi(x_k) \|^2\\
            &- \bigg( \mu_G \phi_v^k \rho_k - 2 \bar{C}_v  \phi_v^k \rho_k^2 - 2 L_{1,G}^2 \gamma_k - \omega_{dx}^k  2C_x^k \gamma_k^2  \bigg) \delta_{v_k}\\
            &- \bigg( \frac{\mu_G}{2}\phi_z^k \rho_k - 4\frac{\omega_{v,1}^2}{\mu_G} \phi_v^k \rho_k   - 2\omega_3 \phi_z^k \rho_k^2 - 2 \omega_{v,1}^2 \gamma_k \bigg) \delta_{z_k}\\
            &+ \bar{C}_k.
        \end{aligned}
    \end{equation*}    
    Recalling that
    \begin{equation*}
        \begin{aligned}
            \frac{1}{\omega_k^z} = \frac{\mu_G \rho_k}{\mu_G \rho_k + 4} \quad \text{and} \quad \frac{1}{\omega_k^v} = \frac{\mu_G \rho_k}{\mu_G \rho_k + 2}.
        \end{aligned}
    \end{equation*}
    Let
    \begin{equation*}
        \phi_z^k = \frac{\bar{\phi}_z}{\omega_k^z} \quad \text{and} \quad \phi_v^k = \frac{\bar{\phi}_v}{\omega_k^v}.   
    \end{equation*}
    Notice that for $\rho_k$ decreasing sequence or constant $\phi_z^{k + 1} \leq \phi_z$ and $\phi_v^{k + 1} \leq \phi_z$. Then, we have
    \begin{equation*}
        \begin{aligned}
            \omega_{dx}^k = \frac{L_\Psi}{2} + 2L_*^2 (\bar{\phi}_z + \bar{\phi}_v) =: \omega_{dx}.
        \end{aligned}
    \end{equation*}
    Moreover, for $h_{k}^2 \leq \frac{1}{16L_{2,G}^2\left((p + 8)^4 + p(d + 12)^3 \right)}$ , we get
    \begin{equation*}
        C_x^k=2 \bigg(\frac{1}{b_1 \ell_1} \bigg( \frac{1}{2} + 6(p + 4)(p + 2)p + 36 (p + 2) \min(p, d)+ 30 p(d + 2)d L_{1,G}^2\bigg) + L_{1,G}^2\bigg) =: \bar{C}_x.
    \end{equation*}
    Therefore, we have
    \begin{equation*}
        \begin{aligned}
            \mathbb{E}_k\left[L_{k + 1}\right] - L_k &\leq - \frac{\gamma_k}{2}\| \nabla \Psi(x_k) \|^2\\
            &- \bigg( \frac{\mu_G}{\omega_k^v} \bar{\phi}_v \rho_k - \frac{2 \bar{C}_v}{\omega_k^v}  \bar{\phi}_v \rho_k^2 - 2 L_{1,G}^2 \gamma_k - \omega_{dx}  2\bar{C}_x \gamma_k^2  \bigg) \delta_{v_k}\\
            &- \bigg( \frac{\mu_G}{2\omega_k^z}\bar{\phi}_z \rho_k - 4\frac{\omega_{v,1}^2}{\mu_G \omega_k^v} \bar{\phi}_v \rho_k   - 2 \frac{\omega_3}{\omega_k^z} \bar{\phi}_z \rho_k^2 - 2 \omega_{v,1}^2 \gamma_k \bigg) \delta_{z_k}\\
            &+ \bar{C}_k.
        \end{aligned}
    \end{equation*}    
    Notice that since $\rho_k <  \min(1,\frac{\mu_G}{2 \bar{C}_v})$, we have
    \begin{equation*}
        \begin{aligned}
            -\bigg( \frac{\mu_G}{2\omega_k^v} \bar{\phi}_v \rho_k  - 2 L_{1,G}^2 \gamma_k - \omega_{dx}  2\bar{C}_x \gamma_k^2  \bigg) \leq  -\bigg( \underbrace{\frac{\mu_G^2 }{2(\mu_G + 2)} \bar{\phi}_v}_{A} \rho_k^2  - 2 L_{1,G}^2 \gamma_k - \omega_{dx}  2\bar{C}_x \gamma_k^2  \bigg).
        \end{aligned}
    \end{equation*}
    Thus for $\gamma_k < \frac{ - 2 L_{1,G}^2 + \sqrt{4 L_{1,G}^4 + 8\omega_{dx} \bar{C}_x A\rho_k^2} }{4 \omega_{dx} \bar{C}_x} $, we have
    \begin{equation*}
        -\bigg( A \rho_k^2  - 2 L_{1,G}^2 \gamma_k - \omega_{dx}  2\bar{C}_x \gamma_k^2  \bigg) < 0.
    \end{equation*}
    Moreover, since for $\rho_k \leq \min(1, \frac{\mu_G}{2 \bar{C}_v})$, we have
    \begin{equation*}
        \begin{aligned}
            \frac{\mu_G \rho_k}{4} &> \frac{1}{w_k^z} =\frac{\mu_G \rho_k}{\mu_G \rho_k + 4} > \frac{\mu_G \rho_k}{\mu_G + 4},\\
            \frac{\mu_G \rho_k}{2} &> \frac{1}{w_k^v} =\frac{\mu_G \rho_k}{\mu_G \rho_k + 2} > \frac{\mu_G \rho_k}{\mu_G + 2}.         
        \end{aligned}
    \end{equation*}
    Let $T_k = \frac{\mu_G}{2\omega_k^z}\bar{\phi}_z \rho_k - 4\frac{\omega_{v,1}^2}{\mu_G \omega_k^v} \bar{\phi}_v \rho_k   - 2 \frac{\omega_3}{\omega_k^z} \bar{\phi}_z \rho_k^2$. Thus, we get
    \begin{equation*}
        \begin{aligned}
            T_k &\geq \frac{\mu_G^2}{2(\mu_G + 4)}\bar{\phi}_z \rho_k^2 - 2 \omega_{v,1}^2\bar{\phi}_v \rho_k^2 -\frac{\omega_3 \mu_G}{2}\bar{\phi}_z \rho_k^3\\
            &= \left(\frac{\mu_G^2}{2(\mu_G + 4)}\bar{\phi}_z  - 2 \omega_{v,1}^2\bar{\phi}_v  -\frac{\omega_3 \mu_G}{2}\bar{\phi}_z\rho_k \right) \rho_k^2.
        \end{aligned}
    \end{equation*}
    Setting $\bar{\phi}_v = \frac{\mu_G^2}{8 (\mu_G + 4)\omega_{v,1}^2} \bar{\phi}_z$, we have
    \begin{equation*}
        \begin{aligned}
            T_k &\geq \left(\frac{\mu_G^2}{2(\mu_G + 4)}\bar{\phi}_z  - 2 \omega_{v,1}^2\bar{\phi}_v  -\frac{\omega_3 \mu_G}{2}\bar{\phi}_z\rho_k \right) \rho_k^2\\
            &= \left(\frac{\mu_G^2}{4(\mu_G + 4)}\bar{\phi}_z   -\frac{\omega_3 \mu_G}{2}\bar{\phi}_z\rho_k \right) \rho_k^2.
        \end{aligned}
    \end{equation*}
    Notice that for $\rho_k < \frac{\mu_G}{2 \omega_3 (\mu_G + 4)}$,
    \begin{equation*}
        \begin{aligned}
            T_k &>\left(\frac{\mu_G^2}{4(\mu_G + 4)}\bar{\phi}_z   -\frac{\omega_3 \mu_G}{2}\bar{\phi}_z\rho_k \right) \rho_k^2\\
            &>\left(\frac{\mu_G^2}{4(\mu_G + 4)}\bar{\phi}_z   -\frac{\mu_G^2}{8 (\mu_G + 4)}\bar{\phi}_z  \right) \rho_k^2\\
            &=\underbrace{\frac{\mu_G^2}{8(\mu_G + 4)}\bar{\phi}_z  \rho_k^2}_{=:\bar{T}_k} > 0.
        \end{aligned}
    \end{equation*}
    Thus, for $\gamma_k <  \frac{\mu_G^2}{16 \omega_{v,1}^2(\mu_G + 4)}\bar{\phi}_z   \rho_k^2$ we have
    \begin{equation*}
        T_k - 2 \omega_{v,1}^2 \gamma_k  > \bar{T}_k - 2 \omega_{v,1}^2 \gamma_k  > 0.
    \end{equation*}
    This implies that
    \begin{equation*}
        -\bigg( T_k - 2 \omega_{v,1}^2 \gamma_k \bigg) \leq 0.
    \end{equation*}
    Thus, for%
    \begin{equation*}
        \boxed{
        \begin{aligned}
            \rho_k &< \min\left(1, \frac{\mu_G}{2 \omega_3 (\mu_G + 4)},  \frac{\mu_G}{2 \bar{C}_v} \right),\\
            \gamma_k &< \min\left(\frac{\mu_G^2}{16 \omega_{v,1}^2(\mu_G + 4)}\bar{\phi}_z   \rho_k^2, \frac{ - 2 L_{1,G}^2 + \sqrt{4 L_{1,G}^4 + 8\omega_{dx} \bar{C}_x A\rho_k^2} }{4 \omega_{dx} \bar{C}_x} \right),\\
            h_k^2 & \leq  \min \left(\frac{{b_1} \ell_1}{4 L_{2,G}^2 (p +16)^4} \left(1 +\frac{15 (p + 6)^2p}{2 b_1 \ell_1} \right)L_{1,G}^2, \frac{\mu_G^2}{16 \omega_{v,2}}, \frac{1}{16L_{2,G}^2\left((p + 8)^4 + p(d + 12)^3 \right)}\right),
        \end{aligned}
        }
    \end{equation*}
    we have    
    \begin{equation*}
        \begin{aligned}
            \mathbb{E}_k\left[L_{k + 1}\right] - L_k &\leq - \frac{\gamma_k}{2}\| \nabla \Psi(x_k) \|^2 + \bar{C}_k.
        \end{aligned}
    \end{equation*}  
    Notice that, for $\rho_k \leq \bar{\rho} < \min\left(1, \frac{\mu_G}{2 \omega_3 (\mu_G + 4)},  \frac{\mu_G}{2 \bar{C}_v} \right)$, taking $\bar{\phi}_z = \mathcal{O} (\frac{1}{\bar{\rho}})$, there exists $c_\gamma > 0$ such that $\gamma_k \leq c_\gamma  \rho_k < \min\left(\frac{\mu_G^2}{16 \omega_{v,1}^2(\mu_G + 4)}\bar{\phi}_z   \rho_k^2, \frac{ - 2 L_{1,G}^2 + \sqrt{4 L_{1,G}^4 + 8\omega_{dx} \bar{C}_x A\rho_k^2} }{4 \omega_{dx} \bar{C}_x} \right)$. Taking the full expectation and summing for $k = 1,\cdots,K$, we get
    \begin{equation*}
        \begin{aligned}
            \sum\limits_{k=0}^K\frac{\gamma_k}{2} \mathbb{E} \left[ \left\| \nabla \Psi(x_k) \right\|^2\right] &\leq \mathbb{E}[L_0 - L_{K + 1}] + \sum\limits_{k=0}^K\bar{C}_k\\
            &= \mathbb{E}[\Psi(x_0) + \phi_z^0 \delta_{z_0} + \phi_v^0 \delta_{v_0} \underbrace{- \Psi(x_{K + 1})}_{\leq - \min \Psi} \underbrace{- \phi_z^{K+1} \delta_{z_{K+1}} - \phi_v^{K+1} \delta_{v_{K+1}}}_{\leq 0}] + \sum\limits_{k=0}^K\bar{C}_k\\
            &\leq [\Psi(x_0) - \min \Psi] + \phi_z^0 \delta_{z_0} + \phi_v^0 \delta_{v_0} + \sum\limits_{k=0}^K\bar{C}_k.
        \end{aligned}
    \end{equation*}  
    Dividing both sides by $\sum\limits_{k=0}^K \gamma_k$, we get the claim i.e.
    \begin{equation*}
        \begin{aligned}
            \frac{1}{\sum\limits_{k=0}^K \gamma_k}\sum\limits_{k=0}^K\frac{\gamma_k}{2} \mathbb{E} \left[ \left\| \nabla \Psi(x_k) \right\|^2\right] &\leq \frac{[\Psi(x_0) - \min \Psi]}{\sum\limits_{k=0}^K \gamma_k} + \frac{\phi_z^0 \delta_{z_0}  + \phi_v^0 \delta_{v_0}}{\sum\limits_{k=0}^K \gamma_k} + \frac{\sum\limits_{k=0}^K\bar{C}_k}{\sum\limits_{k=0}^K \gamma_k}.
        \end{aligned}
    \end{equation*}  

    \subsection*{Proof of Corollary \ref{cor:param_choices}}
    By Theorem \ref{thm:conv_rate}, we have
    \begin{equation*}
        \begin{aligned}
            \frac{1}{\sum\limits_{k=0}^K \gamma_k}\sum\limits_{k=0}^K\frac{\gamma_k}{2} \mathbb{E} \left[ \left\| \nabla \Psi(x_k) \right\|^2\right] &\leq \frac{[\Psi(x_0) - \min \Psi]}{\sum\limits_{k=0}^K \gamma_k} + \frac{\phi_z^0 \delta_{z_0}  + \phi_v^0 \delta_{v_0}}{\sum\limits_{k=0}^K \gamma_k} + \frac{\sum\limits_{k=0}^K\bar{C}_k}{ \sum\limits_{k=0}^K \gamma_k}.
        \end{aligned}
    \end{equation*}
    Recalling that
    \begin{equation*}
        \bar{C}_k = C_1^2 \gamma_k h_{k}^2 + C_2^k \rho_k h_{k}^2 + C_4 \phi_z^k \rho_k^2 h_{k}^2 +C_3 \phi_z^k \rho_k^2 +2 \phi_v^k C_6^k(h_k) \rho_k^2  +\omega_{dx}^k C_5^k(h_k) \gamma_k^2.
    \end{equation*}
    where,
    \begin{equation}\label{eqn:cor_constants}
        \begin{aligned}
            C_2^k &:= \left( \frac{L_{1,G}^2}{4 \mu_G}(d + 3)^{3} \phi_z^k + 2  \frac{L_{1,G}^2(d + 3)^{3}}{\mu_G} \phi_v^k  + 4 \frac{\omega_{v,2}L_{0,F}^2}{\mu_G^3} \phi_v^k \right)\\
            C_3 &:= 4  \left(3 +\frac{p + 2}{\ell_1} \right)\frac{\sigma_{1,G}^2}{b_1}\\
            C_4&:= 2 \left( \frac{(p + 6)^3}{2 b_1 \ell_1} + \left( \frac{3}{b_1} + 1 \right) (p + 3)^3 \right)L_{1,G}^2\\
            C_5^k(h_k)&:=\bigg(2C_x^k \frac{L_{0,F}^2}{\mu_G^2}+2\left( \frac{(d + 6)^3}{2 b_2 \ell_2} + \left(\frac{3}{b_2} + 1 \right)(d + 3)^{3}  \right) L_{1,F}^2 h_{k}^2\\
            &+4 \left(1 + \frac{d + 2}{b_2 \ell_2} \right) L_{0,F}^2 + 4\left(\frac{d + 2}{b_2 \ell_2} +\frac{3}{b_2} \right )\sigma_{1,F}^2\bigg)\\
            C_6^k(h_k)&:=\bigg( C_v^k \frac{L_{0,F}^2}{\mu_G^2} + 4 \left(1 + \frac{p + 2}{b_2 \ell_2} \right)L_{0,F}^2 + 4\left(\frac{p + 2}{\ell_2} + 3 \right ) \frac{\sigma_{1,F}^2}{b_2}\\
            & +2\left( \frac{(p + 6)^3}{2 b_2 \ell_2} + \left(\frac{3}{b_2} + 1 \right)(p + 3)^{3}  \right) L_{1,F}^2 h_{k}^2 \bigg).
        \end{aligned}
    \end{equation}    
    Recalling that
        \begin{equation*}
        \begin{aligned}
             \frac{\mu_G \rho_k}{4} >\frac{1}{\omega_k^z} = \frac{\mu_G \rho_k}{\mu_G \rho_k + 4} \quad \text{and} \quad  \frac{\mu_G \rho_k}{2} >\frac{1}{\omega_k^v} = \frac{\mu_G \rho_k}{\mu_G \rho_k + 2}.
        \end{aligned}
    \end{equation*}
    Let
    \begin{equation*}
        \phi_z^k = \frac{\bar{\phi}_z}{\omega_k^z} \quad \text{and} \quad \phi_v^k = \frac{\bar{\phi}_v}{\omega_k^v}.   
    \end{equation*}
    By the choice of $\phi_v^k,\phi_z^k$ we have
    \begin{equation*}
        \begin{aligned}
            \omega_{dx}^k = \frac{L_\Psi}{2} + 2L_*^2 (\bar{\phi}_z + \bar{\phi}_v) =: \omega_{dx}.
        \end{aligned}
    \end{equation*}
    Since $h_{k}^2 \leq \min(\frac{1}{16L_{2,G}^2\left((p + 8)^4 + p(d + 12)^3 \right)}, \frac{{b_1} \ell_1}{4 L_{2,G}^2 (p +16)^4} \left(1 +\frac{15 (p + 6)^2p}{2 b_1 \ell_1} \right)L_{1,G}^2, \frac{\mu_G^2}{16 \omega_{v,2}} )$, we have,
    \begin{equation*}
        C_v^k = 8 \left(1 +\frac{15 (p + 6)^2p}{2 b_1 \ell_1} \right) L_{1,G}^2 =:  \bar{C}_v,
    \end{equation*} 
    and 
    \begin{equation*}
        C_x^k=2 \bigg(\frac{1}{b_1 \ell_1} \bigg( \frac{1}{2} + 6(p + 4)(p + 2)p + 36 (p + 2) \min(p, d)+ 30 p(d + 2)d L_{1,G}^2\bigg) + L_{1,G}^2\bigg) =: \bar{C}_x.
    \end{equation*}
    By the choice of parameters (I), we bound the terms in eq. \eqref{eqn:cor_constants} as
    \begin{equation*}
        \begin{aligned}
            C_1^2 \gamma_k h_{k}^2 &= C_1^2 \gamma h^2, \qquad C_3 \phi_z^k \rho_k^2 < \underbrace{C_3 \frac{\mu_G}{4} \bar{\phi}_z}_{=:C_{3,1}} \rho^3, \qquad C_4 \phi_z^k \rho_k^2 h_{k}^2 < \underbrace{C_4 \bar{\phi}_z \frac{\mu_G}{4}}_{=:C_{4,1}} \rho^3h^2 \\
            C_2^k \rho_k h_k^2 &< \underbrace{\left(\frac{L_{1,G}^2}{16}(d + 3)^3 \bar{\phi}_z  + L_{1,G}^2 (d + 3)^3 \bar{\phi}_v  + \frac{2 \omega_{v,2} L_{0,F}^2}{\mu_G^2} \bar{\phi}_v \right)}_{=:C_2} \rho h^2\\
            C_5(h_k) &< 2\bar{C}_x \frac{L_{0,F}^2}{\mu_G^2} \omega_{dx} \gamma^2 + 2\left( \frac{(d + 6)^3}{2 b_2 \ell_2} + \left(\frac{3}{b_2} + 1 \right)(d + 3)^{3}  \right) L_{1,F}^2 \omega_{dx} h^2 \gamma^2\\
            &+ 4 \left(1 + \frac{d + 2}{b_2 \ell_2} \right) L_{0,F}^2\omega_{dx} \gamma^2 + 4\left(\frac{d + 2}{b_2 \ell_2} +\frac{3}{b_2} \right )\sigma_{1,F}^2\omega_{dx} \gamma^2\\
            &= \underbrace{\bigg(2\bar{C}_x \frac{L_{0,F}^2}{\mu_G^2}+ 4 \left(1 + \frac{d + 2}{b_2 \ell_2} \right) L_{0,F}^2 + 4\left(\frac{d + 2}{b_2 \ell_2} +\frac{3}{b_2} \right )\sigma_{1,F}^2 \bigg) \omega_{dx}}_{C_{5,1}} \gamma^2\\
            &+ \underbrace{ 2\left( \frac{(d + 6)^3}{2 b_2 \ell_2} + \left(\frac{3}{b_2} + 1 \right)(d + 3)^{3}  \right) L_{1,F}^2 \omega_{dx}}_{C_{5,2}}h^2 \gamma^2\\
            C_6(h_k) &< \underbrace{\bigg(\bar{C}_v\frac{L_{0,F}^2 }{\mu_G}  + 4 \left(1 + \frac{p + 2}{b_2 \ell_2} \right)L_{0,F}^2  \mu_G + 4\left(\frac{p + 2}{\ell_2} + 3 \right ) \frac{\mu_G \sigma_{1,F}^2}{b_2}\bigg)}_{C_{6,1}} \rho^3\\
            &+ \underbrace{2\mu_G\left( \frac{(p + 6)^3}{2 b_2 \ell_2} + \left(\frac{3}{b_2} + 1 \right)(p + 3)^{3}  \right) L_{1,F}^2}_{C_{6,2}} h^2 \rho^3
        \end{aligned}
    \end{equation*}
    Using these bound, we get the claim
    \begin{equation*}
        \begin{aligned}
            \frac{1}{K + 1}\sum\limits_{k=0}^K \mathbb{E} \left[ \left\| \nabla \Psi(x_k) \right\|^2\right] &\leq 
            \frac{[\Psi(x_0) - \min \Psi] + \phi_z^0 \delta_{z_0}  + \phi_v^0 \delta_{v_0}}{\gamma (K + 1)} + \left(C_1^2 + C_2\frac{\rho}{\gamma} \right)h^2\\
            &+ \left(C_{3,1} + C_{4,1}h^2 + C_{6,1} + C_{6,2} h^2 \right) \frac{\rho^3}{\gamma} + \left(C_{5,1} + C_{5,2} h^2 \right)\gamma.%
        \end{aligned}
    \end{equation*}
    By the choice of parameters (II), we have $\gamma_k = \gamma (k + 1)^{-\alpha_1}$, $\rho_k = \rho (k + 1)^{-\alpha_2}$ and $h_k = h (k + 1)^{-\alpha_3}$. Thus, we bound the terms in eq. \eqref{eqn:cor_constants} as
    \begin{equation*}
        \begin{aligned}
            C_1^2 \gamma_k h_k^2 &= C_1^2 \gamma h^2 (k + 1)^{-(\alpha_1 + 2 \alpha_3)}, \quad C_3 \phi_z^k \rho_k^2 \leq \frac{C_3 \mu_G}{4}\bar{\phi}_z \rho^3 (k + 1)^{-3\alpha_2}\\
            C_4 \phi_z^k \rho_k^2 h_k^2 &\leq \frac{C_4 \mu_G}{4} \bar{\phi}_z \rho^3 h^2 (k + 1)^{-(3\alpha_1 + 2\alpha_3)}, \quad C_2^k \rho_k h_k^2 \leq C_2\rho h^2 (k + 1)^{-(\alpha_2 + 2\alpha_3)}\\
            C_5(h_k)\omega_{dx}^k \gamma_k^2 &\leq C_{5,1}\gamma^2 (k+1)^{-2\alpha_1} + C_{5,2} h^2 \gamma^2 (k + 1)^{-2(\alpha_1 + \alpha_3)}\\
            2 \phi_v^k C_6(h_k) \rho_k^2&\leq C_{6,1} \rho^3 (k + 1)^{-3\alpha_2} + C_{6,2}h^2 \rho^3 (k + 1)^{-(3\alpha_2 + 2 \alpha_3)}.
        \end{aligned}
    \end{equation*}
    Since $\alpha_1, \alpha_2 \in (1/2, 1)$ and $\alpha_3 > 1$, we have
    \begin{equation*}
        \sum\limits_{k=0}^K \gamma_k \leq \frac{(K +1)^{1 - \alpha_1} - 1}{1 - \alpha_1} \gamma \quad \text{and} \quad \sum\limits_{k=0}^K \rho_k \leq \frac{(K +1)^{1 - \alpha_2} - 1}{1 - \alpha_2} \rho.
    \end{equation*}
    Moreover, we get
    \begin{equation*}
        \begin{aligned}
            \sum\limits_{k=0}^K C_1^2 \gamma_k h_k^2 &= C_1^2 \gamma^2 h^2 \sum\limits_{k=0}^K (k+1)^{-(\alpha_1 + 2 \alpha_3)} \leq C_1^2\gamma^2h^2 \frac{\alpha_1 + 2 \alpha_3}{\alpha_1 + 2 \alpha_3 - 1},\\
            \sum\limits_{k=0}^K C_2^k \rho_k h_k^2 &\leq C_2\rho h^2 \sum\limits_{k=0}^K(k + 1)^{-(\alpha_2 + 2\alpha_3)} \leq C_2\rho h^2 \frac{\alpha_2 + 2 \alpha_3}{\alpha_2 + 2 \alpha_3 -1},\\
            \sum\limits_{k=0}^K C_3 \phi_z^k \rho_k^2 &\leq \frac{C_3 \mu_G}{4}\bar{\phi}_z \sum\limits_{k=0}^K \rho_k^3 \leq \frac{C_3 \mu_G}{4}\bar{\phi}_z \rho^3 \frac{3\alpha_2}{3\alpha_2 - 1},\\
            \sum\limits_{k=0}^KC_4 \phi_z^k \rho_k^2 h_k^2 &\leq \frac{C_4 \mu_G}{4} \bar{\phi}_z \rho^3 h^2 \sum\limits_{k=0}^K (k + 1)^{-(3\alpha_1 + 2\alpha_3)} \leq \frac{C_4 \mu_G}{4} \bar{\phi}_z \rho^3 h^2  \frac{3 \alpha_1 + 2\alpha_3}{ 3 \alpha_1 + 2\alpha_3 - 1},\\
             \sum\limits_{k=0}^K C_5(h_k)\omega_{dx}^k \gamma_k^2 &\leq C_{5,1}\gamma^2 \sum\limits_{k=0}^K (k+1)^{-2\alpha_1} + C_{5,2} h^2 \gamma^2 \sum\limits_{k=0}^K(k + 1)^{-2(\alpha_1 + \alpha_3)}\\
             &\leq C_{5,1}\gamma^2 \frac{2\alpha_1}{2\alpha_1 - 1} + C_{5,2} h^2 \gamma^2 \frac{2(\alpha_1 + \alpha_3)}{2(\alpha_1 + \alpha_3) - 1}\\
             2 \sum\limits_{k=0}^k \phi_v^k C_6(h_k) \rho_k^2&\leq C_{6,1} \rho^3 \sum\limits_{k=0}^k (k + 1)^{-3\alpha_2} + C_{6,2}h^2 \rho^3 \sum\limits_{k=0}^k (k + 1)^{-(3\alpha_2 + 2 \alpha_3)}\\
             &\leq C_{6,1} \rho^3 \frac{3\alpha_2}{3\alpha_2 - 1}  + C_{6,2}h^2 \rho^3 \frac{3 \alpha_2 + 2 \alpha_3}{3 \alpha_2 + 2 \alpha_3 - 1}.
        \end{aligned}
    \end{equation*}
    Therefore, using these bounds, we get the claim
    \begin{equation*}
        \begin{aligned}
        \frac{1}{\sum\limits_{k=0}^K \gamma_k}\sum\limits_{k=0}^K\gamma_k \mathbb{E} \left[ \left\| \nabla \Psi(x_k) \right\|^2\right] &\leq  \frac{[\Psi(x_0) - \min \Psi] + \phi_z^0 \delta_{z_0}  + \phi_v^0 \delta_{v_0}}{\gamma (K + 1)^{1 - \alpha_1}} + \mathcal{O} \left(\frac{1}{(K + 1)^{1 - \alpha_1}} \right).%
        \end{aligned}
    \end{equation*}
    Finally, to prove the last claim, by choosing $\gamma_k = \gamma$, $\rho_k = \rho$ and $h_k = h$ for every $k \in \mathbb{N}$, we have
    \begin{equation*}
        \begin{aligned}
            \frac{1}{K + 1}\sum\limits_{k=0}^K \mathbb{E} \left[ \left\| \nabla \Psi(x_k) \right\|^2\right] &\leq  \frac{[\Psi(x_0) - \min \Psi] + \phi_z^0 \delta_{z_0}  + \phi_v^0 \delta_{v_0}}{\gamma (K + 1)} + \left(C_1^2 + C_2\frac{\rho}{\gamma} \right)h^2\\
            &+ \left(C_3 + C_4h^2  \right)\mu_G \bar{\phi}_z \frac{\rho^3}{\gamma} + \left(C_{5,1} + C_{5,2} h^2 \right)\gamma + \left(C_{6,1} + C_{6,2} h^2\right) \frac{\rho^2}{\gamma}.
        \end{aligned}
    \end{equation*}
    Rearranging the terms and by $\phi_v^k = \frac{\bar{\phi}_v}{\omega_k^v}$, $\phi_z^k = \frac{\bar{\phi}_z}{\omega_k^z}$ and $\bar{\phi}_v = \frac{\mu_G^2}{8(\mu_G + 4) \omega_{v,1}^2} \bar{\phi}_z$ we have
    \begin{equation*}
        \begin{aligned}
            \frac{1}{K + 1}\sum\limits_{k=0}^K \mathbb{E} \left[ \left\| \nabla \Psi(x_k) \right\|^2\right] &\leq \frac{1}{\gamma(K + 1)} \left(  [\Psi(x_0) - \min \Psi] + \frac{\bar{\phi}_z}{\omega_0^z} \delta_{z_0}  + \frac{\mu_G^2}{8(\mu_G + 4) \omega_{v,1}^2 \omega_0^v} \bar{\phi}_z\delta_{v_0} \right)\\
            &+ \left( C_1^2 + C_2 \frac{\rho}{\gamma} + C_{5,2} \omega_{dx}\gamma + C_4\mu_G \bar{\phi}_z \frac{\rho^3}{\gamma} + C_{6,2} \frac{\rho^2}{\gamma} \right)h^2 \\
            &+  C_3 \mu_G \bar{\phi}_z \frac{\rho^3}{\gamma} + C_{5,1} \gamma + C_{6,1}\frac{\rho^2}{\gamma} .
        \end{aligned}
    \end{equation*}
    Due to the conditions on $\rho_k$, we have
        \begin{equation*}
        \begin{aligned}
             \frac{\mu_G }{4} >\frac{1}{\omega_k^z} = \frac{\mu_G \rho_k}{\mu_G \rho_k + 4} \quad \text{and} \quad  \frac{\mu_G }{2} >\frac{1}{\omega_k^v} = \frac{\mu_G \rho_k}{\mu_G \rho_k + 2}.
        \end{aligned}
    \end{equation*}
    Then, we get
        \begin{equation*}
        \begin{aligned}
            \frac{1}{K + 1}\sum\limits_{k=0}^K \mathbb{E} \left[ \left\| \nabla \Psi(x_k) \right\|^2\right] &\leq \frac{1}{\gamma(K + 1)} \underbrace{\left(  [\Psi(x_0) - \min \Psi] + \frac{\bar{\phi}_z \mu_G^2}{4} \delta_{z_0}  + \frac{\mu_G^3}{16(\mu_G + 4) \omega_{v,1}^2 } \bar{\phi}_z\delta_{v_0} \right)}_{=:\Delta_1}\\
            &+ \left( C_1^2 + C_2 \frac{\rho}{\gamma} + C_{5,2} \omega_{dx}\gamma + C_4\mu_G \bar{\phi}_z \frac{\rho^3}{\gamma} + C_{6,2} \frac{\rho^2}{\gamma} \right)h^2\\
            &+  C_3 \mu_G \bar{\phi}_z \frac{\rho^3}{\gamma} + C_{5,1} \gamma + C_{6,1}\frac{\rho^2}{\gamma} .
        \end{aligned}
    \end{equation*}
    Let $\varepsilon \in (0,1)$. Consider the following inequality,
    \begin{equation*}
        \begin{aligned}
            \frac{\Delta_1}{\gamma(K + 1)} + C(\gamma, \rho)h^2 + \left( C_3 \mu_G \bar{\phi}_z \frac{\rho^3}{\gamma} + C_{5,1} \gamma + C_{6,1}\frac{\rho^2}{\gamma} \right) \leq \varepsilon.
        \end{aligned}
    \end{equation*}
    Since the term with $C(\gamma, \rho)$ can be controlled with $h$, we choose $\gamma$ by minimizing $\frac{\Delta_1}{\gamma(K + 1)}  \left( C_3 \mu_G \bar{\phi}_z \frac{\rho^3}{\gamma} + C_{5,1} \gamma + C_{6,1}\frac{\rho^2}{\gamma} \right)$ i.e.
    \begin{equation*}
        \gamma = \sqrt{\frac{1}{C_{5,1}}\left(\frac{\Delta_1}{ K + 1} +C_3 \mu_G \bar{\phi}_z\rho^3 + C_{6,1} \rho^2 \right)}.
    \end{equation*}
    Thus, we get
     \begin{equation*}
        \begin{aligned}
            2\sqrt{C_{5,1} \left(\frac{\Delta_1}{(K + 1)} + C_3 \mu_G \bar{\phi}_z \rho^3 + C_{6,1} \rho^2 \right)} + C(\gamma, \rho)h^2 \leq \varepsilon.
        \end{aligned}
    \end{equation*}
    Let $h = \sqrt{\frac{\varepsilon}{2 C(\gamma, \rho)}}$, we get
     \begin{equation*}
        \begin{aligned}
            2\sqrt{C_{5,1} \left(\frac{\Delta_1}{(K + 1)} + C_3 \mu_G \bar{\phi}_z \rho^3 + C_{6,1} \rho^2 \right)} \leq \frac{\varepsilon}{2}.
        \end{aligned}
    \end{equation*}
    Multplying both sides by $1/2$, we have
     \begin{equation*}
        \begin{aligned}
            \frac{\Delta_1}{(K + 1)} \leq \frac{\varepsilon^2}{4C_{5,1}} - (C_3 \mu_G \bar{\phi}_z \rho + C_{6,1}) \rho^2.
        \end{aligned}
    \end{equation*}
    Choosing $\rho = \frac{\varepsilon}{\sqrt{16 C_{5,1} C_{6,1}}}$ and $\bar{\phi}_z = \frac{C_{6,1}\sqrt{16 C_{5,1} C_{6,1}}}{C_3 \mu_G \varepsilon} = \frac{C_{6,1}}{C_3 \mu_G \rho}$, we get
     \begin{equation*}
        \begin{aligned}
            \frac{\Delta_1}{(K + 1)} \leq \frac{\varepsilon^2}{8C_{5,1}}.%
        \end{aligned}
    \end{equation*}
    Therefore, choosing 
    \begin{equation*}
        (K + 1) \geq \frac{8 C_{5,1}}{\Delta_1} \varepsilon^{-2},
    \end{equation*}
    we get $\frac{1}{K + 1}\sum\limits_{k=0}^K \mathbb{E} \left[ \left\| \nabla \Psi(x_k) \right\|^2\right] \leq \varepsilon$. 
    Choosing $b_1\ell_1 = \lfloor \mathcal{O}(p(d + p)^2/c_1) \rfloor$ and $b_2 \ell_2 = \lfloor\mathcal{O}( p( p +d)^2/c_2) \rfloor$ for $c_1,c_2 \in \mathbb{R}_+$, the dependence on the dimensions $p,d$ inside the constants (and in the condition on the stepsizes $\rho,\gamma$ simplifies, yielding the claim (III) i.e. the complexity is
    \begin{equation*}
        \mathcal{O}\left( p(d + p)^2 \varepsilon^{-2} \right).
    \end{equation*}

    \subsection{Proof of Theorem \ref{thm:hfzoba_conv_rate}}
    Let $\delta_{z_k} = \|z_k - z^*(x_k)\|^2$ and $\delta_{v_k} = \| v_k - v^*(x_k)\|^2$. For $\phi_z^k, \phi_v^k > 0$ s.t. $\phi_v^{k+1} \leq \phi_v^k$ and $\phi_z^{k+1} \leq \phi_z^k$, define
    \begin{equation*}
        L_k = \Psi(x_k) + \phi_z^k \delta_{z_k} + \phi_v^k \delta_{v_k}.
    \end{equation*}
    Thus, we have
    \begin{equation*}
        \begin{aligned}
            L_{k + 1} - L_k  &\leq \Psi(x_{k + 1}) - \Psi(x_{k})+ \phi_z^{k} \left(\delta_{z_{k + 1}} - \delta_{z_{k}} \right) + \phi_v^{k} \left( \delta_{v_{k+1}} -  \delta_{v_k} \right).
        \end{aligned}
    \end{equation*}
    Taking the conditional expectation,
    \begin{equation*}
        \begin{aligned}
            \mathbb{E}_k[L_{k + 1}] - L_k  &\leq \mathbb{E}_k[\Psi(x_{k + 1})] - \Psi(x_{k})+ \phi_z^{k} \left(\mathbb{E}_k[\delta_{z_{k + 1}}] - \delta_{z_{k}} \right) + \phi_v^{k} \left( \mathbb{E}_k[\delta_{v_{k+1}}] -  \delta_{v_k} \right).
        \end{aligned}
    \end{equation*}
    Let
    \begin{equation*}
        \begin{aligned}
            C_{1,\Psi} &= L_{1,F} + \frac{L_{0,F} L_{1,G}}{\mu_G},\\
            C_{2,\Psi} &= \frac{2L_{2,G}L_{0,F}}{3\mu_G}\left((d + 5)^{5/2} +  (d +3)^{3/2} +  (d + 1) (p +3)^{3/2} \right) + \frac{L_{1,F} d\sqrt{p}}{2}\\
            C_{3,\Psi} &= \frac{L_{2,G}^2L_{0,F}^2}{2\mu_G^2} \quad \text{and}\quad C_{4,\Psi} =4L_{1,G}^2 + \frac{\mu_G^2}{16}.
        \end{aligned}
    \end{equation*}
    By Lemma \ref{lem:hfzoba_fun_dec},
    \begin{equation*}
        \begin{aligned}
            \mathbb{E}_k[L_{k + 1}] - L_k  &\leq - \frac{\gamma_k}{2} \|\nabla \Psi(x_k)\|^2 + \frac{L_\Psi}{2} \gamma_k^2 \mathbb{E}_k[\| \hat{D}_x^k \|^2]\\
            &+4 C_{1,\Psi}^2 \gamma_k \delta_{z_k}  + C_{4,\Psi} \gamma_k \delta_{v_k}\\
            &+ \phi_z^{k} \left(\mathbb{E}_k[\delta_{z_{k + 1}}] - \delta_{z_{k}} \right) + \phi_v^{k} \left( \mathbb{E}_k[\delta_{v_{k+1}}] -  \delta_{v_k} \right)\\
            &+ 2 C_{2,\Psi}^2 \gamma_kh_k^2 + C_{3,\Psi} \hat{h}_k^2\gamma_k.
        \end{aligned}
    \end{equation*}
    Notice that the bounds on $\delta_{z_k}$ of Lemma \ref{lem:smooth_bound_zv} holds also for Algorithm \ref{alg:hfzoba}. 
    Let 
    \begin{equation*}
        \begin{aligned}
            C_1 &= 2 \frac{L_{1,G}^2(d + 3)^2}{\mu_G} + 4\frac{\omega_{v,2} L_{0,F}^2}{\mu_G^3} \quad \text{and} \quad C_2 =\frac{8}{\mu_G} \frac{L_{2,G}^2}{4} \frac{L_{0,F}^2}{\mu_G^2},\\
            C_3 &= \frac{L_{1,G}^2}{4 \mu_G}(d + 3)^{3}.
        \end{aligned}
    \end{equation*}
    Thus, by Lemma \ref{lem:smooth_bound_zv} and Lemma \ref{lem:hfzoba_vbound}, we get
    \begin{equation*}
        \begin{aligned}
            \mathbb{E}_k[L_{k + 1}] - L_k  &\leq - \frac{\gamma_k}{2} \|\nabla \Psi(x_k)\|^2\\
            &- \left( \frac{\mu_G}{2} \phi_v^k \rho_k - C_{4,\Psi} \gamma_k\right) \delta_{v_k}\\
            &- \left( \frac{\mu_G}{2} \phi_z^k \rho_k - 4\frac{\omega_{v,1}^2}{\mu_G} \phi_v^k \rho_k  -4 C_{1,\Psi}^2 \gamma_k \right)  \delta_{z_k}\\
            &+ 2 \phi_z^k\rho_k^2 \mathbb{E}_k\left[\left\| \hat{D}_z^k \right\|^2\right] +2 \phi_v^k\rho_k^2 \mathbb{E}_k[\| \hat{D}_v^k\|^2] \\
            &+\left(2 \omega_k^v L_*^2 \phi_v^k +\frac{L_\Psi}{2} + 2 \omega_k^zL_*^2 \phi_z^k \right) \gamma_k^2 \mathbb{E}_k[\| \hat{D}_x^k\|^2]\\
            &+(C_1 h_k^2+  C_2 \hat{h}_k^2) \phi_v^k \rho_k + 2 C_{2,\Psi}^2 \gamma_kh_k^2 + C_{3,\Psi} \hat{h}_k^2\gamma_k\\
            &+ C_3 \phi_z^k \rho_k h_{k}^2.
        \end{aligned}
    \end{equation*}
    Observing that by  Lemma \ref{lem:hfzoba_search_directions}, we have
    \begin{equation*}
        \begin{aligned}
            \mathbb{E}_k\left[\left\| \hat{D}_z^k \right\|^2\right] &\leq \underbrace{2 \left(2 + \frac{p + 2}{b_1 \ell_1} \right)}_{=:C_{A,1}} \left\|\nabla_z G(z_k, x_k) \right\|^2\\
        &+ \underbrace{\left( \frac{(p + 6)^3}{2 b_1 \ell_1} + \left( \frac{3}{b_1} + 1 \right) (p + 3)^3 \right)L_{1,G}^2}_{=:C_{A,2}} h_{k}^2\\
        &+ \underbrace{2 \left(3 +\frac{p + 2}{\ell_1} \right)\frac{\sigma_{1,G}^2}{b_1}}_{C_{A,3}}
        \end{aligned}
    \end{equation*}
    Since $\nabla_z G(z^*(x_k), x_k) = 0$, by $L_{1,G}$-smoothness we have
    \begin{equation*}
        \begin{aligned}
            \mathbb{E}_k\left[\left\| \hat{D}_z^k \right\|^2\right] &\leq C_{A,1} \left\|\nabla_z G(z_k, x_k) -\nabla_z G(z^*(x_k), x_k) \right\|^2 + C_{A,2} h_{k}^2+ C_{A,3}\\
            &\leq C_{A,1} L_{1,G}^2 \delta_{z_k} + C_{A,2} h_{k}^2+ C_{A,3}.
        \end{aligned}
    \end{equation*}
    Thus, we have
    \begin{equation*}
        \begin{aligned}
            \mathbb{E}_k[L_{k + 1}] - L_k  &\leq - \frac{\gamma_k}{2} \|\nabla \Psi(x_k)\|^2\\
            &- \left( \frac{\mu_G}{2} \phi_v^k \rho_k - C_{4,\Psi} \gamma_k\right) \delta_{v_k}\\
            &- \left( \frac{\mu_G}{2} \phi_z^k \rho_k - 4\frac{\omega_{v,1}^2}{\mu_G} \phi_v^k \rho_k  -4 C_{1,\Psi}^2 \gamma_k - 2 \phi_z^k\rho_k^2 C_{A,1} L_{1,G}^2\right)  \delta_{z_k}\\
            &+2 \phi_v^k\rho_k^2 \mathbb{E}_k[\| \hat{D}_v^k\|^2] \\
            &+\left(2 \omega_k^v L_*^2 \phi_v^k +\frac{L_\Psi}{2} + 2 \omega_k^zL_*^2 \phi_z^k \right) \gamma_k^2 \mathbb{E}_k[\| \hat{D}_x^k\|^2]\\
            &+(C_1 h_k^2+  C_2 \hat{h}_k^2) \phi_v^k \rho_k + 2 C_{2,\Psi}^2 \gamma_kh_k^2 + C_{3,\Psi} \hat{h}_k^2\gamma_k\\
            &+ C_3 \phi_z^k \rho_k h_{k}^2 + 2C_{A,2}\phi_z^k\rho_k^2 h_{k}^2\\
            &+ 2\phi_z^k\rho_k^2C_{A,3}.
        \end{aligned}
    \end{equation*}
    Let $C_{B,1} = 4\left(1 + \frac{3(p + 2)}{b_1 \ell_1} \right)$, $C_{B,2} = 2 \left(1 + \frac{p + 2}{b_2 \ell_2} \right)$, $C_{B,3} = \left( \frac{(p + 6)^3}{2 b_2 \ell_2} + \left(\frac{3}{b_2} + 1 \right)(p + 3)^{3}  \right) L_{1,F}^2$ and $C_{B,4} = 2\left(\frac{p + 2}{\ell_2} + 3 \right ) \frac{\sigma_{1,F}^2}{b_2}$. Let $C_{C,1} = 4\left(1 + \frac{3(d + 2)}{b_1 \ell_1} \right)$, $C_{C,2} = 2 \left(1 + \frac{d + 2}{b_2 \ell_2} \right)$, $C_{C,3} = \left( \frac{(d + 6)^3}{2 b_2 \ell_2} + \left(\frac{3}{b_2} + 1 \right)(d + 3)^{3}  \right) L_{1,F}^2$ and $C_{C,4} = 2\left(\frac{d + 2}{b_2 \ell_2} +\frac{3}{b_2} \right )\sigma_{1,F}^2$. 
    Let
    \begin{equation*}
        S_k = 2 \omega_k^v L_*^2 \phi_v^k +\frac{L_\Psi}{2} + 2 \omega_k^zL_*^2 \phi_z^k. 
    \end{equation*}
    By Lemma \ref{lem:hfzoba_search_directions}, we have
    \begin{equation*}
        \begin{aligned}
            \mathbb{E}_k[L_{k + 1}] - L_k  &\leq - \frac{\gamma_k}{2} \|\nabla \Psi(x_k)\|^2\\
            &- \left( \frac{\mu_G}{2} \phi_v^k \rho_k - C_{4,\Psi} \gamma_k -4 C_{B,1}L_{1,G}^2 \phi_v^k \rho_k^2 - 2 C_{C,1}L_{1,G}^2 S_k\gamma_k^2\right) \delta_{v_k}\\
            &- \left( \frac{\mu_G}{2} \phi_z^k \rho_k - 4\frac{\omega_{v,1}^2}{\mu_G} \phi_v^k \rho_k  -4 C_{1,\Psi}^2 \gamma_k - 2 \phi_z^k\rho_k^2 C_{A,1} L_{1,G}^2\right)  \delta_{z_k}\\
            &+(C_1 h_k^2+  C_2 \hat{h}_k^2) \phi_v^k \rho_k + 2 C_{2,\Psi}^2 \gamma_kh_k^2 + C_{3,\Psi} \hat{h}_k^2\gamma_k\\
            &+2  C_{B,1}\frac{L_{1,G}^2 L_{0,F}^2}{\mu_G^2} \phi_v^k \rho_k^2\\
            &+ C_3 \phi_z^k \rho_k h_{k}^2 + 2C_{A,2}\phi_z^k\rho_k^2 h_{k}^2\\
            &+ 2C_{C,1}\frac{L_{1,G}^2 L_{0,F}^2}{\mu_G^2}S_k\gamma_k^2\\
            &+ 2\phi_z^k\rho_k^2C_{A,3}\\
            &+ \underbrace{\bigg(\frac{12L_{1,G}^2 (p + 6)^3}{b_1 \ell_1} \phi_v^k \rho_k^2 +\frac{6L_{1,G}^2 (d + 6)^3}{b_1 \ell_1}S_k\gamma_k^2\bigg)}_{C_{L,k}}\frac{h_k^2}{\bar{h}_k^2}\\
            &+ ( 2 C_{C,2} L_{0,F}^2 +2 C_{C,3} h_{k}^2 + 2C_{C,4})S_k\gamma_k^2\\
            &+2 \phi_v^k \rho_k^2 ( 2 C_{B,2} L_{0,F}^2 +2 C_{B,3} h_{k}^2 + 2C_{B,4}).\\ %
        \end{aligned}
    \end{equation*}
    Recall that
    \begin{equation*}
        \begin{aligned}
            \bar{h}_k = \begin{cases}
                \frac{\hat{h}_k}{\|v_k\|} & \text{if }\quad\|v_k\| \neq 0\\
                \hat{h}_k
            \end{cases}.
        \end{aligned}
    \end{equation*}
    Therefore, for every $k \in \mathbb{N}$, for which $\|v_k\| \neq 0$, we have
    \begin{equation*}
        \begin{aligned}
            C_{L,k} \frac{h_k^2}{\bar{h}_k^2} &= C_{L,k}\frac{h_k^2}{\hat{h}_k^2} \|v_k\|^2\\
            &\leq 2C_{L,k}\frac{h_k^2}{\hat{h}_k^2} \|v_k - v^*(x_k)\|^2 + 2C_{L,k}\frac{h_k^2}{\hat{h}_k^2} \|v^*(x_k)\|^2\\
            &\leq 2C_{L,k}\frac{h_k^2}{\hat{h}_k^2} \|v_k - v^*(x_k)\|^2 + \frac{2L_{0,F}^2}{\mu_G^2}C_{L,k}\frac{h_k^2}{\hat{h}_k^2},
        \end{aligned}
    \end{equation*}
    where the last inequality follows by Lemma \ref{lem:bound_norm_v}. For every $k \in \mathbb{N}$ for which $\|v_k \| = 0$, we have instead 
    \begin{equation*}
        C_{L,k} \frac{h_k^2}{\bar{h}_k^2} = C_{L,k} \frac{h_k^2}{\hat{h}_k^2}.
    \end{equation*}
    Let $\theta_k = \begin{cases}
        \frac{2L_{0,F}^2}{\mu_G^2} & \text{if } \|v_k\| \neq 0\\
        1
    \end{cases} < 1 + \frac{2 L_{0,F}^2}{\mu_G} =:\theta$.  Thus,  for every $k \in \mathbb{N}$ we have
    \begin{equation*}
        \begin{aligned}
            C_{L,k} \frac{h_k^2}{\bar{h}_k^2} &\leq 2C_{L,k}\frac{h_k^2}{\hat{h}_k^2} \|v_k - v^*(x_k)\|^2 1_{\|v_k \| \neq 0} + \theta_k C_{L,k}\frac{h_k^2}{\hat{h}_k^2}\\
            &\leq 2C_{L,k}\frac{h_k^2}{\hat{h}_k^2} \|v_k - v^*(x_k)\|^2 + \theta C_{L,k}\frac{h_k^2}{\hat{h}_k^2},
        \end{aligned}
    \end{equation*}
    where the last inequality follows from the fact that $2C_{L,k}\frac{h_k^2}{\hat{h}_k^2} \|v_k - v^*(x_k)\|^2 \geq 0$.
    Therefore, we get
    \begin{equation*}
        \begin{aligned}
            \mathbb{E}_k[L_{k + 1}] - L_k  &\leq - \frac{\gamma_k}{2} \|\nabla \Psi(x_k)\|^2\\
            &- \left( \frac{\mu_G}{2} \phi_v^k \rho_k - C_{4,\Psi} \gamma_k -4 C_{B,1}L_{1,G}^2 \phi_v^k \rho_k^2 - 2 C_{C,1}L_{1,G}^2 S_k\gamma_k^2 - 2C_{L,k} \frac{h_k^2}{\hat{h}_k^2}\right) \delta_{v_k}\\
            &- \left( \frac{\mu_G}{2} \phi_z^k \rho_k - 4\frac{\omega_{v,1}^2}{\mu_G} \phi_v^k \rho_k  -4 C_{1,\Psi}^2 \gamma_k - 2 \phi_z^k\rho_k^2 C_{A,1} L_{1,G}^2\right)  \delta_{z_k}\\
            &+(C_1 h_k^2+  C_2 \hat{h}_k^2) \phi_v^k \rho_k + 2 C_{2,\Psi}^2 \gamma_kh_k^2 + C_{3,\Psi} \hat{h}_k^2\gamma_k\\
            &+2  C_{B,1}\frac{L_{1,G}^2 L_{0,F}^2}{\mu_G^2} \phi_v^k \rho_k^2\\
            &+ C_3 \phi_z^k \rho_k h_{k}^2 + 2C_{A,2}\phi_z^k\rho_k^2 h_{k}^2\\
            &+ 2C_{C,1}\frac{L_{1,G}^2 L_{0,F}^2}{\mu_G^2}S_k\gamma_k^2\\
            &+ \theta C_{L,k} \frac{h_k^2}{\hat{h}_k^2}\\
            &+ ( 2 C_{C,2} L_{0,F}^2 +2 C_{C,3} h_{k}^2 + 2C_{C,4})S_k\gamma_k^2\\
            &+2 \phi_v^k \rho_k^2 ( 2 C_{B,2} L_{0,F}^2 +2 C_{B,3} h_{k}^2 + 2C_{B,4}).\\ %
        \end{aligned}
    \end{equation*}
    Let 
    \begin{equation}\label{eqn:thm2_hat_c}
        \begin{aligned}
            \hat{C}_k &= (C_1 h_k^2+  C_2 \hat{h}_k^2) \phi_v^k \rho_k + 2 C_{2,\Psi}^2 \gamma_kh_k^2 + C_{3,\Psi} \hat{h}_k^2\gamma_k\\
            &+2  C_{B,1}\frac{L_{1,G}^2 L_{0,F}^2}{\mu_G^2} \phi_v^k \rho_k^2+ C_3 \phi_z^k \rho_k h_{k}^2 + 2C_{A,2}\phi_z^k\rho_k^2 h_{k}^2+ 2C_{C,1}\frac{L_{1,G}^2 L_{0,F}^2}{\mu_G^2}S_k\gamma_k^2\\
            &+ 2\phi_z^k\rho_k^2C_{A,3} +  \frac{12L_{1,G}^2 }{b_1 \ell_1}  \phi_v^k \rho_k^2 +\frac{6L_{1,G}^2 }{b_1 \ell_1}S_k\gamma_k^2 + \theta C_{L,k} \frac{h_k^2}{\hat{h}_k^2}\\
            &+ ( 2 C_{C,2} L_{0,F}^2 +2 C_{C,3} h_{k}^2 + 2C_{C,4})S_k\gamma_k^2+2 \phi_v^k \rho_k^2 ( 2 C_{B,2} L_{0,F}^2 +2 C_{B,3} h_{k}^2 + 2C_{B,4}).
        \end{aligned}
    \end{equation}
    Then,
    \begin{equation*}
        \begin{aligned}
            \mathbb{E}_k[L_{k + 1}] - L_k  &\leq - \frac{\gamma_k}{2} \|\nabla \Psi(x_k)\|^2\\
            &- \left( \frac{\mu_G}{2} \phi_v^k \rho_k - C_{4,\Psi} \gamma_k -4 C_{B,1}L_{1,G}^2 \phi_v^k \rho_k^2 - 2 C_{C,1}L_{1,G}^2 S_k\gamma_k^2 - 2C_{L,k} \frac{h_k^2}{\hat{h}_k^2}\right) \delta_{v_k}\\
            &- \left( \frac{\mu_G}{2} \phi_z^k \rho_k - 4\frac{\omega_{v,1}^2}{\mu_G} \phi_v^k \rho_k  -4 C_{1,\Psi}^2 \gamma_k - 2 \phi_z^k\rho_k^2 C_{A,1} L_{1,G}^2\right)  \delta_{z_k}\\
            &+\hat{C}_k.\\ %
        \end{aligned}
    \end{equation*}
    Since $\rho_k \leq \min(1, \frac1{16 C_{B,1}L_{1,G}^2}, \frac{1}{4 C_{A,1}L_{1,G}^2})$, we have
    \begin{equation*}
        \begin{aligned}
            \mathbb{E}_k[L_{k + 1}] - L_k  &\leq - \frac{\gamma_k}{2} \|\nabla \Psi(x_k)\|^2\\
            &- \left( \frac{\mu_G}{4} \phi_v^k \rho_k - C_{4,\Psi} \gamma_k  - 2 C_{C,1}L_{1,G}^2 S_k\gamma_k^2- 2C_{L,k} \frac{h_k^2}{\hat{h}_k^2}\right) \delta_{v_k}\\
            &- \left( \left(\frac{\mu_G}{4} \phi_z^k  - 4\frac{\omega_{v,1}^2}{\mu_G} \phi_v^k \right)  \rho_k  -4 C_{1,\Psi}^2 \gamma_k \right)  \delta_{z_k}\\
            &+\hat{C}_k.\\ %
        \end{aligned}
    \end{equation*}
    Let
    \begin{equation*}
        \hat{S} := \frac{ L_{\Psi} + 4 (\bar{\phi}_v +\bar{\phi}_z)L_*^2}{2}.
    \end{equation*}
    Choosing $\phi_z^k =  \bar{\phi}_z/\omega_k^z$ and $\phi_v^k = \bar{\phi}_v/\omega_k^v$ .  we have
    \begin{equation*}
        \begin{aligned}
            \mathbb{E}_k[L_{k + 1}] - L_k  &\leq - \frac{\gamma_k}{2} \|\nabla \Psi(x_k)\|^2\\
            &- \left( \frac{\mu_G}{4} \frac{\bar{\phi}_v}{\omega_k^v}  \rho_k - C_{4,\Psi} \gamma_k  - 2 C_{C,1}L_{1,G}^2 \hat{S}\gamma_k^2 - 2C_{L,k} \frac{h_k^2}{\hat{h}_k^2}\right) \delta_{v_k}\\
            &- \left( \left(\frac{\mu_G}{4} \frac{\bar{\phi}_z}{\omega_k^z}  - 4\frac{\omega_{v,1}^2}{\mu_G} \frac{\bar{\phi}_v}{\omega_k^v} \right)  \rho_k  -4 C_{1,\Psi}^2 \gamma_k \right)  \delta_{z_k}\\
            &+\hat{C}_k.\\ 
        \end{aligned}
    \end{equation*}
    Let
    \begin{equation*}
        \begin{aligned}
            \gamma_\text{max}^{(1)} &=  \frac{-C_{4,\Psi} + \sqrt{C_{4,\Psi}^2 + C_{C,1}L_{1,G}^2 \hat{S} \frac{\mu_G^2}{\mu_G + 2} \bar{\phi}_v \rho_k^2}}{4 C_{C,1}L_{1,G}^2 \hat{S}},\\
            \gamma_\text{max}^{(2)} &=\frac{\mu_G^2}{32 C_{1,\Psi}^2(\mu_G + 4)}  \bar{\phi}_z\rho_k^2.\\
        \end{aligned}
    \end{equation*}
    Notice that since $\rho_k \leq \bar{\rho} <  \min(1, \frac1{16 C_{B,1}L_{1,G}^2}, \frac{1}{4 C_{A,1}L_{1,G}^2})$,  and $\gamma_k \leq \bar{\gamma} < \min(\gamma_\text{max}^{(1)}, \gamma_\text{max}^{(2)})$, we have that for every $k \in \mathbb{N}$,
    \begin{equation*}
        \begin{aligned}
           \frac{\mu_G \rho_k}{4} > \frac{1}{\omega_k^z} = \frac{\mu_G \rho_k}{\mu_G \rho_k + 4} > \frac{\mu_G \rho_k}{\mu_G + 4}\\
           \frac{\mu_G\rho_k}{2} > \frac{1}{\omega_k^v} = \frac{\mu_G \rho_k}{\mu_G \rho_k + 2} > \frac{\mu_G \rho_k}{\mu_G + 2},
        \end{aligned}
    \end{equation*}
    and
    \begin{equation*}
        \begin{aligned}
        C_{L,k} &= \bigg(\frac{12L_{1,G}^2 (p + 6)^3}{b_1 \ell_1} \phi_v^k \rho_k^2 +\frac{6L_{1,G}^2 (d + 6)^3}{b_1 \ell_1}S_k\gamma_k^2\bigg)\\
        &\leq\bigg(\frac{6L_{1,G}^2 (p + 6)^3}{b_1 \ell_1} \bar{\phi}_v \mu_G \bar{\rho}^2 +\frac{6L_{1,G}^2 (d + 6)^3}{b_1 \ell_1}\hat{S}\bar{\gamma}^2\bigg) =: \bar{C}_L
        \end{aligned}
    \end{equation*}
    Therefore, we have
    \begin{equation*}
        \begin{aligned}
            \mathbb{E}_k[L_{k + 1}] - L_k  &\leq - \frac{\gamma_k}{2} \|\nabla \Psi(x_k)\|^2\\
            &- \left( \frac{\mu_G^2}{4(\mu_G + 2)}  \bar{\phi}_v\rho^2_k - C_{4,\Psi} \gamma_k  - 2 C_{C,1}L_{1,G}^2 \hat{S}\gamma_k^2 - 2\bar{C}_{L} \frac{h_k^2}{\hat{h}_k^2}\right) \delta_{v_k}\\
            &- \left( \left(\frac{\mu_G^2}{4(\mu_G + 4)}  \bar{\phi}_z  - 2\omega_{v,1}^2\bar{\phi}_v \right)  \rho_k^2  -4 C_{1,\Psi}^2 \gamma_k \right)  \delta_{z_k}\\
            &+\hat{C}_k.\\ 
        \end{aligned}
    \end{equation*} 
    Since $h_k \leq \frac{\mu_G}{4 \sqrt{(\mu_G + 2) \bar{C}_L}} \sqrt{\bar{\phi}_v}\hat{h}_k$, we get
    \begin{equation*}
        \begin{aligned}
            \mathbb{E}_k[L_{k + 1}] - L_k  &\leq - \frac{\gamma_k}{2} \|\nabla \Psi(x_k)\|^2\\
            &- \left( \frac{\mu_G^2}{8(\mu_G + 2)}  \bar{\phi}_v\rho^2_k - C_{4,\Psi} \gamma_k  - 2 C_{C,1}L_{1,G}^2 \hat{S}\gamma_k^2\right) \delta_{v_k}\\
            &- \left( \left(\frac{\mu_G^2}{4(\mu_G + 4)}  \bar{\phi}_z  - 2\omega_{v,1}^2\bar{\phi}_v \right)  \rho_k^2  -4 C_{1,\Psi}^2 \gamma_k \right)  \delta_{z_k}\\
            &+\hat{C}_k.\\ 
        \end{aligned}
    \end{equation*} 
    Notice that for 
    \begin{equation*}
        \begin{aligned}
            \frac{\mu_G^2}{8 \omega_{v,1}^2(\mu_G + 4)}  \bar{\phi}_z > \bar{\phi}_v,
        \end{aligned}
    \end{equation*}
    we have that
    \begin{equation*}
        \left(\frac{\mu_G^2}{4(\mu_G + 4)}  \bar{\phi}_z  - 2\omega_{v,1}^2\bar{\phi}_v \right)  > 0.
    \end{equation*}
    Thus, due to the choice $\bar{\phi}_v = \frac{\mu_G^2}{16 \omega_{v,1}^2(\mu_G + 4)}  \bar{\phi}_z$, we have
    \begin{equation*}
        \begin{aligned}
            \mathbb{E}_k[L_{k + 1}] - L_k  &\leq - \frac{\gamma_k}{2} \|\nabla \Psi(x_k)\|^2\\
            &- \left( \frac{\mu_G^2}{8(\mu_G + 2)}  \bar{\phi}_v\rho^2_k - C_{4,\Psi} \gamma_k  - 2 C_{C,1}L_{1,G}^2 \hat{S}\gamma_k^2\right) \delta_{v_k}\\
            &- \left( \frac{\mu_G^2}{8(\mu_G + 4)}  \bar{\phi}_z   \rho_k^2  -4 C_{1,\Psi}^2 \gamma_k \right)  \delta_{z_k}\\
            &+\hat{C}_k.\\ 
        \end{aligned}
    \end{equation*} 
    For $\gamma_k <  \gamma_\text{max}^{(2)}$, we have
    \begin{equation*}
        \begin{aligned}
        - \left( \frac{\mu_G^2}{8(\mu_G + 4)}  \bar{\phi}_z   \rho_k^2  -4 C_{1,\Psi}^2 \gamma_k \right) < 0.
        \end{aligned}
    \end{equation*}
    For $\gamma_k < \gamma_\text{max}^{(1)}$, we have
    \begin{equation*}
        - \left( \frac{\mu_G^2}{8(\mu_G + 2)}  \bar{\phi}_v\rho^2_k - C_{4,\Psi} \gamma_k  - 2 C_{C,1}L_{1,G}^2 \hat{S}\gamma_k^2\right) < 0.
    \end{equation*}
    Thus, for $\gamma_k < \min(\gamma_\text{max}^{(1)}, \gamma_\text{max}^{(2)})$ and $\rho_k < \min(1, \frac1{16 C_{B,1}L_{1,G}^2}, \frac{1}{4 C_{A,1}L_{1,G}^2})$, we have
        \begin{equation*}
        \begin{aligned}
            \mathbb{E}_k[L_{k + 1}] - L_k  &\leq - \frac{\gamma_k}{2} \|\nabla \Psi(x_k)\|^2 +\hat{C}_k.
        \end{aligned}
    \end{equation*} 
    Notice that, denoting $\bar{\rho}$ such that $\rho_k \leq \bar{\rho} <\min(1, \frac1{16 C_{B,1}L_{1,G}^2}, \frac{1}{4 C_{A,1}L_{1,G}^2})$ and taking $\bar{\phi}_z = \mathcal{O}(\frac{1}{\bar{\rho}})$ then there exists $\bar{c}_\gamma >0$ s.t. $\gamma_k < \bar{c}_\gamma \rho_k$. Taking the full expectation and summing for $k = 0,\cdots, K$, we get
    \begin{equation*}
        \begin{aligned}
            \mathbb{E}[L_{K + 1}] - L_0  &\leq - \sum\limits_{k=0}^K\frac{\gamma_k}{2} \mathbb{E}\left[\|\nabla \Psi(x_k)\|^2 \right] + \sum\limits_{k=0}^K\hat{C}_k.
        \end{aligned}
    \end{equation*} 
    Rearranging the terms, and observing that $-\min \Psi \geq -\mathbb{E}[\Psi(x_{K + 1})]$, we get 
    \begin{equation*}
        \begin{aligned}
            \sum\limits_{k=0}^K\frac{\gamma_k}{2} \mathbb{E}\left[\|\nabla \Psi(x_k)\|^2 \right]  &\leq L_0- \mathbb{E}[L_{K + 1}] + \sum\limits_{k=0}^K\hat{C}_k\\
            &= \Psi(x_0) \underbrace{- \mathbb{E}[\Psi(x_{K+1})]}_{\leq - \min \Psi} + \phi_z^0 \|z_0 - z^*(x_0)\|^2 + \phi_v^0 \|v_0 - v^*(x_0)\|^2\\
            &\underbrace{- \phi_z^0 \mathbb{E}\left[\|z_{K+1} - z^*(x_{K+1})\|^2\right] - \phi_v^0 \mathbb{E}\left[\|v_{K+1} - v^*(x_{K+1})\|^2\right]}_{\leq 0}\\
            &+ \sum\limits_{k=0}^K \hat{C}_k.
        \end{aligned}
    \end{equation*} 
    Let 
    \begin{equation*}
        R_\text{init} = \Psi(x_0) - \min \Psi + \phi_z^0 \|z_0 - z^*(x_0)\|^2 + \phi_v^0 \|v_0 - v^*(x_0)\|^2.
    \end{equation*}
    Multiplying by $2$ and dividing by $\sum\limits_{k=0}^K \gamma_k$ in both sides, we get the claim
    \begin{equation*}
        \begin{aligned}
            \frac{1}{\sum\limits_{k=0}^K \gamma_k} \sum\limits_{k=0}^K \gamma_k \mathbb{E}\left[\|\nabla \Psi(x_k)\|^2 \right]  &\leq \frac{2}{\sum\limits_{k=0}^K \gamma_k} \left(R_\text{init}+ \sum\limits_{k=0}^K \hat{C}_k\right).
        \end{aligned}
    \end{equation*} 

    \subsection{Proof of Corollary \ref{cor:hfzoba_param_choices}}
    By Theorem \ref{thm:hfzoba_conv_rate}, we have
    \begin{equation*}
        \begin{aligned}
        \frac{1}{\sum\limits_{k=0}^K \gamma_k} \sum\limits_{k=0}^K\frac{\gamma_k}{2} \mathbb{E}\left[\|\nabla \Psi(x_k)\|^2 \right]  &\leq \frac{1}{\sum\limits_{k=0}^K \gamma_k} \left( R_\text{init}+ \sum\limits_{k=0}^K \hat{C}_k\right).            
        \end{aligned}
    \end{equation*}
    Recall that
    \begin{equation*}
        \begin{aligned}
            \hat{C}_k &= (C_1 h_k^2+  C_2 \hat{h}_k^2) \phi_v^k \rho_k + 2 C_{2,\Psi}^2 \gamma_kh_k^2 + C_{3,\Psi} \hat{h}_k^2\gamma_k\\
            &+2  C_{B,1}\frac{L_{1,G}^2 L_{0,F}^2}{\mu_G^2} \phi_v^k \rho_k^2+ C_3 \phi_z^k \rho_k h_{k}^2 + 2C_{A,2}\phi_z^k\rho_k^2 h_{k}^2+ 2C_{C,1}\frac{L_{1,G}^2 L_{0,F}^2}{\mu_G^2}S_k\gamma_k^2\\
            &+ 2\phi_z^k\rho_k^2C_{A,3} +  \frac{12L_{1,G}^2 }{b_1 \ell_1}  \phi_v^k \rho_k^2 +\frac{6L_{1,G}^2 }{b_1 \ell_1}S_k\gamma_k^2 + \theta C_{L,k} \frac{h_k^2}{\hat{h}_k^2}\\
            &+ ( 2 C_{C,2} L_{0,F}^2 +2 C_{C,3} h_{k}^2 + 2C_{C,4})S_k\gamma_k^2+2 \phi_v^k \rho_k^2 ( 2 C_{B,2} L_{0,F}^2 +2 C_{B,3} h_{k}^2 + 2C_{B,4}).
        \end{aligned}
    \end{equation*}
    where $\theta = 1 + \frac{2L_{0,F}^2}{\mu_G^2} $ and
    \begin{equation*}
        \begin{aligned}
         C_1 &= 2 \frac{L_{1,G}^2(d + 3)^2}{\mu_G} + 4\frac{\omega_{v,2} L_{0,F}^2}{\mu_G^3}, \quad C_2 =\frac{8}{\mu_G} \frac{L_{2,G}^2}{4} \frac{L_{0,F}^2}{\mu_G^2},\\
            C_3 &= \frac{L_{1,G}^2}{4 \mu_G}(d + 3)^{3}, \quad C_{A,1} = 2\left( 2 + \frac{p + 2}{b_1 \ell_1} \right), \quad C_{A,2} = \left( \frac{1}{2 b_1 \ell_1} + \left( \frac{3}{b_1} + 1 \right) (p + 3)^3 \right)L_{1,G}^2\\
            C_{A,3} &= 2 \left(3 + \frac{p + 2}{\ell_1} \right) \frac{\sigma_{1,G}^2}{b_1},\quad C_{B,1} = 4\left(1 + \frac{3(p + 2)}{b_1 \ell_1} \right), \quad C_{B,2} = 2 \left(1 + \frac{p + 2}{b_2 \ell_2} \right)\\
            C_{B,3} &= \left( \frac{(p + 6)^3}{2 b_2 \ell_2} + \left(\frac{3}{b_2} + 1 \right)(p + 3)^{3}  \right) L_{1,F}^2, \quad C_{B, 4} = 2\left(\frac{p + 2}{\ell_2} + 3 \right ) \frac{\sigma_{1,F}^2}{b_2}\\
            C_{C,1} &= 4\left(1 + \frac{3(d + 2)}{b_1 \ell_1} \right), \quad C_{C,2} = 2 \left(1 + \frac{d + 2}{b_2 \ell_2} \right), \quad C_{C,3} = \left( \frac{1}{2 b_2 \ell_2} + \left(\frac{3}{b_2} + 1 \right)(d + 3)^{3}  \right) L_{1,F}^2\\
            C_{C,4} &= 2\left(\frac{d + 2}{b_2 \ell_2} +\frac{3}{b_2} \right )\sigma_{1,F}^2, \quad %
            C_{L,k} = \bigg(\frac{12L_{1,G}^2 (p + 6)^3}{b_1 \ell_1} \phi_v^k \rho_k^2 +\frac{6L_{1,G}^2 (d + 6)^3}{b_1 \ell_1}S_k\gamma_k^2\bigg)
        \end{aligned}
    \end{equation*}
    With the choice  $\phi_z^k =  \bar{\phi}_z/\omega_k^z$ and $\phi_v^k = \bar{\phi}_v/\omega_k^v$, we have
    \begin{equation*}
        S_k = \frac{L_\Psi + 4(\bar{\phi}_v + \bar{\phi}_z)L_*^2}{2} = \hat{S}.
    \end{equation*}
    Moreover, due to the conditions on $\rho_k$, we have for every $k \in \mathbb{N}$
    \begin{equation*}
        \begin{aligned}
           \frac{\mu_G}{4} > \frac{\mu_G \rho_k}{4} > \frac{1}{\omega_k^z} = \frac{\mu_G \rho_k}{\mu_G \rho_k + 4} > \frac{\mu_G \rho_k}{\mu_G + 4}\\
           \frac{\mu_G}{2} >\frac{\mu_G\rho_k}{2} > \frac{1}{\omega_k^v} = \frac{\mu_G \rho_k}{\mu_G \rho_k + 2} > \frac{\mu_G \rho_k}{\mu_G + 2},
        \end{aligned}
    \end{equation*}
    and, thus,
    \begin{equation*}
        C_{L,k} \leq 6\bigg(\frac{L_{1,G}^2 (p + 6)^3 \mu_G}{b_1 \ell_1}  \bar{\phi}_v \rho_k^2 +\frac{L_{1,G}^2 (d + 6)^3}{b_1 \ell_1}\hat{S}\gamma_k^2\bigg)
    \end{equation*}
    Replacing $\gamma_k,\rho_k, h_k$ and $\hat{h}_k$ with the parameter choice (I), we have
    \begin{equation}\label{eqn:cor2_rate}
        \begin{aligned}
        \frac{1}{K + 1} \sum\limits_{k=0}^K \mathbb{E}\left[\|\nabla \Psi(x_k)\|^2 \right]  &\leq \frac{2}{\gamma (K + 1)} \left( R_\text{init}+ \sum\limits_{k=0}^K \hat{C}\right),
        \end{aligned}
    \end{equation}
    with
    \begin{equation*}
        \begin{aligned}
\hat C
&= (C_2 \bar{\phi}_v \rho^2 + C_{3,\Psi}\gamma + 2C_{C,3}\hat S \gamma^2) \hat h^2
+ (C_1 \bar{\phi}_v \rho^2 + C_3 \bar{\phi}_z \rho^2 + 2 C_{2,\Psi}^2 \gamma + 2 C_{A,2} \bar{\phi}_z \rho^3 + 2 C_{B,3} \bar{\phi}_v \rho^3) h^2\\
&\quad + 2C_{B,1} \frac{L_{1,G}^2 L_{0,F}^2}{\mu_G^2} \bar{\phi}_v \rho^3 + 2 C_{A,3} \bar{\phi}_z \rho^3 + \frac{12 L_{1,G}^2 }{b_1 \ell_1} \bar{\phi}_v \rho^3\\
&+ 6\theta \bigg(\frac{L_{1,G}^2 (p + 6)^3 \mu_G}{b_1 \ell_1}  \bar{\phi}_v \rho^2 +\frac{L_{1,G}^2 (d + 6)^3}{b_1 \ell_1}\hat{S}\gamma^2\bigg) \frac{h^2}{\hat{h}^2} \\
&\quad +2C_{C,1} \frac{L_{1,G}^2 L_{0,F}^2}{\mu_G^2} \hat S \gamma^2 + \frac{6 L_{1,G}^2 }{b_1 \ell_1} \hat S \gamma^2 + (2 C_{C,2} L_{0,F}^2 + 2 C_{C,4}) \hat S \gamma^2  + 2 \bar{\phi}_v \rho^3 (2 C_{B,2} L_{0,F}^2 + 2 C_{B,4}).
\end{aligned}
    \end{equation*}
    Rearranging the terms, we get the first claim. The proof of the point (II) follows the same line of point (II) of proof of Corollary \ref{cor:param_choices}. Now, we prove the point (III). Recalling that $\bar{\phi}_v = \frac{\mu_G^2}{16 \omega_{v,1}^2(\mu_G + 4)}  \bar{\phi}_z$, we have
    \begin{equation*}
        \begin{aligned}
            \hat C &= \underbrace{(C_2 \bar{\phi}_v \rho^2 + C_{3,\Psi}\gamma + 2C_{C,3}\hat S \gamma^2)}_{C_1(\gamma, \rho)} \hat h^2
+ \underbrace{(C_1 \bar{\phi}_v \rho^2 + C_3 \bar{\phi}_z \rho^2 + 2 C_{2,\Psi}^2 \gamma + 2 C_{A,2} \bar{\phi}_z \rho^3 + 2 C_{B,3} \bar{\phi}_v \rho^3)}_{C_2(\gamma, \rho)} h^2\\
&\quad+ \underbrace{6\theta \bigg(\frac{L_{1,G}^2 (p + 6)^3 \mu_G}{b_1 \ell_1}  \bar{\phi}_v \rho^2 +\frac{L_{1,G}^2 (d + 6)^3}{b_1 \ell_1}\hat{S}\gamma^2\bigg)}_{C_3(\gamma, \rho)} \frac{h^2}{\hat{h}^2}\\
&\quad + \underbrace{\left(2C_{C,1} \frac{L_{1,G}^2 L_{0,F}^2}{\mu_G^2}  + \frac{6 L_{1,G}^2}{b_1 \ell_1}  + (2 C_{C,2} L_{0,F}^2 + 2 C_{C,4}) \right)\hat{S}}_{\tilde{C}_1}  \gamma^2  \\
&\quad \underbrace{\left( C_{B,1} \frac{L_{1,G}^2 L_{0,F}^2}{8 \omega_{v,1}^2(\mu_G + 4)} + 2 C_{A,3} + \frac{12 L_{1,G}^2  \mu_G^2}{16 \omega_{v,1}^2(\mu_G + 4)b_1 \ell_1} + \frac{\mu_G^2 (2 C_{B,2} L_{0,F}^2 + 2 C_{B,4})}{8 \omega_{v,1}^2(\mu_G + 4)}    \right)}_{\tilde{C}_2}\bar{\phi}_z \rho^3.
        \end{aligned}
    \end{equation*}
    Therefore, plugging $\hat{C}$ in eq. \eqref{eqn:cor2_rate}, we get
    \begin{equation*}%
        \begin{aligned}
        \frac{1}{K + 1} \sum\limits_{k=0}^K \mathbb{E}\left[\|\nabla \Psi(x_k)\|^2 \right]  &\leq \frac{2R_\text{init}}{\gamma (K + 1)} + \tilde{C}_1\gamma + \tilde{C}_2 \bar{\phi}_z \frac{\rho^3}{\gamma} +\frac{C_1(\gamma,\rho)}{\gamma} \hat{h}^2 + \frac{C_2(\gamma,\rho)}{\gamma} h^2 + \frac{C_3(\gamma, \rho)}{\gamma} \frac{h^2}{\hat{h}^2}.%
        \end{aligned}
    \end{equation*}
    By the choice $h = \hat{h}^2$ and since $\hat{h} < 1$, we have
    \begin{equation*}%
        \begin{aligned}
        \frac{1}{K + 1} \sum\limits_{k=0}^K \mathbb{E}\left[\|\nabla \Psi(x_k)\|^2 \right]  &\leq \frac{2R_\text{init}}{\gamma (K + 1)} + \tilde{C}_1\gamma + \tilde{C}_2 \bar{\phi}_z \frac{\rho^3}{\gamma} + \frac{C_1(\gamma,\rho)}{\gamma} \hat{h}^2 + \frac{C_2(\gamma,\rho)}{\gamma} \hat{h}^4 + \frac{C_3(\gamma, \rho)}{\gamma} \hat{h}^2\\
        &<        \frac{2R_\text{init}}{\gamma (K + 1)} + \tilde{C}_1\gamma + \tilde{C}_2 \bar{\phi}_z \frac{\rho^3}{\gamma} + \frac{\left(C_1(\gamma,\rho)  + C_2(\gamma,\rho)  + C_3(\gamma, \rho) \right)}{\gamma} \hat{h}^2.%
        \end{aligned}
    \end{equation*}
    Since $\rho < \bar{\rho} < \min(1, \frac{1}{16 C_{B,1} L_{1,G}^2}, \frac{1}{16 C_{A,1} L_{1,G}^2})$, and by the choice  $\gamma = \hat{c}_\gamma \rho$, we have
    \begin{equation*}%
        \begin{aligned}
        \frac{1}{K + 1} \sum\limits_{k=0}^K \mathbb{E}\left[\|\nabla \Psi(x_k)\|^2 \right]  
        &<        \frac{2R_\text{init}}{\hat{c}_\gamma \rho (K + 1)} + \left(\tilde{C}_1\hat{c}_\gamma  + \frac{\tilde{C}_2}{\hat{c}_\gamma} \right)\rho + \frac{\left(C_1(\gamma,\rho)  + C_2(\gamma,\rho)  + C_3(\gamma, \rho) \right)}{\hat{c}_\gamma\rho} \hat{h}^2.%
        \end{aligned}
    \end{equation*}
    Moreover, we have
    \begin{equation*}
        \begin{aligned}
            C_1(\gamma,\rho)  &< \underbrace{\left(C_2 \frac{\mu_G^2}{16 \omega_{v,1}^2 (\mu_G + 4)}  + C_{3,\Psi} \hat{c}_\gamma + 2 C_{C,3} \hat{S} \hat{c}_\gamma^2 \bar{\rho} \right) }_{\bar{\Delta}_1}\rho\\    
            C_2(\gamma, \rho) &< \underbrace{\left( C_1 \frac{\mu_G^2}{16 \omega_{v,1}^2 (\mu_G + 4)} + C_3  + 2 C_{2,\Psi}^2 \hat{c}_\gamma  + 2 C_{A,2} \bar{\rho} + \frac{C_{B,3}\mu_G^2}{8 \omega_{v,1}^2(\mu_G + 4) } \bar{\rho} \right)}_{\bar{\Delta}_2}\rho\\
            C_3(\gamma, \rho) &< \underbrace{6\theta \left(\frac{L_{1,G}^2 (p + 6)^3 \mu_G^3}{ 16 \omega_{v,1}^2 (\mu_G + 4)b_1 \ell_1} + \frac{L_{1,G}^2 (d + 3)^3}{b_1 \ell_1} \hat{S} \hat{c}_\gamma \bar{\rho}  \right) }_{\bar{\Delta}_3}\rho
        \end{aligned}
    \end{equation*}
    Thus, we have
    \begin{equation*}%
        \begin{aligned}
        \frac{1}{K + 1} \sum\limits_{k=0}^K \mathbb{E}\left[\|\nabla \Psi(x_k)\|^2 \right]  
        &<        \frac{2R_\text{init}}{\hat{c}_\gamma \rho (K + 1)} + \left(\tilde{C}_1\hat{c}_\gamma  + \frac{\tilde{C}_2}{\hat{c}_\gamma} \right)\rho + \frac{\left( \bar{\Delta}_1 + \bar{\Delta}_2 + \bar{\Delta}_3 \right)}{\hat{c}_\gamma} \hat{h}^2 .%
        \end{aligned}
    \end{equation*}    
    Let $\varepsilon \in (0,1)$. By the choice $\hat{h} \leq \min(1,  \sqrt{\frac{ \hat{c}_\gamma\varepsilon}{2\left( \bar{\Delta}_1 + \bar{\Delta}_2 + \bar{\Delta}_3 \right)}})$, we get
    \begin{equation*}%
        \begin{aligned}
        \frac{1}{K + 1} \sum\limits_{k=0}^K \mathbb{E}\left[\|\nabla \Psi(x_k)\|^2 \right]  
        &<        \frac{2R_\text{init}}{\hat{c}_\gamma \rho (K + 1)} + \left(\tilde{C}_1\hat{c}_\gamma  + \frac{\tilde{C}_2}{\hat{c}_\gamma} \right)\rho + \frac{\varepsilon}{2} .%
        \end{aligned}
    \end{equation*}    
    Consider the following inequality
    \begin{equation*}
        \begin{aligned}
            \frac{2R_\text{init}}{\hat{c}_\gamma \rho (K + 1)} + \left(\tilde{C}_1\hat{c}_\gamma  + \frac{\tilde{C}_2}{\hat{c}_\gamma} \right)\rho \leq \frac{\varepsilon}{2}.
        \end{aligned}
    \end{equation*}
    We choose $\rho$ by minimizing the left handside i.e.
    \begin{equation*}
        \rho = \sqrt{\frac{2R_\text{init}}{\bar{c}_\gamma \bar{\Delta}_4 (K + 1)}}
    \end{equation*}
    Thus, we get
    \begin{equation*}
        \begin{aligned}
            2 \sqrt{\frac{2R_\text{init} \bar{\Delta}_4}{\hat{c}_\gamma (K+1)}}< \frac{\varepsilon}{2}.
        \end{aligned}
    \end{equation*}
    Hence, for 
    \begin{equation*}
        K + 1 > \frac{32 R_\text{init} \bar{\Delta}_4}{\hat{c}_\gamma} \varepsilon^{-2},
    \end{equation*}
    we have $\frac{1}{K + 1} \sum\limits_{k=0}^k \mathbb{E}\left[ \| \nabla \Psi(x_k) \|^2 \right] < \varepsilon$.
    Therefore, choosing  $b_1 \ell_1 = \mathcal{O}(\lfloor \frac{d +p }{c^\prime_1} \rfloor)$ and $b_2 \ell_2 = \mathcal{O}(\lfloor \frac{d + p}{c_2^\prime} \rfloor)$ with $c_1^\prime, c_2^\prime \in \mathbb{R}_+$, %
    we get that the complexity is 
    \begin{equation*}
        \mathcal{O}\left( (d + p)\varepsilon^{-2} \right).
    \end{equation*}

\end{document}